\documentclass[12pt,a4paper,twoside,reqno]{amsart}

\usepackage{amsmath, amssymb, amsthm, enumerate, enumitem}
\usepackage[normalem]{ulem}

\makeatletter
\def\l@section{\@tocline{1}{10pt}{1pc}{}{}}
\def\l@subsection{\@tocline{2}{0pt}{1pc}{4.6em}{}}
\def\l@subsubsection{\@tocline{3}{0pt}{1pc}{7.6em}{}}
\renewcommand{\tocsection}[3]{%
  \indentlabel{\@ifnotempty{#2}{\makebox[2.3em][l]{%
    \ignorespaces#1 #2.\hfill}}}\textbf{#3}}
\renewcommand{\tocsubsection}[3]{%
  \indentlabel{\@ifnotempty{#2}{\hspace*{2.3em}\makebox[2.3em][l]{%
    \ignorespaces#1 #2.\hfill}}}#3}
\renewcommand{\tocsubsubsection}[3]{%
  \indentlabel{\@ifnotempty{#2}{\hspace*{4.6em}\makebox[3em][l]{%
    \ignorespaces#1 #2.\hfill}}}#3}
\makeatother

\newcommand{\MM}{\mathcal{M}}
\newcommand{\CC}{\mathcal{C}}
\newcommand{\IR}{\mathbb{R}}
\newcommand{\IB}{\mathbb{B}}
\newcommand{\IN}{\mathbb{N}}
\newcommand{\IC}{\mathbb{C}}
\newcommand{\IH}{\mathbb{H}}

\newcommand{\IZ}{\mathbb{Z}}

\newcommand{\LL}{\mathcal{L}}
\newcommand{\TT}{\mathcal{T}}
\renewcommand{\SS}{\mathcal{S}}

\newcommand{\ov}[1]{\overline{#1}}
\newcommand{\un}[1]{\underline{#1}}
\newcommand{\td}[1]{\widetilde{#1}}
\DeclareMathOperator{\Klein}{Klein}
\DeclareMathOperator{\thick}{thick}
\DeclareMathOperator{\thin}{thin}
\DeclareMathOperator{\eucl}{eucl}
\DeclareMathOperator{\stan}{stan}
\DeclareMathOperator{\scal}{scal}
\DeclareMathOperator{\Int}{Int}

\DeclareMathOperator{\hyp}{hyp}
\DeclareMathOperator{\Ric}{Ric}
\DeclareMathOperator{\curv}{curv}
\DeclareMathOperator{\area}{area}

\DeclareMathOperator{\tr}{tr}

\DeclareMathOperator{\dist}{dist}

\DeclareMathOperator{\diam}{diam}

\DeclareMathOperator{\vol}{vol}
\DeclareMathOperator{\const}{const}

\DeclareMathOperator{\Rm}{Rm}
\newcommand{\cangle}{\widetilde{\sphericalangle}}

\newcommand{\EMPTY}[1]{}

\newtheorem{Theorem}{Theorem}[section]
\newtheorem{Lemma}[Theorem]{Lemma}
\newtheorem{Corollary}[Theorem]{Corollary}
\newtheorem{Proposition}[Theorem]{Proposition}
\newtheorem{Definition}[Theorem]{Definition}
\newtheorem{Example}[Theorem]{Example}
\numberwithin{equation}{section}
\newtheorem*{Claim}{Claim}
\newtheorem*{Claim0}{Claim 0}
\newtheorem*{Claim1}{Claim 1}
\newtheorem*{Claim2}{Claim 2}
\newtheorem*{Claim3}{Claim 3}
\newtheorem*{Claim4}{Claim 4}
\newtheorem*{Claim5}{Claim 5}

\title{Long-time analysis of 3 dimensional Ricci flow II}
\author{Richard H Bamler}
\address{Stanford University, Department of Mathematics, 450 Serra Mall, building 380, Stanford, California 94305}
\email{rbamler@stanford.edu}
\date{\today}

\begin{document}
\begin{abstract}
This is the second part of a series of papers analyzing the long-time behaviour of $3$ dimensional Ricci flows with surgery.
We generalize the methods developed in the first part and use them to treat cases in which the initial manifold satisfies a certain purely topological condition which is far more general than the one that we previously had to impose.
Amongst others, we are able to treat initial topologies such as the $3$-torus or $\Sigma \times S^1$ where $\Sigma$ is any surface of genus $\geq 1$.
We prove that under this condition, only finitely many surgeries occur and that after some time the curvature is bounded by $C t^{-1}$.
This partially answers an open question in Perelman's work, which was made more precise by Lott and Tian.
In the process of the proof, we also find an interesting description of the geometry at large times, which even holds when the condition on the initial topology is violated.

The methods presented in this paper will be refined to treat a more general case in a subsequent paper.
\end{abstract}

\maketitle
\tableofcontents

\section{Introduction} \label{sec:Introduction}
\subsection{Statement of the main result}
In this paper we analyze the long-time behaviour of Ricci flows with surgery on $3$ dimensional manifolds which satisfy a certain topological property $\TT_2$ (see Definition \ref{Def:TT2} and subsection \ref{subsec:RemarksTT2} below for more details).
Examples for such manifolds are $\Sigma \times S^1$, where $\Sigma_g$ is a compact surface of genus $g \geq 1$, the $3$-torus $T^3$ or certain glueings of Seifert manifolds.
In rough terms, our main result can be summarized as follows.
We refer to Theorem \ref{Thm:MainTheorem} at the end of this introduction for a more precise statement.

\begin{quote}
\textit{Let $(M,g)$ be a closed $3$ dimensional Riemannian manifold and assume that the topological manifold $M$ satisfies property $\TT_2$. \\
Then there is a long-time existent Ricci flow with only  \emph{finitely} many surgeries whose initial time-slice is $(M, g)$.
Moreover, there is a constant $C$ such that the Riemannian curvature in this flow is bounded everywhere by $C t^{-1}$ for large times $t$.}
\end{quote}

The Ricci flow with surgery has been used by Perelman to solve the Poincar\'e and Geometrization Conjecture (\cite{PerelmanI}, \cite{PerelmanII}, \cite{PerelmanIII}).
Given any initial metric on a closed $3$-manifold, Perelman managed to construct a solution to the Ricci flow with surgery on a maximal time-interval and showed that its surgery times do not accumulate.
Hence every finite time-interval contains only a finite number of surgery times.
Furthermore, he could prove that if the given manifold is a homotopy sphere (or more generally a connected sum of prime, non-aspherical manifolds), then this flow goes extinct in finite time.
This implies that the initial manifold is a sphere if it is simply connected and hence establishes the Poincar\'e Conjecture.
On the other hand, if the Ricci flow continues to exist for infinite time, Perelman could show that the manifold decomposes into a thick part which approaches a hyperbolic metric and an thin part which becomes arbitrarily collapsed on local scales.
Based on this collapse, it is then possible to show that the thin part can be decomposed into more  concrete pieces (\cite{ShioyaYamaguchi}, \cite{MorganTian}, \cite{KLcollapse}).
This decomposition can be reorganized to a geometric decomposition, establishing the Geometrization Conjecture.

Observe that although the Ricci flow with surgery was used to solve such hard problems, some of its basic properties are still unknown, because they surprisingly turned out to be irrelevant in the end.
For example, after Perelman's work the question remained whether in the long-time existent case there are finitely many surgery times, i.e. whether after some time the flow can be continued by a conventional smooth, non-singular Ricci flow defined up to time infinity.
Furthermore, it is still unknown whether and in what way the Ricci flow exhibits the the full geometric decomposition of the manifold.
These questions follow naturally from Perelman's work and are partially explicitly raised there.
It has been conjectured Tian and Lott that they can be answered positively.

In \cite{LottTypeIII}, \cite{LottDimRed} and \cite{LottSesum}, Lott and Lott-Sesum could give a description of the long-time behaviour of certain Ricci flows on manifolds which consist of a single component in their geometric decomposition.
However, they needed to make additional curvature and diameter or symmetry assumptions.
In \cite{Bamler-longtime-I}, the author proved that under a purely topological condition $\TT_1$, which roughly says that the manifold only consists of  hyperbolic components (see Definition \ref{Def:TT1}), there are only finitely many surgeries and the curvature is bounded by $C t^{-1}$ after some time.

In this paper we derive the same conclusions under a far more general topological condition $\TT_2$.
Before we explain this condition and condition $\TT_1$ more precisely, we need to recall some facts on geometric decompositions of 3-manifolds.

\begin{Definition}[Geometric decomposition] \label{Def:geomdec}
Let $M$ be a compact, orientable 3-manifold whose boundary consists of $2$-tori.
A \emph{geometric decomposition} of $M$ is a collection of pairwise disjoint, smoothly embedded $2$-tori $T_1, \ldots, T_m \subset M$ such that
\begin{enumerate}[label=(\roman*)]
\item each torus $T_i$ is incompressible in $M$ (see Definition \ref{Def:incompressible}) and
\item each component of $M \setminus (T_1 \cup \ldots \cup T_m)$ is either \emph{hyperbolic} (i.e. it can be endowed with a complete metric of constant negative sectional curvature and finite volume) or it is \emph{Seifert} (i.e. it carries a Seifert fibration whose exceptional fibers are of cone-type and which can be extended regularly onto the boundary tori).
\end{enumerate}
The decomposition is called \emph{minimal} if no smaller subcollection of tori satisfies properties (i) and (ii).

If all components of $M \setminus (T_1 \cup \ldots \cup T_m)$ are Seifert, then the manifold is called \emph{(prime) graph manifold} and the decomposition is called a \emph{Seifert decomposition}.
\end{Definition}

Note that for a minimal geometric decomposition, the Seifert fibrations coming from either side on each torus $T_i$ are not isotopic and none of the components of $M \setminus (T_1 \cup \ldots \cup T_m)$ are diffeomorphic to $T^2 \times I$ unless $m=1$ and $T_1$ is non-separating.
The statement of the Geometrization Conjecture is that every closed, orientable, irreducible (see Definition \ref{Def:irreducible}) manifold is either a spherical space form or it admits a minimal geometric decomposition.
We also mention that such a minimal geometric decomposition is unique up to isotopy (see \cite[Theorem 1.9]{Hat}).
So it is reasonable to speak of \emph{the} (minimal) geometric decomposition of a manifold.

\begin{Definition}[Property $\TT_1$] \label{Def:TT1}
We say that an orientable, closed and irreducible $3$-manifold $M$ \emph{satisfies property $\TT'_1$} if it only has hyperbolic pieces in its geometric decomposition.

We say that an orientable, closed $3$-manifold $M$ \emph{satisfies property $\TT_1$}, if it is a connected sum of manifolds satisfying condition $\TT'_1$, spherical space forms and copies of $S^1 \times S^2$.
\end{Definition}

\begin{Definition}[Property $\TT_2$] \label{Def:TT2}
Let $M$ be an orientable, closed and irreducible $3$-mani\-fold which is not diffeomorphic to a spherical space form.
Consider a minimal geometric decomposition of $M$ and denote by $M_{\textnormal{hyp}}$ the union of the closures of all its hyperbolic components and by $M_{\textnormal{Seif}}$ the union of the closures of all its Seifert components.

We say that $M$ \emph{satisfies property $\TT'_2$} if there is a map $f : \Sigma \to M$ which is filling for the pair $(M_{\textnormal{hyp}}, M_{\textnormal{Seif}})$ in the sense of Definition \ref{Def:filling} below.

We say that an orientable, closed $3$-manifold $M$ \emph{satisfies property $\TT_2$}, if it is a connected sum of manifolds satisfying condition $\TT'_2$, spherical space forms and copies of $S^1 \times S^2$.
\end{Definition}

\begin{Definition}[filling surface] \label{Def:filling}
Let $M$ be an orientable, closed and irreducible $3$-manifold which is not diffeomorphic to a spherical space form.
Consider a decomposition $M = M_{\textnormal{hyp}} \cup M_{\textnormal{Seif}}$ such that $M_{\textnormal{hyp}} \cap M_{\textnormal{Seif}}$ is the disjoint union of smoothly embedded 2-tori.
Moreover, let $\Sigma$ be a compact, orientable surface (possibly with boundary and not necessarily connected) such that none of its components are spheres.

We say that a continuous map $f : \Sigma \to M$ is \emph{filling for the pair $(M_{\textnormal{hyp}}, M_{\textnormal{Seif}})$} (or sometimes, we call $f(\Sigma)$ a \emph{filling surface}), if the following holds:
\begin{enumerate}
\item $f$ is incompressible, i.e. the induced map $\pi_1(\Sigma) \to \pi_1(M)$ is injective,
\item $f$ maps each boundary loop of $\Sigma$ to an embedded non-contractible loop in one of the boundary tori of $M_{\textnormal{Seif}}$,
\item for every generic Seifert fiber $\gamma \subset M_{\textnormal{Seif}}$ of any (not necessarily minimal) Seifert decomposition of $M_{\textnormal{Seif}}$ and every map $f' : \Sigma \to M$ which is homotopic to $f$ relatively to its boundary, there is a component $\Sigma_0 \subset \Sigma$ such that $f'(\Sigma_0) \cap \gamma \neq \emptyset$ and such that there is no loop $\gamma' : S^1 \to \Sigma_0$ such that $f' \circ \gamma'$ is homotopic to a non-zero multiple of $\gamma$ in $M_{\textnormal{Seif}}$.
\end{enumerate}
\end{Definition}

The reason of why we have to impose condition $\TT_2$ is that we will need to ensure that certain $S^1$-fibers, along which the manifold collapses in certain areas, intersect an area-minimizing representative of the homotopy class of the map $f$.
Loosely speaking, this will give us an upper area bound for the ($2$ dimensional) basis of this $S^1$-fibration.

Condition $\TT_2$ is more general than it might appear at first glance.
For example $\TT_1$ implies $\TT_2$ and, as mentioned in the beginning, the three torus or every manifold of the form $\Sigma_g \times S^1$ for $g \geq 1$ satisfied condition $\TT_2$.
However, e.g. the Heisenberg manifold does not satisfy $\TT_2$.
We refer to subsection \ref{subsec:RemarksTT2} for a more detailed discussion of condition $\TT_2$ and far more general examples.

We now state our main result.
The notions relating to ``Ricci flows with surgery'', which are used in the following, will be introduced in subsection \ref{sec:DefRFsurg}.
\begin{Theorem} \label{Thm:MainTheorem}
Given a surgery model $(M_{\stan}, g_{\stan}, D_{\stan})$, there is a continuous function $\delta : [0, \infty) \to (0, \infty)$ such that the following holds: 

Let $\MM$ be a Ricci flow with surgery with normalized initial conditions which is performed by $\delta(t)$-precise cutoff such that $\MM(0)$ satisfies the topological condition $\TT_2$.

Then $\MM$ has only finitely many surgeries and there are constants $T, C < \infty$ such that $|{\Rm_t}| < C t^{-1}$ on $\MM(t)$ for all $t \geq T$.
\end{Theorem}

During the proof of this theorem, we will be able to give a more detailed description of the geometry of the time-slices $\MM(t)$ for large $t$ which also holds when the topological condition $\TT_2$ is not satisfied (see amongst others Proposition \ref{Prop:firstcurvboundstep2}).

We mention an important direct consequence of Theorem \ref{Thm:MainTheorem} which can be expressed in a more elementary way:
\begin{Corollary}
Let $(M, (g_t)_{t \in [0, \infty)})$ be a non-singular, long-time existent Ricci flow on a compact $3$-manifold $M$ which satisfies the topological condition $\TT_2$.
Then there is a constant $C < \infty$ such that
\[ |{\Rm_t}| < \frac{C}{t+1} \qquad \text{for all} \qquad t \geq 0. \]
\end{Corollary}

Moreover, we obtain the following result which ensures that the condition of the previous Corollary can be satisfied.
\begin{Corollary}
Let $M$ be a compact, orientable $3$-manifold which satisfies the topological condition $\TT'_2$.
Then there is a Riemannian metric $g_0$ on $M$ which is the initial metric of a non-singular, long-time existent Ricci flow $(M, (g_t)_{t \in [0, \infty)})$.
\end{Corollary}

In fact, starting from any given normalized (see Definition \ref{Def:normalized}) Riemannian metric $g$ on $M$, Perelman (\cite{PerelmanII}) could construct a long-time existent Ricci flow with surgery $\MM$ on the time-interval $[0, \infty)$ which is performed by $\delta(t)$-precise cutoff (see also Proposition \ref{Prop:RFwsurg-existence}).
Since $M$ is irreducible and aspherical, we conclude that all surgeries on $\MM$ are trivial and hence that every time-slice of $\MM$ has a component which is diffeomorphic to $M$.
By Theorem \ref{Thm:MainTheorem}, there is a final surgery time $T < \infty$ on $\MM$.
So the flow $\MM$ restricted to the time-interval $[T, \infty)$ is a non-singular Ricci flow on a manifold which is diffeomorphic to $M$.
Shifting this flow in time by $-T$ yields the desired Ricci flow.

This paper is organized as follows: In the next subsection \ref{subsec:RemarksTT2}, we discuss the property $\TT_2$ more carefully from a topological point of view.
Subsection \ref{subsec:Outline} contains an outline of the proof.
Section \ref{sec:IntroRFsurg} provides a brief introduction to Ricci flows with surgery.
It includes a proper definition which is general enough to unify most existing terminologies.
In section \ref{sec:3dtopology}, we recall elementary facts of $3$-dimensional topology which will be needed in the following proof.
In section \ref{sec:Perelman}, we review Perelman's long-time analysis of Ricci flows with surgery (cf \cite{PerelmanII}).
The aim of this section is on one hand to show that Perelman's analysis can be carried out using our notion of Ricci flows with surgery and on the other hand, to generalize Perelman's results to the non-compact and boundary cases.
In section \ref{sec:thinpart}, we describe the decomposition of the thin (i.e. locally collapsed) part of the manifold into more concrete pieces as carried out in \cite{ShioyaYamaguchi}, \cite{MorganTian} and \cite{KLcollapse}.
We then deduce important consequences on the geometry and combinatorics of this decomposition.
Section \ref{sec:maintools} contains new curvature estimates which follow from the generalizations of Perelman's results, as obtained in section \ref{sec:Perelman}.
Those estimates will finally be applied in the main part in section \ref{sec:mainargument}.
Section \ref{sec:Preparations} provides useful Lemmas for this discussion.

Note that in the following, all manifolds are assumed to be $3$ dimensional unless stated otherwise.

I would like to thank Gang Tian for his constant help and encouragement and John Lott for many long conversations.
I am also indebted to Bernhard Leeb and Hans-Joachim Hein, who contributed essentially to my understanding of Perelman's work.
Thanks also go to Simon Brendle, Will Cavendish, Daniel Faessler, Robert Kremser, Tobias Marxen, Rafe Mazzeo, Hyam Rubinstein, Richard Schoen, Stephan Stadler and Brian White.

\subsection{Remarks on property $\TT_2$} \label{subsec:RemarksTT2}
In the following, we present a brief discussion of property $\TT_2$.
The goal of this subsection is to provide an idea of how restrictive the topological condition in Theorem \ref{Thm:MainTheorem} is.
It is not known to the author whether there is a more handy condition which is equivalent to property $\TT_2$.
In particular, it is an interesting question whether the following characterization holds:
A compact, connected, orientable, irreducible $3$-manifold $M$ which is not a spherical space form satisfies property $\TT'_2$ if and only if $M$ either it is not graph (i.e. it contains at least one hyperbolic component in its geometric decomposition) or it is graph and contains an immersed, incompressible surface of genus $\geq 2$ or it is a quotient of a $3$-torus.
This characterization seems reasonable considering the following examples and the work of Wang and Yu (\cite{WangYu}) and Neumann (\cite{Neumann}).

We first present examples of manifolds which satisfy property $\TT_2$.
As mentioned before, property $\TT_1$ implies property $\TT_2$.

\begin{Example} \label{Ex:SigmaS1} \rm
Let $\Sigma_g$ be a surface of genus $g \geq 1$ and consider the manifold $M = \Sigma_g \times S^1$, e.g. the 3-torus $M \approx T^3$.
Then $M$ satisfies property $\TT'_2$.
In the case $M \approx T^3$ this is easy to see: choose $3$ embedded $2$-tori in $M$ which generate the second homology $H_2 (M) \cong \IZ^3$.
Then every non-contractible curve $\gamma \subset M$ represents a non-zero homology class in $H_1(M)$ and hence has non-zero intersection number with one of the given $2$-tori.

Now assume that $g \geq 2$.
Then we proceed as follows: Choose simple closed loops $\alpha_1, \ldots, \alpha_{2g} \subset \Sigma_g$ whose complement is an open topological disk.
We now argue that $\Sigma_g \times \{ \textnormal{pt} \} \cup \alpha_1 \times S^1 \cup \ldots \cup \alpha_{2g} \times S^1 \subset M$ is a filling surface for $M$.
Let $\gamma \subset M$ be a non-contractible loop.
If $[\gamma] \in \pi_1(M) \cong \pi_1(\Sigma) \times \IZ$ has a non-zero component in the second factor, then its intersection number with $\Sigma \times \{ \textnormal{pt} \}$ is non-zero.
So assume that it doesn't, i.e. that $[\gamma]$ is the image of an element $[\gamma']$ under the map $\pi_1(\Sigma_g) \to \pi_1(M)$ where $\gamma' \subset \Sigma_g$.
Consider the universal cover $\IH^2$ of $\Sigma_g$.
By elementary hyperbolic geometry, there is a lift $\td{\alpha}_i$ whose endpoints in the boundary at infinity $\partial \IH^2$ separate the endpoints of a lift $\td{\gamma}'$ of $\gamma'$.
So $\td{\alpha}_i$ and $\td{\gamma}'$ have non-zero intersection number.
This implies that the corresponding lift of $\alpha_i \times S^1$ in $\IH^2 \times \IR$ has non-zero intersection number with the corresponding lift $\td\gamma$ of $\gamma$.
This intersection number does not change if we homotope these lifts equivariantly.
So $\gamma$ intersects every homotopic translate of $\alpha_i \times S^1$.

We remark, that by inspecting the argument, we can show that every finite quotient of $\Sigma \times S^1$ under a group which preserves the $S^1$-fibration, satisfies property $\TT'_2$ as well.
\end{Example}

\begin{Example} \rm
Consider a geometric decomposition of an orientable, closed, irreducible 3-manifold $M$ which is not a spherical space form.
It can be shown by similar methods as in Example \ref{Ex:SigmaS1} that if $M$ has at least one hyperbolic component and if no two Seifert components are adjacent to one another, then $M$ satisfies property $\TT'_2$.
\end{Example}

Next, we provide an example of a graph manifold which satisfies property $\TT'_2$ and which has a non-trivial Seifert decomposition.
For simplicity, we will describe a very special construction.
It will become clear how this construction can be generalized.

\begin{Example} \rm
Let $M$ be a manifold whose geometric decomposition consists of four components $M_1, \ldots, M_4$ which are all diffeomorphic to $\Sigma \times S^1$ where $\Sigma$ denotes a surface of genus 1 with two boundary circles.
So each $M_i$ has exactly two boundary components and we assume the components to be arranged in a circle in the given order.
Let $\alpha \subset \Sigma$ be an embedded curve joining the two boundary circles and $\beta_1, \beta_2 \subset \Sigma$ embedded closed loops such that $\Sigma \setminus (\alpha \cup \beta_1 \cup \beta_2)$ is an open topological disk.
In each component $M_i$ define the annulus $F_i = \alpha \times S^1$, $\Sigma_i \subset M_i=\Sigma \times \{ \textnormal{pt} \}$ and $T_{i,1}, T_{i,2} \subset M_i$ to be the tori $\beta_1 \times S^1, \beta_2 \times S^1$.

Now assume that the components $M_1, \ldots, M_4$ are identified in such a way that $\Sigma'_1 = \Sigma_1 \cup F_2 \cup \Sigma_3 \cup F_4$ and $\Sigma'_2 = F_1 \cup \Sigma_2 \cup F_3 \cup \Sigma_4$ are embedded surfaces of genus 3.
We now show that $M$ satisfies property $\TT'_2$ and that $\Sigma'_1 \cup \Sigma'_2 \cup T_{1,1} \cup \ldots \cup T_{4,2}$ is a filling surface.

Consider an arbitrary geometric decomposition of $M$ and $\gamma \subset M$ a generic Seifert fiber of one of the components.
After isotoping and removing some of the cutting tori, we are able to obtain the decomposition $M = M_1 \cup \ldots \cup M_4$ and we can assume that $\gamma \subset M_i$.
Observe that $\pi_1(M_i) \cong \pi_1(\Sigma) \times \IZ$.
If $[\gamma]$ has a non-trivial component in the second factor, then it has non-zero intersection number with $\Sigma'_1$ or $\Sigma'_2$.
On the other hand, assume that $[\gamma]$ is the image of an element $[\gamma']$ under the map $\pi_1(\Sigma) \to \pi_1(M_i)$.
If $\gamma'$ is a parabolic curve then it has non-zero intersection number with either $\Sigma'_1$ or $\Sigma'_2$ or $\gamma'$.
If $\gamma'$ is not parabolic (i.e. if it can be represented by a closed geodesic when we choose a hyperbolic structure on $\Sigma$), then we can argue as in Example \ref{Ex:SigmaS1} that in some cover of $M$ a lift of $\gamma$ has non-zero intersection number with a lift of one of the surfaces $T_{i, 1}$ or $T_{i, 2}$.
\end{Example}

Finally, we discuss manifolds which are not of type $\TT_2$.
\begin{Example} \rm
Consider a closed, orientable manifold $M$ which admits a Seifert fibration over an orbifold $B$ whose underlying surface is not a sphere and which has only isolated cone-singularities.
Then the Seifert fibration $M \to B$ admits a finite-sheeted cover $\widehat{M} \to \widehat{B}$ which is a regular $S^1$-fibration.
We now claim that $M$ satisfies property $\TT'_2$ if and only if this fibration is trivial.
Hence, the Heisenberg manifold (or nilmanifold) is the most elementary example of a manifold which does not satisfy property $\TT'_2$.

The forward direction of the claim is clear by Example \ref{Ex:SigmaS1}.
For the other direction, assume that $f: \Sigma \to M$ is a filling map.
By a result of Hass (\cite{Hass}), the incompressibility of $f$ implies that we can homotope $f$ to an immersion such that the restriction of $f$ to every component $\Sigma'$ of $\Sigma$ is either horizontal (i.e. transverse to the Seifert fibration) or vertical (i.e. $f$ is tangent to the Seifert fibers).
If $f|_{\Sigma'}$ is horizontal, then the map $\Sigma' \to M \to B$ is a covering and hence the pull back of the Seifert fibration onto $\Sigma'$ is trivial.
However, this contradicts the assumption that the $\widehat{M} \to \widehat{B}$ is non-trivial.
So $f$ has to be vertical on all components of $\Sigma$ and hence $f(\Sigma)$ does not intersect every generic Seifert fiber.
\end{Example}

\begin{Example} \rm
Assume that $M$ is irreducible, graph and not a spherical space form.
The paper \cite{Neumann} provides obstructions for the existence of an incompressible maps $\Sigma \to M$ for which $\Sigma$ has negative Euler characteristic.
It is easy to see that if these obstructions apply, then $M$ cannot have any filling surface and hence does not satisfy property $\TT'_2$ unless it is a finite quotient of a $3$-torus.
We refer to \cite[Example 2.2]{Neumann} for an elementary example of a non-trivial graph manifold which does not satisfy property $\TT_2$.
\end{Example}

\subsection{Outline of the proof} \label{subsec:Outline}
The proof of Theorem \ref{Thm:MainTheorem} makes use of the work of Perelman (\cite{PerelmanI}, \cite{PerelmanII}, \cite{PerelmanIII}) and the subsequent analysis of the thin part (\cite{ShioyaYamaguchi}, \cite{MorganTian}, \cite{KLcollapse}).
Perelman could show that for large times $t$ the time-slice $\MM(t)$ of a given Ricci flow with surgery $\MM$ can be decomposed into a thick part $\MM_{\thick} (t)$ and a thin part $\MM_{\thin}(t)$.
The thick part is diffeomorphic to a disjoint union of hyperbolic manifolds and the metric $g(t)$ has sectional curvatures close to $- \frac1{4t}$ there.
On the thin part however, the curvature is a priori not bounded, but there is a positive function $w(t)$ which goes to zero as $t \to \infty$ such that following holds: at every point $x \in \MM_{\thin}(t)$ there is a scale $\rho(x, t)$ such that the sectional curvatures on the ball $B (x, t, \rho(x,t))$ are bounded from below by $-\rho^{-2} (x, t)$ and such that the volume of $B(x, t, \rho(x,t))$ is bounded from above by $w(t) \rho^3(x,t)$ (see Proposition \ref{Prop:thickthindec}).
In other words, the metric on the thin part locally collapses on the scale $\rho(x,t)$.
Morgan-Tian, Kleiner-Lott and others were able to understand this collapse and decompose the thin part into concrete pieces on which the metric can be approximated by Seifert fibrations with small $S^1$-fiber, or models such as $T^2 \times I$ or $S^2 \times I$ with small first factor.
Using elementary topological arguments, it is possible to construct a Seifert or geometric decomposition from this decomposition.

In the case in which the initial manifold is diffeomorphic to a hyperbolic manifold, it can be shown easily that for large times we have $\MM_{\thick}(t) = \MM(t)$.
This immediately implies Theorem \ref{Thm:MainTheorem}.

To understand the more general case, we need to recall an important Lemma (cf \cite[6.8, 7.3]{PerelmanII}, compare also with Corollary \ref{Cor:Perelman68}) which led to a curvature bound on $\MM_{\thick}(t)$ in Perelman's result.
It roughly states that if for some large $t$ and $x \in \MM(t)$ the constant $\rho > 0$ is chosen maximal with the property that the sectional curvatures on the ball $B(x,t, \rho)$ are bounded from below by $- \rho^{-2}$, and if the volume of $B(x, t, \rho)$ is bounded from below by $w \rho^3$, then $\rho > \ov\rho(w) \sqrt{t}$ and the curvature on $B(x, t, \rho)$ is bounded by $K (w) t^{-1}$.
We may assume in the following without loss of generality that the function $\rho(x,t)$ from before also satisfies this maximality property.
Then obviously, Perelman's Lemma fails to provide a good curvature bound on the thin part $\MM_{\thin} (t)$. 
However, if we pass to the universal cover of $\MM(t)$, then in certain cases the normalized volume of the corresponding $\rho(x,t)$-ball around a lift $\td{x}$ of $x$ is again controlled from below (see Lemma \ref{Lem:unwrapfibration}).
More precisely, this improvement occurs if the ball $B(x, t, \rho(x,t))$ collapses along \emph{incompressible} $S^1$ or $T^2$-fibers.
Loosely speaking, the reason for this is that regions which look like $S^1 \times B^2$ or $T^2 \times I$ with small first factor will lift to regions close to $\IR \times B^2$ or $\IR^2 \times I$.
We will call a point $x \in \MM(t)$ ``good'' if we can observe these types of collapses in a neighborhood of $x$ (see Definition \ref{Def:goodness}).
It is now possible to apply Perelman's Lemma to the universal covering flow of $\MM$ and obtain a curvature bound of the form $K t^{-1}$ at good points (see Proposition \ref{Prop:curvcontrolgood}).

This idea is the starting point of our proof.
After a topological discussion of the decomposition of $\MM_{\thin} (t)$, we find that $\MM(t)$ is good outside fintely many pairwise disjoint, embedded solid tori $S_1, \ldots, S_m \approx S^1 \times D^2$ in $\MM_{\thin} (t)$ (see Proposition \ref{Prop:GGpp}).
It hence remains to focus our analysis on these $S_i$.

Using a bounded curvature at bounded distance type estimate (which also makes use of the local collapse), we will be able to establish a distance dependent curvature bound from every good point (see Proposition \ref{Prop:curvcontrolincompressiblecollapse}).
More precisely, for any $A < \infty$, we can bound the curvature on an $A \sqrt{t}$-tubular neighborhood around $\MM (t) \setminus (S_1 \cup \ldots \cup S_m)$ by $K'(A) t^{-1}$.
Hence, it remains to analyze those $S_i$ whose diameter on the scale $\sqrt{t}$ is large.

It will turn out that we can choose each the solid torus $S_i$ in such a way that we can find a collar neighborhood $P_i \subset S_i$, $P_i \approx T^2 \times I$, next to its boundary whose length increases with its diameter (see Lemma \ref{Lem:firstcurvboundstep1} and observe that $P_i = \Int S_i(t) \setminus \Int W_i$ there) and whose cross-sectional tori are bounded in diameter by $\sqrt{t}$.
In order to establish this stronger characterization, we additionally need to make use of a boundary version of Perelman's Lemma (see Proposition \ref{Prop:curvboundinbetween}).
Moreover, we can show that the diameter of $S_i$ at slightly earlier times cannot be much smaller than the diameter at time $t$, i.e. that the diameter cannot grow too fast on a time-interval of uniform size (see Proposition \ref{Prop:slowdiamgrowth}).
These geometric observations are summarized in Lemma \ref{Lem:firstcurvboundstep1}.

Using this diameter bound at earlier times, we are able to find collar neighborhoods similar to $P_i$ around $\partial S_i$ also at earlier times.
This will imply that the normalized volumes of \emph{local} universal covers of $\rho(x,t)$-balls in bounded distance to $\partial S_i$ and at slightly earlier times are bounded from below.
Together with a more powerful localized version of Perelman's Lemma (see Proposition \ref{Prop:curvboundnotnullinarea}) we obtain a uniform curvature bound of the form $K t^{-1}$ on each $P_i$ (see Proposition \ref{Prop:firstcurvboundstep2}).
Here $K$ does not depend on the diameter of $S_i$ or $P_i$.

Now the topological property $\TT_2$ comes into play.
By a well-known estimate on the evolution of areas of minimal surfaces (see Lemma \ref{Lem:evolminsurfgeneral}), we can find immersed surfaces which intersect all $S^1$-fibers of every $S_i \approx S^1 \times D^2$.
We will use these surfaces to find arbitrarily thin an long torus structures $P'_i \subset P_i$, $P'_i \approx T^2 \times I$ (see Proposition \ref{Prop:firstcurvboundstep3}).

Finally, we consider the Ricci flow on a larger time-interval $[ t_0, t_\omega]$, $t_\omega < L t_0$, and we relate the solid tori $S_i$ and the torus structures $P'_i$, which we obtain from our previous analysis applied at each time of $[ t_0, t_\omega]$, towards one another.
Let $S_1, \ldots, S_m$ be the solid tori which arise from the analysis at time $t_\omega$.
We will find that if $\diam_{t_\omega} S_i > A_0(L) \sqrt{t_\omega}$ for some $i$, then the curvature on $P'_i$ is bounded by $K t^{-1}$ at all times $t \in [t_0, t_\omega]$ (see Proposition \ref{Prop:structontimeinterval}).
Using the immersed minimal surface of bounded area from before, we will construct an immersed disk whose area is bounded at time $t_0$ and which is bounded by a loop which is contained in $P'_i$ and which is short on the whole time-interval $[ t_0, t_\omega]$.
For sufficiently large $L$, we can derive a contradiction to the existence of such a disk using a minimal disk argument.
Hence $\diam_t S_i \leq A_0(L) \sqrt{t_\omega}$ for all $i$ and hence the curvature on $S_i$ at time $t_\omega$ is bounded by $K'(A_0(L)) t_\omega^{-1}$.
This finishes the proof of Theorem \ref{Thm:MainTheorem}.

Upon first reading we recommend to consider the case in which $\MM$ is non-singular.
The proof in the general case follows along the lines, but the existence of surgeries adds a number of technical difficulties.

\section{Introduction to Ricci flows with surgery} \label{sec:IntroRFsurg}
\subsection{Definition of Ricci flows with surgery} \label{sec:DefRFsurg}
In this section, we give a precise definition of the Ricci flows with surgery that we are going to analyze.
We will mainly use the language developed in \cite{Bamler-diploma} here.
In a first step, we define Ricci flows with surgery in a very broad sense.
After explaining some useful new notions, we will make more precise how we assume the surgeries to be performed.
This characterization can be found in Definition \ref{Def:precisecutoff}.
Note that here we have chosen a phrasing which unifies the constructions presented in \cite{PerelmanII}, \cite{KLnotes}, \cite{MTRicciflow}, \cite{BBBMP} and \cite{Bamler-diploma} and hence Theorem \ref{Thm:MainTheorem} can be applied to the outcomes of each of these representations.

\begin{Definition}[Ricci flow with surgery] \label{Def:RFsurg}
Consider a time-interval $I \subset \IR$.
Let $T^1 < T^2 < \ldots$ be times of the interior of $I$ which form a possibly infinite, but discrete subset of $\IR$ and divide $I$ into the intervals
\[ I^1 = I \cap (-\infty, T^1), \quad I^2 = [T^1, T^2), \quad I^3 = [T^2, T^3), \quad \ldots \]
and $I^{k+1} = I \cap [T^k,\infty)$ if there are only finitely many $T^i$'s and $T^k$ is the last such time and $I^1 = I$ if there are no such times.
Consider Ricci flows $(M^1 \times I^1, g^1_t), (M^2 \times I^2, g^2_t), \ldots$ on manifolds $M^1, M^2, \ldots$ and time-intervals $I^1, I^2, \ldots$.
Let $\Omega^i \subset M^i$ be open sets on which the metric $g^i_t$ converges smoothly as $t \nearrow T^i$ to some Riemannian metric $g^i_{T^i}$ on $\Omega_i$ and let 
\[ U^i_- \subset \Omega^i \qquad \text{and} \qquad U^i_+ \subset M^{i+1} \]
be open subsets such that there are isometries
\[ \Phi^i : (U^i_-, g_{T^i}^i) \longrightarrow (U^i_+, g_{T^i}^{i+1}), \qquad (\Phi^i)^* g_{T^i}^{i+1} |_{U_+^i} = g^i_{T^i} |_{U^i_-}. \]
We assume moreover that we never have $U_-^i = \Omega^i = M^i$ and $U_+^i = M^{i+1}$ and that every component of $M^{i+1}$ contains a point of $U^i_+$.
Then, we call $\MM = ((T^i)_i, (M^i \times I^i, g_t^i)_i, (\Omega^i)_i, (U^i_{\pm})_i, (\Phi^i)_i)$ a \emph{Ricci flow with surgery on the time-interval $I$} and the times $T^1, T^2, \ldots$ \emph{surgery times}.

If $t \in I^i$, then $(\MM(t), g(t)) = (M^i \times \{ t \}, g_t^i)$ is called the \emph{time-$t$ slice of $\MM$}.
The points in $\MM(T^i) \setminus U^i_+ \times \{T^i \}$ are called \emph{surgery points}.
For $t = T^i$, we define the \emph{(presurgery) time $T^{i-}$-slice} to be $(\MM(T^{i-}), g(T^{i-})) = (\Omega^i \times \{ T^i \}, g^i_{T^i})$.
The points $\Omega^i \times \{ T^i \} \setminus U^i_- \times \{ T^i \}$ are called \emph{presurgery points}.

If $\MM$ has no surgery points, then we call $\MM$ \emph{non-singular} and write $\MM = M \times I$.
\end{Definition}

We will often view $\MM$ in the \emph{space-time picture}, i.e. we imagine $\MM$ as a topological space $\bigcup_{t \in I} \MM (t) = \bigcup_i M^i \times I^i$ where the components in the latter union are glued together via the diffeomorphisms $\Phi^i$.

The following vocabulary will prove to be useful when dealing with Ricci flows with surgery:

\begin{Definition}[Ricci flow with surgery, space-time curve]
Consider a sub-interval $I' \subset I$.
A map $\gamma : I' \to \bigcup_{t \in I'} \MM(t)$ (also denoted by $\gamma : I' \to \MM$) is called a \emph{space-time curve} if $\gamma(t) \in \MM(t)$ for all $t \in I'$, if $\gamma$ restricted to each sub-time-interval $I^i$ is continuous and if $\lim_{t \nearrow T^i} \gamma(t) \in U^i_-$ and $\gamma(T^i) = \Phi^i(\lim_{t \nearrow T^i} \gamma(t))$ for all $i$.
\end{Definition}

So a space-time curve is a continuous curve in $\MM$ in the space-time picture.

\begin{Definition}[Ricci flow with surgery, points in time] \label{Def:pointsurvives}
For $(x,t) \in \MM$, consider a spatially constant space-time curve $\gamma$ in $\MM$ that starts in $(x,t)$ and goes forward or backward in time for some time $\Delta t \in \IR$ and that doesn't hit any (pre- or \hbox{post-)}surgery points except possibly at its endpoints.
Then we say that the point $(x,t)$ \emph{survives until time} $t + \Delta t$ and we denote the other endpoint by $(x, t + \Delta t)$.

Observe that this notion also makes sense, if $(x, t^-) \in \MM$ is a presurgery point and $\Delta t \leq 0$.
\end{Definition}

Note that the point $(x, t+\Delta t)$ is only defined if $(x, t)$ survives until time $t + \Delta t$ which also includes the fact that $\MM$ is defined at time $t + \Delta t$.
Using this definition, we can define parabolic neighborhoods in $\MM$.

\begin{Definition}[Ricci flow with surgery, parabolic neighborhoods] \label{Def:parabnbhd}
Let $(x,t) \in \MM$, $r \geq 0$ and $\Delta t \in \IR$.
Consider the ball $B = B(x,t,r) \subset \MM(t)$.
For each $(x',t) \in B$ consider the union $I^{\Delta t}_{x',t}$ of all points $(x',t+t') \in \MM$ which are well-defined in the sense of Definition \ref{Def:pointsurvives} for $t' \in [0, \Delta t]$ resp. $t' \in [\Delta t, 0]$.
Define the \emph{parabolic neighborhood} $P(x,t,r,\Delta t) = \bigcup_{x' \in B} I^{\Delta t}_{x',t}$.
We call $P(x,t,r,\Delta t)$ \emph{non-singular} if all points in $B(x,t,r)$ survive until time $t+ \Delta t$.
\end{Definition}

The following notion will be used in sections \ref{sec:maintools} and \ref{sec:mainargument}.
\begin{Definition}[sub-Ricci flow with surgery] \label{Def:subRF}
Consider a Ricci flow with surgery $\MM = ((T^i)_i, (M^i \times I^i, g_t^i)_i, (\Omega^i)_i, (U^i_{\pm})_i, (\Phi^i)_i)$ on the time-interval $I$.
Let $I' \subset I$ be a sub-interval and consider the indices $i$ for which the intervals ${I'}^i = I^i \cap I'$ are non-empty.
For each such $i$ consider a submanifold ${M'}^i \subset M^i$ of the same dimension and possibly with boundary.
Let ${g'_t}^i$ be the restriction of $g^i_t$ onto ${M'}^i \times {I'}^i$ and set ${\Omega'}^i = \Omega^i \cap {M'}^i$ and ${U'_-}^i = U^i_- \cap {M'}^i$ as well as ${U'_+}^i = U^i_+ \cap {M'}^{i+1}$.
Assume that for each $i$ for which ${I'}^i$ and ${I'}^{i+1}$ are non-empty, we have $\Phi^i ({U'_-}^i) = {U'_+}^i$ and let ${\Phi'}^i$ be the restriction of $\Phi^i$ to ${U'_-}^i$.

In the case in which ${U'_-}^i = {\Omega'}^i = {M'}^i$ and ${U'_+}^i = {M'}^{i+1}$ for some $i$, we can combine the Ricci flows ${g'_t}^i$ and ${g'_t}^{i+1}$ on ${M'}^i \times {I'}^i$ and ${M'}^{i+1} \times {I'}^{i+1}$ to a Ricci flow on the time-interval ${I'}^i \cup {I'}^{i+1}$ and hence remove $i$ from the list of indices.

Then $\MM' = (({T'}^i)_i, ({M'}^i \times {I'}^i, {g'_t}^i)_i, ({\Omega'}^i)_i, ({U'_{\pm}}^i)_i, ({\Phi'}^i)_i)$ is a Ricci flow with surgery in the sense of Definition \ref{Def:RFsurg}.

Assume that for all $t \in I'$ the boundary points $\partial \MM' (t) \subset \MM(t)$ (by this we mean all points in $\MM(t)$ which don't lie in the interior of $\MM'(t)$ or $\MM(t) \setminus \MM'(t))$ survive until any other time of $I'$ and that $\partial \MM' (t)$ is constant in $t$.
Then we call $\MM'$ a \emph{sub-Ricci flow with surgery} and we write $\MM' \subset \MM$.
\end{Definition}

We will now characterize three important approximate local geometries that we will frequently be dealing with: $\varepsilon$-necks, strong $\varepsilon$-necks and $(\varepsilon, E)$-caps.
The notions below also make sense for presurgery time-slices.

\begin{Definition}[Ricci flow with surgery, $\varepsilon$-necks] \label{Def:epsneck}
Let $\varepsilon > 0$ and consider a Riemannian manifold $(M,g)$.
We call an open subset $U \subset M$ an $\varepsilon$-neck, if there is a diffeomorphism $\Phi : S^2 \times (-\frac1{\varepsilon}, \frac1{\varepsilon}) \to U$ such that there is a $\lambda > 0$ with $\Vert \lambda^{-2} \Phi^* g(t) - g_{S^2 \times \IR} \Vert_{C^{[\varepsilon^{-1}]}} < \varepsilon$ where $g_{S^2 \times \IR}$ is the standard metric on $S^2 \times (-\frac1{\varepsilon}, \frac1{\varepsilon})$ of constant scalar curvature $2$.

We say that $x \in U$ is a \emph{center} of $U$ if $x \in \Phi(S^2 \times \{0\})$ for such a $\Phi$.

If $\MM$ is a Ricci flow with surgery and $(x,t) \in \MM$, then we say that \emph{$(x,t)$ is a center of an $\varepsilon$-neck} if $(x,t)$ is a center of an $\varepsilon$-neck in $\MM(t)$.
\end{Definition}

\begin{Definition}[Ricci flow with surgery, strong $\varepsilon$-necks]
Let $\varepsilon > 0$ and consider a Ricci flow with surgery $\MM$ and a time $t_2$.
Consider a subset $U \subset \MM(t_2)$ and assume that all points of $U$ survive until some time $t_1 <  t_2$.
Then the subset $U \times [t_1,t_2] \subset \MM$ is called a \emph{strong $\varepsilon$-neck} if there is a factor $\lambda > 0$ such that after parabolically rescaling by $\lambda^{-1}$, the flow on $U \times [t_1,t_2]$ is $\varepsilon$-close to the standard flow on $[-1,0]$.
By this we mean $\lambda^{-2} (t_2 - t_1) = 1$ and there is a diffeomorphism $\Phi : S^2 \times (-\frac1{\varepsilon}, \frac1{\varepsilon}) \to U$ such that for all $t \in [-1, 0]$
\[ \Vert \lambda^{-2} \Phi^* g(t_2 + \lambda^2 t) - g_{S^2 \times \IR} (t) \Vert_{C^{[\varepsilon^{-1}]}} < \varepsilon. \]
Here $(g_{S^2 \times \IR}(t))_{t \in (-\infty,0]}$ is the standard Ricci flow on $S^2 \times \IR$ which has constant scalar curvature $2$ at time $0$.

A point $(x, t_2) \in U \times \{ t_2 \}$ is called a \emph{center of $U \times [t_1, t_2]$} if $(x, t_2) \in \Phi ( S^2 \times \{ 0 \} \times \{ t_2 \} )$ for such a $\Phi$.
\end{Definition}

\begin{Definition}[Ricci flow with surgery, $(\varepsilon, E)$-caps]
Let $\varepsilon, E > 0$ and consider a Riemannian manifold $(M, g)$ and an open subset $U \subset M$.
Suppose that $(\diam U)^2 |{\Rm}|(y) < E^2$ for any $y \in U$ and $E^{-2} |{\Rm}|(y_1) \leq |{\Rm}|(y_2) \leq E^2 |{\Rm}|(y_1)$ for any $y_1, y_2 \in U$.
Furthermore, assume that $U$ is either diffeomorphic to $\IB^3$ or $\IR P^3 \setminus \ov{\IB}^3$ and that there is a compact set $K \subset U$ such that $U \setminus K$ is an $\varepsilon$-neck.

Then $U$ is called an \emph{$(\varepsilon, E)$-cap}.
If $x \in K$ for such a $K$, then we say that $x$ is a \emph{center of $U$}.

Analogously as in Definition \ref{Def:epsneck}, we define $(\varepsilon, E)$-caps in Ricci flows with surgery.
\end{Definition}

With these concepts at hand we can now give an exact description of the surgery process that will be assumed to be carried out at each surgery time.
To do this, we first fix a geometry which models the metric with which we will endow the filling $3$-balls after each surgery.

\begin{Definition}[surgery model]
Consider $M_{\stan} = \IR^3$ with its natural $SO(3)$-action and let $g_{\stan}$ be a complete metric on $M_{\stan}$ such that
\begin{enumerate}
\item $g_{\stan}$ is $SO(3)$-invariant,
\item $g_{\stan}$ has non-negative sectional curvature,
\item for any sequence $x_n \in M_{\stan}$ with $\dist(0, x_n) \to \infty$, the pointed Riemannian manifolds $(M_{\stan}, g_{\stan}, x_n)$ smoothly converge to the standard $S^2 \times \IR$ of constant scalar curvature $2$.
\end{enumerate}
For every $r > 0$, we denote the $r$-ball around $0$ by $M_{\stan}(r)$.

Let $D_{\stan} > 0$ be a positive number.
Then we call $(M_{\stan}, g_{\stan}, D_{\stan})$ a \emph{surgery model}.
\end{Definition}

\begin{Definition}[$\varphi$-positive curvature]
We say that a Riemannian metric $g$ on a manifold $M$ has \emph{$\varphi$-positive curvature} for $\varphi > 0$ if for every point $x \in M$ there is an $X > 0$ such that $\sec_x \geq - X$ and
\[ \scal_x \geq - \tfrac32 \varphi \qquad \text{and} \qquad \scal_x \geq 2 X (\log (2 X) - \log \varphi - 3). \]
\end{Definition}
Observe that by \cite{Ham} this condition is improved by Ricci flow in the following sense: If $(M, (g_t)_{t \in [t_0, t_1]})$ is a Ricci flow on a compact $3$-manifold with $t_0 > 0$ and $g_{t_0}$ is $t_0^{-1}$-positive, then the curvature of $g_t$ is $t^{-1}$-positive for all $t \in [t_0, t_1]$.

\begin{Definition}[Ricci flow with surgery, $\delta(t)$-precise cutoff] \label{Def:precisecutoff}
Let $\MM$ be a Ricci flow with surgery defined on some time-interval $I \subset [0,\infty)$, let $(M_{\stan},g_{\stan},D_{\stan})$ be a surgery model and let $\delta : I  \to (0, \infty)$ be a function.
We say that $\MM$ is \emph{performed by $\delta(t)$-precise cutoff (using the surgery model $(M_{\stan},g_{\stan}, D_{\stan})$)} if
\begin{enumerate}
\item For all $t$ the metric $g(t)$ has $t^{-1}$-positive curvature.
\item For every surgery time $T^i$, the subset $\MM(T^i) \setminus U^i_+$ is a disjoint union $D^i_1 \cup D^i_2 \cup \ldots$ of smoothly embedded $3$-disks.
\item For every such $D^i_j$ there is an embedding 
\[ \Phi^i_j : M_{\stan}(\delta^{-1}(T^i)) \longrightarrow \MM(T^i) \]
such that $D^i_j \subset \Phi^i_j (M_{\stan}(D_{\stan}))$ and such that the images $\Phi^i_j (M_{\stan} \linebreak[1] (\delta^{-1} \linebreak[1] (T^i)))$ are pairwise disjoint and there are constants $0 <\lambda^i_j \leq \delta(T^i)$ such that 
\[ \big\Vert g_{\stan} - (\lambda^i_j)^{-2} (\Phi^i_j)^* g(T^i) \big\Vert_{C^{[\delta^{-1}(T^i)]}(M_{\stan}(\delta^{-1}(T^i)))} < \delta(T^i). \]
\item For every such $D^i_j$, the points on the boundary of $U^i_-$ in $\MM(T^{i-})$ corresponding to $\partial D^i_j$ are centers of strong $\delta(T^i)$-necks.
\item For every $D^i_j$ for which the boundary component of $\partial U^i_-$ corresponding to the sphere $\partial D^i_j$ bounds a $3$-disk component $(D')^i_j$ of $M^i \setminus U^i_-$ (i.e. a ``trivial surgery'', see below), the following holds:
For every $\chi > 0$, there is some $t_\chi < T^i$ such that for all $t \in (t_\chi,T^i)$ there is a $(1+\chi)$-Lipschitz map $\xi : (D')^i_j \to D^i_j$ which corresponds to the identity on the boundary.
\item For every surgery time $T^i$, the components of $\MM(T^{i-}) \setminus U^i_-$ are diffeomorphic to one of the following manifolds: $S^2 \times I$, $D^3$, $\IR P^3 \setminus B^3$, a spherical space form, $S^1 \times S^2$, $\IR P^3 \# \IR P^3$ and (in the non-compact case) $S^2 \times [0,\infty)$, $S^2 \times \IR$, $\IR P^3 \setminus \ov{B}^3$.
\end{enumerate}
We will speak of each $D^i_j$ as \emph{a surgery} and if $D^i_j$ satisfies the property described in (5), we call it a \emph{trivial surgery}.

If $\delta > 0$ is a number, we say that $\MM$ is \emph{performed by $\delta$-precise cutoff} if this is true for the constant function $\delta(t) = \delta$.
\end{Definition}
Observe that we have phrased the Definition so that if $\MM$ is a Ricci flow with surgery which is performed by $\delta(t)$-precise cutoff, it is also performed by $\delta'(t)$-precise cutoff whenever $\delta'(t) \geq \delta(t)$ for all $t$.
Note also that trivial surgeries don't change the topology of the component at which they are performed.

\subsection{Existence of Ricci flows with surgery} \label{sec:ExRFsurg}
Ricci flows with surgery and precise cutoff as introduced in Definition \ref{Def:precisecutoff} can indeed be constructed from any given initial metric.
We will make this more precise below.
To simplify things, we restrict the geometries which we want to consider as initial conditions.

\begin{Definition}[Normalized initial conditions] \label{Def:normalized}
We say that a Riemannian $3$-manifold $(M,g)$ is \emph{normalized} if 
\begin{enumerate}
\item $M$ is compact and orientable,
\item $|{\Rm}| < 1$ everywhere and
\item $\vol B(x,1) > \frac{\omega_3}2$ for all $x \in M$ where $\omega_3$ is the volume of a standard Euclidean $3$-ball.
\end{enumerate}
We say that a Ricci flow with surgery $\MM$ has \emph{normalized initial conditions}, if $\MM(0)$ is normalized.
\end{Definition}
Obviously, any Riemannian metric on a compact and orientable $3$-manifold can be rescaled to be normalized.
Moreover, recall
\begin{Definition}[$\kappa$-noncollapsedness]
Let $\MM$ be a Ricci flow with surgery, $(x,t) \in \MM$ and $\kappa, r_0 > 0$.
We say that $\MM$ is \emph{$\kappa$-noncollapsed in $(x,t)$ on scales less than $r_0$} if $\vol_t B(x,t,r) \geq \kappa r^3$ for all $0 < r < r_0$ for which
\begin{enumerate}
\item the ball $B(x,t,r)$ is relatively compact in $\MM(t)$,
\item the parabolic neighborhood $P(x,t,r, -r^2)$ is non-singular and
\item $|{\Rm}| < r^{-2}$ on $P(x,t,r,-r^2)$.
\end{enumerate}
\end{Definition}

We now present a definition of the canonical neighborhood assumptions which is a slight modification from the definitions that can be found in other sources, but which suits better our purposes.
\begin{Definition}[canonical neighborhood assumptions] \label{Def:CNA}
Let $\MM$ be a Ricci flow with surgery, $(x, t) \in \MM$ and $r, \varepsilon, \eta > 0$, $E < \infty$ be constants.
We say that $(x,t)$ satisfies the \emph{canoncial neighborhood assumptions $CNA(r, \varepsilon, E, \eta)$} if either $|{\Rm}|(x,t) < r^{-2}$ or the following three properties hold
\begin{enumerate}[label=(\textit{\arabic*})]
\item $(x,t)$ is a center of a strong $\varepsilon$-neck or an $(\varepsilon, E)$-cap $U \subset \MM (t)$. \\
If $U \approx \IR P^3 \setminus \ov{B}^3$, then there is a time $t_1 < t$ such that all points on $U$ survive until time $t_1$ and such that flow on $U \times [t_1, t]$ lifted to its double cover contains strong $\varepsilon$-necks and both lifts of $(x,t)$ are centers of such strong $\varepsilon$-necks.
\item $|{\nabla |{\Rm}|^{-1/2}}| (x,t) < \eta^{-1}$ and $| \partial_t |{\Rm}|^{-1} | (x,t)  < \eta^{-1}$.
\item $\vol_t B( x, t, r' ) > \eta (r')^3$ for all $0 < r' \leq |{\Rm}|^{-1/2} (x, t)$.
\end{enumerate}
or property (2) holds and the component of $\MM(t)$ in which $x$ lies, is closed and the sectional curvatures are positive and $E$-pinched on this component, i.e. they are contained in an interval of the form $(\lambda, E \lambda)$ for some $\lambda > 0$ (and hence that component is diffeomorphic to a spherical space form).
\end{Definition}

Note that we have added an additional assumption in the case in which $U \approx \IR P^3 \setminus \ov{B}^3$ to ensure that the canonical neighborhood assumptions are stable when taking covers of Ricci flows with surgery (compare with Lemma \ref{Lem:tdMM}).
We remark that a manifold which contains a set diffeomorphic to $\IR P^3 \setminus \ov{B}^3$, admits a double cover in which this set lifts to a set diffeomorphic to $S^2 \times (0,1)$.
So it is possible to verify this extra assumption if all the other canonical neighborhood assumptions hold in any double cover. 

The following proposition provides a characterization of regions of high curvature in a Ricci flow with surgery which is performed by precise cutoff.
The power of this proposition lies in the fact that none of the parameters depends on the number or the preciseness of the preceding surgeries.
Hence, it provides a tool to perform surgeries in a controlled way and hence is used to construct long-time existent Ricci flows with surgery as presented in Proposition \ref{Prop:RFwsurg-existence} below.
It also plays an important role in their long-time analysis and will in particular be used in sections \ref{sec:Perelman} and \ref{sec:maintools}.

\begin{Proposition}[Canonical neighborhood theorem, Ricci flows with surgery] \label{Prop:CNThm-mostgeneral}
For every surgery model $(M_{\stan}, \linebreak[1] g_{\stan}, \linebreak[1] D_{\stan})$ and every $\varepsilon > 0$ there are constants $\un\eta > 0$ and $\un{E}_\varepsilon < \infty$ and decreasing continuous positive functions $\un{r}_\varepsilon, \un\delta_\varepsilon, \un\kappa : [0,\infty) \to (0, \infty)$ such that the following holds:

Let $\MM$ be a Ricci flow with surgery on some time-interval $[0,T)$ which has normalized initial conditions and which is performed by $\un\delta_\varepsilon (t)$-precise cutoff.
Then for every $t \in [0,T)$
\begin{enumerate}[label=(\textit{\alph*})]
\item $\MM$ is $\un\kappa(t)$-noncollapsed on scales less than $\sqrt{t}$ at all points of $\MM(t)$.
\item All points of $\MM (t)$ satisfy the canonical neighborhood assumptions $CNA \linebreak[1] (\un{r}_\varepsilon (t), \linebreak[1] \varepsilon, \linebreak[1] \un{E}_\varepsilon,  \linebreak[1] \un\eta)$.
\end{enumerate}
\end{Proposition}
For a proof of this proposition and of Proposition \ref{Prop:RFwsurg-existence} see \cite{PerelmanII}, \cite{KLnotes}, \cite{MTRicciflow}, \cite{BBBMP}, \cite{Bamler-diploma}.
The following proposition provides us an existence result for Ricci flows with surgery.
\begin{Proposition} \label{Prop:RFwsurg-existence}
Given a surgery model $(M_{\stan}, g_{\stan}, D_{\stan})$, there is a continuous function $\un\delta : [0, \infty) \to (0, \infty)$ such that if $\delta' : [0, \infty) \to (0, \infty)$ is a continuous function with $\delta'(t) \leq \un\delta(t)$ for all $t \in [0,\infty)$ and $(M,g)$ is a normalized Riemannian manifold, then there is a Ricci flow with surgery $\MM$ defined for times $[0, \infty)$ such that $\MM(0) = (M,g)$ and which is performed by $\delta'(t)$-precise cutoff. (Observe that we can possibly have $\MM(t) = \emptyset$ for large $t$.)

Moreover, if $\MM$ is a Ricci flow with surgery on some time-interval $[0,T)$ which has normalized initial conditions and which is performed by $\un\delta(t)$-precise cutoff, then $\MM$ can be extended to a Ricci flow on the time-interval $[0, \infty)$ which is performed by $\delta'(t)$-precise cutoff on the time-interval $[T, \infty)$.
\end{Proposition}

We point out that the functions $\un\delta_\epsilon (t), \un{r}_\varepsilon(t), \un\kappa (t)$ and the constants $\un{\eta}, \un{E}_\varepsilon$ in Proposition \ref{Prop:CNThm-mostgeneral} as well as the function $\un\delta(t)$ in Proposition \ref{Prop:RFwsurg-existence} depend on the choice of the surgery model.
\begin{quote}
\textit{\textbf{From now on we will fix a surgery model $(M_{\stan}, g_{\stan}, D_{\stan})$ for the rest of this paper and we will not mention this dependence anymore.}}
\end{quote}

\section{Preliminaries on $3$-dimensional topology} \label{sec:3dtopology}
In this section we present important topological facts which we will frequently use in the course of the paper.
A more elaborate discussion of most of these results can be found in \cite{Hat}.
In the following, all manifolds are assumed to be connected and $3$-dimensional.

\begin{Definition}[prime manifold]
A manifold $M$ is called \emph{prime}, if it cannot be represented as a direct sum $M = M_1 \# M_2$ of two manifolds $M_1$, $M_2$ which are not diffeomorphic to spheres ($\approx S^3$).
\end{Definition}

\begin{Definition}[irreducible manifold] \label{Def:irreducible}
A manifold $M$ is called \emph{irreducible}, if every smoothly embedded $2$-sphere $S \subset M$ bounds a smoothly embedded $3$-disk $D \subset M$, $\partial D = S$.
\end{Definition}

Recall a manifold is prime if and only if it is either irreducible or diffeomorphic to $S^1 \times S^2$ (cf \cite[Proposition 1.4]{Hat}).
We furthermore have

\begin{Proposition} \label{Prop:pi2irred}
An orientable manifold $M$ is irreducible if and only if $\pi_2 (M) = 0$.
\end{Proposition}

\begin{proof}
The backward direction follows from the Sphere Theorem (cf \cite[Theorem 3.8]{Hat}).
For the forward direction suppose that $\pi_2(M) = 0$ and let $S \subset M$ be a smoothly embedded sphere.
So $S$ is contractible and by \cite[Proposition 3.10]{Hat}, it bounds a compact contractible submanifold $N \subset M$.
Following \cite[Proposition 3.7]{Hat}, we conclude that if we attach a $3$-disk to $N$, we obtain a closed, simply-connected manifold $M'$.
By the resolution of the Poincar\'e Conjecture, $M' \approx S^3$ and hence $N$ is a $3$-disk.
\end{proof}

\begin{Definition}[incompressibility] \label{Def:incompressible}
Let $X$ be a topological space and $Y \subset X$ a connected subspace.
Then we call $Y$ \emph{(algebraically) incompressible in $X$} if the induced map $\pi_1(Y) \to \pi_1(X)$ is injective.
Otherwise, we call $Y$ \emph{(algebraically) compressible}.
\end{Definition}

\begin{Proposition} \label{Prop:spanningdiskexists}
Let $M$ be a manifold with boundary.
If $C \subset \partial M$ is an embedded circle which is nullhomotopic in $M$, then $C$ bounds an embedded disk $D$ in $M$, i.e. $\partial D = C$ and $D \cap \partial M = C$.
\end{Proposition}
\begin{proof}
See \cite[Corollary 3.2]{Hat}.
\end{proof}

\begin{Proposition} \label{Prop:incompressibleequiv}
Let $M$ be a manifold (possibly with boundary) and $S \subset M$ a $2$-sided embedded, connected surface.
Then $S$ is algebraically compressible if and only if there is an embedded $C \subset S$ which is homotopically non-trivial in $S$ and which bounds an embedded \emph{compressing disk} $D \subset M$ which meets $S$ only in its boundary, i.e. $\partial D = C$ and $D \cap S = C$.

In particular, the statement holds if $S = \partial M$.
\end{Proposition}
\begin{proof}
See \cite[Corollary 3.3]{Hat}.
\end{proof}

\begin{Lemma} \label{Lem:compressingtorus}
Let $M$ be a closed, irreducible manifold and let $T \subset M$ be an embedded, $2$-sided, compressible torus.
Then $T$ separates $M$ into two components $U$, $V$ (i.e. $M = U \cup V$ and $U \cap V = T$) and we can distinguish the following cases:
\begin{enumerate}[label=(\alph*)]
\item Neither of the components $U$ or $V$ is diffeomorphic to a solid torus $S^1 \times D^2$.
Then the compressing disks $D$ for $T$ either all lie in $U$ or in $V$ and for each such $D$ a tubular neighborhood of $D \cup V$ or $D \cup U$ (depending on whether $D \subset U$ or $D \subset V$) is diffeomorphic to a $3$-ball.
\item Only one of the components $U$, $V$ is diffeomorphic to a solid torus.
Assume that this component is $U$.
Then $T$ has compressing disks in $U$.
If it also has compressing disks in $V$, then $U$ is contained in an embedded $3$-ball in $M$ and $U$ is compressible in $M$ (i.e. the map $\IZ \cong \pi_1 (U) \to \pi_1 (M)$ is not injective).
\item Both $U$ and $V$ are diffeomorphic to solid tori.
Then $M$ is diffeomorphic to a spherical space-form.
\end{enumerate}
\end{Lemma}
\begin{proof}
For the first part see \cite[p. 11]{Hat}.
Let $D$ be a compressing disk for $T$ and assume that $D \subset U$.
Again by \cite[p. 11]{Hat}, we know that either $U$ is a solid torus or a tubular neighborhood  of $D \cup V$ is diffeomorphic to a $3$-ball.
So if in case (a) there are compressing disks for $T$ in both $U$ and $V$, then $M$ is covered by two embedded $3$-balls and we have $M \approx S^3$ by Lemma \ref{Lem:coverMbysth}(a) (observe that the proof of Lemma \ref{Lem:coverMbysth}(a) does not make use of this Lemma).
However, this contradicts the fact that an embedded $2$-torus in $S^3$ bounds a solid torus on at least one side (see \cite[p. 11]{Hat}).
Case (b) is clear.

Consider now case (c).
Let $K_1, K_2 \subset \pi_1(T) \cong \IZ^2$ be the kernels of the projections $\pi_1(T) \to \pi_1(U)$ and $\pi_1(T) \to \pi_1(V)$.
If $K_1 = K_2$, then $M \approx S^1 \times S^2$ contradicting the assumptions on $M$.
So $K_1 \not= K_2$.
Let $a_i \in K_i$ be generators.
By an appropriate choice of coordinates, we can assume that $a_1 = (1,0) \in \IZ^2$ and $a_2 = (p,q) \in \IZ^2$ where $0 \leq p < q$.
Then $M$ is diffeomorphic to the lens space $L(p,q)$.
\end{proof}

\begin{Lemma} \label{Lem:coverMbysth}
Let $M$ be closed manifold and assume that $M = U \cup V$.
Then
\begin{enumerate}[label=(\alph*)]
\item If $U$ and $V$ are diffeomorphic to a ball, then $M \approx S^3$.
\item If $U$ is diffeomorphic to a solid torus $S^1 \times D^2$ and $V$ is diffeomorphic to a ball, then $M \approx S^3$.
\item If $U$ and $V$ are diffeomorphic to a solid torus $\approx S^1 \times D^2$, then $M$ is either not irreducible or it is diffeomorphic to a spherical space form.
\end{enumerate}
\end{Lemma}
\begin{proof}
In case (a), we can assume that $U$ and $V$ are the interiors of compact embedded $3$-disks.
So $\partial U \subset V$.
By Alexander's Theorem (cf \cite[Theorem 1.1]{Hat}), $\partial U$ bounds a $3$-disk in $V$.
So $\partial U$ bounds a $3$-disk on both sides and hence $M \approx S^3$.

Case (b) follows along the lines; note that every embedded sphere in a solid torus bounds a ball.

For case (c) we can assume that $M$ is irreducible.
Moreover, by adding collar neighborhoods, we can assume that $\partial U \cap \partial V = \emptyset$.
Let $T = \partial U$ and $V' = M \setminus \Int U$.
Then $T$ is compressible in $V$ and by Proposition \ref{Prop:incompressibleequiv}, we find a spanning disk $D \subset \Int V$.
If also $D \subset U$, then $U \setminus D$ is a $3$-ball and $M = ( U \setminus D ) \cup V$ and we are done by case (b).
So assume that $D \subset V'$.
Then by Lemma \ref{Lem:compressingtorus}(b), either $V'$ is a solid torus or $U$ is contained in an embedded $3$-ball $B \subset M$.
In the latter case $M = B \cup V$ and we are again done by case (b).
Finally, if $V'$ is a solid torus, we are done by Lemma \ref{Lem:compressingtorus}(c).
\end{proof}

\begin{Lemma} \label{Lem:Kleinandsolidtorus}
Let $M$ be a manifold and $T \subset M$  an embedded $2$-torus which separates $M$ into two connected components whose closures $U, V \subset M$ are diffeomorphic to $\Klein^2 \td\times I$ and $S^1 \times D^2$ each.
Then $M$ is either not irreducible or it is diffeomorphic to a spherical space form.
\end{Lemma}
\begin{proof}
Consider the double cover $\widehat{U} \to U$ for which $\widehat{U} \approx T^2 \times I$.
This cover extends to a double cover $\widehat{M} \to M$.
Let $T' \subset \widehat{M}$ be the torus which projects to the zero section in $\Klein^2 \td\times I$.
Then as in the last part of the proof of Lemma \ref{Lem:compressingtorus}(c) we can write $\widehat{M} = S_1 \#_{T'} S_2$ where $S_1$ and $S_2$ are solid tori.
So $\widehat{M}$ is either diffeomorphic to $S^1 \times S^2$ or a lens space.
In the first case, $M$ is either diffeomorphic to $S^1 \times S^2$ or $\IR P^3 \# \IR P^3$ and in the second case, $M$ is still spherical (see also \cite{Asa}).
\end{proof}

The following Lemma will be important in the proof of Lemma \ref{Lem:SStori}.

\begin{Lemma} \label{Lem:Seifertfiberincompressible}
Let $M$ be compact, orientable, irreducible manifold (possibly with boundary) which is not diffeomorphic to a spherical space form.
Consider a compact, connected $3$-dimensional submanifold $N \subset M$ whose boundary components are tori and which carries a Seifert fibration that is compatible with these boundary tori.
Assume that each boundary component $T \subset \partial N$ which is compressible in $M$, either bounds a solid torus $\approx S^1 \times D^2$ on the other side or $T$ separates $M$ into two components and is incompressible in the component of $M \setminus T$ which does not contain $N$ (if $T \subset \partial M$, then this component is empty).

Then either there is one boundary torus $T \subset \partial N$ which bounds a solid torus on the same side as $N$ or every boundary component of $N$ either bounds a solid torus on the side opposite to $N$ or it is even incompressible in $M$.
Moreover, in the latter case, the generic Seifert fibers of $N$ are incompressible in $M$.
\end{Lemma}
\begin{proof}
Some of the following arguments can also be found in \cite{Faessler} and \cite{MorganTian}.
Denote the boundary tori of $N$ by  $T_1, \ldots, T_m$.
Assume that there is a component $T_i$ which bounds a solid torus $S_i$ on the side opposite to $N$ such that the Seifert fibers in $T_i$ are incompressible in $S_i$.
Then we can extend the Seifert fibration of $N$ to $S_i$.
So assume in the following that for any $T_i$ which bounds a solid torus $S_i$ on the other side, the Seifert fibers of $T_i$ are nullhomotopic in $S_i$.
Denote by $B$ the base orbifold of the Seifert fibration on $N$ and call the projection $\pi : N \to B$.
We remark that since $M$ is orientable, the only singular points of $B$ are cone points.
Each $T_i$ corresponds to a boundary circle $C_i = \pi(T_i) \subset \partial B$.

We first show that that there is at most one $T_i$ which bounds a solid torus $S_i$ on the side opposite to $N$ (we will call it from now on $T_1$):
Assume, there were two such components $T_1$ and $T_2$ and denote the respective solid tori by $S_1$ and $S_2$.
Let $\alpha \subset B$ be an embedded curve connecting $C_1$ and $C_2$ which does not meet any singular points.
The preimage $Z_{\alpha} = \pi^{-1}(\alpha) \subset N$ is a cylinder whose boundary components are each nullhomotopic in $S_1$ resp. $S_2$.
Let $D_1 \subset S_1$ resp. $D_2 \subset S_2$ be spanning disks for $Z_{\alpha} \cap \partial S_1$ resp. $Z_{\alpha} \cap \partial S_2$.
Then $\Sigma_\alpha = D_1 \cup Z_\alpha \cup D_2$ is an embedded $2$-sphere.
Since $D_1$ or $D_2$ are non-separating in $S_1$ resp. $S_2$, we conclude that $\Sigma_\alpha$ is non-separating in $M$.
This contradicts the assumption that $M$ is irreducible.

Next, we show that if $T_1$ bounds a solid torus $S_1$ on the side opposite to $N$, then the topological surface underlying $B$ is (a disk or) a multi-annulus:
Assume not.
Then there is an embedded, non-separating curve $\alpha \subset B$ whose endpoints are distinct and lie in $C_1$.
As before, this yields a non-separating sphere $\Sigma_\alpha \subset M$ contradiction the irreducibility assumption of $M$.

Assume now for the rest of the proof that none of the tori $T_i$ bound a solid torus on the same side as $N$.
We will show in the following that then none of the tori $T_i$ bounds a solid torus on either side and that all $T_i$ as well as that the generic Seifert fibers on $N$ are incompressible in $M$.

First assume that $T_1$ bounds a solid torus $S_1$ (on the side opposite to $N$).
So the topological surface underlying $B$ is a multi-annulus.
We can find a collection of embedded curves $\alpha_1, \ldots, \alpha_k \subset B$ with endpoints in $C_1$ which do not meet any singular points and which cut $B$ into smaller pieces, each of which contain at most one singular point or one boundary component and which are bounded by at most two of the curves $\alpha_i$ and parts of $C_1$.
The corresponding spheres $\Sigma_{\alpha_1}, \ldots, \Sigma_{\alpha_k} \subset M$ bound closed $3$-balls $B_1, \ldots, B_k \subset M$.
Two such balls are either disjoint or one is contained in the other.
Hence, either there is one $B_i$ containing all other balls or there are two balls $B_i$, $B_j$ such that any ball is contained in one of them.
It is easy to conclude from the position of these balls relatively to $S_1$ that $U = S_1 \cup B_i$ resp. $S_1 \cup B_i \cup B_j$ is diffeomorphic to a solid torus.
We can now distinguish the following cases:
\begin{enumerate}[label=$-$]
\item If $\alpha_i$ (and possibly $\alpha_j$) enclose an orbifold singularity, then the complement of $U$ is a solid torus and we obtain a contradiction using Lemma \ref{Lem:coverMbysth}(c).
\item If they enclose a boundary component $C_k$ of $B$, then we argue as follows:
Let $\alpha'$ be a curve connecting $C_k$ with $C_1$ which does not intersect $\alpha_i \cup \alpha_j$ and choose a spanning disk $D' \subset S_1$ for the curve $Z_{\alpha'} \cap \partial S_1$.
Then $Z_{\alpha'} \cup D'$ is a compressing disk for $T_k$ and $T_k$ does not bound a solid torus.
So by Lemma \ref{Lem:compressingtorus}(a), a tubular neighborhood of $T_k \cup Z_{\alpha'} \cup D'$ is diffeomorphic to a $3$-ball.
This implies that $M$ is covered by a solid torus and a ball and Lemma \ref{Lem:coverMbysth}(b) gives us a contradiction.
\end{enumerate}
Hence, none of the $T_i$ bound a solid torus on either side.

We argue that the generic Seifert fibers of $N$ are incompressible in $N$:
Using Lemma \ref{Lem:coverMbysth}(c), it is easy to see that $B$ cannot be a bad orbifold (i.e. the tear drop or the football) or a quotient of the $2$-sphere.
So, we can find a (possibly non-compact) cover $\widehat B \to B$ such that $\widehat B$ is smooth and corresponding to this a cover $\widehat N \to N$ such that we have an $S^1$-fibration $\widehat N \to \widehat B$.
Observe that $\widehat B$ is not a $2$-sphere, because otherwise by Lemma \ref{Lem:coverMbysth}(c) $\widehat N \approx S^3$ in contradiction to our assumptions.
Using the long exact homotopy sequence and the fact that $\pi_2(\widehat B) = 0$, we see that a lift of any generic $S^1$-fiber $\gamma$ is incompressible in $\widehat N$ implying that $\gamma$ is incompressible in $N$.

Next we show that any generic $S^1$-fiber $\gamma$ of $N$ is incompressible in $M$:
Assume there is a nullhomotopy $f : D^2 \to M$ for a non-zero multiple of $\gamma$.
By a small perturbation, we can assume that $f$ is transversal to the boundary tori $T_1, \ldots, T_m$.
So $f^{-1} ( T_1 \cup \ldots \cup T_m)$ consists of finitely many circles.
Look at one of those circles $\gamma' \subset D^2$ which is innermost in $D^2$ and assume $f(\gamma') \subset T_i$.
If $f|_{\gamma'}$ is homotopically trivial in $T_i$, then we can alter $f$ such that $\gamma'$ is removed from the list.
So assume that $f|_{\gamma'}$ is homotopically non-trivial in $T_i$.
Let $D' \subset D^2$ be the disk which is bounded by $\gamma'$.
Then by Proposition \ref{Prop:incompressibleequiv} and Lemma \ref{Lem:compressingtorus}(b) we have $f(D') \subset N$.
Since the generic Seifert fibers of $N$ are incompressible in $N$, $f|_{\gamma'}$ cannot be homotopic to such a fiber, so it projects down to a curve which is homotopic to a non-zero multiple of the boundary circle $C_i$ under $\pi$.
Hence, a non-zero multiple of $C_i$ is homotopically trivial in $\pi_{1, \textnormal{orbifold}}(B)$.
We conclude that $B$ can only be a disk with possibly one orbifold singularity.
But this implies that $N$ is diffeomorphic to a solid torus, in contradiction to our assumptions.

It remains to show that all tori $T_i$ are incompressible in $M$.
By Lemma \ref{Lem:compressingtorus}(a), we conclude that of $T_i$ is compressible in $M$, then $T_i$ is contained in an embedded $3$-ball.
But this however contradicts the fact that the generic Seifert fibers of $N$ are incompressible in $M$.
\end{proof}

\section{Perelman's long-time analysis results and certain generalizations} \label{sec:Perelman}
\subsection{Perelman's long-time curvature estimates}
In this subsection, we will review some of Perelman's long-time analysis results (see \cite{PerelmanII}).
We will generalize these results to the boundary case and go through most of their proofs.
The most important result of this section will be Proposition \ref{Prop:genPerelman} below and will be used in section \ref{sec:maintools}, however many of the Lemmas leading to this Proposition will also be used in that section.
The boundary case will be important for use, because we want to do analysis in local covers.

The following notation will be used throughout the whole paper.
\begin{Definition} \label{Def:rhoscale}
Let $(M,g)$ be a Riemannian manifold and $x \in M$ a point.
We define
\[ \rho(x) = \sup \{ r \;\; : \;\; \sec \geq - r^{-2} \;\; \text{on} \;\; B(x,r) \}. \]
For $r _0 > 0$ we set furthermore $\rho_{r_0}(x) = \min \{ \rho(x), r_0 \}$.
If $(M, g) = \MM(t)$ is the time-slice of a Ricci flow (with surgery) $\MM$, then we often use the notation $\rho(x,t)$ and $\rho_{r_0} (x,t)$.
\end{Definition}

We now present the main result of this section.
We point out a small inaccuracy which we will encounter in the next Proposition as well as in the following results of this section and which we will from now on accept without further mention:
We will often be dealing with Ricci flows with surgery $\MM$ defined on a time-interval of the form $[ t_0 - r_0^2, t_0]$ and most results require certain canonical neighborhood assumptions to hold on $\MM$.
For times which are very close to $t_0 - r_0^2$ this may be problematic since strong $\varepsilon$-necks might stick out of the time-interval.
This inconsistency however does not create any problems since in our applications $\MM$ will always arise from the restriction of a Ricci flow with surgery defined on the time-interval $[0, \infty)$.
One could resolve this issue e.g. by requiring the canonical neighborhood assumptions to hold on a slightly smaller time-interval $[t_0 - 0.99 r_0^2, t_0]$ or by adapting the definition of strong $\varepsilon$-necks.

\begin{Proposition}[\hbox{\cite[6.8]{PerelmanII}} in the non-compact case] \label{Prop:genPerelman}
There is a constant $\varepsilon_0 > 0$ such that for all $w, r, \eta > 0$ and $E < \infty$ and $1 \leq A < \infty$ and $m \geq 0$ there are $\tau = \tau(w, A, E, \eta), \ov{r} = \ov{r}(w, A, E, \eta), \widehat{r} = \widehat{r}(w, E, \eta), \delta = \delta(r, w, A, E, \eta, m) > 0$ and $K_m = K_m (w, A, E, \eta), C_1 = C_1(w, A, E, \eta), Z = Z(w, A, \linebreak[1] E, \linebreak[1] \eta) \linebreak[1] < \infty$ such that: \\
Let $r_0^2 \leq t_0/2$ and let $\MM$ be a Ricci flow with surgery (whose time-slices are allowed to have boundary) on the time-interval $[t_0 - r_0^2, t_0]$ which is performed by $\delta$-precise cutoff and consider a point $x_0 \in \MM(t_0)$.
Assume that the canonical neighborhood assumptions $CNA (r, \varepsilon_0, E, \eta)$ as described in Definition \ref{Def:CNA} are satisfied on $\MM$.
We also assume that the curvature on $\MM$ is uniformly bounded on compact time-intervals which don't contain surgery times and that all time slices of $\MM$ are complete.

In the case in which some time-slices of $\MM$ have non-empty boundary, we assume that
\begin{enumerate}[label=(\roman*)]
\item For all $t_1, t_2 \in [t_0 - \frac1{10} r_0^2, t_0]$, $t_1 < t_2$ we have:
if some $x \in B(x_0, t_0, r_0)$ survives until time $t_2$ and $\gamma : [t_1, t_2] \to \MM$ is a space-time curve with endpoint $\gamma(t_2) \in B(x, t, (A+3) r_0)$ which meets the boundary $\partial\MM$ somewhere, then it has $\LL$-length $\LL(\gamma) > Z r_0$ (based in $t_2$, see (\ref{eq:defLLlength})).
\item For all $t \in [t_0 - \frac1{10} r_0^2, t_0]$ we have: if some $x \in B(x_0, t_0, r_0)$ survives until time $t$, then $B(x, t, 2(A+3) r_0 + r)$ does not meet the boundary $\MM(t)$.
\end{enumerate}
Now assume that
\begin{enumerate}[label=(\roman*), start=3]
\item $r_0 \leq \ov{r} \sqrt{t_0}$,
\item $\sec_{t_0} \geq - r_0^{-2}$ on $B(x_0, t_0, r_0)$ and
\item $\vol_{t_0} B(x_0, t_0, r_0) \geq w r_0^3$.
\end{enumerate}
Then $|{\nabla^k \Rm}| < K_m r_0^{-2 - k}$ on $B(x_0, t_0, A r_0)$ for all $k \leq m$.
In particular, if $r_0 = \rho(x_0, t_0)$, then $r_0 > \widehat{r} \sqrt{t_0}$.

If moreover $C_1 \delta \leq r_0$, then the parabolic neighborhood $P(x_0, t_0, A r_0, -\tau r_0^2)$ is non-singular and we have $|{\nabla^k \Rm}| < K_k r_0^{-2-k}$ on $P(x_0, t_0, A r_0, -\tau r_0^2)$ for all $k \leq m$.
\end{Proposition}

The following Corollary is a consequence of Propositions \ref{Prop:genPerelman} and \ref{Prop:CNThm-mostgeneral}.

\begin{Corollary}[cf \hbox{\cite[6.8, 7.3]{PerelmanI}}] \label{Cor:Perelman68}
There is a continuous positive function $\delta : [0, \infty) \to (0, \infty)$ such that for every $w > 0$, $1 \leq A < \infty$ and $m \geq 0$ there are constants $\tau = \tau(w, A), \ov{\rho} = \ov{\rho} (w, A), \ov{r} = \ov{r} (w, A), c_1 = c_1(w, A) > 0$ and $T = T(w, A, m), K_m = K_m(w, A) < \infty$ such that:

Let $\MM$ be a Ricci flow with surgery on the time-interval $[0, \infty)$ with normalized initial conditions (whose time-slices are all compact) and which is performed by $\delta(t)$-precise cutoff.
Let $t > T$ and $x \in \MM (t)$.
\begin{enumerate}[label=(\textit{\alph*})]
\item If $0 < r \leq \min \{ \rho(x,t), \ov{r} \sqrt{t} \}$ and $\vol_t B(x,t, r) \geq w r^3$, then $|{\nabla^k \Rm}| < K_m r_0^{-2-k}$ on $B(x,t, Ar)$ for all $k \leq m$.
Moreover, if all surgeries on the time-interval $[t - r^2, t]$ are performed by $c_1 r$-precise cutoff, then the parabolic neighborhood $P(x, t, Ar, - \tau r^2)$ is non-singular and we have $|{\nabla^k \Rm}| < K_m r^{-2 - k}$ on $P(x, t, Ar, - \tau r^2)$ for all $k \leq m$.
\item If $\vol_t B(x,t,\rho(x,t)) \geq w \rho^3 (x,t)$, then $\rho(x,t) > \ov{\rho} \sqrt{t}$ and the parabolic neighborhood $P(x,t, A \sqrt{t}, - \tau t)$ is non-singular and we have $|{\nabla^k \Rm}| < K_m t^{-1-k/2}$ on $P(x,t, A \sqrt{t}, - \tau t)$ for all $k \leq m$.
\end{enumerate}
\end{Corollary}

In the case $A = 1$, this Corollary implies \cite[6.8]{PerelmanII} and parts of \cite[7.3]{PerelmanII}.

In the following, we will present proofs of Proposition \ref{Prop:genPerelman} and Corollary \ref{Cor:Perelman68}.
They require a few rather complicated Lemmas which we will establish first.
The proofs of Proposition \ref{Prop:genPerelman} and Corollary \ref{Cor:Perelman68} can be found at the end of this section.
Note that the following arguments will be very similar to those presented in \cite{PerelmanII} and \cite{KLnotes} with small modifications according to the author's taste.
Occasionally, we will omit shorter arguments and refer to \cite{KLnotes}.
The main objective in the proofs will be the discussion of the influence of the boundary.
Upon the first reading, it is recommended to skip the rather uninnovative remainder this section.
The boundary case of Proposition \ref{Prop:genPerelman} will only be used in subsection \ref{subsec:curvboundinbetween} and the exact phrasing of conditions (i) and (ii) will turn out to be not as important as it might appear here.

The following distance distortion estimates will be used frequently throughout this paper.
\begin{Lemma}[distance distortion estimates] \label{Lem:distdistortion}
Let $(M, (g_t)_{t \in [t_1, t_2]})$ be a Ricci flow whose time-slices are complete and let $x_1, x_2 \in M$.
Then
\begin{enumerate}[label=(\alph*)]
\item If $\Ric_t \leq K$ along any minimizing geodesic between $x_1$ and $x_2$ in $(M, g_t)$, then at time $t$ we have $\frac{d}{d t} \dist_t (x_1, x_2) \geq - K \dist_t (x_1, x_2)$.
Likewise if $\Ric_t \geq - K$ along any such minimizing geodesic, then $\frac{d}{d t} \dist_t (x_1, x_2) \leq K \dist_t (x_1, x_2)$.
\item If at some time $t$ we have $\dist_t (x_1, x_2) \geq 2 r$ and $\Ric_t \leq r^{-2}$ on $B(x_1, r) \cup B(x_2, r)$ for some $r > 0$, then $\frac{d}{dt^+} \dist_t (x_1, x_2) \geq - \frac{16}3 r^{-1}$.
\end{enumerate}
\end{Lemma}

\begin{proof}
See \cite[sec 27]{KLnotes}, \cite[8.3]{PerelmanI}, \cite[sec 2.3]{Bamler-diploma}.
\end{proof}

We will also need
\begin{Lemma} \label{Lem:shortrangebounds}
Let $\MM$ be a Ricci flow with surgery which satisfies the canonical neighborhood assumptions $CNA (r, \varepsilon, E, \eta)$ for some $r, \varepsilon, E, \eta > 0$, let $(x, t) \in \MM$ and set $Q = |{\Rm}|(x,t)$.
\begin{enumerate}[label=(\alph*)]
\item If $Q \leq r^{-2}$, then $|{\Rm}| < 2 r^{-2}$ on $P(x, t, \frac{\eta}{10} r, - \frac{\eta}{10} r^2)$.
\item If $Q \geq r^{-2}$, then $|{\Rm}| < 2 Q$ on $P(x, t, \frac{\eta}{10} Q^{-1/2}, - \frac{\eta}{10} Q^{-1})$.
\end{enumerate}
\end{Lemma}

\begin{proof}
See \cite[4.2]{PerelmanII}, \cite[Lemma 70.1]{KLnotes}, \cite[sec 6.2]{Bamler-diploma}.
\end{proof}

Before we present the first main Lemma, we recall the $\LL$-functional as defined in \cite[sec 7]{PerelmanI}:
For any smooth space-time curve $\gamma : [t_1, t_2] \to \MM$ ($t_1 < t_2 \leq t_0$) in a Ricci flow with surgery $\MM$ set
\begin{equation} \label{eq:defLLlength}
 \LL(\gamma) = \int_{t_1}^{t_2} \sqrt{t_0 - t'} \big(|\gamma'|^2(t') + \scal (\gamma(t'), t') \big) dt'.
\end{equation}
We say that $\LL$ is \emph{based in $t_0$} and call $\LL(\gamma)$ the \emph{$\LL$-length} of $\gamma$.
Every $\gamma$ which is a critical point of $\LL$ is called \emph{$\LL$-geodesic}.

\begin{Lemma}[cf \hbox{\cite[6.3(a)]{PerelmanII}}] \label{Lem:6.3a}
For any $1 \leq A < \infty$ and $w, r, \varepsilon, E, \eta > 0$ there are $\kappa = \kappa(w, A, \eta), \delta = \delta(w, A, r, \eta) > 0$, $Z = Z(A) < \infty$ such that: \\
Let $r_0^2 < t_0/2$ and let $\MM$ be a Ricci flow with surgery (whose time-slices are allowed to have boundary) on the time-interval $[ t_0 - r_0^2, t_0]$ which is performed by $\delta$-precise cutoff and consider a point $x_0 \in \MM(t_0)$.
Assume that the canonical neighborhood assumptions $CNA (r, \varepsilon, E, \eta)$ are satisfied on $\MM$.
We also assume that the curvature on $\MM$ is uniformly bounded on compact time-intervals which don't contain surgery times and that all time-slices of $\MM$ are complete.

Assume that the parabolic neighborhood $P(x_0, t_0, r_0, - r_0^2)$ is non-singular, that $|{\Rm}| \leq r_0^{-2}$ on $P(x_0, t_0, r_0, - r_0^2)$ and $\vol_{t_0} B(x_0, t_0, r_0) \geq w r_0^3$.

In the case in which some time-slices of $\MM$ have non-empty boundary, we assume that
\begin{enumerate}[label=(\roman*)]
\item every space-time curve $\gamma : [t, t_0] \to \MM$ with $t \in [t_0 - r_0^2, t_0)$ which ends in $\gamma(t_0) \in B(x_0, t_0, A r_0)$ and which meets the boundary $\partial \MM (t')$ at some time $t' \in [t, t_0]$, has $\LL$-length $\LL(\gamma) > Z r_0$ (based in $t_0$),
\item the ball $B (x_0, t_0, (2A+1) r_0 + r)$ does not hit the boundary $\partial \MM(t_0)$ and for every $t \in [t_0 - \frac12 r_0^2, t_0]$ the ball $B(x_0, t, A (1 - 2 (t_0 - t) r_0^{-2} ) r_0 + \frac1{10} r_0)$ does not hit the boundary $\partial \MM(t)$.
\end{enumerate}

Then $\MM$ is $\kappa$-noncollapsed on scales less than $r_0$ at all points in the ball $B(x_0, t_0, A r_0)$.
\end{Lemma}

\begin{proof}
We first consider the case in which the component of $\MM(t)$ which contains $x_0$ is closed and has positive sectional curvature.
Then the sectional curvature is also positive on the component of $\MM (t_0)$ which contains $x_0$ and we are done by volume comparison.
So in the following, we exclude this case and hence the last option in the Definition \ref{Def:CNA} of the canonical neighborhood assumptions will not occur.

Let $x_1 \in B(x_0, t_0, A r_0)$ and $0 < r_1 < r_0$ such that $B(x_1, t_0, r_1)$ does not hit the boundary $\partial \MM (t_0)$, that $P(x_1, t_0, r_1, - r_1^2)$ is non-singular and $|{\Rm}| < r_1^{-2}$ on $P(x_1, t_0, r_1, - r_1^2)$.

\begin{Claim1}
There is a universal constant $\delta_0 > 0$ such that if $\delta < \delta_0$, then we can restrict ourselves to the case $r_1 > \frac12 r$.
By this we mean that if the Lemma holds under this additional restriction for some $\kappa' = \kappa' (w, A, \eta) > 0$ then it also holds whenever $r_1 \leq \frac12 r$ for some $\kappa = \kappa (w, A, \eta) > 0$.
\end{Claim1}

\begin{proof}
Let $s > 0$ be the supremum over all $r_1$ which satisfy the properties above.
If $s \leq \frac12 r$, then there are several cases:
\begin{enumerate}[label=(\arabic*)]
\item The closure of $B(x_1, t_0, s)$ hits the boundary $\partial \MM(t_0)$.
This case is excluded by conditin (ii).
\item The closure of $P(x_1, t_0, s, -s^2)$ hits a singular point $(x', t')$.
By Definition \ref{Def:precisecutoff}(3), there is a neighborhood $U \subset \MM (t')$ of $(x',t')$ whose geometry is modeled on a standard solution on a scale of at least $c_1 s$ for some universal $c_1 > 0$.
So for sufficiently small $\delta$, we can find point $(x'', t') \in P(x_1, t_0, s, - s^2) \cap U$ such that $B(x'', t', c_2 s) \subset P(x_1, t_0, s, - s^2) \cap U$ for some universal $c_2 > 0$.
Since the standard solution is uniformly noncollapsed, we find $\vol_{t'} B(x'', t', c_2 s) > \kappa' s^3$ for some universal $\kappa' > 0$.
So by volume distortion $\vol_{t_0} B(x_1, t_0, s) > \kappa'' s^3$ for some universal $\kappa'' > 0$.
By volume comparison, this implies $\vol_{t_0} B(x_1, t_0, r_1) > \kappa r_1^3$ for some universal $\kappa > 0$.
\item There is a point $(x', t')$ in the closure of $P(x_1, t_0, s, -s^2)$ with $|{\Rm}|(x',t') = s^{-2}$.
Then by Lemma \ref{Lem:shortrangebounds} and the canonical neighborhood assumptions, we can find a point $(x'', t') \in P(x_1, t_0, s, -s^2) \cap \MM(t')$ such that $|{\Rm}|(x'', t') > \frac12 s^{-2} > r^{-2}$ and $B(x'', t', c_2 s) \subset P(x_1, t_0, s, -s^2)$ for some universal $c_2 > 0$.
By the canonical neighborhood assumptions, we have $\vol_{t'} B(x'', t', c_2 s) > \eta c_2^3 s^3$ and as in case (2) we can conclude that $\vol_{t_0} B(x_1, t_0, r_1) > \kappa r_1^3$ for some universal $\kappa = \kappa(\eta) > 0$.
\item We have $s = r_0$.
So $r_0 \leq \frac12 r$.
In this case choose $0 < d \leq (A+1) r_0$ maximal with the property that $|{\Rm}| < r_0^{-2} = s^{-2}$ on $B(x_1, t_0, d)$.
So $d \geq r_0$.
If $d = (A+1) r_0$, then
\[ \vol_{t_0} B(x_1, t_0, d) \geq \vol_{t_0} B(x_0, t_0, r_0) \geq w r_0^{3} = \frac{w}{(A+1)^3} d^3. \]
So by volume comparison and assumption (ii) we obtain a lower volume bound on the normalized volume of $B(x_1, t_0, r_1)$ since $r_1 < r_0 < d$.
Assume now that $d < (A+1) r_0$.
Then $|{\Rm}|(x', t_0) = r_0^{-2} \geq 4 r^{-2}$ for some $(x', t_0)$ in the closure of $B(x_1, t_0, d)$.
As in case (3), we can find a point $(x'', t_0) \in B(x_1, t_0, d - c_1 r_0)$ with $|{\Rm}|(x'', t_0) > \frac12 r_0^{-2} > r^{-2}$.
By the canonical neighborhood assumptions, we have $\vol_{t_0} B(x_1, t_0, d) \geq \vol_{t_0} B(x'', t_0, c_1 r_0) > \eta c_1^3 r_0^3$.
So again by volume comparison, we find that $\vol_{t_0} B(x_1, t_0, r_1) > \kappa r_1^3$ for some $\kappa = \kappa (\eta, A) > 0$.
\end{enumerate}
Lastly, if $s > \frac12 r$, then the conditions hold for some $r_1' > \frac12 r$.
If the assertion of the Lemma holds for $r'_1$ and some $\kappa' = \kappa'(w, A, \eta) > 0$, then by volume comparison, it also holds for any $r_1 \leq r'_1$ and some $\kappa = \kappa(w, A, \eta) > 0$.
\end{proof}

So assume in the following that $r_1 > \frac12 r$.
We will now set up an $\LL$-geometry argument.
Define for any $t \in [t_0 - r_0^2, t_0]$ and $y \in \MM(t)$
\[
 L(y, t) = \inf \Big\{ \LL(\gamma) \;\; : \;\; \gamma : [t, t_0] \to \MM \; \text{smooth}, \; \gamma(t) = y, \; \gamma(t_0) = x_1 \Big\}.
\]
Moreover, set
\[ \ov{L}(y, t) = 2 \sqrt{t_0 - t} L(y, t) \qquad \text{and} \qquad \ell(y, t) = \frac1{2\sqrt{t_0 - t}} L(y, t). \]
Let
\begin{multline*}
 D_t = \{ y \in \MM (t) \;\; : \;\; \text{there is a minimizing $\LL$-geodesic $\gamma : [t, t_0] \to \MM \setminus \partial \MM$} \\ \text{with $\gamma(t) = y$ and $\gamma(t_0) = x_1$ which does not hit any surgery points} \}.
\end{multline*}
We can then define the reduced volume
\[ \widetilde{V}(t) = (t_0 - t)^{-n/2} \int_{D_t} e^{-\ell(\cdot, t)} d {\vol_t}. \]
It is shown in \cite{PerelmanI} that $\widetilde{V}(t)$ is non-decreasing in $t$.

We will now show that the quantity $\ell (\cdot, t_0 - r_0^2)$ is uniformly bounded on $B(x_0, t_0, r_0)$ by a constant only depending on $A$ if $\delta$ is chosen small enough depending on $A$, $r$ and $\eta$.
To do this we will make use of a maximum principle argument on $D_t$.
The following claim will ensure hereby that extremal points of $L$ lie inside $D_t$.

\begin{Claim2}
For any $\Lambda < \infty$ there is a constant $\delta^* = \delta^* (\Lambda, r, \eta) > 0$ such that whenever $\delta \leq \delta^*$ and $Z \geq \Lambda$, then the following holds:
Assume that $r_1 > \frac12 r$.
If $t \in [t_0 - r_0^2, t_0]$, $y \in \MM(t)$ and $L(y,t) < \Lambda r_0$, then $y \in D_t$ and $(y,t)$ is not a surgery point.
\end{Claim2}

\begin{proof}
Assume that $y \in \MM (t) \setminus D_t$.
Then there is a space-time curve $\gamma : [t, t_0] \to \MM$ with $\LL(\gamma) < \Lambda r_0$ which either touches $\partial \MM$ or a surgery point.
The first case is excluded by assumption (i), so $\gamma$ touches a surgery point.
Now the Claim follows from \cite[5.3]{PerelmanII}, \cite[Lemma 79.3]{KLnotes}, \cite[p 92]{Bamler-diploma}.
Note that this Lemma is still true in the boundary case since by Definition \ref{Def:precisecutoff}(3) for every surgery point $(y', t') \in \MM$, we have $\dist_{t'} (y', \partial \MM(t') ) > c \delta^{-1}$ for some universal constant $c > 0$.
\end{proof}

Observe that for all $t \in [t_0 - r_0^2, t_0]$, we have
\[ \scal (\cdot, t) \geq - \frac{3}{2t} \geq - 3 r_0^{-2}. \]
So
\[ \ov{L}(\cdot, t) \geq - 6 \sqrt{t_0 - t} \int_{t_0 - t}^{t_0} r_0^{-2} \sqrt{t_0 - t'} dt' = - 4 r_0^{-2} (t_0 - t)^2. \]
Hence, for $t \in [t_0 - \frac12 r_0^2, t_0]$ we have 
\[ \widehat{L}(\cdot, t) = \ov{L}(\cdot, t) + 2 r_0 \sqrt{t_0 - t} > 0. \]
Let $\phi$ be a cutoff function which is constantly equal to $1$ on $(-\infty, \frac1{20}]$ and $\infty$ on $[\frac1{10}, \infty)$ and satisfies
\[ 2 \frac{(\phi')^2}{\phi} - \phi '' \geq (2A + 300) \phi' - C(A) \phi. \]
Here $C(A) < \infty$ is a positive constant which only depends on $A$.
Then set for all $t \in [t_0 - \frac12 r_0^2, t_0]$ and $y \in \MM (t)$
\[ h(y, t) = \phi \big( r_0^{-1} \dist_t (x_0, y) - A(1 - 2 r_0^{-2} (t_0 -  t) ) \big) \widehat{L}(y,t). \]
So $h(\cdot, t)$ is infinite outside $B(x_0, t, A(1 - 2(t_0 - t) r_0^{-2}) r_0 + \frac1{10} r_0) \subset \MM (t) \setminus \partial \MM (t)$ (compare with assumption (ii)) and hence it attains a minimum $h_0(t)$ at some interior point $y \in \MM (t)$.

Assume first that $h(y, t) < 2 r_0 \sqrt{t_0 - t} \exp (C(A) + 100)$.
So $L(y, t) < r_0 \exp \linebreak[2] (C(A) + \linebreak[1] 100)$.
Then by Claim 2, assuming $\delta < \delta^* (\exp (C(A) + 100), r, \eta)$,  and $Z > \exp (C(A) + 100)$ we have $y \in D_t$ and we can compute (cf \cite[6.3]{PerelmanII}, \cite[sec 85]{KLnotes})
\[ r_0^2 \Big( \frac{\partial}{\partial t} - \triangle \Big) h(y,t) \geq - C(A) h(y,t) - \Big( 6 + \frac{r_0}{\sqrt{t_0 - t}} \Big) \phi r_0^2.  \]
So we have
\[ r_0^2 \frac{d}{d t} \Big( \log \frac{h_0(t)}{\sqrt{t_0 - t}} \Big) \geq - C(A) - \frac{50 r_0}{\sqrt{t_0 - t}} \]
if $h_0(t) < 2 r_0 \sqrt{t_0 - t} \exp (C(A) + 100)$.

Since $\frac{h_0(t)}{r_0 \sqrt{t_0 - t}} \to 2$ for $t \to t_0$, we find that if $h_0 (t') < 2 r_0 \sqrt{t_0 - t'} \exp (C(A) + 100)$ for all $t' \in [t, t_0]$, then
\begin{multline*}
 h_0 (t) \leq 2 r_0 \sqrt{t_0 - t} \exp \big( C(A) r_0^{-2} (t_0 - t) + 100 r_0^{-1} \sqrt{t_0 - t} \big) \\
  \leq 2 r_0 \sqrt{t_0 - t} \exp (C(A) + 100).
\end{multline*}
This implies that the assumption $h_0 (t) < 2 r_0 \sqrt{t_0 - t} \exp (C(A) + 100)$ is actually satisfied for all $t \in [t_0 - \frac12 r_0^2, t_0]$.
So we can find a $y \in B(x_0, t_0 - \frac12 r_0^2, \frac1{10} r_0)$ such that $L(y, t_0 - \frac12 r_0^2) < r_0 \exp (C(A) + 100) = C'(A) r_0$.

Since by length distortion estimates $B(x_0, t_0 - \frac12 r_0^2, \frac1{10} r_0) \subset B(x_0, t_0, \frac12 r_0)$, we conclude by joining paths that for all $x \in B(x_0, t_0, r_0)$ we have $L(x, t_0 - r_0^2) < C''(A) r_0$.
So assuming $\delta < \delta^* (C''(A), r, \eta)$ and $Z > C''(A)$, we can use Claim 2 to conclude that $P(x_0, t_0, r_0, - r_0^2) \cap \MM(t_0 - r_0^2) \subset D_{t_0 - r_0^2}$ and we have
\[ \widetilde{V}(t_0 - r_0^2) > v(w, A) \]
for some constant $v(w, A) > 0$ which only depends on $A$ and $w$.
This gives us a uniform lower bound on $r_1^{-3} \vol_{t_0} B(x_1, t_0, r_1)$ (cf \cite[7.3]{PerelmanI}, \cite[Theorem 26.2]{KLnotes}, \cite[Lemma 4.2.3]{Bamler-diploma}).
\end{proof}

The noncollapsing result from Lemma \ref{Lem:6.3a} will be applied in Lemma \ref{Lem:6.3bc} below.
Before we continue, we introduce the following type of solutions which will be used as a model for singularities and for regions of high curvature.
The definition makes sense in all dimensions.

\begin{Definition}[$\kappa$-solution]
Let $\kappa > 0$.
An ancient solution to the Ricci flow $(M, (g_t)_{t \in (-\infty, 0]})$ is called a \emph{$\kappa$-solution} if
\begin{enumerate}[label=(\arabic*)]
\item The curvature is uniformly bounded on $M \times (-\infty, 0]$.
\item The metric on every time-slice is complete and has non-negative curvature operator.
\item The scalar curvature at time $0$ is positive.
\item At every point the scalar curvature is non-decreasing in time.
\item The solution is $\kappa$-noncollapsed on all scales at all points.
\end{enumerate}
\end{Definition}
We mention that there is a $\kappa_0 > 0$ such that every $3$ dimensional $\kappa$-solution which is not round, is in fact a $\kappa_0$-solution (cf \cite[11.9]{PerelmanI}, \cite[Proposition 50.1]{KLnotes}).
$\kappa$-solutions can be used to detect strong $\varepsilon$-necks or $(\varepsilon, E)$-caps or more generally to verify the canonical neighborhood assumptions as explained in the next Lemma.

\begin{Lemma} \label{Lem:kappasolCNA}
There is an $\eta > 0$ and for any $\varepsilon > 0$ there is an $E = E(\varepsilon) < \infty$ such that for every orientable $3$ dimensional $\kappa$-solution $(M, (g_t)_{t \in (-\infty, 0]})$ the following holds:
For all $r > 0$, the canonical neighborhood assumptions $CNA (r, \varepsilon, E, \eta)$ hold everywhere on $M \times (- \infty, 0]$.
More precisely, either $M$ is a spherical space form or for any $(x,t) \in M \times (-\infty, 0]$ we have:
\begin{enumerate}[label=(\alph*)]
\item $(x,t)$ is a center of a strong $\varepsilon$-neck or an $(\varepsilon, E)$-cap $U \subset M$. \\ 
If $U \subset \IR P^3 \setminus \ov{B}^3$, then there is a double cover of $M$ such that any lift of $(x,t)$ is the center of a strong $\varepsilon$-neck.
\item $|\nabla |{\Rm}|^{-1/2}| (x,t) < \eta^{-1}$ and $| \partial_t |{\Rm}|^{-1} | (x,t) < \eta^{-1}$.
\item $\vol_t B(x,t, r') > \eta r^3$ for all $0 < r' < |{\Rm}|^{-1/2} (x,t)$.
\end{enumerate}
\end{Lemma}

\begin{proof}
See \cite[11.8]{PerelmanI}, \cite[Corollary 48.1]{KLnotes}, \cite[Theorem 5.4.11]{Bamler-diploma}.
\end{proof}

The following Lemma will enable us to identify $\kappa$-solutions as limits of Ricci flows with surgeries under very weak curvature bound assumptions.

\begin{Lemma} \label{Lem:lmiitswithCNA}
There is an $\varepsilon_0 > 0$ such that:
Let $\MM^\alpha$ be a sequence of $3$-dimensional Ricci flows with surgery on the time-intervals $[- \tau_0^\alpha, 0]$, $\tau^\alpha \leq \tau_0^\alpha$, $x^\alpha_0 \in \MM(0)$ a sequence of basepoints which survive until time $\tau^\alpha$, and $a^\alpha \to \infty$ a sequence of positive numbers such that for $P^\alpha = \{ (x,t) \in \MM^\alpha \;\; : \;\; t \in [- \tau^\alpha, 0], \; \dist_t(x_0^\alpha, x) < a^\alpha \}$ the following statements hold
\begin{enumerate}[label=(\roman*)]
\item the ball $B(x_0^\alpha, t, a^\alpha)$ is relatively compact in $\MM^\alpha (t)$ and does not hit the boundary $\partial \MM^\alpha (t)$ for all $t \in [- \tau^\alpha, 0]$,
\item $|{\Rm}|(x^\alpha,0) \leq 1$,
\item the curvature on $P^\alpha$ is $\varphi^\alpha$-positive for some $\varphi^\alpha \to 0$,
\item all points of $P^\alpha$ are $\kappa$-noncollapsed on scales $< a^\alpha$ for some uniform $\kappa > 0$,
\item all points on $P^\alpha$ satisfy the canonical neighborhood assumptions $CNA(\frac12, \linebreak[1] \varepsilon_0, \linebreak[1] E, \eta)$ for some uniform $E, \eta > 0$,
\item there is a sequence $K^\alpha \to \infty$ such that for every surgery point $(x', t') \in P^\alpha$ we have $R(x', t') > K^\alpha$.
\end{enumerate}
Then the pointed Ricci flows with surgery $(\MM^\alpha, (x_0^\alpha, 0))$ subconverge to some non-singular Ricci flow $(M^\infty, (g_t^\infty)_{t \in (- \tau^\infty, 0]}, (x_0^\infty,0))$ where $\tau^\infty = \limsup_{\alpha \to \infty} \tau^\alpha$.
Moreover, this limiting Ricci flow has complete time-slices and bounded, non-negative sectional curvature.
If $\tau^\infty = \infty$ and $|{\Rm^\infty}|(x^\infty, 0) > 0$, then $(M^\infty, \linebreak[1] (g_t^\infty)_{t \in (- \infty, 0]})$ is a $\kappa$-solution.
\end{Lemma}

\begin{proof}
See \cite[Proposition 6.3.1]{Bamler-diploma}, \cite[4.2]{PerelmanII} or the proofs of \cite[12.1]{PerelmanI} or \cite[Theorem 52.7]{KLnotes}.
\end{proof}

We now state the second main Lemma.

\begin{Lemma}[cf \hbox{\cite[6.3(b)+(c)]{PerelmanII}}] \label{Lem:6.3bc}
There are constants $\eta_0, \varepsilon_0 > 0$ and for every $\varepsilon \in (0, \varepsilon_0]$ and $E < \infty$ there is a constant $E_0 = E_0(\varepsilon) < \infty$ such that: \\
For any $1 \leq A < \infty$, $w, r > 0$, $\eta \in (0, \eta_0]$ and $E \in [E_0, \infty)$ there are constants $K = K(w, A, E, \eta), Z = Z(A) < \infty$ and $\td\rho = \td\rho (w, A, \varepsilon, E, \eta), \ov{r}=\ov{r}(A, w, E, \eta), \delta = \delta (w, A, r, \varepsilon, E, \eta) > 0$ such that: \\
Let $r_0^2 < t_0/2$ and let $\MM$ be a Ricci flow with surgery (whose time-slices are allowed to have boundary) on the time-interval $[ t_0 - r_0^2, t_0]$ which is performed by $\delta$-precise cutoff and consider a point $x_0 \in \MM(t_0)$.
Assume that the canonical neighborhood assumptions $CNA (r, \varepsilon, E, \eta)$ hold on $\MM$.
We also assume that the curvature on $\MM$ is uniformly bounded on compact time-intervals which don't contain surgery times and that all time-slices of $\MM$ are complete.

Assume that the parabolic neighborhood $P(x_0, t_0, r_0, -r_0^2)$ is non-singular, that $|{\Rm}| \leq r_0^{-2}$ on $P(x_0, t_0, r_0, - r_0^2)$ and $\vol_{t_0} B(x_0, t_0, r_0) \geq w r_0^3$.

In the case in which some time-slices of $\MM$ have non-empty boundary, we assume that
\begin{enumerate}[label=(\roman*)]
\item every space-time curve $\gamma : [t_1, t_2] \to \MM$ with $t_2 \in [t_0 - \frac1{10} r_0^2, t_0]$ and $\gamma(t_2) \in B(x_0, t_2, (A+1) r_0)$ which meets the boundary $\partial \MM$ somewhere, has $\LL(\gamma) > Z r_0$ (based in $t_2$),
\item for all $t \in [t_0 - \frac15 r_0^2, t_0]$, the ball $B(x_0, t, 2 (A+3)r_0 + r)$ does not meet the boundary $\partial \MM (t)$.
\end{enumerate}

Then
\begin{enumerate}[label=(\alph*)]
\item Every point $x \in B(x_0, t_0, A r_0)$ satisfies the canonical neighborhood assumptions $CNA( \td\rho r_0, \varepsilon, E , \eta )$.
\item If $r_0 \leq \ov{r} \sqrt{t_0}$, then $|{\Rm}| \leq K r_0^{-2}$ on $B(x_0, t_0, A r_0)$.
\end{enumerate}
\end{Lemma}

\begin{proof}
By choosing $\td\rho$ small and $K$ large enough we can again exclude the case in which the component of $\MM(t)$ which contains $x_0$ has positive, $E$-pinched sectional curvature for some time $t \leq t_0$.

We first establish part (a).
Choose $\eta_0$ and $E_0 = E_0 (\varepsilon)$ to be strictly less/larger than the constants $\eta$, $E(\varepsilon)$ in Lemma \ref{Lem:kappasolCNA}.
Assume now that given some small $\td\rho$, there is a point $x \in B(x_0, t_0, A r_0)$ such that $(x, t_0)$ does not satisfy the canonical neighborhood assumptions $CNA(\td\rho r_0, \varepsilon, E, \eta)$, i.e. we have $|{\Rm}|(x, t_0) \geq \td{\rho}^{-2} r_0^{-2}$ and $(x, t_0)$ does not satisfy the assumptions (1)--(3) in Definition \ref{Def:CNA}.
Set for $\ov{t} \in [t_0 - r_0^2, t_0]$, $\ov{x} \in \MM(\ov{t})$
\begin{multline*}
 P_{\ov{x}, \ov{t}} = \big\{ (y,t) \in \MM \;\; : \;\; t \in  [\ov{t} - \tfrac1{20} \td{\rho}^{-2} |{\Rm}|^{-1} (\ov{x},\ov{t}), \ov{t}], \;\; y \in \MM(t), \\
 \;\; \dist_t (x_0, y) \leq \dist_{\ov{t}}(x_0, \ov{x}) + \tfrac14 \td{\rho}^{-1} |{\Rm}|^{-1/2}(\ov{x},\ov{t}) \big\}.
\end{multline*}
We will now find a  particular $(\ov{x}, \ov{t}) \in \MM$ with $\ov{t} \in [t_0 - \frac1{10} r_0^2, t_0]$ and $\ov{x} \in B(x_0, \ov{t}, (A+\frac12) r_0)$ by a point-picking process:
Set first $(\ov{x}, \ov{t}) = (x, t_0)$.
Let $\ov{q} = |{\Rm}|^{-1/2} (\ov{x}, \ov{t}) \leq \td{\rho} r_0$.
If every $(x',t') \in P_{\ov{x}, \ov{t}}$ satisfies the canonical neighborhood assumptions $CNA(\frac12 \ov{q}, \varepsilon, E, \eta)$, then we stop.
If not, we replace $(\ov{x}, \ov{t})$ by such a counterexample an start over.
In every step of this algorithm, $\ov{q}$ decreases by at least a factor of $\frac12$ which implies that the algorithm has to terminate after a finite number of steps since after a finite number of steps we have $\ov{q} < r$ and we can make use of the canonical neighborhood assumptions $CNA(r, \varepsilon, E, \eta)$ given in the assumptions of Lemma.
So the algorithm yields an $(\ov{x}, \ov{t}) \in \MM$ and a $\ov{q} = |{\Rm}|^{-1/2} (\ov{x}, \ov{t}) \leq \td\rho r_0$ such that $(\ov{x}, \ov{t})$ does not satisfy the canonical neighborhood assumptions $CNA (\ov{q}, \varepsilon, E, \eta)$, but all points in $P_{\ov{x}, \ov{t}}$ satisfy the canonical neighborhood assumptions $CNA( \frac12 \ov{q}, \varepsilon, E, \eta)$.
By convergence of the geometric series, we conclude $\ov{t} - \frac1{20} \td\rho^{-2} \ov{q}^2 \in [t_0 - \frac1{10} r_0^2, t_0]$ and $\dist_{\ov{t}}(x_0, \ov{x}) < (A + \frac12) r_0$.
Moreover, for all $(x', t') \in P_{\ov{x}, \ov{t}}$ we have $\dist_{t'} (x_0, x') < (A+1) r_0$.

Now assume that for fixed parameters $w, A, \varepsilon, E, \eta$ there is no $\td\rho$ such that assertion (a) holds for a constant $Z$ and a constant $\delta$ which can additionally depend on $r$.
Then we can find a sequence $\td\rho^\alpha \to 0$ and a sequence of counterexamples $\MM^\alpha$, $t_0^\alpha$, $r_0^\alpha$, $x_0^\alpha$ together with parameters $r^\alpha$ which satisfy the assumptions of the Lemma, but there are points $x^\alpha \in B(x_0^\alpha, t_0^\alpha, A r_0^\alpha)$ such that $(x^\alpha, t_0^\alpha)$ don't satisfy the canonical neighborhood assumptions $CNA (\td{\rho}^\alpha r_0, \varepsilon, E^\alpha, \eta)$.
The choice of the constant $\delta^\alpha$ will be explicit and arise from Lemma \ref{Lem:6.3a}.
We will also assume that $\delta^\alpha / r^\alpha \to 0$ for $\alpha \to \infty$.

First, let $(\ov{x}^\alpha, \ov{t}^\alpha)$ and $\ov{q}^\alpha$ be the points and the constant obtained by the algorithm from the last paragraph.
We now apply Lemma \ref{Lem:6.3a} with 
\begin{multline*}
 r_0 \leftarrow \tfrac1{10} r_0^\alpha, \; x_0 \leftarrow x_0^\alpha, \; t_0 \leftarrow t \in [\ov{t}^\alpha - \tfrac1{20} (\td\rho^\alpha)^{-2} (\ov{q}^\alpha)^{2}, \ov{t}^\alpha], \\ \; w \leftarrow c w, \; A \leftarrow 10(A+1), \; r \leftarrow r^\alpha, \; E \leftarrow E^\alpha
\end{multline*}
to conclude that for sufficiently large $Z$ and small $\delta^\alpha$ (depending on $w, A, r^\alpha, \eta$) for any $(x',t') \in \MM^\alpha$ with $t' \in [\ov{t}^\alpha - \frac1{20} (\td\rho^\alpha)^{-2} (\ov{q}^\alpha)^2, \ov{t}^\alpha]$ and $x' \in B(x_0^\alpha, t', (A+1) r_0^\alpha)$ is $\kappa$-noncollapsed for some uniform $\kappa > 0$ on scales less than $\frac1{10} r_0^\alpha$.
This implies that the points on $P_{\ov{x}^\alpha, \ov{t}^\alpha}$ are $\kappa$-noncollapsed on scales less than $\frac1{10} r_0^\alpha$.

Observe that the assumption on $\delta^\alpha$ and Definition \ref{Def:precisecutoff}(3) imply that there is a universal constant $c' > 0$ such that for every surgery point $(x',t') \in \MM^\alpha$ we have 
\begin{equation} \label{eq:surgpointhighcurv}
 |{\Rm}|(x',t') > c' (\delta^\alpha)^{-2} > c' \Big( \frac{\delta^\alpha}{r^\alpha} \Big)^{-2} (r^\alpha)^{-2} > c' \Big( \frac{\delta^\alpha}{r^\alpha} \Big)^{-2} (\ov{q}^\alpha)^{-2}.
\end{equation}
Here we have made use of the inequality $r^\alpha < \ov{q}^\alpha$ which follows from the fact that the point $(\ov{x}^\alpha, \ov{t}^\alpha)$ satisfies the canonical neighborhood assumptions $CNA (r^\alpha, \linebreak[1] \varepsilon, \linebreak[1] E, \linebreak[1] \eta)$, but not $CNA (\ov{q}^\alpha, \varepsilon, E, \eta)$.

By Lemma \ref{Lem:shortrangebounds}(a), we have $|{\Rm}| < 8 (\ov{q}^\alpha)^{-2}$ on $P(\ov{x}^\alpha, \ov{t}^\alpha, c \ov{q}^\alpha, - c (\ov{q}^\alpha)^2)$ where $c = \frac{\eta}{40}$.
Using (\ref{eq:surgpointhighcurv}), we conclude that this parabolic neighborhood is non-singular for large $\alpha$.
Choose now $\tau \geq 0$ maximal with the property that for all $\tau' < \tau$ there is some $D_{\tau'} < \infty$ such that for infinitely many $\alpha$ the point $\ov{x}^\alpha$ survives until time $\ov{t}^\alpha - \tau' (\ov{q}^\alpha)^2$ and we have $|{\Rm}| (x,t) \leq D_{\tau'} (\ov{q}^\alpha)^{-2}$ whenever $t \in [\ov{t}^\alpha - \tau' (\ov{q}^\alpha)^2, \ov{t}^\alpha]$ and $\dist_t (\ov{x}^\alpha, x) < \frac{c}2 \ov{q}^\alpha$ (so by (\ref{eq:surgpointhighcurv}) none of these points is a surgery point for infinitely many $\alpha$).
After passing to a subsequence, we can assume that for all $\tau' < \tau$ this property even holds for \emph{sufficiently large} $\alpha$.
Obviously, $\tau > 0$ by the result at the beginning of the paragraph.
By distance distortion estimates (Lemma \ref{Lem:distdistortion}(b)) we then obtain for all such $\alpha$ and all $t \in [\ov{t}^\alpha - \tau' (\ov{q}^\alpha)^2, \ov{t}^\alpha]$
\[ \frac{d}{dt} \dist_t (x_0^\alpha, \ov{x}^\alpha) \geq - C \sqrt{D_{\tau'}} (\ov{q}^\alpha)^{-1} . \]
So for all $t \in [\ov{t}^\alpha - \tau' (\ov{q}^\alpha)^2, \ov{t}^\alpha]$
\[ \dist_t (x_0^\alpha, \ov{x}^\alpha) < \dist_{\ov{t}^\alpha}(x_0^\alpha, \ov{x}^\alpha) + C \tau' \sqrt{D_{\tau'}}  \ov{q}^\alpha. \]
This implies that for every $a < \infty$ and $\tau' < \tau$ we have $B(x_0^\alpha, t, a \ov{q}^\alpha) \subset P_{\ov{x}^\alpha, \ov{t}^\alpha}$ for all $t \in [\ov{t}^\alpha - \tau' (\ov{q}^\alpha)^2, \ov{t}^\alpha]$ for large enough $\alpha$.
Hence we can find sequences $a^\alpha \to \infty$ and $\tau^\alpha \to \tau$ such that $B(\ov{x}^\alpha, t, a^\alpha \ov{q}^\alpha) \subset P_{\ov{x}^\alpha, \ov{t}^\alpha}$ for all $\alpha$ and $t \in [\ov{t}^\alpha - \tau^\alpha (\ov{q}^\alpha)^2, \ov{t}^\alpha]$.

After rescaling by $(\ov{q}^\alpha)^{-1}$, the Ricci flows restricted to the time-interval $[\ov{t}^\alpha - \tau^\alpha (\ov{q}^\alpha)^2, \ov{t}^\alpha]$ satisfy the assumptions of Lemma \ref{Lem:lmiitswithCNA} (we also need to make use of assumption (ii) here) and hence they subconverge to some non-singular Ricci flow on $M_\infty \times (- \tau, 0]$ of bounded curvature.
Using Lemma \ref{Lem:shortrangebounds} we can argue that if $\tau$ was finite, then by its choice we could increase it and hence we must have $\tau = \infty$.
So $M_\infty \times (- \infty, 0]$ is a $\kappa$-solution.
Using Lemma \ref{Lem:kappasolCNA}, we finally obtain a contradiction to the assumption that the points $(\ov{x}^\alpha, \ov{t}^\alpha)$ do not satisfy the canonical neighborhood assumptions $CNA(\ov{q}^\alpha, \varepsilon, E, \eta)$.

Part (b) follows exactly the same way as in \cite[6.3]{PerelmanII}.
See also \cite[Lemma 70.2]{KLnotes} and \cite[Proposition 6.2.4]{Bamler-diploma}.
The boundary $\partial \MM (t_0)$ does not create any issues since it is far enough away from $x_0$.
\end{proof}

We now prepare for the proof of the next main result, Lemma \ref{Lem:6.4}.
We believe that we have to alter the following Lemma with respect to \cite[6.5]{PerelmanII} to make it's proof work.
\begin{Lemma}[\hbox{\cite[6.5]{PerelmanII}}] \label{Lem:6.5}
For all $w > 0$ there exists $\tau_0 = \tau_0(w) > 0$ and $K_0 = K_0(w) < \infty$, such that: \\
Let $(g_t)_{t \in (-\tau,0]}$ be a smooth solution to the Ricci flow on a non-singular parabolic neighborhood $P(x_0, 0, 1, -\tau)$, $\tau \leq \tau_0$.
Assume that $\sec \geq -1$ on $B = P(x_0, 0, 1, -\tau) \cap \bigcup_{t \in [-\tau,0]} B(x_0, t, 1)$ and $\vol_0 B(x_0, 0, 1) \geq w$.
Then
\begin{enumerate}[label=(\alph*)]
\item $|{\Rm}| \leq K_0 \tau^{-1}$ in $P(x_0, 0, \frac14, -\tau/2)$.
\item $B(x_0, - \tau, \frac14)$ is relatively compact in $B(x_0, 0, 1)$ and
\item $\vol_{-\tau} B(x_0, -\tau, \frac14) > \frac12 w (\frac14)^3$.
\end{enumerate}
\end{Lemma}
\begin{proof}
See \cite[Lemma 82.1]{KLnotes} for a proof of the first part and the proof of \cite[Corollary 45.1(b)]{KLnotes} for the third.
The second part is due to the lower bound on the sectional curvature.
\end{proof}

\begin{Lemma}[\hbox{\cite[6.6]{PerelmanII}}] \label{Lem:6.6}
For any $w > 0$ there is a $\theta_0 = \theta_0 (w) > 0$ such that:
Let $(M, g)$ be a Riemannian $3$-manifold and $B(x,1) \subset M$ a ball of volume at least $w$, which is relatively compact and does not meet the boundary of $M$.
Assume that $\sec \geq - 1$ on $B(x, 1)$.
Then there exists a ball $B(y,\theta_0) \subset B(x,1)$, such that every subball $B(z,r) \subset B(y, \theta_0)$ of any radius $r$ has volume at least $\frac1{10} r^3$.
\end{Lemma}
\begin{proof}
See \cite[Lemma 83.1]{KLnotes}.
\end{proof}

\begin{Lemma} \label{Lem:HIbound}
For any $K < \infty$ there is an $\ov{r} = \ov{r} (K) < \infty$ such that:
Let $r_0 \leq \ov{r} \sqrt{t_0}$ and $\frac12 t_0 \leq t \leq t_0$.
Assume that $(M, g)$ is a Riemannian manifold of $t^{-1}$-positive curvature and $|{\Rm}| < K r_0^{-2}$ on $M$.
Then the sectional curvature is bounded from below: $\sec \geq - \frac12 r_0^{-2}$.
\end{Lemma}

\begin{proof}
The claim is clear for $r_0 = 1$.
The rest follows from rescaling.
\end{proof}

\begin{Lemma}[\hbox{\cite[6.4]{PerelmanII}}] \label{Lem:6.4}
There is a constant $\varepsilon_0 > 0$ such that for all $r, \eta > 0$ and $E < \infty$ there are constants $\tau = \tau(\eta, E), \ov{r} = \ov{r} (\eta, E), \delta = \delta (r, \eta, E) > 0$ and $K = K(\eta, E),  C_1 = C_1 (\eta, E), Z = Z( \eta, E) < \infty$ such that: \\
Let $r_0^2 < t_0 /2$ and let $\MM$ be a Ricci flow with surgery (whose time-slices are allowed to have boundary) on the time-interval $[ t_0 - r_0^2, t_0]$ which is performed by $\delta$-precise cutoff and consider a point $x_0 \in \MM(t_0)$.
Assume that the canonical neighborhood assumptions $CNA (r, \varepsilon_0, E, \eta)$ hold on $\MM$.
We also assume that the curvature on $\MM$ is uniformly bounded on compact time-intervals which don't contain surgery times and that all time-slices of $\MM$ are complete.

In the case in which some time-slices of $\MM$ have non-empty boundary, we assume that
\begin{enumerate}[label=(\roman*)]
\item For all $t_1 < t_2 \in [t_0 - \frac1{10} r_0^2, t_0]$ we have: if some $x \in B(x_0, t_0, r_0)$ survives until time $t_2$ and $\gamma : [t_1, t_2] \to \MM$ is a space-time curve with endpoint $\gamma(t_2) \in B(x, t, 3 r_0)$ which meets the boundary $\partial\MM$ somewhere, then $\LL(\gamma) > Z r_0$ (where $\LL$ is based in $t_2$).
\item For all $t \in [t_0 - \frac1{10} r_0, t_0]$ we have: if some $x \in B(x_0, t_0, r_0)$ survives until time $t$, then $B(x, t, 5 r_0 + r)$ does not meet the boundary $\partial \MM (t)$.
\end{enumerate}
Now assume that 
\begin{enumerate}[label=(\roman*), start=3]
\item $C_1 \delta \leq r_0 \leq \ov{r} \sqrt{t_0}$,
\item $\sec \geq - r_0^{-2}$ on $B(x_0, t_0, r_0)$ and
\item $\vol_{t_0} B(x_0, t_0, r_0) \geq \frac1{10} r_0^3$.
\end{enumerate}
Then the parabolic neighborhood $P(x_0, t_0, \frac14 r_0, - \tau r_0^2)$ is non-singular and we have $|{\Rm}| < K r_0^{-2}$ on $P(x_0, t_0, \frac14 r_0, - \tau r_0^2)$.
\end{Lemma}
\begin{proof}
Before we start with the main argument, we first discuss the case in which $r_0 \leq r$.
We will first show that for a universal $K' = K' (E) < \infty$ and sufficiently small $\varepsilon_0$, we can guarantee that $|{\Rm}| < \frac12 K' r_0^{-2}$ on $B(x_0, t_0, \frac14 r_0)$.
The constant $K'$ and the smallness of the constant $\varepsilon_0$ will be determined in the course of this paragraph.
Assume the assumption was wrong, i.e. there is a point $x \in B(x_0, t_0, \frac14 r_0)$ such that $Q = |{\Rm}|(x,t_0) \geq \frac12 K' r_0^{-2}$.
By the canonical neighborhood assumptions $CNA (r, \varepsilon_0, E, \eta)$, we know that $(x,t_0)$ is either a center of a strong $\varepsilon_0$-neck or of an $(\varepsilon_0, E)$-cap or $\MM (t_0)$ has positive $E$-pinched curvature (here we assumed that $K' > 2$).
The latter case cannot occur by assumption (v) for large enough $K'$, so assume that $(x, t_0)$ is a center of a strong $\varepsilon_0$-neck or an $(\varepsilon_0, E)$-cap.
So in both cases there is a $y \in \MM(t_0)$ with $\dist_{t_0}(x, y) < E Q^{-1/2}$ such that $(y,t_0)$ is a center of an $\varepsilon_0$-neck and $E^{-1} Q < |{\Rm}|(y, t_0) < E Q$.
So for every $w > 0$ there is an $\varepsilon'_0 = \varepsilon'_0 (w) > 0$ and a $D = D(w) < \infty$ such that if $\varepsilon_0 < \varepsilon'_0$, then $\vol_{t_0} B(y, t_0, D Q^{-1/2}) < w D^3 Q^{-3/2}$.
Assuming $K' > 4 E^2$, we conclude that $y \in B(x_0, t_0, \frac12 r_0)$.
So by volume comparison, there is a universal constant $w_0 > 0$ such that $\vol_{t_0} B(y, t_0, d) \geq w_0 d^3$ for all $0 < d < \frac12 r_0$.
Assume now that $\varepsilon_0 < \varepsilon'_0 (w_0)$ and $K' > 4 D^2 (w_0)$.
Then we obtain a contradiction for $d = D(w_0) Q^{-1/2} < \frac12 r_0$.
So we indeed have $|{\Rm}| < \frac12 K' r_0^{-2}$ on $B(x_0, t_0, \frac14 r_0)$.
Choose now $C_1 = {K'}^{1/2}$.
Then by Lemma \ref{Lem:shortrangebounds} and the fact that at every surgery point we have $|{\Rm}| > c' \delta^{-2} > C_1^2 r_0^{-2} = K' r_0^{-2}$ at every surgery point (compare with (\ref{eq:surgpointhighcurv})), we conclude finally, that there are $\tau' = \tau' (\eta, E), \delta' = \delta' (\eta, E) > 0$ such that if $\delta \leq \delta'$ then $P(x_0, t_0, \frac14 r_0, - \tau' r_0^2)$ is non-singular and $|{\Rm}| < K' r_0^{-2}$ on $P(x_0, t_0, \frac14 r_0, - \tau' r_0^2)$.

Now we return to the general case.
We will first fix some constants:
Consider the constants $\tau_{0, \ref{Lem:6.5}}$ and $K_{0, \ref{Lem:6.5}}$ from Lemma \ref{Lem:6.5}, $\theta_{0, \ref{Lem:6.6}}$ from Lemma \ref{Lem:6.6}, $K_{\ref{Lem:6.3bc}}$, $\ov{r}_{\ref{Lem:6.3bc}}$, $Z_{\ref{Lem:6.3bc}}$ and $\delta_{\ref{Lem:6.3bc}}$ from Lemma \ref{Lem:6.3bc} and $\ov{r}_{\ref{Lem:HIbound}}$ from Lemma \ref{Lem:HIbound} and set:
\begin{alignat*}{1}
\tau &= \min \{ \tau', \tfrac12 \tau_{0, \ref{Lem:6.5}}(\tfrac1{10}), \tfrac1{100}\} \\
K &= \max\{ K', K_{0, \ref{Lem:6.5}} (\tfrac1{10}) \}, \\
\theta_0 &= \min \{ \tfrac14 \theta_{0, \ref{Lem:6.6}}(\tfrac1{20}), \linebreak[1] \tfrac1{10} \} \\
r^* &= \theta_0 \min \{ \tau^{1/2}, \linebreak[1] K^{-1/2}, \linebreak[1] \tfrac1{10} \} \\
K^* &= (r^*)^{-2} K_{\ref{Lem:6.3bc}}(\tfrac1{100}, 2 (r^*)^{-1}, E, \eta) \\
Z &= Z_{\ref{Lem:6.3bc}}(2 (r^*)^{-1}) \\
\ov{r} &= \min \{ \ov{r}_{\ref{Lem:6.3bc}}(\tfrac1{100}, 2 (r^*)^{-1}, E, \eta), \ov{r}_{\ref{Lem:HIbound}} (K^*) \} \\
\delta &= \min \{ \delta', C_1^{-1} \theta_0 r, \delta_{\ref{Lem:6.3bc}} ( \tfrac1{100}, 2 (r^*)^{-1}, \cdot, E, \eta), {c'}^{1/2} (K^*)^{-1/2} \}
\end{alignat*}
The constants $C_1$ and $\varepsilon_0$ from the first paragraph will be kept.
We will also assume that $\varepsilon_0$ is smaller than the constant from Lemma \ref{Lem:6.3bc}.

Assume that the conclusion of the Lemma is not true for some $x_0$, $t_0$ and $r_0$.
Then $r_0 > r$.
We first carry out a point-picking process.
In the first step set $x'_0 = x_0$, $t'_0 = t_0$ and $r'_0 = r_0$.
If there are $x''_0$, $t''_0$ and $r''_0$,  such that
\begin{enumerate}
\item  $t''_0 \in [t'_0 - 2\tau (r'_0)^2, t'_0]$,
\item the point $x'_0$ survives until time $t''_0$ and for all $t \in (t''_0, t'_0]$ there are no surgery points in $B(x'_0, t, r'_0)$ and $B(x'_0, t, r'_0) \cap \partial \MM (t) = \emptyset$,
\item $\sec \geq - (r'_0)^{-2}$ on $\bigcup_{t \in [t''_0, t'_0]} B(x'_0, t, r'_0)$,
\item $x''_0 \in B(x'_0, t''_0, r'_0/4)$,
\item $r''_0 = \theta_0 r'_0$,
\item $\vol_{t''_0} B(x''_0, t''_0, r''_0) \geq \frac1{10} (r''_0)^3$ and
\item we don't have $|{\Rm}| < K (r''_0)^{-2}$ on $P(x''_0, t''_0, \frac14 r''_0, - \tau (r''_0)^2)$ or the parabolic neighborhood $P(x''_0, t''_0, \frac14 r''_0, - \tau (r''_0)^2)$ is singular,
\end{enumerate}
then we replace $x'_0$, $t'_0$, $r'_0$ by $x''_0$, $t''_0$, $r''_0$ and repeat.
If not, we stop the process.
Observe that here and in the rest of the proof the parabolic neighborhoods are not assumed to be non-singular unless otherwise noted (compare with Definition \ref{Def:parabnbhd}).
Since by the choice of $\delta$ we have $C_1 \delta < \theta_0 r$, we find by the discussion at the beginning of the proof and conditions (2), (3), (6), (7) that we always have $r'_0 > r_0$.
So the process has to terminate after a finite number of steps and yield a triple $(x'_0, t'_0, r'_0)$ with $r'_0 > r_0$.

Observe that by the smallness of $\tau$, $\theta_0$ and distance distortion estimates, we have in every step of this process
\[ P(x''_0, t''_0, r''_0, - \tfrac1{10} (r''_0)^2) \subset P(x'_0, t'_0, r'_0, - \tfrac1{10} (r'_0)^2). \]
So the parabolic neighborhoods of each step are nested have for the final triple $(x'_0, t'_0, r'_0)$
\[  P(x'_0, t'_0, r'_0, - \tfrac1{10} (r'_0)^2) \subset P(x_0, t_0, r_0, - \tfrac1{10} (r_0)^2). \]
So the triple $(x'_0, t'_0, r'_0)$ satisfies assumptions (i) and (ii) of the Lemma.
By (3) and (6), also assumptions (iv) and (v) are satisfied and by the choice of $\delta$, assumption (iii) holds.
However, by (7) the assertion of the Lemma fails for the triple $(x'_0, t'_0, r'_0)$.
Thus, without loss of generality, we can assume that we have $x_0 = x'_0$, $t_0 = t'_0$ and $r_0 = r'_0$ and add to our assumptions that whenever we find $x''_0$, $t''_0$ and $r''_0$ satisfying the assumptions (1)--(6) above, assumption (7) cannot be satisfied (and hence we have curvature control on $P(x''_0, t''_0, \frac14 r''_0, - \tau (r''_0)^2)$).

Now let $\ov{\tau} \leq 2\tau$ be maximal with the property that
\begin{enumerate}[label=$-$]
\item the point $x_0$ survives until time $t_0 - \ov\tau r_0^2$,
\item for all $t \in (t_0 - \ov\tau r_0^2, t_0]$, there are no surgery points in $B(x_0, t, r_0)$ and $B(x_0, t, r_0) \cap \partial\MM(t) = \emptyset$,
\item $\sec \geq - r_0^{-2}$ on $\bigcup_{t \in [t_0 - \ov{\tau} r_0^2, t_0]} B(x_0, t, r_0)$.
\end{enumerate}
If $\ov{\tau} = 2\tau$, then we can conclude the Lemma using Lemma \ref{Lem:6.5}.

So assume now $\ov{\tau} < 2\tau$.
We will show that we then have curvature control at times $[t_0 - \ov\tau r_0^2, t_0]$ which implies a better lower bound on the sectional curvature and hence contradicts the maximality of $\ov\tau$.
Fix for a moment $t \in [t_0 - \ov\tau r_0^2, t_0]$.
By Lemma \ref{Lem:6.5} we first find $\vol_t B(x_0, t, \frac14 r_0) > \frac1{20} (\frac14)^3 r_0^3$.
Hence using Lemma \ref{Lem:6.6}, we can find a ball $B(y, t, \theta_0 r_0) \subset B(x_0, t, \frac14 r_0)$ such that $\vol_t B(y, t, \theta_0 r_0) > \frac1{10} \theta_0^3 r_0^3$.
So the triple $(y, t, \theta_0 r_0)$ satisfies the assumptions (1)--(6) above and hence by choice of the triple $(x_0, t_0, r_0)$, we find that the parabolic neighborhood $P(y, t, \tfrac14 \theta_0 r_0, -\tau \theta_0^2 r_0^2)$ is non-singular and
\[ |{\Rm}| < K \theta_0^{-2} r_0^{-2} \qquad \text{on} \qquad P(y, t, \tfrac14 \theta_0 r_0, - \tau \theta_0^2 r_0^2). \]

This implies that $|{\Rm}| < (r^* r_0)^{-2}$ on $P(y, t, r^* r_0, - (r^* r_0)^2)$.
Applying Lemma \ref{Lem:6.3bc}(b) for $x_0 \leftarrow y$, $t_0 \leftarrow t$, $r_0 \leftarrow r^* r_0$, $w \leftarrow \frac1{100}$, $A \leftarrow A^* = 2 (r^*)^{-1}$, $E \leftarrow E$, $\eta \leftarrow \eta$, $r \leftarrow r$ yields
\[ |{\Rm}|(\cdot, t) < K^* r_0^{-2} \qquad \text{on} \qquad B(y, t, 2 r_0) \qquad \text{for all} \qquad t \in [t_0 - \ov\tau r_0^2, t_0]. \] 
Observe here that by the right choice of $Z$ and assumptions (i), (ii) of this Lemma, the assumptions (i), (ii) of Lemma \ref{Lem:6.3bc} are satisfied.
We conclude that
\[ |{\Rm}|(\cdot, t) < K^* r_0^{-2} \quad \text{on} \quad B(x_0, t, r_0) \quad \text{for all} \quad t \in [t_0 - \ov{\tau} r_0^2, t_0]. \]
In particular by (\ref{eq:surgpointhighcurv}) and the choice of $\delta$, there are no surgery points on $B(x_0, t, r_0)$ for all $t \in [t_0 - \ov{\tau} r_0^2 - (r^* r_0)^2, t_0]$.
By Lemma \ref{Lem:HIbound}, this curvature bound implies $\sec \geq - \frac12 r_0^{-2}$ on $B(x_0, t, r_0)$ for all $t \in [t_0 - \ov{\tau} r_0^2, t_0]$.
So the point $x_0$ survives until some time which is strictly smaller than $t_0 - \ov\tau r_0^2$ and $B(x_0, t, r_0)$ does not contain surgery points or meet the boundary for times which are slightly smaller than $t_0 - \ov\tau r_0^2$.
This contradicts the maximality of $\ov{\tau}$.
\end{proof}

\begin{proof}[Proof of Proposition \ref{Prop:genPerelman}]
By Lemma \ref{Lem:6.6} we can find a ball $B(y, t_0, \theta_0(w) r_0) \subset B(x_0, t_0, r_0)$ with $\vol_{t_0} B(y, t_0, \theta_0(w) r_0) > \frac1{10} (\theta_0 r_0)^3$.
So we can apply Lemma \ref{Lem:6.4} with $t_0 \leftarrow t_0$, $x_0 \leftarrow y$, $r_0 \leftarrow \theta_0 r_0$, $\varepsilon_0 \leftarrow \varepsilon_0$, $E \leftarrow E$, $\eta \leftarrow \eta$, $r \leftarrow r$ and obtain that if 
\[ C_{1, \ref{Lem:6.4}} (\eta, E) \delta \leq \theta_0 r_0, \]
$\delta < \delta_{\ref{Lem:6.4}} (r, \eta, E)$, $Z > Z_{\ref{Lem:6.4}} (r, \eta, E)$ and $r_0 < \ov{r}_{\ref{Lem:6.4}} (\eta, E) \sqrt{t_0}$ then the parabolic neighborhood $P(y, t_0, \frac14 \theta_0 r_0, - \tau_{\ref{Lem:6.4}}(\eta, E) \theta_0^2 r_0^2)$ is non-singular and
\[ |{\Rm}| < K_{\ref{Lem:6.4}} (\eta, E) \theta_0^{-2} r_0^{-2} \qquad \text{on} \qquad P(y, t_0, \tfrac14 \theta_0 r_0, - \tau_{\ref{Lem:6.4}} \theta_0^2 r_0^2). \]
Now choose $r^* = r^*(w, \eta, E)> 0$ so small that $P(y, t, r^* r_0, (r^* r_0)^2) \subset P(y, t_0, \linebreak[1] \frac14 \theta_0 r_0 , \linebreak[1] - \tau_{\ref{Lem:6.4}} \theta^2_0 r_0^2)$ for all $t \in [t_0 - (r^* r_0)^2, t_0]$ and $|{\Rm}| < (r^* r_0)^{-2}$ there.
We can then invoke Lemma \ref{Lem:6.3bc}(b) with $t_0 \leftarrow t \in [t_0 - (r^* r_0)^2, t_0]$, $x_0 \leftarrow y$, $r_0 \leftarrow r^* r_0$, $w \leftarrow \frac1{100}$,  $A \leftarrow (A+2) (r^*)^{-1}$, $r \leftarrow r$, $\eta \leftarrow \eta$, $E \leftarrow E$ and obtain that if $\delta < \delta_{\ref{Lem:6.3bc}} (\frac1{100}, (A+2) (r^*)^{-1}, r, \cdot, E, \eta)$, $Z > Z_{\ref{Lem:6.3bc}} ( (A+2) (r^*)^{-1})$ and $r_0 < \ov{r}_{\ref{Lem:6.3bc}} ( (A+2) (r^*)^{-1}, \frac1{100}, E, \eta) \sqrt{t_0}$, then
\[ |{\Rm}| < K r_0^{-2} \qquad \text{on} \qquad B(y, t, (A+ 2)r_0) \qquad \text{for all} \qquad t \in [t_0 - (r^* r_0)^2, t_0] \]
for $K = K_{\ref{Lem:6.3bc}} (\frac1{100}, (A+2) (r^*)^{-1}, E, \eta) (r^*)^{-2} < \infty$.
As in (\ref{eq:surgpointhighcurv}), we conclude that if $C'_1 \delta \leq r_0$ for some $C'_1 = C'_1 (w, A, \eta, E) < \infty$, then there are no surgery points in $B(y, t, (A+ 2)r_0)$ for all $t \in [t_0 - (r^* r_0)^2, t_0]$ and we can find a $\tau = \tau(w, A, \eta, E) > 0$ such that $P(y, t_0, (A+1) r_0, -\tau r_0^2)$ is non-singular and $B(y, t_0, (A+1) r_0) \subset B(y, t, (A+2) r_0)$ for all $t \in [t_0 - \tau r_0^2, t_0]$.
This implies that $|{\Rm}| < K r_0^{-2}$ on $P(x_0, t_0, A r_0, - \tau r_0^2) \subset P(y, t_0 (A+1) r_0, - \tau r_0^2)$ in the case in which $C_1 \delta \leq r_0$ for some $C_1 = C_1 (w, A, \eta, E) < \infty$.
The higher derivative estimates follow from Shi's estimates on a slightly smaller parabolic neighborhood.

It remains to consider the case $C_1 \delta \leq r_0$.
Assuming $\delta$ to be sufficiently small depending on $r$, we can conclude that then $r_0 < r$.
Let $Q = |{\Rm}|(x_0, t_0)$.
In the next paragraph we show that $Q r_0^2$ is bounded by a constant which only depends on $w$, $E$ and $\eta$.

For this paragraph, fix $w$, $E$ and $\eta$ and assume that $Q r_0^2 > 1$.
Using the same reasoning as in the proof of Lemma \ref{Lem:6.3bc}(b) (compare with the ``bounded curvature at bounded distance''-estimate in \cite[Claim 2, 4.2]{PerelmanII}, see also the proof of \cite[Lemma 89.2]{KLnotes} or \cite[Lemma 70.2]{KLnotes} or \cite[Proposition 6.2.4]{Bamler-diploma}) we can conclude that $|{\Rm}| > K^*_1 (Q r_0^2) r_0^{-2}$ on $B(x_0, t_0, r_0)$ if $r_0 < \ov{r}^*(Q r_0^2) \sqrt{t_0}$ for some functions $K^*_1, \ov{r}^* : [0, \infty) \to (0, \infty)$ with $K^*_1 (s) \to \infty$ and $s \to \infty$ (we remark that for this argument the basepoint has to be chosen at a point $x' \in B(x_0, t_0, r_0)$ with $r_0^{-2} \leq |{\Rm}| (x', t_0) \leq K^*_1(Q r_0^2)$).
So there is some $S_1 < \infty$ such that if $Q r_0^2 > S_1$ and $r_0 < \ov{r}^* (S_1) \sqrt{t_0}$, then all points on $B(x_0, t_0, r_0)$ are centers of strong $\varepsilon$-necks or $(\varepsilon, E)$-caps whose cross-sectional $2$-spheres have diameter at most $C (K^*_1 (Q r_0^2))^{-1/2} r_0$.
These necks and caps can be glued together to give long tubes as described in \cite[Proposition 5.4.7]{Bamler-diploma} or \cite[sec 58]{KLnotes} and we conclude that $\vol_{t_0} B(x_0, t_0, r_0) < w^* (Q r_0^2) r_0^3$ for some function $w^* : [0, \infty) \to (0, \infty)$ with $w^*(s) \to 0$ as $x \to \infty$.
This finally shows that there is a universal constant $S_2 < \infty$ such that $Q r_0^2 < S_2$.

Again, by the same reasoning as before (this time, we choose the basepoint to be $(x_0, t_0)$), we obtain the estimate $|{\Rm}| < K^*_2 r_0^{-2}$ on $B(x_0, t_0, (A+1) r_0)$ for some universal constant $K^*_2 = K^*_2 (w, A, E, \eta) < \infty$ if $r_0 < \ov{r}^{**}(w, A, E, \eta) \sqrt{t_0}$.
Let $\ov\tau \geq 0$ be maximal such that the parabolic neighborhood $P(x_0, t_0, (A+1) r_0, - \ov\tau r_0^2)$ is non-singular.
By Lemma \ref{Lem:shortrangebounds}, we conclude that there is a constant $\tau_0 > 0$ such that  $|{\Rm}| < 2 K^*_2 r_0^{-2}$ on $P (x_0, t_0, (A+1) r_0, - \min \{ \ov\tau, \tau_0 \} r_0^2)$.
If $\tau' \geq \tau_0$, then we can deduce curvature derivative bounds on $B(x_0, t_0, A r_0)$ by Shi's estimates.
On the other hand, if $\ov\tau < \tau_0$, then by assuming $\delta$ to be sufficiently small depending on $m$, we can use Definition \ref{Def:precisecutoff}(3) to conclude that $|{\nabla^m \Rm}| < C_k r_0^{-2-k}$ for all $k \leq m$ on initial time-slice of $P (x_0, t_0, (A+1) r_0, - \tau' r_0^2)$.
So by a modified version of Shi's estimates we obtain a bound on $r_0^{2+k} |{\nabla^k \Rm}|$ in $B(x_0, t_0, A r_0)$ for all $k \leq m$.

Finally, we consider the case $r_0 = \rho(x_0, t_0)$.
Applying the Proposition with $A \leftarrow 1$ yields $|{\Rm}| < K r_0^{-2}$ on $B(x_0, t_0, r_0)$ for some $K = K(w, E, \eta) < \infty$.
So by Lemma \ref{Lem:HIbound}, if we had $r_0 < \ov{r}_{\ref{Lem:HIbound}} (K) \sqrt{t_0}$, then $\sec \geq - \frac12 r_0^{-2}$ on $B(x_0, t_0, r_0)$ which would contradict the choice $r_0$.
\end{proof}

\begin{proof}[Proof of Corollary \ref{Cor:Perelman68}]
Let $\varepsilon_0$ be the constant from Proposition \ref{Prop:genPerelman}.
Observe that by Proposition \ref{Prop:CNThm-mostgeneral} there are constants $\un\eta > 0$, $\un{E}_{\varepsilon_0} < \infty$ and decreasing, continuous, positive functions $\un{r}_{\varepsilon_0}, \un{\delta}_{\varepsilon_0} : [0, \infty) \to (0, \infty)$ such that if $\delta(t) \leq \un\delta_{\varepsilon_0} (t)$ for all $t \in [0, \infty)$, then every point $(x,t) \in \MM$ satisfies the canonical neighborhood assumptions $CNA (\un{r}_{\varepsilon_0} (t), \varepsilon_0, \un{E}_{\varepsilon_0}, \un\eta)$.
Now consider the constant $\delta_{\ref{Prop:genPerelman}} = \delta_{\ref{Prop:genPerelman}}(r, w, A, E, \eta, m)$ from Proposition \ref{Prop:genPerelman}.
We can assume that it depends on its parameters $r$, $w$ and $A$ in a monotone way, i.e. $\delta_{\ref{Prop:genPerelman}}(r', w', A', E, \eta, m') \leq \delta_{\ref{Prop:genPerelman}}(r, w, A, E, \eta, m)$ if $r' \leq r$, $w' \leq w$, $A' \geq A$ and $m' \geq m$.
Assume now that for all $t \geq 0$
\[ \delta(t) < \min \big\{ \delta_{\ref{Prop:genPerelman}} (\un{r}_{\varepsilon_0} (2 t), t^{-1}, t, \un{E}_{\varepsilon_0}, \un\eta, [t^{-1}]), \; \un\delta_{\varepsilon_0} (t) \big\}. \]
Let $w, A, m$ be given.
Choose $T = T(w, A, m) < \infty$ such that $2T^{-1} < w$, $\frac12 T > A$ and $\frac12 T > m$.

Consider now the point $x$, the time $t > T$ and the scale $r$.
Observe that $\MM$ satisfies the canonical neighborhood assumptions $CNA (\un{r}_{\varepsilon_0} (t), \varepsilon_0, \un{E}_{\varepsilon_0}, \un\eta)$ on $[t - r^2, t]$ and the surgeries on $[t - r^2, t]$ are performed by $\delta_{\ref{Prop:genPerelman}} (\un{r}_{\varepsilon_0} (t), 2 t^{-1}, \frac12 t,\un{E}_{\varepsilon_0}, \un\eta, [\frac12 t^{-1}])$-precise cutoff.
So we can apply Proposition \ref{Prop:genPerelman} with $x_0 \leftarrow x$, $t_0 \leftarrow t$, $r_0 \leftarrow r$, $r \leftarrow \un{r}_{\varepsilon_0} (t)$, $w \leftarrow w$, $A \leftarrow A$, $E \leftarrow \un{E}_{\varepsilon_0}$, $\eta \leftarrow \un\eta$, $m \leftarrow m$ to conclude that if $r \leq \ov{r}_{\ref{Prop:genPerelman}} (w, A, \un{E}_{\varepsilon_0}, \un\eta) \sqrt{t_0}$, then $|{\nabla^m \Rm}| < K_{m, \ref{Prop:genPerelman}} (w, A, \un{E}_{\varepsilon_0}, \un\eta) r_0^{-2}$ on $B(x, t, A r)$.
If the surgeries on $[t-r^2, t]$ are performed by $C_{1, \ref{Prop:genPerelman}}^{-1} (w, A, \un{E}_{\varepsilon_0}, \un\eta)$-precise cutoff, then $P(x, t, A r, - \tau_{\ref{Prop:genPerelman}}(w, A, \un{E}_{\varepsilon_0}, \un\eta))$ is non-singular and $|{\nabla^m \Rm}| < K_{m, \ref{Prop:genPerelman}} (w, A, \un{E}_{\varepsilon_0}, \un\eta)  r_0^{-2}$ there.
This establishes assertion (a)

For assertion (b) we argue as follows.
If $\rho(x,t) \leq \ov{r} \sqrt{t}$, then as discussed in the last paragraph for $r = \rho(x,t)$, we obtain $r = \rho(x,t) > \widehat{r} \sqrt{t}$.
So in general, we have $\rho(x,t) > \min \{ \ov{r}, \widehat{r} \} \sqrt{t}$ and we can apply assertion (a) with $r = \ov{r} \sqrt{t}$.
Observe that sufficiently large $T$, we indeed have $C_1 \delta (\frac12 t) < \min \{ \ov{r}, \widehat{r} \} \sqrt{t} < r$.
\end{proof}

\subsection{The thick-thin decomposition} \label{subsec:thickthin}
We now describe how, in the long-time picture, Ricci flows with surgery decompose the manifold into a thick and a thin part.
In this process, the thick part approaches a hyperbolic metric while the thin part collapses on local scales.
Compare this Proposition with \cite[7.3]{PerelmanII} and \cite[Proposition 90.1]{KLnotes}.

\begin{Proposition} \label{Prop:thickthindec}
There is a function $\delta : [0, \infty) \to (0, \infty)$ such that given a Ricci flow with surgery $\MM$ with normalized initial conditions which is performed by $\delta(t)$-precise cutoff defined on the interval $[0,\infty)$, we can find a constant $T_0 < \infty$, a function $w : [T_0, \infty) \to (0, \infty)$ with $w(t) \to 0$ as $t \to \infty$ and a collection of orientable, complete, finite volume hyperbolic (i.e. of constant sectional curvature $-1$) manifolds $(H'_1, g_{\hyp,1}), \ldots, (H'_k, g_{\hyp, k})$ such that: \\
There are finitely many embedded $2$-tori $T_{1,t}, \ldots, T_{m,t} \subset \MM(t)$ for $t \in [T_0, \infty)$ which move by isotopies and don't hit any surgery points and which separate $\MM(t)$ into two (possibly empty) closed subsets $\MM_{\thick}(t), \linebreak[1] \MM_{\thin}(t) \subset \MM(t)$ such that
\begin{enumerate}[label=(\textit{\alph*})]
\item $\MM_{\thick}(t)$ does not contain surgery points for any $t \in [T_0, \infty)$.
\item The $T_{i,t}$ are incompressible in $\MM(t)$ and $t^{-1/2} \diam_t T_{i,t} < w(t)$.
\item The topology of $\MM_{\thick}(t)$ stays constant in $t$ and $\MM_{\thick}(t)$ is a disjoint union of components $H_{1,t}, \ldots, H_{k,t} \subset \MM_{\thick}(t)$ such that the interior of each $H_{i,t}$ is diffeomorphic to $H'_i$.
\item We can find an embedded cross-sectional torus $T'_{j,t}$ at least $w^{-1}(t)$-deep inside each cusp of the $H'_i$ which moves by isotopy and speed at most $w(t)$ such that the following holds:
Chop off the ends of the $H'_i$ along the $T'_{j,t}$ and call the remaining open manifolds $H''_{i,t}$.
Then there are smooth families of diffeomorphisms $\Psi_{i,t} : H''_{i,t} \to H_i$ which become closer and closer to being isometries, i.e.
\[ \big\Vert \tfrac1{4t} \Psi_{i,t}^* g(t) - g_{\hyp,i} \big\Vert_{C^{[w^{-1}(t)]}(H''_{i,t})} < w(t) \]
and which move slower and slower in time, i.e.
\[  \sup_{H''_{i,t}} t^{1/2} | \partial_t \Psi_{i,t} |  < w(t) \]
for all $t \in [T_0, \infty)$ and $i = 1, \ldots, k$.
Moreover, the sectional curvatures on a $w^{-1}(t)$-tubular neighborhood of $\MM_{\thick}(t)$ lie in the interval $(\frac1t (-\frac14 - w(t)), \frac1t (-\frac14 + w(t)))$.
\item A large neighborhood of the part $\MM_{\thin}(t)$ is better and better collapsed, i.e. for every $t \geq T_0$ and $x \in \MM(t)$ with 
\[ \dist_t (x, \MM_{\thin}(t)) < w^{-1}(t) \sqrt{t} \]
we have
\[ \vol_t B ( x,t, \rho_{\sqrt{t}} (x,t) ) < w(t) \rho_{\sqrt{t}}^3 (x,t). \]
\end{enumerate}
\end{Proposition}

\section{The analysis of the collapsed part and consequences} \label{sec:thinpart}
Based on property (e) of Proposition \ref{Prop:thickthindec} we can analyze the thin part $\MM_{\thin}(t)$ for large $t$ and recover its graph structure geometrically.
More precisely, we can decompose the thin part into pieces on which the collapse can be approximated by certain models.
We will explain this in the first subsection.
Afterwards we prove important geometric and topological consequences concerning this decomposition.

\subsection{Analysis of the collapse}
The following result follows from the work of Morgan and Tian (\cite{MorganTian}).
We have altered its phrasing to include more geometric information.
The reader can find an explanation below of where to find each of the following conclusions in their paper.
Similar results can also be found in \cite{KLcollapse}, \cite{BBMP2}, \cite{CG} and \cite{Faessler}.

\begin{Proposition} \label{Prop:MorganTianMain}
For every two continuous functions $\ov{r}, K : (0, 1) \to (0, \infty)$ and every $\mu > 0$ there are constants $w_0 = w_0(\mu, \ov{r}, K) > 0$, $0 < s_2 (\mu, \ov{r}, K) < s_1(\mu, \ov{r}, K) < \frac1{10}$ and $a(\mu) > 0$, monotone in $\mu$, such that:

Let $(M,g)$ be a connected, orientable, closed Riemannian manifold and $M' \subset M$ a closed subset such that
\begin{enumerate}[label=(\textit{\roman*})]
\item Each component $T$ of $\partial M'$ is an embedded torus and for each such $T$ there is a closed subset $U'_T \subset M'$ which is diffeomorphic to $T^2 \times I$ such that $T \subset \partial U'_T$ and such that the boundary components of $U'_T$ have distance of at least $2$.
Moreover, there is a fibration $p_T : U \to I$ such that the $T^2$-fiber through every $x \in U'_T$ has diameter $< w_0 \rho_1(x)$.
\item For all $x \in M'$ we have (with $\rho_1 (x) = \min \{ \rho(x), 1 \}$)
\[ \vol B(x, \rho_1(x) ) < w_0 \rho_1^3(x). \]
\item For all $w \in (w_0, 1)$, $r < \ov{r}(w)$ and $x \in M'$ we have: if $\vol B(x,r) > w r^3$ and $r < \rho(x)$, then $|{\Rm}|, r |\nabla \Rm |, r^2 | \nabla^2 \Rm | < K(w) r^{-2}$ on $B(x,r)$.
\end{enumerate}

Then either $M' = M$ and $\diam M < \mu \rho_1(x)$ for all $x \in M$ and $M$ is diffeomorphic to an infra-nilmanifold or to a manifold which carries a metric of non-negative sectional curvature, or the following holds:

There are finitely many embedded $2$-tori $\Sigma^T_i$ and $2$-spheres  $\Sigma^S_i \subset \Int M'$ which are pairwise disjoint as well as closed subsets $V_1, V_2, V'_2 \subset M'$ such that
\begin{enumerate}[label=(a\textit{\arabic*})]
\item $M' = V_1 \cup V_2 \cup V'_2$, the interiors of the sets $V_1, V_2$ and $V'_2$ are pairwise disjoint and $\partial V_1 \cup \partial V_2 \cup \partial V'_2 = \partial M' \cup \bigcup_i \Sigma^T_i \cup \bigcup_i \Sigma^S_i$.
Obviously, no two components of the same set share a common boundary.
\item $\partial V_1 = \partial M' \cup \bigcup_i \Sigma^T_i \cup \bigcup_i \Sigma^S_i$.
In particular, $V_2 \cap V_2' = \emptyset$ and $V_2 \cup V_2'$ is disjoint from $\partial M'$.
\item $V_1$ consists of components diffeomorphic to one of the following manifolds:
\[ T^2 \times I, \; S^2 \times I, \; \Klein^2 \widetilde{\times} I, \; \IR P^2 \widetilde{\times} I, \;  S^1 \times D^2, \; D^3, \]
a $T^2$-bundle over $S^1$, $S^1 \times S^2$ or the union of two (possibly different) components listed above along their $T^2$- or $S^2$-boundary.
\item Every component of $V_2'$ has exactly one boundary component and this component borders $V_1$ on the other side.
More precisely, every component of $V'_2$ is diffeomorphic to one of the following manifolds:
\[ S^1 \times D^2, \; D^3, \; L(p,q) \setminus B^3, \; \Klein^2 \widetilde{\times} I. \]
\end{enumerate}

We can further characterize the components of $V_2$:
In $\Int V_2$ we find embedded $2$-tori $\Xi^T_i$ and $\Xi^O_i$  which are pairwise disjoint.
Furthermore, there are embedded closed $2$-annuli $\Xi^A_i \subset V_2$ whose interior is disjoint from the $\Xi^T_i$, $\Xi^O_i$ and $\partial V_2$ and whose boundary components lie in the components of $\partial V_2$ which are spheres.
Each spherical component of $\partial V_2$ contains exactly two such boundary components which separate the sphere into two (polar) disks and one (equatorial) annulus $\Xi^E_i$.
We also find closed subsets $V_{2,\textnormal{reg}}, \linebreak[1] V_{2, \textnormal{cone}}, \linebreak[1] V_{2, \partial} \subset V_2$ such that
\begin{enumerate}[label=(b\textit{\arabic*})]
\item  $V_{2, \textnormal{reg}} \cup V_{2, \textnormal{cone}} \cup V_{2, \partial} = V_2$ and the interiors of these subsets are pairwise disjoint.
Moreover, $\partial V_{2, \textnormal{reg}}$ is the union of $\bigcup_i \Xi^T_i \cup \bigcup \Xi^O_i \cup \bigcup_i \Xi^A_i \bigcup_i \Xi^E_i$ with the components of $\partial V_2$ which are diffeomorphic to tori.
\item $V_{2, \textnormal{reg}}$ carries an $S^1$-fibration which is compatible with its boundary components and all its annular regions.
\item The components of $V_{2, \textnormal{cone}}$ are diffeomorphic to solid tori ($\approx S^1 \times D^2$) and each of these components is bounded by one of the $\Xi^T_i$ such that the fibers of $V_{2, \textnormal{reg}}$ on each $\Xi^T_i$ are not nullhomotopic inside $V_{2, \textnormal{cone}}$.
\item The components of $V_{2, \partial}$ are diffeomorphic to solid tori ($\approx S^1 \times D^2$) or solid cylinders ($\approx I \times D^2$).
The solid tori are bounded the $\Xi^O_i$ such that the $S^1$-fibers of $V_{2, \textnormal{reg}}$ on the $\Xi^O_i$ are nullhomotopic inside the $V_{2, \partial}$.
The diskal boundary components of each solid torus of $V_{2, \partial}$ are polar disks on spherical components of $\partial V_2$ and the annular boundary component is one of the $\Xi^A_i$.
Every polar disk and every $\Xi^A_i$ bounds such a component on exactly one side.
\end{enumerate}

We now explain the geometric properties of this decomposition:
\begin{enumerate}[label=(c\textit{\arabic*})]
\item If $\CC$ is a component of $V_1$, then there is a closed subset $U \subset \CC$ with smooth boundary, as well as a Riemannian $1$-manifold $J$ of whose diameter is larger than $s_1 \rho_1(x)$ for each $x \in \CC$ and a fibration $p : U \to J$ such that
\begin{enumerate}
\item[({$\alpha$})] If $\CC \approx T^2 \times I$ or $S^2 \times I$, then $U = \CC$ and $J$ is a closed interval. \\
If $\CC \approx S^1 \times D^2, \Klein^2 \td\times I, D^3$ or $\IR P^3 \setminus B^3$, then $U \approx T^2 \times I$ (in the first two cases) or $U \approx S^2 \times I$ (in the latter two cases), $\partial \CC \subset \partial U$, $J$ is a closed interval and for all $x \in \CC \setminus U$ we have $\diam \CC \setminus U < \mu \rho_1(x)$. \\
If $\CC$ is the union of two such components as listed in (a3), then $U \approx T^2 \times I$ or $S^2 \times I$ depending on whether these two components have toroidal or spherical boundary and $\CC \setminus \Int U$ is diffeomorphic to the disjoint union of these two components.
Moreover for all $x \in \CC \setminus U$, the diameter of the component of $\CC \setminus U$ in which $x$ lies, has diameter $< \mu \rho_1(x)$. \\
If $\CC$ is diffeomorphic to a $T^2$-bundle over $S^1$ or to $S^1 \times S^2$, then $J$ is a circle and $U = \CC$.
\item[({$\beta$})] If $U$ is diffeomorphic to $T^2 \times I$, $S^2 \times I$ or $S^1 \times S^2$, then $p$ corresponds to the projection onto the interval or the circle factor.
\item[({$\gamma$})] $p$ is $1$-Lipschitz.
\item[({$\delta$})] For every $x \in U$, the fiber of $p$ through $x$ has diameter less than $\mu s_1 \rho_1(x)$.
\end{enumerate}
\item For every $x \in V_2$, the ball $(B(x,\rho_1(x)), \rho_1^{-1}(x) g, x)$ is $\mu$-close to a $2$-dimensional pointed Alexandrov space $(X, \ov{x})$ of area $> a$.
\item For every $x \in V_{2, \textnormal{reg}}$, the ball $(B(x, s_2 \rho_1(x)), s_2^{-2} \rho_1^{-2} (x) g, x)$ is $\mu$-close to a standard $2$ dimensional Euclidean ball $(B = B_1(0), g_{\eucl}, \ov{x} = 0)$.

Moreover, there is an open subset $U$ with $B(x,\frac12 s_2 \rho_1(x)) \subset U \subset \linebreak[1] B(x, \linebreak[0] s_2 \rho_1 (x))$, a smooth map $p : U \to \IR^2$ such that
\begin{enumerate}
\item[({$\alpha$})] There are smooth vector fields $X_1, X_2$ on $U$ such that $dp (X_i) = \frac{\partial}{\partial x_i}$ and $X_1, X_2$ are almost orthonomal, i.e. $| \langle X_i, X_j \rangle - \delta_{ij} | < \mu$ for all $i,j = 1,2$.
\item[({$\beta$})] $U$ is diffeomorphic to $B^2 \times S^1$ such that $p : U \to p(U)$ corresponds to the projection onto $B^2$ and the $S^1$-fibers are isotopic in $U$ to the $S^1$-fibers of the fibration on $V_{2, \textnormal{reg}}$.
\item[({$\gamma$})] The $S^1$-fibers of $p$ the $S^1$-fibers of $V_{2, \textnormal{reg}}$ on $U$ enclose an angle $< \mu$ with each other and an angle $\in (\frac\pi{2} - \mu, \frac\pi{2} + \mu)$ with $X_1$ and $X_2$.
\item[({$\delta$})] The $S^1$-fiber of the fibration on $V_{2, \textnormal{reg}}$ which passes through $x$ is isotopic in $U$ to the $S^1$-fibers of $p$.
\item[({$\varepsilon$})] The $S^1$-fibers of $p$ as well as the $S^1$-fibers of $V_{2, \textnormal{reg}}$ on $U$ have diameter less than $(\vol U)^{1/3}$.
\end{enumerate}
\item For every $x \in V_{2, \textnormal{cone}}$, the ball $B(x, \mu \rho_1(x))$ covers the component of $V_{2, \textnormal{cone}}$ in which $x$ lies.
\item For every $x \in V_{2, \partial}$ there is an $x' \in V_{2, \textnormal{reg}}$ with $\dist(x,x') < \mu \rho_1(x)$.
\item For every $x \in V'_2$, the ball $B(x, \mu \rho_1(x))$ covers the component of $V'_2$ in which $x$ lies.
\end{enumerate}
\end{Proposition}

\begin{proof}
Conditions 1., 3. in \cite[Theorem 0.2]{MorganTian} follow from assumptions (ii), (iii) if we replace $w_0$ by a sequence $w_n \to 0$; except for the higher derivative bounds in 3. resp. (iii) which are not really needed in the proof.
Condition 2. and the condition that the boundary is convex are replaced by the more general condition (i).
We will see that this replacement does not change the proof significantly.

We now go through the proof of \cite[Theorem 0.2]{MorganTian}.
If $(\diam M) \rho_1^{-1} (x)$ is sufficiently small for some $x \in M$, then we can use \cite[Corollary 0.13]{FY} or arguments of the proof of \cite[Lemma 1.5]{MorganTian} to conclude that either $M$ carries a metric of non-negative curvature or it is an infra-nilmanifold.
So in the following we can assume that $(\diam M) \rho_1^{-1}(x) \geq \min \{c, \mu \} > 0$ for all $x \in M$ for some universal $c > 0$.
The arguments in \cite{MorganTian} will still work if we replace Assumption 1 by this bound.
So by choosing the function $\rho_n$ in \cite[Lemma 1.5]{MorganTian} to be equal to $\frac{1}{100} \min \{ c, \mu \} \rho_1$, where $\rho_1$ is our lower sectional curvature scale, we can ensure that the properties of this Lemma are satisfied.
Moreover, we will profit from the fact that whenever we will find a geometric bound of the form $a \rho_n(x)$ for some $a > 0$, then this bound translates to a bound of the form $\frac{1}{100} \min \{ c, \mu \} a \rho_1(x)$.

Using \cite[Theorem 1.1]{MorganTian}, we choose our set $V_1$ to be the set $V_{n,1}^{MT}$ from \cite[Theorem 1.1]{MorganTian} minus the closed $3$-balls which were added in \cite[subsection 5.4.2]{MorganTian}.
Then assertion (a3) is clear (in our presentation we have to include a few more topological possibilities, because we did not impose any topological restrictions on $M$).
By the discussion in \cite[subsections 5.1--5.3]{MorganTian} and \cite[Lemma 4.38]{MorganTian} it is clear that assertion (c1) holds for the sets $U'_{n,1}$ and $W'_{n,1}$ if the constant $\varepsilon'$ there is chosen sufficiently small.
The set $V_{n,1}^{MT}$ differs from $W'_{n, 1}$ by collar neighborhoods of diameter $\ll \rho_n(x) < \frac\mu{100} \rho_1(x)$ for all points $x$ inside these collars (see \cite[subsection 5.5]{MorganTian}) and hence assertion (c1) still holds for $V_1$.
The constant $s_1$ can be determined based on the definition of the function $\rho_n$ (see above) and a lower diameter bound on components of $V_{n,1}^{MT}$ on the $\rho_n$ scale.
We also remark that condition (i) allows us to patch together the subsets $U'_T$ with nearby interval product structures (compare this with the proof of \cite[Proposition 5.2]{MorganTian}).

Let $V'_2$ be the union of the closures of all components of $M' \setminus V_1$ which correspond to components of $M \setminus U'_{n,1}$ which are close to an interval but which expand to be close to a standard $2$ dimensional ball in \cite[Lemma 5.9]{MorganTian}.
The topology of the components of $V'_2$ can be deduced using a better lower bound on the sectional curvature at the local scale.
This establishes assertions (a4) and (c6) (the latter holds by our choice of $\rho_n$).
For our purposes it will just be important that each of these components has small diameter and only one boundary component.

Finally, let $V_2$ be the closure of $M' \setminus (V_1 \cup V'_2)$ and let $V_{2, \textnormal{reg}} \subset V_2$ the set $V_{n,2}^{MT}$ from \cite[Theorem 1.1]{MorganTian}.
Then $V_{2, \textnormal{reg}}$ carries and $S^1$-fibration.
The closure of $V_2 \setminus V_{2, \textnormal{reg}}$ is the union of the closed $3$-balls, which we have previously deleted from $V_{n,1}^{MT}$ with the solid cylinders and the solid tori mentioned in property 6. of \cite[Theorem 1.1]{MorganTian}.
Denote the union of those solid tori in which the $S^1$-fiber of $V_{2, \textnormal{reg}}$ is not nullhomotopic, by $V_{2, \textnormal{cone}}$ and denote by $\Xi^T_i$ their boundaries.
Let $V_{2, \partial}$ be the closure of $V_2 \setminus (V_{2, \textnormal{reg}} \cup V_{2, \textnormal{cone}})$.
So $V_{2, \partial}$ consists of solid tori and circular or linear chains composed of solid cylinders and closed $3$-balls.
Denote the boundaries of these solid tori and circular chains by $\Xi^O_i$ and the annular boundaries of the linear chains by $\Xi^A_i$.
The annular regions on the spherical boundary components of $V_2$ are denoted by $\Xi^E_i$.
After a smoothing argument, assertions (a1), (a2) and (b1)--(b4) are satisfied.
Assertions (c4) and (c5) are a consequence of the construction.

Assertion (c2) can be deduced from \cite[Lemma 5.3]{MorganTian}.
Observe that $V_{2, \textnormal{reg}}$ is contained in the set $U_{2, \textnormal{generic}}^{MT}$ which is introduced in \cite[Lemma 5.7]{MorganTian}.
By construction (see \cite[Proposition 4.4]{MorganTian}), the set $U_{2, \textnormal{generic}}^{MT}$ has the following property: there is a subset $K \subset U_{2, \textnormal{generic}}^{MT}$ such that for all $x \in K$ the ball $(B(x, s_2^{MT} \rho_n (x) ), (s_2^{MT} \rho_n (x))^{-2} g, x)$ is $\widehat\varepsilon^{MT}$-close to a standard $2$ dimensional Euclidean ball (here $s_2^{MT}$ denotes the constant from \cite[Theorem 3.22]{MorganTian} and $\widehat\varepsilon^{MT}$ the constant from \cite[Proposition 4.3]{MorganTian}).
Moreover, we can assume that $U_{2, \textnormal{generic}}^{MT}$ is contained in the union of a tenth of these balls.
So if $\widehat\varepsilon^{MT}$ is sufficiently small, then for all $x \in U_{2, \textnormal{generic}}^{MT}$ the ball $(B(x, s_2 \rho_n(x)), (s_2 \rho_n(x))^{-2} g, x)$ is $\mu$-close to a standard $2$ dimensional Euclidean ball where $s_2$ is a constant depending on $\mu$.
This establishes the very first part of assertion (c3).
The rest of assertion (c3) follows from \cite[Proposition 4.3]{MorganTian}.
Observe that the diameter bound on the fibers in $U$ follows from the fact around every such fiber of diameter $d$, we can find neighborhood inside $B(x, s_2 \rho_1(x))$ which is close to $S^1(d) \times B^2(10 d)$.
\end{proof}

\subsection{Geometric consequences}
We now identify parts in the decomposition of Proposition \ref{Prop:MorganTianMain} which become non-collapsed when we pass to the universal cover or to a local cover.

\begin{Lemma} \label{Lem:unwrapfibration}
There is a constants $\mu_1 > 0$ such that:
Assume that we're in the situation of Proposition \ref{Prop:MorganTianMain} and assume $\mu \leq \mu_1$.
Then there is a constant $w_1 = w_1 (\mu) > 0$ which only depends on $s_2(\mu, \ov{r}, K)$ such that the following holds:
Consider a subset $N \subset M$, a point $x \in N$ such that $B(x,\rho_1(x)) \subset N$.
Assume that we are in one of the following cases:
\begin{enumerate}[label=(\roman*)]
\item $x \in \CC \subset N$ where $\CC$ is a component of $V_2$ with the property that the $S^1$-fiber of $\CC \cap V_{2, \textnormal{reg}}$ is incompressible in $N$.
\item $x \in \CC \subset N$ where $\CC$ is a component of $V_1$ which is diffeomorphic to $T^2 \times I$, $\Klein^2 \td\times I$, a $T^2$-bundle over $S^1$ or the union of two copies of $\Klein^2 \td\times I$ along their boundary and in all of these cases the generic $T^2$-fiber is incompressible in $N$.
\item $x \in \CC$ where $\CC$ is a component of $V_1$ which is diffeomorphic to $S^1 \times D^2$ or to a union of two (possibly different) copies of $S^1 \times D^2$ or $\Klein^2 \td\times I$.
Let $U \subset \CC$ be a subset as described in Lemma \ref{Prop:MorganTianMain}(c1).
Then we assume that $U \subset N$ and that the $T^2$-fiber of $U$ is incompressible in $N$.
\item $x \in \CC \subset N$ where $\CC$ is a component of $V'_2$ which is diffeomorphic to $\Klein^2 \td\times I$ and whose generic $T^2$-fiber is incompressible in $N$.
\item We are in the case $\diam M < \mu \rho_1(y)$ for all $y \in M$ as mentioned in the beginning of Proposition \ref{Prop:MorganTianMain}, $N = M$ and $M$ is either an infra-nilmanifold or a quotient of $T^3$.
\end{enumerate}
Now consider the universal cover $\td{N}$ of $N$ and choose a lift $\td{x} \in \td{N}$ of $x$.
Then we claim that
\[ \vol B(\widetilde{x}, \rho_1(x)) > w_1 \rho_1^3(x). \]
\end{Lemma}

In other words, $x$ is $w_0$-good at scale $1$ relatively to $N$ in the sense of Definition \ref{Def:goodness} from subsection \ref{subsec:goodness}.
In the following proofs (for each of the cases (i)--(v)) we will denote by $\delta_k(\mu_1)$ a positive constant, which depends on $\mu_1 > 0$ and which goes to zero as $\mu_1$ goes to zero.
In the end of each proof we will choose $\mu_1$ small enough so that all constants $\delta_k$ are sufficiently small.

\begin{proof}[Proof in case (i).]
We have either $x \in V_{2, \textnormal{reg}}$ or $x \in V_{2, \textnormal{cone}} \cup V_{2, \partial}$.
In the second case, there is an $x' \in B(x, \mu \rho_1(x)) \cap \CC \cap V_{2, \textnormal{reg}}$ and $\frac12 \rho_1(x) < \rho_1(x') < 2 \rho_1(x)$.
So $B(x', \frac14 \rho_1(x')) \linebreak[1] \subset B(x, \linebreak[1] \rho_1(x)) \linebreak[1] \subset N$.
Let $\td{x}' \in \td{N}$ be a lift of $x'$ such that $\dist(x, x') = \dist(\td{x}, \td{x}')$.
Then $B(\td{x}', \frac14 \rho_1(x')) \subset B(\td{x}, \rho_1(x))$ and hence $\vol B(\td{x}, \rho_1(x)) \geq \vol B(\td{x}', \frac14 \rho_1(x'))$.
So if we relax the assumption $B(x, \rho_1(x)) \subset N$ to $B(x, \frac14 \rho_1(x)) \subset N$, then we can replace $x$ by $x'$.
This shows that under this relaxed assumption, we can assume without loss of generality that $x \in V_{2, \textnormal{reg}}$.

Consider now the subset $U$ with $B(x, \frac12 s_2 \rho_1(x)) \subset U \subset B(x, s_2 \rho_1(x)) \subset N$ and the map $p : U \to \IR^2$ from Proposition \ref{Prop:MorganTianMain}(c3).
For the rest of the proof of case (i), we will only work with the metric $g' = s_2^{-2} \rho_1^{-2}(x) g$ on $M$ as opposed to $g$, and we will bound the $g'$-volume of the $1$-ball around $x$ in the universal cover $\td{N}$ from below by a universal constant.
This will then imply the Lemma.
Observe that the sectional curvatures of the metric $g'$ are bounded from below by $-1$ on this ball.
In the following paragraphs, we carry out concepts which can also be found in \cite{BBI} or \cite{BGP}.

By the properties of $x$, we can find a $(2, \delta_1(\mu_1))$-strainer $(a_1, b_1, a_2, b_2)$ of size $\frac12$ around $x$ (here $\delta_1(\mu_1)$ is a suitable constant as mentioned above).
Recall that this means that for the comparison angle $\cangle$ in the model space of constant sectional curvature $-1$ we have
\[ \cangle a_i x b_i > \pi - \delta_1, \quad \cangle a_i x b_j > \tfrac{\pi}2 - \delta_1, \quad \cangle a_i x a_j > \tfrac{\pi}2 - \delta_1, \quad \cangle b_i x b_j > \tfrac{\pi}2 - \delta_1 \]
and $\dist(a_i, x) = \dist(b_i, x) = \frac12$ for all $i \not= j$.
In the universal cover $\widetilde{N}$, we can now choose lifts $\widetilde{x}, \widetilde{a}_i, \widetilde{b}_i$ such that $\dist(\widetilde{a}_i, \widetilde{x}) = \dist(a_i, x) = \frac12$ and $\dist(\widetilde{b}_i, \widetilde{x}) = \dist(b_i, x) = \frac12$.
Since the universal covering map is $1$-Lipschitz, we obtain furthermore $\dist(\widetilde{a}_i, \widetilde{b}_j) \geq \dist(a_i, b_j)$, $\dist(\widetilde{a}_1, \widetilde{a}_2) \geq \dist(a_1, a_2)$ and $\dist(\widetilde{b}_1, \widetilde{b}_2) \geq \dist(b_1, b_2)$.
So all the comparison angles in the universal cover are at least as large as those on $M$ and hence we conclude that $(\widetilde{a}_1, \widetilde{b}_1, \widetilde{a}_2, \widetilde{b}_2)$ is a $(2, \delta_1(\mu_1))$-strainer around $\td{x}$ of size $\frac12$.

Next, we extend this strainer to a $2\frac12$-strainer around $\td{x}$.
By the property of the map $p$ there is a sequence $\widetilde{x}_n$ of lifts of $x$ in $\td{N}$ which is unbounded and whose consecutive distance is at most $2 (\vol U)^{1/3}$.
We can assume that $2 (\vol U)^{1/3} < \mu_1$, because otherwise we have a lower bound on $\vol U$ and we are done.
So for sufficiently small $\mu_1$ we can find a point $\widetilde{y} \in \td{N}$ with $\dist(\td{x},\td{y}) = 2 \sqrt{\mu_1}$ which projects to a point $y \in M$ with $\dist(x,y) < \mu_1$.
It follows that
\[ \dist(\widetilde{y}, \widetilde{a}_i) > \tfrac12 - \mu_1 \qquad \text{and} \qquad \dist(\widetilde{y}, \widetilde{b}_i) > \tfrac12 - \mu_1. \]
This implies
\begin{equation} \label{eq:212atx}
 \cangle \widetilde{y} \widetilde{x} \widetilde{a}_i > \tfrac{\pi}2 - \delta_2(\mu_1) \qquad \text{and} \qquad \cangle \widetilde{y} \widetilde{x} \widetilde{b}_i > \tfrac{\pi}2 - \delta_2(\mu_1).
\end{equation}
Hence $(\td{a}_1, \td{b}_1, \td{a}_2, \td{b}_2, \td{y})$ is a $2\frac12$-strainer around $\td{x}$.

Since $| {\dist(\widetilde{y}, \widetilde{a}_i) - \dist(\widetilde{x}, \widetilde{a}_i)}| < 2 \sqrt{\mu_1}$ and $|{\dist(\widetilde{y}, \widetilde{b}_i) - \dist(\widetilde{x}, \widetilde{b}_i)}| < 2 \sqrt{\mu_1}$, we conclude that $(\td{a}_1, \td{b}_1, \td{a}_2, \td{b}_2)$ is a $(2, \delta_3(\mu_1))$-strainer around $\td{y}$ of size $\geq \frac12 - 2 \sqrt{\mu_1}$.
We now show that symmetrically $(\td{a}_1, \td{b}_1, \td{a}_2, \td{b}_2, \td{x})$ is a $2\frac12$-strainer around $\td{y}$ of arbitrary good precision:
By comparison geometry
\[ \cangle \td{a}_i \td{x} \td{y} + \cangle \td{y} \td{x} \td{b}_i + \cangle \td{a}_i \td{x} \td{b}_i \leq 2 \pi. \]
Together with (\ref{eq:212atx}) and the strainer inequality at $\td{x}$, this yields
\[ \cangle \td{a}_i \td{x}  \td{y} < \tfrac{\pi}2 + \delta_1(\mu_1) + \delta_2(\mu_1). \]
By hyperbolic trigonometry,
\[ \cangle  \td{x} \td{y} \td{a}_i + \cangle \td{a}_i \td{x}  \td{y} + \cangle \td{y} \td{a}_i \td{x} > \pi - \delta_4(\mu_1) \qquad \text{and} \qquad \cangle \td{y} \td{a}_i \td{x}  < \delta_4(\mu_1). \]
Combining the last three inequalities yields
\[ \cangle \td{x} \td{y} \td{a}_i  > \tfrac{\pi}2 - \delta_1(\mu_1) - \delta_2(\mu_1) - 2 \delta_4(\mu_1) = \tfrac{\pi}2 - \delta_5(\mu_1). \]
The same estimate holds for $\cangle \td{x} \td{y} \td{b}_i$.

Let $\widetilde{m}$ be the midpoint of a minimizing segment joining $\widetilde{x}$ and $\widetilde{y}$.
We will now show that $(\td{a}_1, \td{b}_1, \td{a}_2, \td{b}_2, \td{y}, \td{x})$ is a $3$-strainer around $\td{m}$ of arbitrary good precision.
Since the distances of $\widetilde{a}_i$ resp. $\widetilde{b}_i$ to $\widetilde{m}$ differ from the distances to $\widetilde{x}$ by at most $\sqrt{\mu_1}$, we can conclude that $(\widetilde{a}_1, \widetilde{b}_1, \widetilde{a}_2, \widetilde{b}_2)$ is a $(2, \delta_6(\mu_1))$-strainer of size $\geq \frac12 - \sqrt{\mu_1}$ around $\td{m}$.
It remains to bound comparison angles involving the points $\td{x}$, $\td{y}$:
By the monotonicity of comparison angles, we have
\[ \cangle \td{m} \td{x} \td{a}_i \geq \cangle \td{y} \td{x} \td{a}_i > \tfrac{\pi}2 - \delta_2(\mu_1) \qquad \text{and} \qquad \cangle \td{m} \td{x} \td{b}_i \geq \cangle \td{y} \td{x} \td{b}_i > \tfrac{\pi}2 - \delta_2(\mu_1). \]
Now, if we apply the same argument as in the last paragraph, replacing $\td{y}$ with $\td{m}$, we obtain $\cangle \td{x} \td{m} \td{a}_i, \cangle \td{x} \td{m} \td{b}_i > \frac{\pi}2 - \delta_6(\mu_1)$.
For analogous estimates on the opposing angles, we then interchange the roles of $\td{x}$ and $\td{y}$.
Finally, $\cangle \td{x} \td{m} \td{y} = \pi$ is trivially true.

Set $\td{a}_3 = \td{y}$ and $\td{b}_3 = \td{x}$.
We have shown that $(\td{a}_1, \td{b}_1, \td{a}_2, \td{b}_2, \td{a}_3, \td{b}_3)$ is a $(3, \delta_7(\mu_1))$-strainer around $\td{m}$ of size $\geq \sqrt{\mu_1}$ for a suitable $\delta_7(\mu_1)$.
We will now use this fact to estimate the volume of the $\lambda \sqrt{\mu_1}$-ball around $\td{m}$ from below for sufficiently small $\lambda$ and $\mu_1$.
We follow here the ideas of the proof of \cite[Theorem 10.8.18]{BBI}.
Define the function
\begin{multline*}
 f : B(\td{m}, \lambda \sqrt{\mu_1}) \longrightarrow \IR^3 \qquad 
  z \longmapsto (\dist(\td{a}_1, z) - \dist(\td{a}_1, \td{m}), \\ \dist(\td{a}_2, z) - \dist(\td{a}_2, \td{m}), \dist(\td{a}_3, z) - \dist(\td{a}_3, \td{m})).
\end{multline*}
We will show that $f$ is $100$-bilipschitz for sufficiently small $\mu_1$ and $\lambda$.
Obviously, $f$ is $3$-Lipschitz, so it remains to establish the lower bound $\frac1{100}$.
Assume that this was false, i.e. that there are $z_1, z_2 \in B(\td{m}, \lambda \sqrt{\mu_1})$ with $\dist(z_1, z_2) > 100 | f(z_1) - f(z_2) |$.
Then for all $i = 1, 2, 3$
\begin{equation} \label{eq:almostiso}
 \dist(z_1, z_2) > 100 | {\dist(a_i, z_1) - \dist(a_i, z_2)} |.
\end{equation}
It is easy to see that given some $\delta > 0$, we can choose $\lambda > 0$ and $\mu_1 > 0$ sufficiently small, to ensure that $(\td{a}_1, \td{b}_1, \td{a}_2, \td{b}_2, \td{a}_3, \td{b}_3)$ is a $(3, \delta)$-strainer around $z_1$ and around $z_2$.
Now look at the comparison triangle corresponding to the points $z_1, z_2, \td{a}_i$.
By (\ref{eq:almostiso}), it is almost isosceles and hence by elementary hyperbolic trigonometry we conclude for $\lambda$ sufficiently small
\[ \tfrac{9}{10} \tfrac{\pi}2 < \cangle z_2 z_1 \td{a}_i, \; \cangle z_1 z_2 \td{a}_i < \tfrac{11}{10} \tfrac{\pi}2 . \]
Using comparison geometry
\[ \cangle z_1 z_2 \td{b}_i \leq 2 \pi - \cangle \td{a}_i z_2 \td{b}_i - \cangle z_1 z_2 \td{a}_i < \tfrac{11}{10} \tfrac{\pi}2 + \delta. \]
For $\lambda$ sufficiently small, we obtain furthermore by hyperbolic trigonometry
\[  \cangle \td{b}_i z_1 z_2 + \cangle z_1 z_2 \td{b}_i +\cangle z_2 \td{b}_i z_1 > \pi - \delta \qquad \text{and} \qquad \cangle z_2 \td{b}_i z_1 < \delta. \]
So
\[ \cangle \td{b}_i z_1 z_2 > \tfrac{9}{10} \tfrac{\pi}2 - 3 \delta. \]

Now join $z_1$ with $\td{a}_1, \td{b}_1, \td{a}_2, \td{b}_2, \td{a}_3$ by minimizing geodesics.
By comparison geometry, these geodesics enclose angles of at least $\frac{\pi}2 - \delta$ resp. $\pi - \delta$ between each other.
So their unit direction vectors approximate the negative and positive directions of an orthonormal basis.
By the same argument, the minimizing geodesic which connects $z_1$ with $z_2$ however encloses an angle of at least $\frac9{10} \frac{\pi}2 - 3\delta$ with each of these geodesics.
For sufficiently small $\delta$ this contradicts the fact that the tangent space at $z_1$ is $3$-dimensional.
So $f$ is indeed $100$-bilipschitz for sufficiently small $\lambda$ and $\mu_1$.

From the bilipschitz property we can conclude that
\[ \vol B(\td{m}, \lambda \sqrt{\mu_1} ) > c (\lambda \sqrt{\mu_1})^3 \]
for some universal $c > 0$.
Fixing $\mu_1 < \frac14$ and $\lambda < 1$ such that the argument above can be carried out, we obtain
\[ \vol B(\td{x}, 1) > \vol B(\td{m}, \lambda \sqrt{\mu_1}) > c (\lambda \sqrt{\mu_1})^3 = c' > 0. \]
By rescaling, this implies the desired inequality for the metric $g$.
\end{proof}

\begin{proof}[Proof in cases (ii)--(iv).]
By Proposition \ref{Prop:MorganTianMain}(c1) we know that there is an $x' \in \CC$ (or $x' \in \CC'$ in case (iv) where $\CC'$ is the component of $V_1$ adjacent to $\CC$) with $\dist(x, x') < \frac1{10} \rho_1(x)$ such that there is a subset $U$ with $B(x', \frac14 s_1 \rho_1(x')) \subset U \subset B(x', \frac12 s_1 \rho_1(x'))$ which is diffeomorphic to $T^2 \times I$ and incompressible in $N$ and the ball $(B(x', \frac12 s_1 \rho_1(x')), \linebreak[1] 4 s_1^{-2} \rho_1^{-2}(x') g, \linebreak[1] x')$ is $\delta_1(\mu_1)$-close to $((-1,1), g_{\eucl}, 0)$.
As in the proof in case (i) we can replace $x$ by $x'$ and hence assume without loss of generality that $B(x, \frac14 s_1 \rho_1(x)) \subset U \subset B(x, \frac12 s_1 \rho_1(x))$ and that $(B(x, \frac12 s_1 \rho_1(x)), \linebreak[1] 4 s_1^{-2} \rho_1^{-2}(x) g, \linebreak[1] x)$ is $\delta_1(\mu_1)$-close to $((-1,1), \linebreak[1] g_{\eucl}, \linebreak[1] 0)$.

Choose $q \in \pi_1(N)$ corresponding to a non-trivial simple loop in one of the cross-sectional tori of $U$ and denote by $\widehat{N}$ the covering of $N$ corresponding to the cyclic subgroup generated by $q$, i.e. if we also denote by $q$ the deck-transformation of $\widetilde{N}$ corresponding to $q$, then $\widehat{N} = \td{N} / q$.
So we have a tower of coverings $\td{N} \to \widehat{N} \to N$.

Consider first the rescaled metric $g' = 4 s_1^{-2} \rho_1^{-2} (x) g$.
With respect to this metric we can construct a $(1, \delta_2(\mu_1))$ strainer $(a_1, b_1)$ around $x$ of size $\frac12$ on $M$ for a suitable $\delta_2(\mu_1)$.
By the same arguments as in case (i), but using the covering $\widehat{N} \to N$, we can find a point $\widehat{m} \in \widehat{N}$ within $\sqrt{\mu_1}$-distance away from a lift $\widehat{x}$ of $x$ and a $(2, \delta_2(\mu_1))$ strainer $(\widehat{a}_1, \widehat{b}_1, \widehat{a}_2, \widehat{b}_2)$ around $\widehat{m}$ of size $\geq \sqrt{\mu_1}$.
Connect the points $\widehat{a}_i$ and $\widehat{b}_i$ with $\widehat{m}$ by minimizing geodesics and choose points $\widehat{a}'_i$ and $\widehat{b}'_i$ of distance $\sqrt{\mu_1}$ from $\widehat{m}$.
By monotonicity of comparison angles, $(\widehat{a}'_1, \widehat{b}'_1, \widehat{a}'_2, \widehat{b}'_2)$ is a $(2, \delta_2(\mu_1))$-strainer of size $\sqrt{\mu_1}$.

Let $g'' = \frac12 \mu_1^{-1} g'$.
Then $(\widehat{a}'_1, \widehat{b}'_1, \widehat{a}'_2, \widehat{b}'_2)$ has size $\frac12$ with respect to $g''$.
Using this strainer, the metric $g''$ and the covering $\td{N} \to \widehat{N}$, we can apply the same argument from case (i) again and obtain a $(3, \delta_3(\mu_1))$ strainer $(\td{a}_1, \td{b}_1, \td{a}_2, \td{b}_2, \td{a}_3, \td{b}_3)$ around a point $\td{m}' \in \td{N}$ which is $\sqrt{\mu_1}$-close to a lift $\td{m}$ of $\widehat{m}$ in $\td{N}$.

As in case (i), for a sufficiently small $\mu_1$ we can deduce a lower volume bound $\vol_{g''} \td{B} (\td{m}', 1) > c'$.
With respect to $g'$, the point $\td{m}'$ is within a distance of $\mu_1 + \sqrt{\mu_1}$ from a lift $\td{x}$ of $\widehat{x}$.
Hence 
\[ \vol_{g'} B (\td{x}, 1) > \vol_{g'} B(\td{m}', \sqrt{2 \mu_1}) > c' (2 \mu_1)^{3/2} = c'' > 0. \]
The desired inequality follows by rescaling.
\end{proof}

\begin{proof}[Proof in case (v).]
In this case, there is a covering $\widehat{M} \to M$ such that $\widehat{M} \approx T^2 \times \IR$ and whose group of deck transformations is isomorphic to $\IZ$.
Let $\widehat{x} \in \widehat{M}$ be a lift of $x$.
Then $(B(\widehat{x}, \rho_1(x), \rho_1^{-1} (x) g, \widehat{x})$ is $\delta_1(\mu_1)$-close to $((-1,1), g_{\textnormal{eucl}}, 0)$ and there is a subset $U$ with $B(\widehat{x}, \frac12 \rho_1(x)) \subset U \subset B(\widehat{x}, \rho_1(x))$ which is diffeomorphic to $T^2 \times I$ and incompressible in $\widehat{M}$.
We can now apply the previous proof.
\end{proof}

\subsection{Topological consequences} \label{subsec:topimplications}
We now discuss the topological structure of the decomposition obtained in Proposition \ref{Prop:MorganTianMain}.
So let in the following $M$ be an orientable, closed manifold, $M' \subset M$ a closed subset with smooth torus boundary components and consider a decomposition $M' = V_1 \cup V_2 \cup V'_2$, along with the surfaces $\Sigma_i^T, \Sigma_i^S, \Xi_i^T, \Xi_i^O, \Xi_i^A, \Xi_i^E$ and subsets $V_{2, \text{reg}}, V_{2, \text{cone}}, V_{2, \partial}$ which satisfies all the topological assertions of Proposition \ref{Prop:MorganTianMain} (a1)--(a4), (b1)--(b4).
We assume that the very first option of this Proposition, in which the manifold is infra-nil or a torus quotient, does not occur.
For future applications, we will discuss the following three cases:
\begin{description}
\item[case A] $M'$ is closed, i.e. $\partial M' = \emptyset$ and $M'$ is irreducible and not a spherical space form.
\item[case B] $M'$ is irreducible, has a boundary and all its boundary components are tori which are incompressible in $M'$,
\item[case C] $M' \approx S^1 \times D^2$.
\end{description}

The main result of this subsection will be Proposition \ref{Prop:GGpp}.
We first need to make some preparations.

\begin{Definition} \label{Def:GG}
Let $\mathcal{G} \subset M'$ be the union of
\begin{enumerate}[label=(\arabic*)]
\item all components of $V_2$ whose generic $S^1$-fiber is incompressible in $M'$,
\item all components of $V_1$ which are diffeomorphic to $T^2 \times I$, $\Klein^2 \td{\times} I$ and whose generic $T^2$-fiber is incompressible in $M'$, or components of $V_1$ which are diffeomorphic to a $T^2$-bundle over $S^1$ or the union of two copies of $\Klein^2 \td{\times} I$ along their common torus boundary.
\item all components of $V'_2$ which are diffeomorphic to $\Klein^2 \td{\times} I$ and whose generic $T^2$-fiber is incompressible in $M'$.
\end{enumerate}
We call the components of $V_1$, $V_2$ or $V'_2$ which are contained in $\mathcal{G}$ \emph{good (in $M'$)}.
\end{Definition}
Obviously, every good component of $V_2$ does not contain points of $V_{2, \partial}$.

\begin{Lemma} \label{Lem:bdrygoodisV2}
Consider the cases A--C.
Every component of $V_1$, $V_2$ or $V'_2$ which is contained in $\mathcal{G}$ and shares a boundary component with $\mathcal{G}$ either belongs to $V_2$ and is not adjacent to $\partial M'$ or belongs to $V_1$ and is adjacent to $\partial M'$.
\end{Lemma}
\begin{proof}
This follows directly from the definition of $\mathcal{G}$ and Proposition \ref{Prop:MorganTianMain}(a2).
Observe that any component which is adjacent to a good component of $V_1$ is good.
\end{proof}

\begin{Lemma} \label{Lem:SStori}
Consider the cases A or B.
There is a unique subset $\mathcal{S} \subset M'$ which is the disjoint union of finitely many embedded solid tori $\approx S^1 \times D^2$, bounded by some of the $\Sigma^T_i$, such that for any $\Sigma^T_i$ the following statement holds: $\Sigma^T_i$ is compressible in $M'$ if and only if $\Sigma^T_i \subset \mathcal{S}$ (i.e. it either bounds a component of $\mathcal{S}$ or it is contained in its interior).

In particular, $M' = \mathcal{G} \cup \mathcal{S}$.
\end{Lemma}

Note that an important consequence of this Lemma is that in cases A and B we have $\mathcal{G} \neq \emptyset$.

\begin{proof}
Without loss of generality, we assume in the following that $M'$ is connected.

First observe that by the irreducibility of $M'$ in case A and the fact that $\partial M' \neq \emptyset$ in case B, each sphere $\Sigma^S_i \subset M'$ bounds a ball $B_i \subset M'$ on exactly one side.
By Lemma \ref{Lem:coverMbysth}, any two of those balls are either disjoint or one is contained in the other.
Let $B'_1, \ldots, B'_{m'}$ be a disjoint subcollection of those balls which are maximal with respect to inclusion.

Now, consider all $\Sigma^T_i$ which already bound a solid torus $S_i$ in $M$.
In case B, we have $S_i \cap \partial M' = \emptyset$.
By Lemma \ref{Lem:coverMbysth} again, any two such tori are either disjoint or one is contained in the other.
So there is a unique subcollection of those solid tori which are maximal with respect to inclusion.
Denote the union of those solid tori by $\mathcal{S}$.

We will now show by contradiction that every torus $\Sigma^T_i \subset M' \setminus ( \mathcal{S} \cup B'_1 \cup \ldots \cup B'_{m'})$ is incompressible in $M'$.
Observe that each such torus does not bound a solid torus in $M$.
For each compressible torus $\Sigma^T_i \subset M' \setminus ( \mathcal{S} \cup B'_1 \cup \ldots \cup B'_{m'})$ we choose a spanning disk $D_i \subset M'$.
By Lemma \ref{Lem:compressingtorus} and a maximum argument, it is easy to see that there is at least one such torus $\Sigma^T_j \subset M' \setminus ( \mathcal{S} \cup B'_1 \cup \ldots \cup B'_{m'})$ with the following property:
For any other compressible torus $\Sigma^T_i \subset M' \setminus ( \mathcal{S} \cup B'_1 \cup \ldots \cup B'_{m'})$ which lies in the same component of $M \setminus \Sigma^T_j$ as $D_j$, the disk $D_i$ lies in the same component of $M \setminus \Sigma^T_i$ as $\Sigma^T_j$.

Let $\CC$ be the component of $V_1$, $V_2$ or $V'_2$ whose boundary contains $\Sigma^T_j$ and which lies on the same side of $\Sigma^T_j$ as $D_j$.
It can be seen easily that $\CC \not\subset V_1$.
Moreover $\CC \not\subset V'_2$, because otherwise $\CC$ would be diffeomorphic to $S^1 \times D^2$ by Proposition \ref{Prop:MorganTianMain}(a4) which would contradict the choice of $\Sigma^T_j$.
So $\CC \subset V_2$.

We will now analyze the boundary of $\CC$.
Consider $\Sigma^T_i \subset \partial\CC$ (possibly $\Sigma^T_i = \Sigma^T_j$).
If $\Sigma^T_i$ bounds a solid torus $S_i \subset \mathcal{S}$, then by choice of $\Sigma^T_j$, $S_i$ lies on the opposite side of $\CC$.
If $\Sigma^T_i$ does not bound a solid torus, then $\Sigma^T_i \subset M' \setminus ( \mathcal{S} \cup B'_1 \cup \ldots \cup B'_{m'})$.
So if $\Sigma^T_i$ has compressing disks, then by Lemma \ref{Lem:compressingtorus}(a) and again by the choice of $\Sigma^T_j$, these disks can only lie on the same side as $\CC$.
By Proposition \ref{Prop:incompressibleequiv}, this implies that then $\Sigma^T_i$ is incompressible in the component of $M' \setminus \Sigma^T_i$ which does not contain $\CC$.
For every spherical boundary component $\Sigma^S_i \subset \partial\CC$, the ball $B_i$ lies on the opposite side of $\CC$ and has to be one of the maximal balls $B'_i$ because otherwise $\Sigma^T_j$ would be contained in $B'_1 \cup \ldots \cup B'_{m'}$.

Set now $N = \CC \cap V_{2, \textnormal{reg}}$ and define $\mathcal{S}^*$ to be the union of $\CC \setminus N$ with the balls $B'_i$ whose boundary lies in $\partial \CC$.
So $N$ carries an $S^1$-fibration and is bounded by some of the tori $\Sigma^T_i$ and $\partial\mathcal{S}^*$.
The set $\mathcal{S}^*$ consists of components of $V_{2, \textnormal{cone}}$ (those are are solid tori), components of $V_{2, \partial}$ which are solid tori and chains made out of components of $V_{2, \partial}$ which are solid cylinders and balls $B'_i$.
So (after smoothing out the edges of the chains) all components of $\mathcal{S}^*$ are solid cylinders.
We can hence apply Lemma \ref{Lem:Seifertfiberincompressible} to conclude that either there is a boundary torus $T \subset \partial N$ which bounds a solid torus in $M'$ on the same side as $N$, or every boundary torus of $N$ either bounds a solid torus on the side opposite to $N$ or it is incompressible in $M'$.
The latter case cannot occur, since $\Sigma^T_j$ is compressible and does not bound a solid torus.
In the first case we conclude by Lemma \ref{Lem:coverMbysth} that $T \not\subset \partial\mathcal{S}^*$, so $T = \Sigma^T_i$ for some $i$.
But this would imply that $\Sigma^T_j \subset \CC \subset S_i$ in contradiction to our assumptions.

We have shown so far that every $\Sigma^T_i \subset M \setminus ( \mathcal{S} \cup B'_1 \cup \ldots \cup B'_{m'})$ is incompressible in $M'$.

Now assume that there was some $B'_i$ which is not contained in $\mathcal{S}$.
Then $\partial B'_i \cap \mathcal{S} = \emptyset$ by Lemma \ref{Lem:coverMbysth}.
By maximality, $B'_i$ borders a component $\CC$ of $V_2$ or $V'_2$ on the other side.
$\CC \not\subset V'_2$, since otherwise $M'$ would be a spherical space form by Proposition \ref{Prop:MorganTianMain}(a4), so $\CC \subset V_2$.
Since $\CC$ has a spherical boundary component, $\CC \cap V_{2, \partial} \neq \emptyset$ and hence the $S^1$-fibers on $\CC \cap V_{2, \text{reg}}$ are contractible in $M'$.
So every boundary torus of $\CC$ must be compressible and hence contained in $\mathcal{S} \cup B'_1 \cup \ldots \cup B'_{m'}$ and, in case B, $\partial \CC \cap \partial M' = \emptyset$.
Since $\CC \not\subset \mathcal{S} \cup B'_1 \cup \ldots \cup B'_{m'}$, all boundary tori of $\CC$ bound solid tori on the other side.
Define $N$ and $\mathcal{S}^*$ in the same way as above.
Then $N$ carries an $S^1$-fibration and $\mathcal{S}^*$ is a disjoint union of solid tori.
So every boundary component of $N$ bounds a solid torus on the other side.
In particular by Lemma \ref{Lem:coverMbysth}, no boundary component of $N$ bounds a solid torus on the same side.
Hence by Lemma \ref{Lem:Seifertfiberincompressible}, the $S^1$-fibers on $N$ are incompressible in $M'$ in contradiction to our previous conclusion.

We conclude that $B'_1 \cup \ldots \cup B'_{m'} \subset \mathcal{S}$ and one direction of the first claim follows.
The other direction is clear since $\pi_1(S^1 \times D^2) \cong \IZ$.
It remains to show that $M' = \mathcal{G} \cup \mathcal{S}$.
Let $\CC$ be a component of $V_1$, $V_2$ or $V'_2$ whose interior is contained in $M' \setminus \mathcal{S}$ and assume that $\CC \not\subset \mathcal{G}$.
Observe that since all $\Sigma^S_i$ are contained in $\mathcal{S}$, $\partial \CC$ only consists of tori.

Consider first the case $\CC \subset V_1$.
If $\CC$ has no boundary, then it must be diffeomorphic to either $S^1 \times S^2$ or the union of $D^3$ and $D^3$, $D^3$ and $\IR P^3 \setminus B^3$, two copies of $\IR P^3 \setminus B^3$ along their boundary, two copies of $S^1 \times D^2$ along their boundary or to to the union of $\Klein^2 \td\times I$ and $S^1 \times D^2$ along their boundary.
The first four cases can be excluded immediately, since $M'$ is assumed to be irreducible and not a spherical space form and the last two cases are excluded by Lemma \ref{Lem:coverMbysth} and Lemma \ref{Lem:Kleinandsolidtorus} respectively.
So $\CC$ has a boundary and hence it is diffeomorphic to $T^2 \times I$, $\Klein^2 \td\times I$ or $S^1 \times D^2$.
The last case cannot occur, since otherwise $\CC \subset \mathcal{S}$.
In the other two cases, the boundary component(s) are compressible in $M'$ and hence not contained in $\partial M'$.
So they bound a component of $\mathcal{S}$ on the side opposite to $\CC$, i.e. $M' = \CC \cup \mathcal{S}$.
But this is again ruled out by Lemma \ref{Lem:coverMbysth} and Lemma \ref{Lem:Kleinandsolidtorus}.

Similarly, in the case $\CC \subset V'_2$, $\CC$ would be diffeomorphic to $\Klein^2 \td\times I$ or $S^1 \times D^2$.
The second case can not occur since otherwise $\CC \subset \mathcal{S}$ and in the first case, $M'$ would be the union of $\Klein^2 \td\times I$ with a solid torus which is ruled out by Lemma \ref{Lem:Kleinandsolidtorus}.

Finally, assume that $\CC \subset V_2$.
Since the generic fiber in $\CC$ is assumed to be compressible in $M'$, all boundary tori of $\CC$ are compressible in $M'$ and hence $\partial \CC$ is disjoint from $\partial M'$.
So $M' = \CC \cup \mathcal{S}$ which gives a contradiction already in case B.
In case A, define $N$ and $\mathcal{S}^*$ again in the same way as above.
$N$ carries an $S^1$-fibration, $\mathcal{S}^*$ is a disjoint union of solid tori and $M' = N \cup \mathcal{S} \cup \mathcal{S}^*$.
By Lemma \ref{Lem:Seifertfiberincompressible} and Lemma \ref{Lem:coverMbysth}, we conclude that the generic Seifert fiber of $N$ is in fact incompressible in $M'$ and hence $\CC \subset \mathcal{G}$.
\end{proof}

We now focus on the intersection $\mathcal{G} \cap \mathcal{S}$.
\begin{Lemma} \label{Lem:solidtoriinsideSS}
Consider the cases A--C.
In the cases A, B let $\mathcal{S}$ be the set defined in Lemma \ref{Lem:SStori} and in case C let $\mathcal{S} = M'$.
Then $\mathcal{G} \cap \mathcal{S} \subset V_2$ and every component $\CC$ of $V_2$, which is contained in $\mathcal{G} \cap \mathcal{S}$ is bounded by tori which bound solid tori inside $\mathcal{S}$.
\end{Lemma}
\begin{proof}
It follows by the definition of $\mathcal{G}$ that $\mathcal{S}$ cannot contain any components of $V_1$ or $V'_2$.
Let now $\CC$ be a component of $V_2$ which is contained in $\mathcal{G} \cap \mathcal{S}$.
So $\CC \cap V_{2, \partial} = \emptyset$ and hence the boundary of $\CC$ is a disjoint union of tori which are of course compressible in $\mathcal{S}$.
Consider a component $T \subset \partial \CC$ and let $D \subset \mathcal{S}$ be a spanning disk for $T$.
Then by Lemma \ref{Lem:compressingtorus}, either $T$ bounds a solid torus, or $T \cup D$ is contained in an embedded ball.
But the latter case cannot occur, since the $S^1$-fiber direction of $\CC$ in $T$ is incompressible in $M'$.
So $T$ bounds a solid torus which by Lemma \ref{Lem:coverMbysth} has to lie inside $\mathcal{S}$.
\end{proof}

\begin{Definition} \label{Def:GGp}
Let the subset $\mathcal{G}' \subset M'$ to be the union of $\mathcal{G}$ with
\begin{enumerate}[label=(\arabic*)]
\item all components of $V_1$ which are diffeomorphic to $\Klein^2 \td{\times} I$ and adjacent to $\mathcal{G}$ or $\partial M'$,
\item all components of $V_1$ which are diffeomorphic to $T^2 \times I$ and which are adjacent to $\mathcal{G} \cup \partial M'$ on both sides,
\item all unions $\CC_1 \cup \CC_2$ where $\CC_1$ is a component of $V_1$ diffeomorphic to $T^2 \times I$ and adjacent to $\mathcal{G}$ or $\partial M'$ on one side and adjacent to $\CC_2$, which is a component of $V'_2$ diffeomorphic to $\Klein^2 \td{\times} I$, on the other side.
\end{enumerate}
\end{Definition}

\begin{Lemma} \label{Lem:GGpisGGoutside}
Consider the cases A--C.
Every component of $V_1$, $V_2$ or $V'_2$ which is contained in $\mathcal{G}'$ and meets the boundary $\partial \mathcal{G}'$, already belongs to $\mathcal{G}$ or it is adjacent to $\partial M'$.
In other words, $\partial \mathcal{G}' \subset \partial \mathcal{G} \cup \partial M'$.
In the cases A and B, the second option can be omitted.

In particular, any such component is either contained in $V_2$ if it is not adjacent to $\partial M'$ or it is contained in $V_1$ if it is adjacent to $\partial M'$.
\end{Lemma}
\begin{proof}
This is a direct consequence of the definition of $\mathcal{G}'$ and Lemma \ref{Lem:bdrygoodisV2}.
\end{proof}

We can now state the main result of this subsection.

\begin{Proposition} \label{Prop:GGpp}
Consider the cases A--C.
There is a unique subset $\mathcal{G}'' \subset \mathcal{G}'$ which is the union of certain components of $\mathcal{G}'$ such that
\begin{enumerate}[label=(\alph*)]
\item In cases A and B, every connected component of $M'$ contains exactly one component of $\mathcal{G}''$.
In case C, $\mathcal{G}''$ is connected and possibly empty.
\item Let $\mathcal{S}''$ be the closure of $M' \setminus \mathcal{G}''$.
Then $\mathcal{S}''$ is a disjoint union of finitely many solid tori ($\approx S^1 \times D^2$) each of which is incompressible in $M'$
\end{enumerate}
Moreover, 
\begin{enumerate}[label=(\alph*), start=3]
\item Every component of $V_1$, $V_2$ or $V'_2$ which is contained in $\mathcal{G}''$ and shares a boundary component with $\mathcal{G}''$ is contained in $\mathcal{G}$.
If it is adjacent to $\partial M'$, then it is also contained in $V_1$ and does not intersect $\partial\mathcal{G}'' \setminus \partial M'$ and otherwise it is contained in $V_2$.
\item $\mathcal{G} \setminus \mathcal{G}'' \subset V_2$.
\item For every component $\CC''$ of $\mathcal{S}''$ there is a component $\CC$ of $V_1$ which is contained in $\CC''$ and adjacent to $\partial \CC''$.
Moreover, $\CC$ is diffeomorphic to either $S^1 \times D^2$ (and hence $\CC'' = \CC$) or $T^2 \times I$.
In the latter case, the component $\CC'$ of $V_2$ or $V'_2$, which is adjacent to $\CC$ on the other side, is not contained in $\mathcal{G}$.
More precisely, if $\CC' \subset V'_2$, then $\CC' \approx S^1 \times D^2$ and if $\CC' \subset V_2$, then the generic Seifert fiber of $\CC'$ is contractible in $\CC''$.
\end{enumerate}
\end{Proposition}

\begin{proof}
In the cases A and B observe that $M' = \mathcal{G} \cup \mathcal{S}$ and every component of $M'$ contains exactly one component of $M' \setminus \mathcal{S}$.
So let $\mathcal{G}''$ be the union of all components of $\mathcal{G}'$ which share points with $M' \setminus \mathcal{S}$.
In case C, let $\mathcal{G}''$ be the component of $\mathcal{G}'$ which is adjacent to $\partial M'$ if it exists.
This establishes assertion (a).

Assertion (b)--(d) follow from Lemmas \ref{Lem:bdrygoodisV2}, \ref{Lem:solidtoriinsideSS} and \ref{Lem:GGpisGGoutside}:
For assertion (b) observe that the solid tori of $\mathcal{S}''$ are incompressible $M'$, because they are adjacent to Seifert fibrations with incompressible $S^1$-fiber or equal to $M'$.
And for assertion (d) observe that $\mathcal{G} \setminus \mathcal{G}'' \subset \mathcal{S}$.

For assertion (e) observe that $\CC''$ is either adjacent to $\partial M'$ or to $\mathcal{G}''$ and hence to $V_2$.
So the component $\CC$ of $V_1$, $V_2$ or $V'_2$, which is adjacent to $\partial \CC''$ inside $\CC''$, is contained in $V_1$.
Since $\CC$ is contained in a solid torus, it cannot be diffeomorphic to $\Klein^2 \td\times I$.
So it is diffeomorphic to $S^1 \times D^2$ or $T^2 \times I$.
The rest follows from the definition of $\mathcal{G}'$.
Observe that in the case $\CC' \subset V_2$, the generic Seifert fiber of $\CC'$ is compressible in $M'$ since $\CC' \not\subset \mathcal{G}$.
Since $\CC''$ is incompressible in $M'$, we conclude that the generic Seifert fiber of $\CC'$ is even contractible in $\mathcal{S}''$.
\end{proof}

\section{The main tools} \label{sec:maintools}
In the following we derive more specialized estimates using Perelman's methods and results presented in section \ref{sec:Perelman}.
Those statements will be used in section \ref{subsec:firstcurvbounds}.

\subsection{The goodness property} \label{subsec:goodness}
The following notion will become important for us.
It is inspired by Lemma \ref{Lem:unwrapfibration}.

\begin{Definition}[goodness] \label{Def:goodness}
Let $(M,g)$ be a complete Riemannian 3-manifold, $r_0 > 0$ and consider the function $\rho_{r_0} : M \to (0, \infty)$ from Definition \ref{Def:rhoscale}.
Let $w > 0$ be a constant and $x \in M$ be a point.
\begin{enumerate}[label=(\arabic*)]
\item Let $\td{x}$ be a lift of $x$ in the universal cover $\td{M}$ of $M$.
Then $x \in M$ is called \emph{$w$-good at scale $r_0$} if $\vol B(\td{x}, \rho_{r_0}(x)) > w \rho_{r_0}^3(x)$.
Here $B(\td{x}, \rho_{r_0}(x))$ denotes the $\rho_{r_0}(x)$-ball in the universal cover $\td{M}$ of $M$.
\item Let $U \subset M$ be an open subset and assume that $x \in U$.
Assume now that $\td{x}$ is a lift of $x$ in the universal cover $\td{U}$ of $U$.
Then $x$ is called \emph{$w$-good at scale $r_0$ relatively to $U$} if either $B(x, \rho_{r_0}(x)) \not\subset U$ or $\vol B(\td{x}, \rho_{r_0}(x)) > w \rho_{r_0}^3(x)$ where now $B(\td{x}, \rho_{r_0}(x))$ denotes the $\rho_{r_0}(x)$-ball in the universal cover $\td{U}$ of $U$.
\item The point $x$ is called \emph{locally $w$-good at scale $r_0$} if it is $w$-good at scale $r_0$ relatively to $B(x, \rho_{r_0}(x))$.
\end{enumerate}
\end{Definition}

Observe that the choice of the lift $\td{x}$ of $x$ is not essential.
We remark that the property ``$w$-good'' implies the properties ``$w$-good relatively to a subset $U$'' and ``locally $w$-good''.
The opposite implication however is generally false: Consider for example a smoothly embedded solid torus $S \subset M$, $S \approx S^1 \times D^2$ and a collar neighborhood $U$ of $\partial S$ in $S$, i.e. $U \subset S$, $U \approx T^2 \times (0,100)$ and $\partial S \subset \partial U$, such that the geometry on $U$ is close to product geometry $T^2 \times (0,100)$ in which the $T^2$-factor is very small.
Then for some $w > 0$ all points of $U$ are $w$-good relatively to $U$ as well as locally $w$-good, but none of the points of $U$ are $w$-good.

We remark that by volume comparison, there is a $\td{c} > 0$ such that if $x \in M$ is $w$-good at scale $r_0 > 0$ for some $w > 0$, then $x$ is also $\td{c} w$-good at any scale $r'_0 \leq r_0$.

\subsection{Universal covers of Ricci flows with surgery}
In the following we will need to carry out Perelman's methods in the universal covering flow $\td\MM$ of a given Ricci flow with surgery $\MM$.
In the case in which $\MM$ is non-singular, $\td\MM$ is just the pullback of the time-dependent metric onto the universal cover of the underlying manifold.
In the general case, the existence of $\td\MM$ is established by the following Lemma.

\begin{Lemma} \label{Lem:tdMM}
Let $\MM$ be a Ricci flow with surgery on a time-interval $I \subset [0, \infty)$ which is performed by precise cutoff and assume that if $\MM$ has surgeries, then there is a minimal surgery time.
Then there is a Ricci flow with surgery $\td\MM$ (called the \emph{universal covering flow}) which is performed by precise cutoff and a family of Riemannian coverings $\pi_t : \td\MM(t) \to \MM(t)$ which are locally constant in time away from surgery points such that the components of all time-slices $\td\MM(t)$ are all simply connected (i.e. $\td\MM(t)$ is the disjoint union of components which are isometric to the universal cover of $\MM(t)$).

Moreover, if $\MM$ is performed by $\delta(t)$-cutoff for some $\delta : I \to (0, \infty)$, then so is $\td\MM$.
If all time-slices of $\MM$ are complete, then the same is true for $\td\MM$.
If the curvature on $\MM$ is bounded on compact time-intervals which don't contain surgery times, then this property also holds on $\td\MM$.
\end{Lemma}

\begin{proof}
Recall that $\MM = ((T^i)_i, (M^i \times I^i, g_t^i)_i, (\Omega^i)_i, (U^i_{\pm})_i, (\Phi^i)_i)$ where each $g^i_t$ is a Ricci flow on the $3$-manifold $M^i$ defined for times $I^i$.
We can lift each of these flows to the universal cover $\widetilde{M}_0^i$ of $M^i$ via the natural projections $\pi^i_0 : \td{M}^i_0 \to M^i$ and obtain families of metrics $\widetilde{g}^i_{0, t}$ which still satisfy the Ricci flow equation.

We will now assemble the flows $(\widetilde{M}_0^i \times I^i, \widetilde{g}_{0, t}^i)$ to a Ricci flow with surgery $\td\MM$.
Each time-slice $\td\MM (t)$ of the resulting flow will be composed of a (possibly infinite) number of copies of components of $(\td{M}_0^i, \td{g}_{0, t}^i)$ if $t \in I^i$.
If there are no surgery times in $I$, i.e. $I = I^1$, then we set $\MM = (\cdot, (\td{M}_0^1, \td{g}_{0, t}^1), \cdot, \cdot, \cdot)$.
For any $i$ let $\MM^i$ be the restriction of $\MM$ to the time-interval $I \cap (-\infty, T^i)$ and if $T^{i-1}$ is the last surgery time, set $\MM^i = \MM$.
By induction, we can assume that $\td\MM^i$ already exists and we only need to prove that we can extend this flow to a Ricci flow with surgery $\td\MM^{i+1}$ which is the universal covering flow of $\MM^{i+1}$.
In order to do this, it suffices to construct the objects $(\td{M}^{i+1} \times I^{i+1}, \td{g}^{i+1}_t)$, $\Omega^i$, $U^i_\pm$, $\Phi^i$ and the projection $\pi^{i+1} : \td{M}^{i+1} \to M^{i+1}$.

Fix $i$ and consider $(\td{M}^i \times I^i, \td{g}^i_t)$ from $\td{\MM}^i$ and the projection $\pi^i : \td{M}^i \to M^i$ corresponding to $\pi_t$.
Denote by $\widetilde{\Omega}^i \subset \td{M}^i$ the preimage of $\Omega^i$ and by $\td{U}^i_- \subset \td{\Omega}^i$ the preimage of $U^i_-$ under $\pi^i$ and let $\td{U}^i_{0, +} \subset \td{M}^{i+1}_0$ be the preimage of $U^i_+$ under $\pi^{i+1}_0$.
Observe that the subset $U^i_- \subset M^i$ is bounded by pairwise disjoint, embedded $2$-spheres.
So for every point $p \in U^i_-$, the natural map $\pi_1(U^i_-, p) \to \pi_1(M^i,p)$ is an injection.
Consider now the set $\widetilde{U}^i_{0, +} \subset \widetilde{M}^{i+1}_0$.
The complement of this subset is still a collection of pairwise disjoint, embedded $3$-disks and hence each component of $\widetilde{U}^i_{0, +}$ is simply connected.
Via $(\Phi^i)^{-1} : U^i_+ \to U^i_-$ we find a covering map $\widetilde{U}^i_{0, +} \to U^i_+ \to U^i_- \subset M^i$.
Since every component $\CC_+$ of $\td{U}^i_{0, +}$ is simply-connected, we find a lift $\phi^i_{\CC_+, \CC_-} : \widetilde{U}^i_{0, +} \to \widetilde{M}^i$ for every component $\CC_-$ of $\td{U}^i_-$ with $\pi^{i+1}_0 (\CC_+) = \Phi^i (\pi^i (\CC_-))$ such that $\phi^i (\CC_+) = \CC_-$.
Using the fact that $U^i_- \to M^i$ is $\pi_1$-injective, we conclude that $\phi^i_{\CC_+, \CC_-}$ is injective.

Choose for every component $\CC_-$ of $\td{U}^i_-$ the (unique) component $\CC_+ = \CC_+(\CC_-)$ of $\td{U}^i_{0, +}$ such that $\pi^{i+1}_0 (\CC_+) = \Phi^i (\pi^i (\CC_-))$.
Let $\td{M}^{i+1}_{0}(\CC_+)$ be the component of $\td{M}^{i+1}_0$ which contains $\CC_+$.
Observe that $\CC_+$ is the only component of $\td{U}^i_{0, +}$ in $\td{M}^{i+1}_{0}(\CC_+)$.
Now let $\td{M}^{i+1}$ the disjoint union of all components $\td{M}^{i+1}_0 (\CC_+(\CC_-))$ where $\CC_-$ runs through all components of $\td{U}^i_-$.
The set $U^i_+$ is the disjoint union of all the $\CC_+(\CC_-)$ and the diffeomorphism $\Phi^i$ is defined to be the inverse of $\phi^i_{\CC_+(\CC_-), \CC_+}$ on each $\CC_-$.
This finishes the proof.
\end{proof}

\subsection{Quotients of necks}
Before we discuss the main tools, we need to establish the following Lemma which asserts that sufficiently precise $\varepsilon$-necks cannot have arbitrarily small quotients.

\begin{Lemma} \label{Lem:neckhasfewquotients}
There are constants $\td\varepsilon_0, \td{w}_0 > 0$ such that:
Let $(M, g)$ be a Riemannian manifold, $\varepsilon \leq \td\varepsilon_0$ and assume that $x_0 \in M$ is a center of an $\varepsilon$-neck and assume that $r < |{\Rm}|^{-1/2} (x)$.
Consider a local Riemannian covering $\pi : (M, g) \to (M', g')$ (i.e. $\pi^* g' = g$) such that $\pi(M) \subset M'$ is not closed and let $x'_0 = \pi (x_0) \in M'$.
Then $\vol B(x', r) > \td{w}_0 r^3$.
\end{Lemma}
\begin{proof}
We may assume without loss of generality that the scale $\lambda$ in Definition \ref{Def:epsneck} is equal to $1$ (and hence $r < 1.1$ for small $\varepsilon$), that $M$ is an $\varepsilon$-neck and that $\pi$ is surjective.
So, we can identify $M = S^2 \times (- \frac1{\varepsilon}, \frac1{\varepsilon} )$ and assume that $\Vert g - g_{S^2 \times \IR} \Vert_{C^{[\varepsilon^{-1}]}} < \varepsilon$.
If $\varepsilon$ is small enough, there is a smooth unit vector field $X$ on $M$ pointing in the direction of the eigenspace of $\Ric$ associated to the smallest eigenvalue which is unique up to sign.
For any $y_1, y_2 \in M$ with $\pi (y_1) = \pi(y_2)$, we have $d\pi (X_{y_1}) = \pm d\pi (X_{y_2})$.
So by possibly passing to a $2$-folded cover of $M'$, we can assume that $d\pi ( X )= X'$ for some smooth vector field $X'$ on $M'$.
Moreover, by possibly passing to another $2$-folded cover, we can assume that $M'$ is orientable.
Let $\Sigma \subset M$ be the embedded $2$-sphere corresponding to $S^2 \times \{ 0 \}$.
If $\varepsilon$ is small enough, the trajectories of $X$ cross $\Sigma$ exactly once and transversely.
Finally, let $U_0 \subset M$ be the open set corresponding to $S^2 \times (-20, 20)$ and assume that $\varepsilon^{-1} > 100$.

We will first show by contradiction that $\pi$ restricted to the ball $B(x_0, 1)$ is injective.
So assume that there are two distinct points $y_1, y_2 \in B(x_0, 1)$ with $\pi(y_1) = \pi(y_2)$.
Consider a geodesic segment $\gamma$ between $y_1$ and $x_0$ and lift its projection $\pi \circ \gamma$ starting from $y_2$.
This produces a point $x_1 \in M$ with $\pi(x_0) = \pi(x_1)$ and $\dist(x_0, x_1) < 2$.
Moreover, we can construct an isometric local deck transformation $\varphi : U_0 \to U_1 \subset M$ with $\varphi(x_0) = x_1$.
We note that $\varphi$ preserves orientation and the vector field $X$.

Now for $x \in \Sigma$ define $\varphi'(x)$ to be the unique intersection point of the $X$-trajectory passing through $\varphi(x)$ with $\Sigma$.
Then $\varphi' : \Sigma \to \Sigma$ is bijective continuous and orientation preserving.
Hence it has a fixed point $z_0 \in \Sigma$.

Let now $z_k = \varphi^{(k)}(z_0) \in U_1$ as long as this is defined.
Those points all lie on the trajectory through $z_0$ and have consecutive distance less than $10$.
Hence, there is a point $z_{k_0} \in U_1$ of distance more than $10$ to $z_0$.
This implies that $\Sigma' = \varphi^{(k_0)}(\Sigma)$ is disjoint from $\Sigma$.
But then the part of $M$ which is enclosed between $\Sigma$ and $\Sigma'$, maps to a closed manifold contradicting the assumptions.

So for all $r \leq 1$ we have $\vol B(x'_0, r) = \vol B(x_0, r) > c r^3$ for some universal $c > 0$.
This finishes the proof.
\end{proof}

\subsection{Bounded curvature around good points}
We start out by presenting a simple generalization of Corollary \ref{Cor:Perelman68} and consequence of Proposition \ref{Prop:genPerelman} which exhibits the flavor of the subsequent results.
We point out that the following Proposition is also a consequence of the far more general Proposition \ref{Prop:curvcontrolincompressiblecollapse} below.

\begin{Proposition} \label{Prop:curvcontrolgood}
There is a continuous positive function $\delta : [0, \infty) \to (0, \infty)$ such that for any $w, \theta > 0$ there are $\tau = \tau(w), \ov{\rho} = \ov{\rho} (w) > 0$ and $K = K(w), T = T(w, \theta) < \infty$ such that: \\
Let $\MM$ be a Ricci flow with surgery on the time-interval $[0, \infty)$ with normalized initial conditions which is performed by $\delta(t)$-precise cutoff.
Let $t_0 > T$ be a time, $x_0 \in \MM (t_0)$ a point and $r _0 > 0$ and assume that
\begin{enumerate}[label=(\roman*)]
\item $\theta \sqrt{t_0} \leq r_0 \leq \sqrt{t_0}$,
\item $x_0$ is $w$-good at scale $r_0$ and time $t_0$.
\end{enumerate}
Then we have $\rho(x_0, t_0) > r_1 = \min \{ \ov{\rho} \sqrt{t_0}, r_0 \}$ and the parabolic neighborhood $P(x_0, t_0,r_1 , - \tau r_1^2)$ is non-singular and $|{\Rm}| < K r_0^{-2}$ on $P(x_0, t_0,r_1 , - \tau r_1^2)$.
\end{Proposition}

\begin{proof}
The proof is very similar to that of Corollary \ref{Cor:Perelman68}.
Choose the constants $\varepsilon_0$, $\un{E}_{\varepsilon_0}$, $\un\eta$ as well as the functions $\un{r}_{\varepsilon_0}$, $\un{\delta}_{\varepsilon_0}$, $\delta$ as described in the first paragraph of this proof.
Then every point $(x,t) \in \MM$ with $t \in [\frac12 t_0, t_0]$ satisfies the canonical neighborhood assumptions $CNA (\un{r}_{\varepsilon_0} (t_0), \varepsilon_0, \un{E}_{\varepsilon}, \un\eta)$.
This implies that also every point $(x,t) \in \td\MM$  in the universal covering flow (see Lemma \ref{Lem:tdMM}) with $t \in [\frac12 t_0, t_0]$ satisfies the same the canonical neighborhood assumptions $CNA (\un{r}_{\varepsilon_0} (t_0), \varepsilon_0, \un{E}_{\varepsilon_0}, \un\eta)$.
As in the proof of Corollary, we can assume that $\delta (t) \leq \delta_{\ref{Prop:genPerelman}} (\un{r}_{\varepsilon_0} (t_0), w, 1, \un{E}_{\varepsilon_0}, \un\eta, 0)$ for all $t \in [\frac12 t_0, t_0]$ where $\delta_{\ref{Prop:genPerelman}}$ is the constant from Proposition \ref{Prop:genPerelman}.

We now apply Proposition \ref{Prop:genPerelman} to the universal covering flow $\td{\MM}$ at a lift $\td{x}_0$ of $x_0$ with $x_0 \leftarrow \td{x}_0$, $t_0 \leftarrow t_0$, $r_0 \leftarrow r_2 = \min\{ \rho_{r_0}(x_0, t_0), \ov{r}_{\ref{Prop:genPerelman}} \sqrt{t_0}, \frac12 \sqrt{t_0} \}$, $r \leftarrow \un{r}_{\varepsilon_0} (t_0)$, $w \leftarrow w$, $A \leftarrow 1$, $E \leftarrow \un{E}_{\varepsilon_0}$, $\eta \leftarrow \un\eta$, $m \leftarrow 0$.
Here $\ov{r}_{\ref{Prop:genPerelman}}  = \ov{r}_{\ref{Prop:genPerelman}} (w,1, \un{E}_{\varepsilon_0}, \un\eta)$.
So we obtain that  if $r_2 = \rho(x_0, t_0)$, then $r_2 > \widehat{r}_{\ref{Prop:genPerelman}}(w, 1, \un{E}_{\varepsilon_0}, \un\eta) \sqrt{t_0}$.
This implies that $\rho(x_0, t_0) > \min \{ \min \{ \widehat{r}_{\ref{Prop:genPerelman}}, \ov{r}_{\ref{Prop:genPerelman}}, \frac12 \} \sqrt{t_0}, r_0 \}$ and hence the first claim for $\ov{\rho} = \min \{ \widehat{r}_{\ref{Prop:genPerelman}}, \ov{r}_{\ref{Prop:genPerelman}}, \frac12 \}$.

Consider now the constant $C_{1, \ref{Prop:genPerelman}} = C_{1, \ref{Prop:genPerelman}} (w, 1, \un{E}_{\varepsilon_0}, \un\eta)$ from Proposition \ref{Prop:genPerelman} and assume that $T  = T(w, \theta) > 1$ is chosen large enough to ensure that $C_{1, \ref{Prop:genPerelman}} \delta(t) \leq \min \{ \ov\rho, \theta \}$ for all $t \in [\frac12 t_0, t_0]$.
Then in particular
\[ C_{1, \ref{Prop:genPerelman}} \delta(t) \leq  \min \{  \ov\rho, \theta \} \sqrt{t_0} \leq \min \{ \rho_{r_0} (x_0, t_0), \ov{r}_{\ref{Prop:genPerelman}} \sqrt{t_0}, \tfrac12 \sqrt{t_0} \} = r_2 \]
and hence by Proposition \ref{Prop:genPerelman} we have $|{\Rm}| < K_{\ref{Prop:genPerelman}}(w, 1, \un{E}_{\varepsilon_0}, \un\eta) r_2^{-2}$ on $P(x_0, t_0, r_2, \linebreak[1] - \tau_{\ref{Prop:genPerelman}} \linebreak[1] (w, \linebreak[1] 1, \linebreak[1] \un{E}_{\varepsilon_0}, \linebreak[1] \un\eta) r_2^2)$ and this parabolic neighborhood is non-singular.
\end{proof}

\subsection{Bounded curvature at bounded distance in sufficiently collapsed and good regions}
We now extend the curvature bound from Proposition \ref{Prop:curvcontrolgood} to balls of larger radii $A r_0$.
It is crucial here that by assuming sufficient collapsedness around the basepoint (depending on $A$), we don't have to impose an assumption of the form  $r_0 < \ov{r}(w, A) \sqrt{t_0}$ as in Proposition \ref{Prop:genPerelman}.
So the product $A r_0$ can indeed be chosen arbitrarily large.
\begin{Proposition} \label{Prop:curvcontrolincompressiblecollapse}
There is a continuous positive function $\delta : [0, \infty) \to (0, \infty)$ such that for any $w, \theta > 0$ and $1 \leq A < \infty$ there are $\tau = \tau(w), \ov{\rho} = \ov{\rho} (w, A), \ov{w} = \ov{w}(w, A) > 0$ and $K(w, A), T (w, A, \theta) < \infty$ such that: \\
Let $\MM$ be a Ricci flow with surgery on the time-interval $[0, \infty)$ with normalized initial conditions which is performed by $\delta(t)$-precise cutoff and let $t_0 > T$.
Choose $x_0 \in \MM(t_0)$ and $r_0 > 0$ and assume that
\begin{enumerate}[label=(\roman*)]
\item $\theta \sqrt{t_0} \leq r_0 \leq \sqrt{t_0}$,
\item $x_0$ is $w$-good at scale $r_0$ and time $t_0$,
\item $\vol_{t_0} B(x_0, t_0, r_0) < \ov{w} r_0^3$.
\end{enumerate}
Then $|{\Rm}| < K r_0^{-2}$ on $B = \bigcup_{t \in [t_0 - \tau r_0^2, t_0]} B(x_0, t, A r_0)$ and there are no surgery points on $B$.

In particular, if $r_0 \geq \rho(x_0, t_0)$, then $\rho( x_0, t_0) > \ov\rho \sqrt{t_0}$ and the curvature estimate becomes $| {\Rm} | < K t_0^{-1}$.
\end{Proposition}

\begin{proof}
We first set up an argument in the spirit of the proof of Corollary \ref{Cor:Perelman68}.
Choose $\varepsilon_0 > 0$ to be smaller than the corresponding constants in Lemma \ref{Lem:6.3bc} and the constant $\td\varepsilon_0$ in Lemma \ref{Lem:neckhasfewquotients}.
By Proposition \ref{Prop:CNThm-mostgeneral} there are constants $0 < \eta < \un\eta$, $\un{E}_{\varepsilon_0} < E < \infty$ and decreasing continuous positive functions $\un{r}_{\varepsilon_0}, \un{\delta}_{\varepsilon_0} : [0, \infty) \to (0, \infty)$ such that if $\delta(t) \leq \un\delta_{\varepsilon_0} (t)$ for all $t \in [0, \infty)$, then every point $(x,t) \in \MM$ satisfies the canonical neighborhood assumptions $CNA (\un{r}_{\varepsilon_0} (t), \varepsilon_0, E, \eta)$.
Without loss of generality, we can assume that $E > E_{0, \ref{Lem:6.3bc}} (\varepsilon_0)$ and $\eta < \eta_{0, \ref{Lem:6.3bc}}$ where $E_{0, \ref{Lem:6.3bc}}$ and $\eta_{0, \ref{Lem:6.3bc}}$ are the constants from Lemma \ref{Lem:6.3bc}.
Consider the constant $\delta_{\ref{Lem:6.3bc}} (w', A', r', \varepsilon', E', \eta')$, assume that it depends on its parameters $w', A', r'$ in a monotone way, i.e. $\delta_{\ref{Lem:6.3bc}} (w'', A'', r'', \varepsilon', E', \eta') \leq \delta_{\ref{Lem:6.3bc}} (w', A', r', \varepsilon', E', \eta')$ whenever $w'' \leq w'$, $A'' \geq A'$ and $r'' \leq r'$, and assume that for all $t \geq 0$
\[ \delta(t) < \min \big\{ \delta_{\ref{Lem:6.3bc}} (t^{-1}, t, \un{r}_{\varepsilon_0}(2t), \varepsilon_0, E, \eta), \; \un{\delta}_{\varepsilon_0} (t), \; \delta_{\ref{Prop:curvcontrolgood}}(t) \big\}. \]

By Proposition \ref{Prop:curvcontrolgood}, we have $\rho(x_0, t_0) > r_ 1 = \min \{ \ov{\rho}_{\ref{Prop:curvcontrolgood}} (w)  \sqrt{t_0}, r_0 \}$ and $|{\Rm}|< K_{\ref{Prop:curvcontrolgood}} (w) r_0^{-2}$ on the non-singular parabolic neighborhood $P(x_0, \linebreak[1] t_0, \linebreak[1] r_1, \linebreak[1] -\tau_{\ref{Prop:curvcontrolgood}} \linebreak[1] (w) r_1^2)$ (here we need to assume that $T$ is large enough).
In particular, this shows how the last assertion of the Proposition follows from the first one.

It remains to show the first assertion.
Consider the constants $\tau_{\ref{Prop:curvcontrolgood}}(w)$, $K_{\ref{Prop:curvcontrolgood}}(w)$ from Proposition \ref{Prop:curvcontrolgood} and set
\[ \gamma = \gamma(w) = \tfrac1{10} \min \big\{ 1, \tau_{\ref{Prop:curvcontrolgood}}^{1/2}(w), K_{\ref{Prop:curvcontrolgood}}^{-1/2}(w) \big\}. \]
Consider the universal covering flow $\td\MM$ of $\MM$ as described in Lemma \ref{Lem:tdMM} and let $\td{x}_0 \in \td\MM$ be a lift of $x_0$.
By the choice of $\gamma$ we have
\[ |{\Rm}| < (\gamma r_0)^{-2} \quad \text{on} \quad P(\td{x}_0, t, \gamma r_0, - (\gamma r_0)^2) \quad \text{for all} \quad t \in [t_0 - (\gamma r_0)^2, t_0] \]
and $\vol_t B(\td{x}_0, t, \gamma r_0) > \frac1{10} \td{c} (\gamma r_0)^3$ for all such $t$.
We now argue that for sufficiently large $T$ we can apply Lemma \ref{Lem:6.3bc}(a) with $\MM \leftarrow \td\MM$, $\widetilde{x}_0 \leftarrow x_0$, $t_0 \leftarrow t \in [t_0 - (\gamma r_0)^2, t_0]$, $r_0 \leftarrow \gamma r_0$, $w \leftarrow \frac1{10} \td{c} w$, $A \leftarrow \gamma^{-1} (A + 1)$, $r \leftarrow \un{r}_{\varepsilon_0} (t_0)$, $\eta \leftarrow \eta$, $\varepsilon \leftarrow \varepsilon_0$, $E \leftarrow E$.
In order to do this, we need to assume that $T = T(w, A, \theta)$ is large enough such that $2 T^{-1} < \frac1{10} \td{c} w$ and $\frac12 T > \gamma^{-1} (A+1)$.
Observe that for all $(x',t') \in \MM$ with $[\frac12 t_0, t_0]$ the canonical neighborhood assumptions $CNA (\un{r}_{\varepsilon} (t_0), \varepsilon_0, E, \eta)$ hold.
So these canonical neighborhood assumptions also hold for all $(x',t') \in \td\MM$ with $t' \in [\frac12 t_0, t_0]$.
Moreover, by the choice of $T$ we have $\delta(t') <  \delta_{\ref{Lem:6.3bc}} (\td{c} w, \gamma^{-1}(A+1), \un{r}_{\varepsilon_0}(t_0), \varepsilon_0, E, \eta)$ for all $t' \in [\frac12 t_0, t_0]$.
So Lemma \ref{Lem:6.3bc}(a) can be applied and we conclude that for any $t \in [t_0 - (\gamma r_0)^2, t_0]$ the points on $B(\td{x}_0, t, (A+1)r_0)$ satisfy the canonical neighborhood assumptions $CNA(\gamma \td{\rho}_{\ref{Lem:6.3bc}} r_0, \varepsilon_0, E, \eta)$.
Here $\td{\rho}_{\ref{Lem:6.3bc}} = \td{\rho}_{\ref{Lem:6.3bc}}(\frac1{10} \td{c} w, \gamma^{-1} (A+1), \varepsilon_0, E, \eta)$.
Set $K = \gamma^{-2} \max \{ \td{\rho}^{-2}_{\ref{Lem:6.3bc}}, E^2 \}$.
So, whenever $|{\Rm}|(x,t) \geq K r_0^{-2}$ for some $x \in B(\td{x}_0, t, (A+1)r_0)$, then $(x,t)$ is a center of a strong $\varepsilon_0$-neck or an $(\varepsilon_0, E)$-cap and
\begin{equation} \label{eq:deretabound}
 |\nabla |{\Rm}|^{-1/2}|(x,t) < \eta^{-1}.
\end{equation}

Fix some $t \in [t_0 - (\gamma r_0)^2, t_0]$.
Let $a \leq A$ be maximal with the property that $|{\Rm_t}| < K r_0^{-2}$ on $B(\td{x}_0, t, a r_0)$.
If $a = A$, we are done, so assume $a < A$.
By (\ref{eq:deretabound}), we can conclude that (compare also with Lemma \ref{Lem:shortrangebounds})
\begin{equation} \label{eq:etaextcurvbound}
 |{\Rm_t}| < 4 K r_0^{-2} \qquad \text{on} \qquad B(\td{x}_0, t, a r_0 + \tfrac12 \eta K^{-1/2} r_0).
\end{equation}
By the choice of $a$ we can find a point $\td{x}_1 \in \td\MM(t)$ of time-$t$ distance exactly $a r_0$ from $\td{x}_0$ with $|{\Rm}| (\td{x}_1, t) = K r_0^{-2}$.
So $(\td{x}_1, t_0)$ is either a center an $\varepsilon_0$-neck or an $(\varepsilon_0, E)$-cap in $\td\MM (t)$.

Let $x_1 \in \MM(t)$ be the projection of $\td{x}_1$.
By (\ref{eq:etaextcurvbound}) and volume comparison, we can rudely estimate
\begin{multline}
 \vol_{t} B(x_1, t, \tfrac12 \eta K^{-1/2} r_0) < \vol_t B(x_0, t, a r_0 + \tfrac12 \eta K^{-1/2} r_0) \\
 < C(A, K) \vol_t B(x_0, t, r_0) < C(w, A) \ov{w} r_0^3. \label{eq:etaballsmallvol}
\end{multline}
If $(\td{x}_1, t)$ is a center of an $\varepsilon_0$-neck, then we obtain a contradiction using Lemma \ref{Lem:neckhasfewquotients} assuming $C(w, A) \ov{w} < \td{w}_0 (\frac12 \eta K^{-1/2})^3$.
So assume for the rest of the proof that $(\td{x}_1, t)$ is a center of an $(\varepsilon_0, E)$-cap $U \subset \td\MM (t)$.
Let $K \subset U$ be a compact subset such that $\td{x}_1 \in K$ and $U \setminus K$ is an $\varepsilon_0$-neck and let $\td{y} \in U$ be a center of this neck.
We have $\gamma^{-2} r_0^{-2} \leq E^{-2} K r_0^{-2} \leq |{\Rm}| \leq E^2 K r_0^{-2}$ on $U$.
So $\td{x}_0 \not\in U$ and hence the minimizing geodesic segment between $\td{x}_0$ and $\td{x}_1$ passes through the whole $\varepsilon_0$-neck $U \setminus K$.
So for sufficiently small $\varepsilon_0$ we have $\dist_t (\td{x}_0, \td{y}) < \dist_t(\td{x}_0, \td{x}_1) = a r_0$.
In particular, for the projection $y$ of $\td{y}$ we find $B(y, t, \frac12 \eta K^{-1/2} r_0) \subset B(x_0, t, ar_0 + \frac12 \eta K^{-1/2} r_0)$.
Now again using Lemma \ref{Lem:neckhasfewquotients} and (\ref{eq:etaballsmallvol}), we conclude
\[ \td{w}_0 \big( \tfrac12 \eta E^{-1} K^{-1/2} \big)^3 r_0^3 < \vol_t B(y, t, \tfrac12 \eta E^{-1} K^{-1/2} r_0) < C(w, A) \ov{w} r_0^3. \]
This yields a contradiction for sufficiently small $\ov{w}$.

It remains to show that there are no surgery points on $B$.
To see this, observe that $|{\Rm}| < K \theta^{-2} t_0^{-2}$ on $B$, but by (\ref{eq:surgpointhighcurv}) we have $|{\Rm}|(x, t) > c' \delta^{-2} (t)$ at for every surgery point $(x,t) \in \MM$ for some universal $c' > 0$.
So choosing $T$ large enough yields the desired result.
\end{proof}

\subsection{Curvature control in points which are good relatively to regions whose boundary is geometrically controlled} \label{subsec:curvboundinbetween}
Next, we generalize Proposition \ref{Prop:curvcontrolgood} to include points which are good relatively to some open set $U$.
In order to do this, we need to assume that the metric around the boundary of $U$ is sufficiently controlled on a small time-interval.

\begin{Proposition} \label{Prop:curvboundinbetween}
There is a continuous positive function $\delta : [0, \infty) \to (0, \infty)$ such that for any $w, \theta > 0$ there are $\alpha = \alpha (w), \widehat{r} = \widehat{r}(w) > 0$ and $K = K(w), T = T(w, \theta) < \infty$ such that: \\
Let $\MM$ be a Ricci flow with surgery on the time-interval $[0, \infty)$ with normalized initial conditions which is performed by $\delta(t)$-precise cutoff and let $t_0 > T$.
Let $r_0 > 0$ and consider a sub-Ricci flow with surgery $U \subset \MM$ (see Definition \ref{Def:subRF}) on the time-interval $[t_0 - r_0^2, t_0]$ whose time-slices $U(t)$ are closed subsets of $\MM(t)$ and a point $x_0 \in U (t_0)$ such that
\begin{enumerate}[label=(\roman*)]
\item $\theta \sqrt{t_0} < r_0 \leq \sqrt{t_0}$,
\item for all $x \in \partial U (t_0)$, the parabolic neighborhood $P(x, t_0, r_0, - r_0^2)$ is non-singular and we have $| {\Rm} | < r_0^{-2}$ there,
\item $x_0$ is $w$-good at scale $r_0$ relatively to $U (t_0)$ at time $t_0$.
\end{enumerate}
Then the parabolic neighborhood $P(x_0, t_0,  \alpha r_0, - \alpha^2 r_0^2)$ is non-singular and we have $|{\Rm}| < K r_0^{-2}$ there.
\end{Proposition}
\begin{proof}
The idea of the proof will be to apply Proposition \ref{Prop:genPerelman} to the universal covering flow $\td{U}$ of $U$ (see Lemma \ref{Lem:tdMM}).
The main problem in the proof is to verify assumptions (i) and (ii) of this Proposition.
Apart from this, the proof essentially goes along the lines of the proof of Proposition \ref{Prop:curvcontrolgood}.

We first choose the function $\delta(t)$.
Let $\varepsilon_0 > 0$ be the constant from Proposition \ref{Prop:genPerelman} and consider the constants $\un{E}_{\varepsilon_0}$, $\un\eta$ and the functions $\un\delta_{\varepsilon_0} (t), \un{r}_{\varepsilon_0} (t)$ from Proposition \ref{Prop:CNThm-mostgeneral}.
So if $\delta(t) < \un\delta_{\varepsilon_0}(t)$ for all $t \geq 0$, then $\MM$ satisfies the canonical neighborhood assumptions $CNA (\un{r}_{\varepsilon_0} (t), \varepsilon_0, \un{E}_{\varepsilon_0}, \un\eta)$.
Without loss of generality, we assume that $\un{r}_{\varepsilon_0} (t) \to 0$ as $t \to \infty$.
Similarly as in the proof of Proposition \ref{Prop:curvcontrolgood} or Corollary \ref{Cor:Perelman68}, we assume that
\[ \delta (t) < \min \big\{ \delta_{\ref{Prop:genPerelman}} ( \un{r}_{\varepsilon_0} (2t), t^{-1}, 1, \un{E}_{\varepsilon_0}, \un\eta, 0), \; \un\delta_{\varepsilon_0} (t), \; 1 \big\} \]
where $\delta_{\ref{Prop:genPerelman}}$ is the constant in Proposition \ref{Prop:genPerelman} which we assume to satisfy the before-mentioned monotonicity property.
Furthermore, we choose $T = T(w, \theta)$ such that $2 T^{-1} < \td{c} w$ and that $\un{r}_{\varepsilon_0} (t) < \frac1{20}  \theta \min \{ 1, E^{-1}, \varepsilon_0 \} \sqrt{t}$, $\delta(t) < \frac12 \theta \sqrt{t}$ for all $t \geq \frac12 T$.

We now present the main argument.
By assumption (ii), we can consider the case in which $B(x_0, t_0, r_0) \subset U(t_0)$.
Our goal will be to apply Proposition \ref{Prop:genPerelman} in the universal covering flow $\td{U}$ of $U$ (see Lemma \ref{Lem:tdMM}) at a lift $(\td{x}_0, t_0)$ of $(x_0, t_0)$.
We first check that all points $(x,t) \in \td{U}$ with $t \in [t_0 - \frac12 r_0^2, t_0]$ satisfy the canonical neighborhood assumptions $CNA (\un{r}_{\varepsilon_0} (t_0), \varepsilon_0, \un{E}_{\varepsilon_0}, \un\eta)$.
To do this, consider first a point $(x,t) \in U \subset \MM$ with $t \in [ t_0 - \frac12 r_0^2, t_0] \subset [\frac12 t_0, t_0]$.
By the previous conclusion, $(x,t)$ satisfies the desired canonical neighborhood assumptions in $\MM$.
We now argue that $(x,t)$ satisfies those canonical neighborhood assumptions also in $U$.
If $|{\Rm}|^{-1/2} (x,t) > \un{r}_{\varepsilon_0} (t_0)$, then there is nothing to show.
So assume that
\[|{\Rm}|^{-1/2} (x,t) \leq \un{r}_{\varepsilon_0} (t_0) < \tfrac1{20}  \theta \min \{ 1, E^{-1}, \varepsilon_0 \} \sqrt{t_0} < \tfrac1{10} \big(\max \{ 1, E, 2 \varepsilon_0^{-1} \} \big)^{-1} r_0. \]
Then in particular $|{\Rm}|(x,t) > r_0^{-2}$ which implies by assumption (ii) that $(x,t) \not\in P(x', t_0, r_0, - r_0^2)$ for all $x' \in \partial U(t_0)$ and hence $B(x, t, \frac1{10} r_0) \subset U(t)$.
The point $(x,t)$ is a center of a strong $\varepsilon_0$-neck or an $(\varepsilon_0, E)$-cap in $\MM$.
The time-$t$ slice of this strong $\varepsilon_0$-neck or $(\varepsilon_0, E)$-cap is contained in the ball
\[ B \big( x, t, \max \{ E, 2 \varepsilon_0^{-1} \} |{\Rm}|^{-1/2} (x,t) \big) \subset B(x,t, \tfrac1{10} r_0) \subset U(t). \]
Moreover, if $(x,t)$ is the center of a strong $\varepsilon_0$-neck, then this neck reaches at most until time $t - 2 |{\Rm}|^{-1} (x,t) > t_0 - \frac12 r_0^2 - \frac1{50} r_0^2 > t - r_0^2$.
So $(x,t)$ in fact satisfies the canonical neighborhood assumptions $CNA (\un{r}_{\varepsilon_0} (t_0), \varepsilon_0, \un{E}_{\varepsilon_0}, \un\eta)$ in $U$.
It follows that all points $(x,t) \in \td{U}$ with $t \in [t_0 - \frac12 r_0^2, t_0]$ satisfy those canonical neighborhood assumptions in $\td{U}$.

It is easy to see that all surgeries on $\td{U}$ in the time-interval $[t_0 - \frac12 r_0^2, t_0]$ are performed by $\delta_{\ref{Prop:genPerelman}} ( \un{r}_{\varepsilon_0} (t_0), \td{c} w, 1, \un{E}_{\varepsilon_0}, \un\eta, 0)$-precise cutoff.
So the first paragraph of the assumptions as well as assumption (iv) and (v) of Proposition \ref{Prop:genPerelman} is satisfied for $\MM \leftarrow \td{U}$, $t_0 \leftarrow t_0$, $x_0 \leftarrow \td{x}_0$, $r_0 \leftarrow r_1$, $w \leftarrow \td{c} w$, $A \leftarrow 1$, $r \leftarrow \un{r}_{\varepsilon_0}(t_0)$, $E \leftarrow \un{E}_{\varepsilon_0}$, $\eta \leftarrow \un\eta$ whenever $r_1 \leq \min \{ \rho(x_0, t_0), \frac12 r_0 \}$.
We will now argue that assumptions (i) and (ii) are satisfied for the right choice of $r_1$, i.e. we show that there is a constant $\beta = \beta (w) > 0$ (depending only on $w$) such that these assumptions hold whenever $r_1 \leq \beta r_0$.

Consider first assumption (ii).
Let $x \in B(\td{x}_0, t_0, \beta r_0)$ be a point which survives until some time $t \in [t_0 - \frac1{10} \beta^2 r_0^2, t_0]$.
Then $\dist_{t_0} (x, \partial \td{U}(t_0)) > \frac12 r_0$ for $\beta < \frac12$ and we conclude that $\dist_t (x, \partial \td{U}(t) ) > \frac1{20} r_0$.
So assumption (ii) holds if $\beta < \frac1{200}$.

Assumption (i) requires more work.
Set $Z = Z_{\ref{Prop:genPerelman}} ( \td{c} w, 1, \un{E}_{\varepsilon_0}, \un\eta)$.
Let $t_1, t_2 \in [t_0 - \frac1{10} \beta^2 r_0^2, t_0]$, $t_1 < t_2$ and consider some point $x \in B(\td{x}_0, t_0, \beta r_0)$ which survives until time $t_2$ and a space-time curve $\gamma : [t_1, t_2] \to \MM$ with endpoint $\gamma(t_2) \in B(x, t, 4 \beta r_0)$ and which meets the boundary $\partial \td{U}$.
We want to show that for a sufficiently small choice of $\beta$ we have $\LL(\gamma) > Z \beta r_0$.
Similarly as in the last paragraph, we conclude that $\dist_{t_0} (\gamma(t_2), \partial \td{U}(t_0) ) > \frac12 r_0$ if $\beta < \frac1{100}$.
Let now
\[ P = \bigcup_{x' \in \partial \td{U}(t_0)} P(x', t_0, \tfrac12 r_0, - r_0^2) \]
be a parabolic collar neighborhood of $\partial \td{U}$.
Recall that $P$ is non-singular, $|{\Rm}| < r_0^{-2}$ on $P$ and $(\gamma(t_2), t_2) \not\in P$.
Let $[t'_1, t'_2] \subset [t_1, t_2]$ be a time-interval such that $\gamma(t'_1) \in \partial \td{U}(t'_1)$ and $\dist_{t_0} (\gamma(t'_2), \partial \td{U}' (t'_2)) > \frac12 r_0$ and such that $\gamma( [t'_1, t'_2) ) \subset P$.
Then we can estimate using the $t^{-1}$-positivity condition
\begin{multline*}
 \LL(\gamma) \geq \int_{t_1}^{t_2} \sqrt{t_2 - t} \big( |\gamma'(t)|^2_t - \tfrac32 t^{-1} \big) dt \geq \int_{t'_1}^{t'_2} \sqrt{t_2 - t} \big| \gamma'(t) \big|_t^2 dt - \tfrac32 \beta r_0 \\
 \geq c \int_{t'_1}^{t'_2} \sqrt{t_2' - t} |\gamma'(t)|_{t_0}^2 dt - \tfrac32 \beta r_0. 
\end{multline*}
Substituting $s^2 = t'_2 - t$ with $s_1^2 = t'_2 - t'_1$ yields
\begin{multline*}
 \int_{t'_1}^{t'_2} \sqrt{t_2' - t} |\gamma'(t)|_{t_0}^2 dt = \tfrac12 \int_0^{s_1} \Big| \frac{d}{ds} \gamma (t'_2 - s^2) \Big|^2_{t_0} ds \\ \geq \frac1{2\sqrt{t'_2 - t'_1}} \dist_{t_0}^2 (\gamma(t'_2), \gamma (t'_1)) 
 \geq \frac{r^2_0}{8 \sqrt{t'_2 - t'_1}} \geq \frac1{8 \beta} r_0.
\end{multline*}
Thus
\[ \LL(\gamma) > \Big( \frac{c}{8 \beta} - \tfrac32 \beta \Big) r_0. \]
For sufficiently small $\beta$, the right hand side is larger than $Z \beta r_0$.

We can finally apply Proposition \ref{Prop:genPerelman} with the parameters mentioned above and $r_1 = \min \{ \rho(x_0, t_0), \min \{ \beta, \ov{r}_{\ref{Prop:genPerelman}} (\td{c} w, 1, \un{E}_{\varepsilon_0}, \un\eta ) \} r_0 \}$.
We first obtain, that there is an $\widehat{r}_1 = \widehat{r}_1 (w) > 0$ such that $r_1 > \widehat{r}_1 r_0$.
Consider moreover the constant $C_{1, \ref{Prop:genPerelman}} = C_{1, \ref{Prop:genPerelman}} (\td{c} w, 1, \un{E}_{\varepsilon_0}, \un\eta)$ from Proposition \ref{Prop:genPerelman}.
Assuming $T$ to be large enough, we conclude $C_{1, \ref{Prop:genPerelman}} \delta(t) \leq C_{1, \ref{Prop:genPerelman}} \leq \widehat{r}_1 \theta \sqrt{t_0} < \widehat{r}_1 r_0 < r_1$ for all $t \in [\frac12 t_0, t_0]$.
This implies that the parabolic neighborhood $P(\td{x}_0, t_0, r_1, - \tau_{\ref{Prop:genPerelman}} (\td{c} w, 1, \un{E}_{\varepsilon_0}, \un\eta) r_1^2)$ is non-singular and that we have $|{\Rm}| < K_{0, \ref{Prop:genPerelman}}(\td{c} w, 1, \un{E}_{\varepsilon_0}, \un\eta) r_1^{-2}$ there which finishes the proof.
\end{proof}

\subsection{Controlled diameter growth of regions whose boundary is sufficiently collapsed and good}
In this subsection we show that if a region $U$ in a Ricci flow with surgery has controlled diameter at some time $t_1$, then we can control its curvature and diameter at some slightly later time $t_2 > t_1$ if the geometry around the boundary $\partial U$ satisfies certain collapsedness and goodness assumptions.
The important point is hereby that the size of the time-interval $[t_1, t_2]$ does not depend on the diameter of $U$ at time $t_1$.
We obtain this independence by requiring instead a collapsedness assumption which depends on the diameter of $U$ at time $t_1$.
The idea of the following proof is that by an $\LL$-geometry argument similar to Lemma \ref{Lem:6.3a}, we can deduce a $\kappa$-noncollapsedness result where the constant $\kappa$ only depends on the diameter of $U$ at \emph{time $t_1$}.
Then, an argument similar to the one in the proof of Lemma \ref{Lem:6.3bc}(b) will help us to conclude more uniform canonical neighborhood assumptions on $U$ and finally an argument similar to the one in the proof of Proposition \ref{Prop:curvcontrolincompressiblecollapse} gives us a curvature bound on $U$.

\begin{Proposition} \label{Prop:slowdiamgrowth}
There is a continuous positive function $\delta : [0, \infty) \to (0, \infty)$ and for every $w > 0$ there is a $\tau_0 = \tau_0(w) > 0$ such that for all $\theta > 0$ and $A < \infty$ there are constants $\kappa = \kappa(w, A), \td\rho = \td\rho (w, A), \ov{w} = \ov{w} (w, A) > 0$ and $K = K (w, A), A' = A'(w, A), T = T(w, A, \theta) < \infty$ such that: \\
Let $\MM$ be a Ricci flow with surgery on the time-interval $[0, \infty)$ with normalized initial conditions which is performed by $\delta(t)$-precise cutoff and let $t_0 > T$.
Let $\tau \in (0, \tau_0]$ and consider a sub-Ricci flow with surgery $U \subset \MM$ on the time-interval $[t_0 - \tau r_0^2, t_0]$ whose time-slices are connected.
Let $x_0 \in U(t_0)$ be a point which survives until time $t_0 - \tau r_0^2$ and $r_0 > 0$ be a constant.
Assume that
\begin{enumerate}[label=(\roman*)]
\item $\theta \sqrt{t_0} < r_0 \leq \sqrt{t_0}$,
\item $x_0$ is $w$-good at scale $r_0$ and time $t_0$,
\item $\vol_{t_0} B(x_0, t_0, r_0) < \ov{w} r_0^3$,
\item $\partial U(t) \subset B(x_0, t, A r_0)$ for all $t \in [t_0 - \tau r_0^2, t_0]$,
\item $U(t_0 - \tau r_0^2) \subset B(x_0, t_0 - \tau r_0^2, A r_0)$.
\end{enumerate}
Consider the universal covering flow $\td\MM$ of $\MM$ as described in Lemma \ref{Lem:tdMM} and let $\td{U} \subset \td\MM$ be a sub-Ricci flow with surgery such that $\td{U}(t) \subset \td\MM(t)$ is a family of lifts of $U(t)$.
Then
\begin{enumerate}[label=(\alph*)]
\item For all $t \in [t_0 - \tau r_0^2, t_0]$ all points of $\td{U}(t)$ are $\kappa$-noncollapsed on scales $< r_0$ in $\td\MM$.
\item There are universal constants $\eta > 0$, $E < \infty$ such that for every $t \in [t_0 - \tau r_0^2, t_0]$ the points in $\td{U}(t)$ satisfy the canonical neighborhood assumptions $CNA( \td\rho r_0, \linebreak[1] \td\varepsilon_0, \linebreak[1] E, \linebreak[1] \eta)$ in $\td\MM$.
Here $\td\varepsilon_0$ is the constant from Lemma \ref{Lem:neckhasfewquotients}.
\item There are no surgery points in $U$ (i.e. the Ricci flow with surgery $U$ is non-singular and we can write $U = U(t_0) \times [t_0 - \tau r_0^2, t_0]$) and we have $|{\Rm}| < K r_0^{-2}$ on $U(t_0) \times [t_0 - \tau r_0^2, t_0]$.
\item $U(t) \subset B(x_0, t, A' r_0)$ for all $t \in [t_0 - \tau r_0^2, t_0]$.
\end{enumerate}
\end{Proposition}
\begin{proof}
Consider the functions $\underline{\delta}_{\td\varepsilon_0} (t), \underline{r}_{\td\varepsilon_0} (t)$ and the constants $\un{E}_{\td\varepsilon_0}, \un\eta$ from Proposition \ref{Prop:CNThm-mostgeneral} ($\td\varepsilon_0$ is the constant from Lemma \ref{Lem:neckhasfewquotients}) and the function $\delta_{\ref{Prop:curvcontrolincompressiblecollapse}} (t)$ from Proposition \ref{Prop:curvcontrolincompressiblecollapse}.
Let furthermore $\delta^* (\Lambda, r, \eta, t_0)$ be the function from Claim 2 in the proof of Lemma \ref{Lem:6.3a}.
We can assume without loss of generality that $\delta^*$ is monotone in the sense that $\delta^* (\Lambda', r', \eta) \leq \delta^* (\Lambda, r, \eta)$ whenever $\Lambda' \geq \Lambda$ and $r' \leq r$.
Assume now that for all $t \geq 0$
\[ \delta(t) < \min \big\{ \delta^*( t, \tfrac12 \un{r}_{\td\varepsilon_0} (2 t), \un\eta ), \; \un\delta_{\td\varepsilon_0} (t), \; \delta_{\ref{Prop:curvcontrolincompressiblecollapse}} (t), \; t^{-1} \un{r}_{\td{\varepsilon}_0} (2 t) \big\}. \]
We note that then the flows $\MM$ and $\td\MM$ satisfy the canonical neighborhood assumptions $CNA (\un{r}_{\td\varepsilon_0} (t), \td\varepsilon_0, \un{E}_{\td\varepsilon_0}, \un\eta)$.

Set $\tau_0(w) = \frac12 \tau_{\ref{Prop:curvcontrolincompressiblecollapse}}(w)$ and assume that $\ov{w} < \ov{w}_{\ref{Prop:curvcontrolincompressiblecollapse}} (w, 2A)$ and $T > T_{\ref{Prop:curvcontrolincompressiblecollapse}} (w, 2A, \theta)$ where $\tau_{\ref{Prop:curvcontrolincompressiblecollapse}}$, $\ov{w}_{\ref{Prop:curvcontrolincompressiblecollapse}}$ and $T_{\ref{Prop:curvcontrolincompressiblecollapse}}$ are the constants from Propsition \ref{Prop:curvcontrolincompressiblecollapse}.
Then there is a constant $0 < \tau' = \tau' (w, A) < \tau$ such that the parabolic neighborhood $P(x_0, t_0 - \tau r_0^2, A r_0, - \tau' r_0^2)$ is non-singular and
\begin{equation}  \label{eq:curvboundonAparnbdtauprime}
 |{\Rm}| < K_1^* (w, A) r_0^{-2} \qquad \text{on} \qquad P(x_0, t_0 - \tau r_0^2, A r_0, - \tau' r_0^2)
\end{equation}
and such that the distance distortion on $P(x_0, t_0 - \tau r_0^2, A r_0, - \tau' r_0^2)$ can be controlled by a factor of $2$, i.e. $U(t) \subset B(x_0, t, 2 A r_0)$ for all $t \in [t_0 - \tau r_0^2, t_0 - (\tau + \tau') r_0^2]$ (note that since the previous parabolic neighborhood is non-singular, we can extend $U$ to the time-interval $[t_0 - (\tau + \tau') r_0^2, t_0]$).
Moreover, we obtain the bound
\begin{equation} \label{eq:curvboundon2Afortau}
 |{\Rm}| < K_1^*(w, A) r_0^{-2} \qquad \text{on} \qquad B = \bigcup_{t \in [t_0 - (\tau + \tau') r_0^2, t_0]} B(x_0, t, 2 A r_0)
\end{equation}
and we can assume that there are no surgery points in $B$.

\textit{Proof of assertion (a). \quad}
Here, we follow a modified version of the proof of Lemma \ref{Lem:6.3a}.
Let $t_1 \in [t_0 - \tau r_0^2, t_0]$, $\td{x}_1 \in \td{U}(t_1) \subset \td\MM (t_1)$, $0 < r_1 < r_0$ such that $P(\td{x}_1, t_1, r_1, - r_1^2)$ is non-singular and $|{\Rm}| < r_1^{-2}$ on $P(\td{x}_1, t_1, r_1, - r_1^2)$.

We first explain that for sufficiently large $T$ we can restrict ourselves to the case $r_1 > \frac12 \underline{r}_{\td\varepsilon_0} (t_1) \geq \frac12 \underline{r}_{\td\varepsilon_0} (t_0)$.
Compare this statement with Claim 1 in the proof of Lemma \ref{Lem:6.3a} (applied to $\td\MM$).
As in the proof of this claim, we chose $s > 0$ to be the supremum over all $r_1$ which satisfy the properties above and if $s \leq \frac12 \un{r}_{\td\varepsilon_0} (t_1)$, we argue as in the cases (2),(3).
Case (1) does not occur since $\td\MM$ has no boundary and case (4) does not occur since we can assume that $s \leq \frac12 \un{r}_{\td\varepsilon_0} (t_1) \leq \theta \sqrt{t_0} < r_0$.

Let $x_1 \in \MM(t_1)$ be the projection of $\td{x}_1$.
Consider the functions $L$, $\ov{L}$ and the family of domains $D_t$ on $\MM$ based in $(x_1, t_1)$.
Our first goal will be to show that $L(x_0, t_0 - (\tau + \frac12 \tau') r_0^2) < C_3 r_0$ for some univsersal $C_3 = C_3 (w, A, a, \tau) < \infty$.
An important tool will hereby be the following claim which is analogous to Claim 2 in the proof of Lemma \ref{Lem:6.3a}:

\begin{Claim}
For any $\Lambda < \infty$ there is a $T^* = T^* (\Lambda) < \infty$ such that whenever $t_0 \geq T^*$, then the following holds: If $t \in [t_0 - (\tau + \tau') r_0^2, t_1)$, $x \in \MM(t)$, $r_1 > \frac12 \un{r}_{\td\varepsilon_0} (t_0)$ and $L(x, t) < \Lambda r_0$, then $x \in D_t$ and $(x,t)$ is not a surgery point.
\end{Claim}

\begin{proof}
This follows by the choice of $\delta$ in (\ref{eq:curvboundon2Afortau}) along with Claim 2 in the proof of Lemma \ref{Lem:6.3a} (applied to $\MM$).
\end{proof}

In contrast to the proof of Lemma \ref{Lem:6.3a}, we don't need to localize the function $\ov{L}$.
So we will make use of the inequality
\begin{equation} \label{eq:evolovLinU}
 \Big( \frac{\partial}{\partial t} - \triangle \Big) \ov{L} (x, t) \geq - 6,
\end{equation}
which holds on $D_t$ (cf \cite[7.1]{PerelmanI}).
We will now apply a maximum principle argument to (\ref{eq:evolovLinU}) to show that either $\inf_{U(t)} \ov{L} (\cdot, t) \leq 6 (t_1 - t)$ for all $t \in [t_0 - \tau r_0^2, t_1)$ or there is a $t \in [t_0 - \tau r_0^2, t_1)$ such that $\inf_{\partial U(t)} \ov{L} (\cdot, t) \leq 6 (t_1 - t)$.
Assume not.
Since, $\ov{L}(x_1, t) \sim \const (t_1 - t)^2$ as $t \to t_1$, there is some $t' \in [t_0 - \tau r_0^2, t_1)$ such that $\inf_{U(t)} \ov{L} (\cdot, t) \leq 6 (t_1 - t)$ for all $t \in [t', t_1)$.
Choose $\nu > 0$ small enough such that $\inf_{\partial U(t)} \ov{L} (\cdot, t) > (6+\nu) (t_1 - t)$ for all $t \in [t_0 - \tau r_0^2, t']$ and choose $t^* \in [t_0 - \tau r_0^2, t']$ minimal with the property that $\inf_{U(t)} \ov{L} (\cdot, t) \leq (6 + \nu) (t_1 - t)$ for all $t \in (t^*, t_1)$.
So $\ov{L} (\cdot, t^*)$ attains its minimum at an interior point $x^* \in U(t^*)$.
This implies that $\triangle \ov{L} (x^* , t^*) \geq 0$.
Since $\ov{L} (x^*, t^*) \leq (6+\nu) (t_1 - t^*)$, we have $L (x^*, t^*) \leq (3 + \nu) \sqrt{t_1 - t^*} \leq 4 r_0$.
Hence by the Claim, assuming $T \geq T^*(4)$, we conclude $x^* \in D_{t^*}$ and $(x^*, t^*)$ is not a surgery point.
By the assumption on $t^*$, we must then either have $t^* = t_0 -  \tau r_0^2$ or $\ov{L} (x^*, t^*) = (6+\nu) (t_1 - t^*)$ and $\frac{\partial}{\partial t} \ov{L} (x^*, t^*) \leq - 6$ which however contradicts (\ref{eq:evolovLinU}).
So $\inf_{U(t)} \ov{L} (\cdot, t) \leq (6 + \nu) (t_1 - t)$ holds for all $\nu > 0$ and $t \in [t_0 - \tau r_0^2, t_1)$ and by letting $\nu$ go to zero, we obtain a contradiction.

Consider now the case in which there is a $t \in [t_0 - \tau r_0^2, t_1)$ such that $\inf_{\partial U(t)} \ov{L} (\cdot, t) \linebreak[1] \leq 6 (t_1 - t)$.
Let $x \in \partial U (t)$ such that $\ov{L} (x, t) \leq 6 (t_1 - t)$, i.e. $L (x, t) \leq 3 \sqrt{t_1 - t} \leq 3 r_0$.
By concatenating an $\LL$-geodesic between $(x_1, t_1)$ and $(x, t)$ with a constant space-time curve on the time-interval $[t_0 - \tau r_0^2, t]$, we conclude using (\ref{eq:curvboundon2Afortau}) and assumption (iv)
\[ L (x, t_0 - \tau r_0^2) \leq L (x, t) + C_1 K^*_1 r_0^{-2} \int^{ t_0 - \tau r_0^2}_{t} \sqrt{t_1 - t'} dt' 
 \leq 3 r_0  + C_1 K^*_1  r_0. \]
Thus, in both cases (i.e. in the case in which the infimum of $\ov{L}$ can be controlled on the boundary of $U$ as well as in the case in which it can be controlled everywhere on $U$), we can find some point $y \in U( t_0 -  \tau r_0^2 )$ such that $L(y, t_0 - \tau r_0^2) < C_2 r_0$ for some constant $C_2 = C_2 (w, A) < \infty$.
Observe that by (v) we have $y \in B( x_0, t_0 - \tau r_0^2, A r_0)$.
So by extending an $\LL$-geodesic between $(x_1, t_1)$ and $(y,  t_0 - \tau r_0^2)$ by a time-$(t_0 - \tau r_0^2)$ geodesic segment, we can conclude using (\ref{eq:curvboundonAparnbdtauprime}) that there is a constant $C_3 = C_3(w, A, \tau') = C_3 (w, A) < \infty$ such that $L(x_0,  t_0 - (\tau + \frac12 \tau') r_0^2) < C_3 r_0$.

By the Claim, assuming $T \geq T^* (C_3)$, we find that there is a smooth minimizing $\LL$-geodesic $\gamma$ between $(x_1, t_1)$ and $(x_0,  t_0 - (\tau + \frac12 \tau') r_0^2)$ which does not hit any surgery points.
We now lift $\gamma$ to an $\LL$-geodesic $\td\gamma$ in $\td\MM$ starting from $(\td{x}_1,  t_1)$ and going backwards in time.
If there are no surgery times on the time-interval $[t_0 - (\tau + \frac12 \tau') r_0^2, t_1]$, then this is trivial.
If there are, then let $T^i$ be the last surgery time which is $\leq t_1$ and lift $\gamma$ on the time-interval $[T^i, t_1]$ to $\td\MM(T^i)$.
Obviously, $\td\gamma (T^i) \in \td{U}^i_+$ and we can use the diffeomorphism $\td{\Phi}^i$ to determine the limit $\lim_{t \nearrow T^i} \td\gamma(t)$.
Starting from this limit point, we can lift $\gamma$ on the interval $[T^{i-1}, T^i)$ or $[t_0 - (\tau + \frac12 \tau') r_0^2), T^i)$ and continue the process until we reach time $t_0 - (\tau + \frac12 \tau') r_0^2$.
Let $\td{x}_0 = \td\gamma (t_0 - ( \tau + \frac12 \tau') r_0^2) \in \td\MM (t_0 - (\tau + \frac12 \tau') r_0^2)$ be the initial point of $\td\gamma$.
Then $\td{x}_0$ is a lift of $x_0$ and there is a $\nu_1 = \nu_1(w) > 0$ such that
\[  \vol_{t_0 - (\tau + \tau') r_0^2} B(\td{x}_0, t_0 - (\tau + \tau') r_0^2, r_0) > \nu_1 r_0^3. \]

We consider now the functions $L^{\td\MM}, \ell^{\td\MM}$, the domains $D^{\td\MM}_t$ and the reduced volume $\td{V}^{\td\MM} (t)$ in $\td\MM$ based in $(\td{x}_1, t_1)$.
By concatenating $\td\gamma$ with time-$(t_0 - (\tau + \frac12 \tau') r_0^2)$ geodesic segments we conclude, using the curvature bound in (\ref{eq:curvboundon2Afortau}), that there is some $C_4 = C_4 (w, A) < \infty$ such that
\[ L^{\td\MM} (x, t_0 - (\tau + \tau') r_0^2) < C_4 r_0 \qquad \text{for all} \qquad x \in B(\td{x}'_0, t_0 - (\tau + \tau') r_0^2, r_0). \]
Again, using the Claim and assuming $T \geq T^*(C_4)$, we conclude that $B(\td{x}'_0, t_0 - (\tau+ \tau') r_0^2, r_0) \subset D^{\td\MM}_{t_0 - (\tau + \tau') r_0^2}$.
So together with the inequality $t_1 - (t_0 - (\tau + \tau') r_0^2) \geq \frac12 \tau' r_0^2$, this implies that there is some $\nu_2 = \nu_2(w, A) > 0$ such that
\[ \td{V}^{\td\MM} (t_0 - (\tau + \tau') r_0^2) > \nu_2. \]
This implies the noncollapsedness in $(\td{x}_1, t_1)$.

\textit{Proof of assertion (b). \quad} The proof of this part goes along the lines of the proof of Lemma \ref{Lem:6.3bc}(a).
The main difference is however that instead of invoking Lemma \ref{Lem:6.3a} for the non-collapsing statement, we make use of assertion (a) of this Proposition.

Observe that by (\ref{eq:curvboundon2Afortau}), (ii) and basic volume comparison, we can choose $\kappa = \kappa (w, A) > 0$ such that the $\kappa$-noncollapsedness from assertion (a) even holds for all $t \in [t_0 - (\tau + \tau') r_0^2, t_0]$.

Let $w, A$ be given and let $E = \max \{ \un{E}_{\td\varepsilon_0}, E_{\ref{Lem:kappasolCNA}}(\td\varepsilon_0) \}$ and $\eta = \min \{ \un\eta, \eta_{\ref{Lem:kappasolCNA}} \}$ where $E_{\ref{Lem:kappasolCNA}}(\td\varepsilon_0)$, $\eta_{\ref{Lem:kappasolCNA}}$ are the constants from Lemma \ref{Lem:kappasolCNA}.

Assume first that the statement is false for some small $\td\rho$, i.e. there is a time $t \in [t_0 - \tau r_0^2, t_0]$ and a point $\td{x} \in \td{U} (t)$ such that $(x,t)$ does not satisfy the canonical neighborhood assumptions $CNA (\td\rho r_0, \td\varepsilon_0, E, \eta)$ on $\td\MM$.
In particular $|{\Rm}|(\td{x}, t) \geq \td\rho^{-2} r_0^{-2}$.

By a point picking argument, much easier than the one used in the proof of Lemma \ref{Lem:6.3bc}(a), we can find a time $\ov{t} \in [t_0 - \tau r_0^2, t_0]$ and a point $\ov{x} \in \td{U} ( \ov{t} )$ which have the same property and which additionally satisfy the following condition:
Set $\ov{q} = |{\Rm}|^{-1/2} (\ov{x}, \ov{t})$.
Then for any $t' \in [t_0 - (\tau + \tau') r_0^2, \ov{t}]$, all points in $\td{U}(t')$ satisfy the canonical neighborhood assumptions $CNA (\frac12 \ov{q}, \td\varepsilon_0, E, \eta)$.
Observe that here we assumed that $\td{\rho}^{-2} > K_1^*$ and hence by (\ref{eq:curvboundon2Afortau}) we did not need to extend the interval $[t_0 - \tau r_0^2, t_0]$ to pick $(\ov{x}, \ov{t})$.
Moreover, this implies that $\dist_{\ov{t}} (\ov{x}, \partial \td{U} (\ov{t})) > A r_0$.

We now assume that there are no uniform constants $\td\rho$ and $T$ such that assertion (b) holds.
Then for some given $w, A$, we can find a sequence $\td\rho^\alpha \to 0$ and a sequence of counterexamples $\td\MM^\alpha$, $U^\alpha$, $t_0^\alpha$, $r_0^\alpha$, $\tau^\alpha$, $\theta^\alpha$, $x_0^\alpha$ with $t_0^\alpha \to \infty$ and $t^\alpha_0 > T_{(a)} (w, A, \theta^\alpha)$ (here $T_{(a)}$ is the constant for which assertion (a) holds) such that there are times $t^\alpha \in [t^\alpha_0 - \tau^\alpha (r_0^\alpha)^2, t_0^\alpha]$ and points $\td{x}^\alpha \in \td{U}^\alpha (t^\alpha)$ which do not satisfy the canonical neighborhood assumptions $CNA (\td\rho^\alpha r_0^\alpha, \td\varepsilon_0, E, \eta)$ on $\td\MM$.
We can additionally assume that $t_0^\alpha \to \infty$.
By the last paragraph, we find times $\ov{t}^\alpha \in [t_0^\alpha - \tau^\alpha (r_0^\alpha)^2, t_0^\alpha]$ and points $\ov{x}^\alpha \in \td{U}(\ov{t}^\alpha)$ such that $(\ov{x}^\alpha, \ov{t}^\alpha)$ does not satisfy the canonical neighborhood assumptions $CNA(\ov{q}^\alpha, \td\varepsilon_0, E, \eta)$ with $\ov{q}^\alpha = |{\Rm}|^{-1/2} (\ov{x}^\alpha, \ov{t}^\alpha)$, but for any $t' \in [\ov{t}^\alpha - (\tau^\alpha + \tau') (r_0^\alpha)^2, \ov{t}^\alpha]$, all points in $\td{U}^\alpha (t')$ satisfy the canonical neighborhood assumptions $CNA (\frac12 \ov{q}^\alpha, \td\varepsilon_0, E, \eta)$.

Recall that we must have $\ov{q}^\alpha > \un{r}_{\td\varepsilon_0} (\ov{t}^\alpha) \geq \un{r}_{\td\varepsilon_0} (t_0^\alpha)$.
Let $(x', t') \in \MM^\alpha$ be a surgery point with $t' \in [t_0^\alpha - \tau^\alpha (r_0^\alpha)^2, t_0^\alpha]$.
Then as in (\ref{eq:surgpointhighcurv}) we have by the choice of $\delta$
\[ |{\Rm}| (x', t') > c' \delta^{-2} (t') \geq c' t^2 \un{r}_{\td\varepsilon_0}^{-2} (2 t') \geq c' (t_0^\alpha)^2 \un{r}_{\td\varepsilon_0}^{-2} (t_0^\alpha) \]
and hence $(\ov{q}^\alpha)^2 |{\Rm}| (x',t') > c' (t_0^\alpha)^2 \to \infty$.
As in the proof of Lemma \ref{Lem:6.3bc}(a) we conclude using Lemma \ref{Lem:shortrangebounds} that there is a constant $c > 0$ such that for large $\alpha$ the parabolic neighborhood $P(\ov{x}^\alpha, \ov{t}^\alpha, c \ov{q}^\alpha, - c(\ov{q}^\alpha)^2)$ is non-singular and we have $|{\Rm}| < 8 (\ov{q}^\alpha)^{-2}$ there.

Again, as in the proof of Lemma \ref{Lem:6.3bc}(a), we choose $\tau^* \geq 0$ maximal with the property that for all $\tau^{**}< \tau^*$ the point $\ov{x}^\alpha$ survives until time $\ov{t}^\alpha - \tau^{**} (\ov{q}^\alpha)^2$ for infinitely many $\alpha$.
After passing to a subsequence, we can assume that for all $\tau^{**}< \tau^*$ the point $\ov{x}^\alpha$ survives until time $\ov{t}^\alpha - \tau^{**} (\ov{q}^\alpha)^2$ for \emph{sufficiently large} $\alpha$.
Recall that $\dist_{\ov{t}^\alpha} (\ov{x}^\alpha, \partial \td{U}^\alpha ( \ov{t}^\alpha )) > A r_0^\alpha$.
By (\ref{eq:curvboundon2Afortau}) and simple distance distortion estimates, we obtain that $\dist_{t} (\ov{x}^\alpha, \partial \td{U}^\alpha ( t )) > b r_0^\alpha$ for all $t \in [ \ov{t}^\alpha - \tau^{**} (\ov{q}^\alpha)^2, \ov{t}^\alpha]$ and some $b = b(w, A) > 0$ (actually we can choose $b = b(w) > 0$).
So for every $a < \infty$ and $\tau^{**} < \tau^*$, we have $\dist_{t} (\partial \td{U}^\alpha (t), \ov{x}^\alpha) > a \ov{q}^\alpha$ for all $t \in [\ov{t}^\alpha - \tau^{**} (\ov{q}^\alpha)^2, \ov{t}^\alpha]$ whenever $\alpha$ is sufficiently large.

So by assertion (a) of this Proposition and the choice of $(\ov{x}^\alpha, \ov{t}^\alpha)$ there is a uniform constant $\kappa > 0$ such that:
For all $\tau^{**} < \tau^*$, $a < \infty$ and sufficiently large $\alpha$ we have that for all $t \in [\ov{t}^\alpha - \tau^{**} (\ov{q}^\alpha)^2, \ov{t}^\alpha]$ the points in the ball $B(\ov{x}^\alpha, t, a \ov{q}^\alpha)$ are $\kappa$-noncollapsed on scales $< r_0^\alpha$ and satisfy the canonical neighborhood assumptions $CNA(\frac12 \ov{q}^\alpha, \td\varepsilon_0, E, \eta)$.
Therefore we can follow the reasoning of the proof of Lemma \ref{Lem:6.3bc}(a) and apply Lemma \ref{Lem:lmiitswithCNA} to conclude that there is a $K^*_2 < \infty$ such that that for all $\tau^{**} < \tau^*$ we have $(\ov{q}^\alpha)^2 |{\Rm}|(\ov{x}^\alpha,t) < K_2^*$ for all $t \in [\ov{t}^\alpha - \tau^{**} (\ov{q}^\alpha)^2, \ov{t}^\alpha]$ for infinitely many $\alpha$.
If $\tau^{**} < \infty$, this implies using Lemma \ref{Lem:shortrangebounds} that there is a constant $c'' > 0$ such that $(\ov{q}^\alpha)^2 |{\Rm}|(\ov{x}^\alpha,t) < 2 K_2^*$ for all $t \in [\ov{t}^\alpha - (\tau^* + c'') (\ov{q}^\alpha)^2, \ov{t}^\alpha]$ for infinitely many $\alpha$.
In particular, $\ov{x}^\alpha$ survives until time $\ov{t}^\alpha - (\tau^* + c'') (\ov{q}^\alpha)^2$ for infinitely many $\alpha$, contradicting the choice of $\tau^{**}$.
So $\tau^* = \infty$ and again Lemma \ref{Lem:lmiitswithCNA} yields that the pointed Ricci flows with surgery $(\MM^\alpha, (\ov{x}^\alpha, \ov{t}^\alpha))$ subconverge to a $\kappa$-solution after rescaling by $(\ov{q}^\alpha)^{-1}$.
Using Lemma \ref{Lem:kappasolCNA}, this yields a contradiciton to the assumption that the points $(\ov{x}^\alpha, \ov{t}^\alpha)$ don't satisfy the canonical neighborhood assumptions $CNA(\ov{q}^\alpha, \td\varepsilon_0, E, \eta)$.

\textit{Proof of assertion (c). \quad}
The proof is similar to the proof of Proposition \ref{Prop:curvcontrolincompressiblecollapse}.
However, instead of using Lemma \ref{Lem:6.3bc}(a), we will invoke the canonical neighborhood assumptions from assertion (b) which are independent of the distance to $x_0$.
Choose $E$ and $\eta$ according to assertion (b) and set $K = \max \{ \td\rho^{-2} (w, A), E^2 K^*_1 \}$.

Fix some time $t \in [t_0 - \tau r_0^2, t_0]$ and let $a < \infty$ be maximal with the property that $|{\Rm_t}| < K r_0^{-2}$ on $B(x_0, t, ar_0) \cap U (t)$.
We assume that $U(t) \not\subset B(x_0, t, a r_0)$.
Analogously to (\ref{eq:etaextcurvbound}) we obtain using assertion (b) that
\[ |{\Rm_t}| < 4 K r_0^{-2} \qquad \text{on} \qquad \big( B(x_0, t, a r_0 + \tfrac12 \eta K^{-1/2} r_0) \cap U(t) \big) \cup B(x_0, t, 2 A r_0). \]
Choose $x_1 \in U(t)$ of time-$t$ distance exactly $a r_0$ from $x_0$ with $|{\Rm}| (x_1, t) = K r_0^{-2}$.
Then $x_1 \not\in B(x_0, t, 2A r_0)$ by (\ref{eq:curvboundon2Afortau}).
Let $\td{x}_1 \in \td{U}(t)$ be a lift of $x_1$ and $\td{x}_0 \in \td{M}(t)$ be a lift of $x_0$ such that $\dist_t (\td{x}_0, \td{x}_1) = \dist_t ( x_0, x_1 ) = a r_0$.
By assertion (b), we conclude that $(\td{x}_1, t_0)$ is either a center of an $\td\varepsilon_0$-neck or an $(\td\varepsilon_0, E)$-cap in $\td\MM (t)$.

Without loss of generality, we can assume that $\tau_0$ is chosen small enough depending on $w$ such that $B(x_0, t, \frac1{10} r_0) \subset B(x_0, t_0, r_0)$.
So we can find a $C_1 = C_1(w) < \infty$ such that $\vol_t B(x_0, t, \frac1{10} r_0) < C_1 \ov{w} r_0^3$ (see assumption (iii)).
Observe that by assumption (iv) any minimizing geodesic in $\MM(t)$ connecting $x_0$ with a pont $x \in U(t)$ is contained in $B(x_0, t, 2A r_0) \cup U(t)$.
So by (\ref{eq:curvboundon2Afortau}) and volume comparison we obtain analogously to (\ref{eq:etaballsmallvol}) that
\begin{multline*}
 \vol_t B(x_1, t, \tfrac12 \eta K^{-1/2} r_0 ) < \vol_t B(x_0, t, a r_0 + \tfrac12 \eta K^{-1/2} r_0) \cap U(t) \\
< C_2 (w, A) \vol_t B(x_0, t, \tfrac1{10} r_0) < C_2 C_1 \ov{w} r_0^3.
\end{multline*}
Using the same reasoning as in the proof of Proposition \ref{Prop:curvcontrolincompressiblecollapse}, we obtain a contradiction if $\ov{w}$ is small enough depending on $w$ and $A$.

The fact that $U$ is free of surgery points for sufficiently large $T$ follows as usual.

\textit{Proof of assertion (d). \quad} Assertion (d) follows immediately from assertion (c) by a simple distance distortion estimate.
\end{proof}

\subsection{Curvature control in large regions which are locally good}
We will now show that if we only have \emph{local} goodness control within some distance to some geometrically controlled region and if we can guarantee this control on a time-interval of uniform size, then we can deduce a curvature bound which is independent of this distance.

In this section, we will use the following notion:
Let $U \subset \MM$ be a sub-Ricci flow with surgery of $\MM$, $t$ be a time for which $U(t)$ is defined and $d \geq 0$.
Then we denote the time-$t$ $d$-tubular neighborhood around $\partial U (t)$ in $U(t)$ by $B^U (\partial U, t, d) = B( \partial U(t), t, d) \cap U(t)$.
The parabolic neighborhood $P^U (\partial U, t, d, \Delta t)$ is defined similarly.

\begin{Proposition} \label{Prop:curvboundnotnullinarea}
There is a continuous positive function $\delta : [0, \infty) \to (0, \infty)$ such that for every $w, \theta > 0$ and $A < \infty$ there are constants $K = K(w, A), \linebreak[1] T (w, \linebreak[1] A, \linebreak[1] \theta) < \infty$ such that the following holds: \\
Let $\MM$ be a Ricci flow with surgery on the time-interval $[0, \infty)$ which is performed by $\delta(t)$-precise cutoff and whose time-slices are compact and let $t_0 > T$.
Consider a sub-Ricci flow with surgery $U \subset \MM$ on the time-interval $[t_0 - r_1^2, t_0]$ whose time-slices $U(t)$ are open subsets of $\MM(t)$.
Let $r_1, r_0, b > 0$ be constants such that
\begin{enumerate}[label=(\roman*)]
\item $\theta \sqrt{t_0} \leq 2 r_1 < r_0 \leq \sqrt{t_0}$,
\item for all $x \in \partial U(t_0)$ we have $| {\Rm} | < A r_1^{-2}$ on $P (x, t_0, r_1, - r_1^2)$,
\item for every $t \in [t_0 - r_1^2, t_0]$ and $x \in B^U (\partial U, t, b)$, either $x$ is locally $w$-good at scale $r_0$ and time $t$ or $|{\Rm}| (x,t) < A r_1^{-2}$.
\end{enumerate}
Then for every $t \in (t_0 - r_1^2, t_0]$ and $x \in B^U (\partial U, t, b)$ we have
\[ | {\Rm} |(x,t) < K \big( (b - \dist_t(\partial U(t), x))^{-2} + (t - t_0 + r_1^2)^{-1} \big). \]
\end{Proposition}
\begin{proof}
Let $\delta(t)$ be an arbitrary function which goes to zero as $t \to \infty$.
Then for sufficiently large $t$ (depending on $w$, $A$ and $\theta$), we can use Definition \ref{Def:precisecutoff}(3) to conclude that no surgery point of $\MM(t)$ is locally $w$-good at scale $r_0$ and the curvature at every surgery point satisfies $|{\Rm}| > A r_1^{-2}$.
So we can assume in the following that there are no surgery points in the space-time neighborhood
\[ B = \bigcup_{t \in (t_0 - r_1^2, t_0]} B^U (\partial U, t, b) . \]
Consider the function
\[ f \quad : 
\quad (x,t) \quad \longmapsto \quad | {\Rm} |(x,t) \big( (b - \dist_t(\partial U(t), x))^{-2} + (t - t_0 + r_1^2)^{-1} \big)^{-1}  \]
on $B$.
Since $B$ is free of surgery points, we find that $|{\Rm}|$ and hence $f$ is bounded on $B$.

Denote by $H$ the supremum of $f$.
Choose some $(x_1, t_1) \in B$ where this supremum is attained up to a factor of $2$, i.e. $f(x_1, t_1) > \frac12 H$ and set $Q = r_1^2 | {\Rm} |(x_1, t_1)$.
Observe that $Q > f (x_1, t_1)$.
Now if $H \leq \max \{ 32, 2 A \}$, then we are done.
So assume in the following that $H > \max \{ 2, 2 A \}$.
This implies in particular that $Q > f(x_1, t_1) > \frac12 H > \max \{ 1, A \}$ and hence by assumption (iii) the point $x_1$ is locally $w$-good at scale $r_0$ and time $t_1$ and by assumption (ii) $(x_1, t_1) \not\in P (x, t_0, r_1, -r_1^2)$ for all $x \in \partial U(t_0)$.

Set $d_1 = \dist_{t_1}(\partial U(t_1), x_1)$, $a = r_1^{-1} (b -d_1)$ and observe that
\begin{equation}  \label{eq:4QaH2}
 Q a^2 >  f(x_1, t_1) > \tfrac12 H  \qquad \text{and} \qquad   Q (t_1 - t_0 + r_1^2) r_1^{-2} > f(x_1, t_1) > \tfrac12 H . 
\end{equation}
For all $t \geq \frac14 (t_1 - t_0 + r_1^2) + t_0 - r_1^2$ and $x \in B^U (\partial U, t, d_1 + \frac12 a r_1)$ we have
\begin{multline}
 | {\Rm} |(x,t) \leq H \big( (b - \dist_t (\partial U (t), x) )^{-2} + (t - t_0 + r_1^2)^{-1} \big) \\
\leq H \big( 4 a^{-2} r_1^{-2} + 4 (t - t_0 + r_1^2)^{-1} \big) < 16 Q r_1^{-2}. \label{eq:8Qr1}
\end{multline}

For a moment fix some arbitrary $x \in B^U (\partial U, t, d_1 + \frac14 a r_1)$ and choose $\Delta t > 0$ maximal with the property that $t_1 - \Delta t \geq \frac14 (t_1 - t_0 + r_1^2) + t_0 - r_1^2$ and $\dist_t(\partial U (t), x) < d_1 + \frac38 a$ for all $t \in [t_1 - \Delta t, t_1]$.
We will now estimate the distance distortion between $x$ and any point $x_0 \in \partial U$ using Lemma \ref{Lem:distdistortion}(b).
Observe that for all $t \in [t_1 - \Delta t, t_1]$ we have $| {\Rm} | < 16 Q r_1^{-2}$ on $B(x, t, \frac18 a r_1)$ by (\ref{eq:8Qr1}).
Moreover, by (\ref{eq:4QaH2}) we have $\frac18 Q^{-1/2} r_1 < \frac18 a r_1$.
By assumption (ii) we can also find a $\beta = \beta(A) > 0$ such that $|{\Rm}| < \beta^{-2} r_1^{-2}$ on $B(x_0, t, \beta r_1)$ and such that $B(x_0, t, \beta r_1) \subset B(x_0, t_0, r_1)$ for all $t \in [t_1 - \Delta t, t_1]$.
So Lemma \ref{Lem:distdistortion}(b) yields
\[ \frac{d}{dt} \dist_t (x_0, x) > - C \big( \sqrt{Q} + \beta^{-1} \big) r_1^{-1} \]
for some universal constant $C$ and hence using (\ref{eq:4QaH2}) we obtain
\begin{multline*}
 \Delta t \geq \min \Big\{ \frac{\tfrac18  a r_1^2}{C (Q^{1/2} + \beta^{-1})}, \tfrac34 (t_1 - t_0 + r_1^2) \Big\} > c' \min \big\{ a Q^{-1/2}, a \beta, H Q^{-1} \big\} r_1^2 \\
> c \min \big\{  H^{1/2}  Q^{-1}, H^{1/2} Q^{-1/2}, H Q^{-1} \big\} r_1^2 = c H^{1/2} Q^{-1} r_1^2
\end{multline*}
for some universal $c' > 0$ and some $c = c( A) > 0$.
So by varying $x$ and $x_0$, we conclude that $| {\Rm} | < 16 Q r_1^{-2}$ on $P ' = P^U (\partial U, t_1, d_1 + \frac14 a r_1, - c H^{1/2} Q^{-1} r_1^2)$ and that
\[ P' \subset \bigcup_{t \in [t_1 - c H^{1/2} Q^{-1} r_1^2, t_1]} B^U (\partial U, t, d_1 + \tfrac12 a r_1). \]

By the $t^{-1}$-positivity of the curvature on $\MM$, we have $\sec \geq - F (Q r_1^{-2} t_0 ) Q r_1^{-2}$ on $P'$ where $F : [0, \infty) \to [0, \infty)$ is a decreasing function which goes to zero on the open end.
Observe that $F(Q r_1^{-2} t_0) \leq F(4Q) \leq F(H)$ and hence we have the bound $\sec \geq - F(H) Q r_1^{-2}$ on $P'$.
Recall now that there is a constant $0 < \beta = \beta(A) < 1$ such that $\dist_{t_1} (\partial U (t_1), x_1) > \beta r_1$.
Then
\[ P(x_1, t_1, \min \{ \beta, \tfrac14 a \} r_1, - c H^{1/2} Q^{-1} r_1^2) \subset P'. \]
Define $S : (0, \infty) \to (0, \infty)$ by $S(x) = \min \{ F^{-1/2} (x), \frac18 x^{1/2}, \frac12 \beta  x^{1/2}, c^{1/2} x^{1/4} \}$.
Then $S(x) \to \infty$ as $x \to \infty$ and we find using (\ref{eq:4QaH2}):
\begin{alignat*}{1}
 \tfrac14 a &> \tfrac18 H^{1/2} Q^{-1/2} \geq S(H) Q^{-1/2}, \\
 \beta &\geq \tfrac12 \beta H^{1/2} Q^{-1/2} \geq S (H) Q^{-1/2}, \\
 c H^{1/2} Q^{-1} &\geq S^2(H) Q^{-1}. 
\end{alignat*}
This yields the bound
\[ \sec \geq - S^{-2} (H) Q r_1^{-2} \qquad \text{on} \qquad P(x_1, t_1, S(H) Q^{-1/2} r_1, -S^2(H) Q^{-1} r_1^2). \]
In particular $\rho_{r_0} (x_1, t_1) \geq  S(H) Q^{-1/2} r_1$ (observe that $S(H) Q^{-1/2} r_1 \leq \beta r_1 \leq r_0$).

So by property (iii), we conclude that for $r = S(H) Q^{-1/2} r_1$ we have $\vol_{t_1} \widetilde{B}(\td{x}_1, \linebreak[1] t_1, \linebreak[1] r) > \td{c} w r^3$ where $\td{B} (\td{x}_1, t_1, r) $ denotes the universal cover of the ball $B(x_1, t_1, r)$.
We can now lift the flow from on $P(x_1, t_1, r, -r^2)$ to this universal cover, rescale it by $r^{-1}$ and use Lemma \ref{Lem:6.5} to conclude
\[ Q r_1^{-2} = |{\Rm}| (x_1, t_1) < K_0 (\td{c} w) \tau_0^{-1}( \td{c} w) r^{-2} = K_0  \tau_0^{-1} S^{-2} (H) Q r_1^{-2}. \]
This implies $S^2(H) < K_0 \tau_0$ and hence $H$ is bounded by some universal constant $H_0 = H_0(w) < \infty$.
This finishes the proof.
\end{proof}

\section{Preparations for the main argument} \label{sec:Preparations}
In this section we list smaller Lemmas which will be used in the main argument in section \ref{sec:mainargument}.
\subsection{Evolution of areas of minimal surfaces} \label{subsec:minsurf}
The following Lemma will be used in the proof of Proposition \ref{Prop:irreducibleafterfinitetime} to show that after some time, all time-slices in a Ricci flow with surgery are irreducible.
\begin{Lemma} \label{Lem:evolsphere}
Let $\MM$ be a Ricci flow with surgery and precise cutoff and closed time-slices, defined on the time-interval $[T_1, T_2]$ $(0 < T_1 < T_2)$, assume that the surgeries are all trivial and that $\pi_2 (\MM(t)) \not= 0$ for all $t \in [T_1, T_2]$.
For every time $t \in [T_1, T_2]$ denote by $A(t)$ the infimum of the areas of all homotopically non-trivial immersed $2$-spheres.
Then the quantity
\[ t^{1/4} \big( t^{-1} A(t) + 16 \pi \big) \]
is monotonically non-increasing on $[T_1, T_2]$.
Moreover,
\[ T_2 < \Big( 1 +  \frac{1 }{16 \pi}  T_1^{-1} A(T_1) \Big)^4 T_1. \]
\end{Lemma}
\begin{proof}
Compare also with \cite[Lemma 18.10 and 18.11]{MTRicciflow}.
Let $t_0 \in [T_1, T_2)$.
By \cite{SU81} and \cite{Gul} or \cite{MY}, there is a non-contractible, conformal, minimal immersion $f : S^2 \to \MM(t_0)$ with $\area_{t_0} f = \area_{S^2} f^* (g(t_0)) = A(t_0)$.
We remark, that using the methods in the proof of Lemma \ref{Lem:evolminsurfgeneral} below, it is enough to assume that $f$ is only smooth.
Call $\Sigma = f(S^2) \subset \MM(t)$.
Then $\Sigma$ is either a $2$-sphere or an $\IR P^2$ with a finite number of self-intersections.
We can estimate the infinitesimal change of the area of $\Sigma$ (we count the area twice if $\Sigma$ is an $\IR P^2$) while we vary the metric in positive time direction (and keep $f$ constant!).
Using the $t_0^{-1}$-positivity of the curvature on $\MM(t_0)$, the fact that the interior sectional curvatures are not larger than the ambient ones as well as Gau\ss-Bonnet, we conclude:
\begin{multline*}
\frac{d}{dt^+} \Big|_{t = t_0} \area_{t} (\Sigma ) = - \int_{\Sigma} \tr_{t_0} (\Ric_{t_0} |_{T \Sigma}) d {\vol}_{t_0} \\
= - \frac12 \int_{\Sigma} \scal_{t_0} d {\vol}_{t_0} - \int_{\Sigma} \sec_{t_0}^{\MM(t_0)}(T \Sigma) d {\vol}_{t_0} 
 \leq \frac3{4t_0} \area_{t_0} (\Sigma) - \int_{\Sigma} \sec^\Sigma d {\vol}_{t_0} \\
 \leq \frac3{4t_0} \area_{t_0}(\Sigma) - 2 \pi \chi(\Sigma) = \frac3{4t_0} A(t_0) - 4 \pi.
\end{multline*}
Here, $\sec^{\MM(t_0)}_{t_0}(T\Sigma)$ denotes the ambient sectional curvature of ${\MM(t_0)}$ tangential to $\Sigma$ and $\sec^\Sigma_{t_0}$ denotes the interior sectional curvature of $\Sigma$.
We conclude from this calculation that $\frac{d}{dt^+}|_{t = t_0} (t^{1/4} (t^{-1} A(t) + 16 \pi )) \leq 0$ in the barrier sense and hence, the quantity $t^{1/4} (t^{-1} A(t) + 16 \pi )$ is monotonically non-increasing in $t$ away from the singular times.

We will now show that $A(t)$ is lower semi-continuous.
We can restrict ourselves to the case in which $t_0$ is a surgery time.
Let $t_k \nearrow t_0$ be a sequence converging to $t_0$ and choose minimal $2$-spheres $\Sigma_k \subset \MM(t_k)$ with $\area_{t_k} \Sigma_k = A(t_k)$.
By property (5) of Definition \ref{Def:precisecutoff}, we find diffeomorphisms $\xi_k : \MM(t_k) \to \MM(t_0)$ which are $(1+\chi_k)$-Lipschitz for $\chi_k \to 0$.
So $A(t_0) \leq \lim \inf_{k \to \infty} (1+\chi_k)^2 A(t_k) = \lim \inf_{k \to \infty} A(t_k)$.

The lower semi-continuity implies that $t^{1/4} (t^{-1} A(t) + 16 \pi )$ is monotonically non-increasing on $[T_1, T_2]$.
The bound on $T_2$ follows from the fact that $A(T_2) > 0$.
\end{proof}

The next Lemma estimates the evolution of the areas of minimal surfaces of higher genus which can have boundary.
It describes two different kinds of behaviors:
First, it shows that the area of a minimal disk spanning a short geodesic loop of controlled geodesic curvature and speed, goes to zero in finite time.
This fact is highlighted in part (a) of the Lemma and will be used in the proof of Theorem \ref{Thm:MainTheorem} to exclude the long-time existence of short contractible loops as asserted in Proposition \ref{Prop:structontimeinterval}.
Secondly, it demonstrates that the normalized area of a minimal surface of higher genus is asymptotically bounded by a constant which essentially only depends on its Euler characteristic.
See part (b) of the Lemma for more details.
We will use this asymptotic bound in the proof of Theorem \ref{Thm:MainTheorem} to deduce the existence of filling surfaces of controlled area at large times.
This will then enable use to apply Proposition \ref{Prop:structontimeinterval}.

\begin{Lemma} \label{Lem:evolminsurfgeneral}
Let $\MM$ be a Ricci flow with surgery and precise cutoff and closed time-slices, defined on the time-interval $[T_1, T_2]$ $(T_2 > T_1 > 0)$, assume that the surgeries are all trivial and that $\pi_2(\MM(t)) = 0$ for all $t \in [T_1, T_2]$.

Let $\gamma_{1,t}, \ldots, \gamma_{m,t}  \subset \MM(t)$ be families of smoothly embedded, pairwise disjoint loops in $\MM(t)$ parameterized by $t \in [T_1, T_2]$, which move by isotopies in $t$ and which don't meet surgery points.
Let $\Sigma$ be an orientable (not necessarily connected) surface none of whose components are spheres which has $m$ boundary circles and let $f_{T_1} : \Sigma \to \MM(T_1)$ be an incompressible map (i.e. the induced map $\pi_1(\Sigma) \to \pi_1(\MM(T_1))$ is injective) such that $f_{T_1}$ restricted to the boundary circles of $\Sigma$ provides parameterizations of the loops $\gamma_{1, T_1}, \ldots, \gamma_{m, T_1}$.
For every time $t \in [T_1, T_2]$ denote by $A(t)$ the infimum of the areas of all smooth maps $f' : \Sigma \to \MM(t)$ whose restriction to $\partial \Sigma$ parameterizes the loops $\gamma_{1,t}, \ldots, \gamma_{m,t}$ and for which there is a homotopy to $f_{T_1}$ in space-time which extends the isotopies given by the $\gamma_{i,t}$.
(In the case $m = 0$, this means that $f_{T_1}$ and $f'$ are homotopic.)

Assume that there are constants $\Gamma, a, b > 0$ such that for all $t \in [T_1, T_2]$ and $i = 1, \ldots, m$
\begin{enumerate}[label=(\roman*)]
\item the geodesic curvatures along $\gamma_{i, t}$ satisfy the bound $|\kappa(\gamma_{i, t})| < \Gamma t^{-1}$,
\item the normalized length of $\gamma_{i, t}$ satisfies the bound $\ell(\gamma_{i, t}) < a t^{-1/2}$,
\item the velocity by which $\gamma_{i, t}$ moves satisfies the bound $|\partial_t \gamma_{i, t}| < b t^{-1/2}$.
\end{enumerate}

Then the quantity
\[ t^{1/4} \big( t^{-1} A(t) + 4 (2\pi \chi(\Sigma) - m a ( \Gamma + b) ) \big) \]
is monotonically non-increasing on $[T_1, T_2]$.

In particular,
\begin{enumerate}[label=(\alph*)]
\item if $m a (\Gamma + b) < 2 \pi \chi(\Sigma)$, then
\[ T_2 < \Big( 1 +  \frac{ T_1^{-1} A(T_1) }{4 (2\pi \chi(\Sigma) - m a (\Gamma + b))} \Big)^4 T_1. \]
\item if the conditions even hold for $[T_1, \infty)$ instead of $[T_1, T_2]$, then
\[ \limsup_{t \to \infty} t^{-1} A(t) \leq 4\big( - 2\pi \chi(\Sigma) + m a (\Gamma + b) \big). \]
Note that if $\Sigma$ is a torus, then this implies $\lim_{t \to \infty} t^{-1} A(t) = 0$.
\end{enumerate}
\end{Lemma}

\begin{proof}
In the case in which $\Sigma$ has no boundary, the proof is almost the same as the proof of Lemma \ref{Lem:evolsphere}.
We now have to make use of \cite{Sacks-Uhlenbeck-1982} or \cite{Schoen-Yau-1979} and \cite{Gul} to show that at every time $t$ there is a representative $f_t : \Sigma \to \MM(t)$ in the homotopy class of $f$ which minimizes area and which is an immersion.

If $\Sigma$ has a boundary, we have to be more careful due to the existence of possible branch points.
Let $t_0 \in [T_1, T_2]$.
In the case in which $\Sigma$ is a disk, we can use the results of \cite{Mor} to find a time-$t_0$ area minimizing continuous map $f : \Sigma \to \MM(t_0)$ with the following properties: $f$ is homotopic to $f_{T_1}$ in the sense explained above (in fact all such maps are homotopic to one another in $\MM(t_0)$ relative boundary) and $f$ restricted to the boundary $\partial \Sigma$ represents a parameterization of $\gamma$. Then $f$ is moreover smooth, conformal and harmonic on the interior of $\Sigma$ (with respect to a conformal structure on $\Sigma = D^2$) and we have $A(t_0) = \area f^*(g(t_0))$.
By combining the methods of \cite{Mor} and \cite{Sacks-Uhlenbeck-1982} or \cite{Schoen-Yau-1979}, it is possible to see that this construction also works in the case in which $\Sigma$ has higher genus.
Note that however in this case, the conformal structure on $\Sigma$ is not unique and depends on the geometric setting.
We use \cite{HH} to conclude that $f$ is smooth up to the boundary.

Analogously to the proof of Lemma \ref{Lem:evolsphere}, we can carry out the first part of the computation of the infinitesimal change of the area of $f$ as we vary the metric only:
\begin{multline*}
 \frac{d}{dt} \Big|_{t = t_0} \area f^*(g(t)) = - \int_{\Sigma} \tr f^*(\Ric_{t_0}^{\MM(t_0)}) \\
 \leq \frac{3}{4t_0} A(t_0) - \int_{\Sigma}  \sec^{\MM(t_0)}( df) d {\vol}_{f^* ( g(t_0) )},
\end{multline*}
where $\sec^{\MM(t_0)} (df)$ denotes the sectional curvature in the normalized tangential direction of $f$. 
Observe that the last integrand is a continuous function on $\Sigma$ since the volume form vanishes wherever this tangential sectional curvature is not defined.

In order to avoid issues arising from possible branch points (especially on the boundary of $\Sigma$), we employ the following trick (compare with \cite{PerelmanIII}):
Fix a metric $g'$ on $\Sigma$ which represents the conformal structure with respect to which $f$ is conformal and harmonic.
Let $\varepsilon > 0$ be a small constant and consider the Riemannian manifold $(N_\varepsilon = \Sigma, \varepsilon g')$.
The identity map $h_\varepsilon : \Sigma \to (\Sigma, \varepsilon g')$ is a conformal and harmonic diffeomorphism and hence the map $f_\varepsilon = (f, h_\varepsilon) : \Sigma \to \MM(t_0) \times N_\varepsilon$ is a conformal and harmonic \emph{embedding}.
Denote its image by $\Sigma_\varepsilon = f_\varepsilon(\Sigma)$.
Since the sectional curvatures on the target manifold are bounded, we have
\[ \lim_{\varepsilon \to 0} \int_{\Sigma_\varepsilon} \sec^{\MM(t_0) \times N_\varepsilon} ( T \Sigma_\varepsilon) d {\vol}_{t_0} =  \int_{\Sigma}  \sec^{\MM(t_0)}( df) d {\vol}_{f^*(g(t_0))}. \]
We can now proceed as in the proof of Lemma \ref{Lem:evolsphere}, using the fact that the interior sectional curvatures of $\Sigma_\varepsilon$ are not larger than the corresponding ambient ones as well as the Theorem of Gau\ss-Bonnet:
\[ \int_{\Sigma_\varepsilon} \sec^{\MM(t_0) \times N_\varepsilon} ( T \Sigma_\varepsilon) d {\vol}_{t_0} \geq \int_{\Sigma_\varepsilon} \sec^{\Sigma_\varepsilon} ( T \Sigma_\varepsilon) d {\vol}_{t_0} = 2 \pi \chi(\Sigma_\varepsilon) + \int_{\partial \Sigma_\varepsilon} \kappa^{\Sigma_\varepsilon}_{\partial \Sigma_\varepsilon} d s_{t_0}. \]
We now estimate the last integral.
Let $\gamma_{i,\varepsilon} : S^1(l_{i, \varepsilon}) \to \partial \Sigma_{\varepsilon}$, $i = 1, \ldots, m$ be unit-speed parameterizations of the boundary of $\Sigma_\varepsilon$.
Denote by $\gamma^{\MM(t_0)}_{i, \varepsilon}(s)$ their component functions in $\MM(t_0)$ and by $\gamma^{N_\varepsilon}_{i, \varepsilon}(s)$ those in $N_\varepsilon$.
Furthermore, let $\nu_{i, \varepsilon}(s)$ be the outward-pointing unit-normal fields along $\gamma_{i, \varepsilon}(s)$ which are tangent to $\Sigma_{\varepsilon}$.
It is not difficult to see that due to conformality, the $\MM(t_0)$-component of $\nu_{i, \varepsilon}(s)$ has the same length as the component of the velocity vector $(\gamma^{\MM(t_0)}_{i, \varepsilon})' (s)$ at that point.
Hence, we can compute
\begin{multline*}
 - \int_{\gamma_{i, \varepsilon}} \kappa_{\gamma_{i, \varepsilon}}^{\Sigma_\varepsilon} d s_{t_0} = - \int_0^{l_\varepsilon} \Big\langle \frac{D}{ds} \Big( \frac{d}{ds}  \gamma^{\MM(t_0)}_{i, \varepsilon} (s) \Big), \nu^{\MM(t_0)}_{i, \varepsilon}(s) \Big\rangle ds \\
 - \int_0^{l_{i, \varepsilon}} \Big\langle \frac{D}{ds} \Big( \frac{d}{ds}  \gamma^{N_\varepsilon}_{i, \varepsilon} (s) \Big), \nu^{N_\varepsilon}_{i, \varepsilon}(s) \Big\rangle ds
\end{multline*}
The first integral is bounded by $ \Gamma t_0^{-1/2} l_{i, \varepsilon}$ and the second integral is bounded by $|\kappa_{\partial N_\varepsilon}| | \nu^{N_\varepsilon}_{i, \varepsilon}|$ which goes to zero as $\varepsilon \to 0$.
So passing to the limit $\varepsilon \to 0$ and using $l_{i, \varepsilon} \to \ell(\gamma_{t_0}) < a t_0^{1/2}$, we hence obtain
\[ \frac{d}{dt} \Big|_{t = t_0} \area f^* ( g(t)) \leq \frac{3}{4 t_0} A(t_0) - 2 \pi \chi(\Sigma) + m a \Gamma . \]

In order to bound the derivative of $A(t)$ in the barrier sense, we have to account for the fact that the boundary curves move by isotopies.
The maximal additional infinitesimal increase per boundary curve is then
\[ \ell(\gamma_{i, t_0}) \sup_{\gamma_{i, t_0}} | \partial_t \gamma_{i, t_0} | < a b . \]
So in the barrier sense
\[ \frac{d}{dt^+} \Big|_{t = t_0} A(t) \leq \frac{3}{4 t_0} A(t_0) - 2 \pi \chi(\Sigma) + m a \Gamma + m ab. \]
Thus
\[
 \frac{d}{dt^+} \big[ t^{1/4} \big( t^{-1} A(t) + 4 (2\pi \chi(\Sigma) - m a (\Gamma + b)) \big) \big] \leq 0.
\]

Analogously as in the proof of Lemma \ref{Lem:evolsphere}, we conclude that $A(t)$ is lower semi-continuous.
The desired monotonicity follows now immediately.
The bound in assertion (a) follows again from the fact that $A(T_2) > 0$ and assertion (b) is clear.
\end{proof}

\subsection{Torus structures and torus collars}
We will make use of the following terminology to describe the geometry of collar neighborhoods in an approximate sense.

\begin{Definition} \label{Def:torusstruc}
Let $a > 0$.
A subset $P \subset M$ of a Riemannian manifold $(M, g)$ is called \emph{torus structure of width $\leq a$} if there is a diffeomorphism $\Phi : T^2 \times [0,1] \to P$ such that $\diam \Phi(T^2 \times \{ s \}) \leq a$ for all $s \in [0,1]$.
The \emph{length} of $P$ is the distance between the two boundary components of $P$ inside $P$. \\
If $h, r_0 > 0$, then we say that $P$ is \emph{$h$-precise (at scale $r_0$)} if it has width $\leq h r_0$ and length $> h^{-1} r_0$.
\end{Definition}
Obviously, every torus structure of width $\leq a$ and length $L_1$ can be shortened to a torus structure of width $\leq a$ and length $L_2$ for any $L_2 < L_1$.

In the proof of Proposition \ref{Prop:firstcurvboundstep2}, we will moreover make use of the following related notion:
\begin{Definition} \label{Def:toruscollars}
Consider a constant $a> 0$, a Riemannian manifold $(M,g)$ and be a smoothly embedded solid torus $S \subset M, S \approx S^1 \times D^2$.
We say that $S$ \emph{has torus collars of width $\leq a$ and length up to $b$}, if for every point $x \in \Int S$ with $\dist(x, \partial S) \leq b$ there is a set $P \subset S$ which is diffeomorphic to $T^2 \times I$ such that:
$P$ is bounded by $\partial S$ and another smoothly embedded $2$-torus $T \subset S$ with $x \in T$ and $\diam T \leq a$.
\end{Definition}
So if $P \subset S$ (such that $\partial S \subset \partial P$) is a torus structure of width $\leq a$ and length $b$, then $S$ has torus collars of width $\leq a$ and length up to $b$.

We mention two conclusions which we will use frequently.
\begin{Lemma} \label{Lem:collardiameter}
Assume that $S$ has torus collars of width $\leq a$ and length up to $b$.
Let $x \in \Int S$ with $\dist(x, \partial S) < b - 2a$ and choose $P \subset S$ according to Defintion \ref{Def:toruscollars}.
Then $\dist(x, \partial S) \leq \diam P \leq \dist(x, \partial S) + 4a$.
\end{Lemma}
\begin{proof}
The first inequality is clear.
For the second inequality consider a minimizing geodesic $\gamma$ joining $\partial S$ with $x$.
By minimality, $\gamma \subset S$ and all points of $\gamma$ have distance $< b - 2a$ from $\partial S$.
Let $y \in P \setminus \partial S$ and assume that $\dist(y, \partial S) \leq b$.
So there is an embedded $2$-torus $T' \subset S$ with $\diam T' \leq a$ and a set $P'$ which is diffeomorphic to $T^2 \times I$ and bounded by $\partial S$ and $T'$.
Then $T'$ intersects either $\gamma$ or $T$.
In the first case, $\dist(y, \gamma) \leq a$ and in the second case $\dist(y, \gamma) \leq \dist(y, x) \leq 2a$.
So in fact $\dist(y, \partial S) < b$.
This implies that \emph{all} points of $P \setminus \partial S$ have distance less than $b$ from $\partial S$ and hence are not more than $2a$ away from $\gamma$.
This implies the diameter bound.
\end{proof}

\begin{Lemma} \label{Lem:2T2timesI}
Consider two subsets $P_1, P_2 \subset M$ of a smooth $3$-manifold which are diffeomorphic to $T^2 \times I$.
Assume that one boundary component, $T_1$, of $P_1$ is contained in the interior of $P_2$ and the other boundary component, $T'_1$, is disjoint from $P_2$.
Assume also that conversely one boundary component, $T_2$, of $P_2$ is contained in the interior of $P_1$ and the other boundary component, $T'_2$ is disjoint from $P_1$.
Then $P_1 \cup P_2$ is diffeomorphic to $T^2 \times I$
\end{Lemma}
\begin{proof}
Observe that $P_1 \setminus P_2$ and $P_2 \setminus P_1$ are connected.
Consider the sequence of maps which are induced by inclusion
\[ \pi_1( T'_1 ) \longrightarrow \pi_1 (P_1 \setminus P_2) \longrightarrow \pi_1 (P_1) \longrightarrow \pi_1 (P_1 \cup P_2). \]
The composition of the first two maps, $\pi_1(T'_1) \to \pi_1(P_1)$, is an isomorphism and since $P_1 \setminus P_2$ is a deformation retract of $P_1 \cup P_2$, the composition of the last two maps, $\pi_1 ( P_1 \setminus P_2 ) \to \pi_1 ( P_1 \cup P_2)$, is also an isomorphism.
So all maps are isomorphisms.
In particular $\pi_1 (T'_1) \to \pi_1 (P_1 \cup P_2)$ and hence also $\pi_1 (T_1) \to \pi_1 (P_1 \cup P_2)$ are isomorphisms.
We thus conclude that $T_1$ is incompressible in $P_2$.

By elementary $3$-manifold topology (see e.g. the proof of \cite[Proposition 1.7]{Hat}), this implies that $P_2 \setminus P_1 \approx T^2 \times [0,1)$ and hence $P_1 \cup P_2 = (P_2 \setminus P_1) \cup P_1 \approx T^2 \times I$.
\end{proof}

The next Lemma asserts that under the presence of a curvature bound, we can find a torus structure of small width around a cross-section of small diameter inside a given torus structure.
This fact will be used in the proof of Proposition \ref{Prop:firstcurvboundstep3}.
In the subsequent Lemma \ref{Lem:2loopstorus} we show that such a small cross-section exists if we can find two short loops which represent linearly independent homotopy classes inside the torus structure.

\begin{Lemma} \label{Lem:bettertorusstructure}
For any $K < \infty$, $L < \infty$ and $h > 0$ there is a constant $0 < \td\nu = \td\nu(K, L, h) < 1$ such that: \\
Let $(M,g)$ be a complete Riemannian manifold and consider a torus structure $P' \subset M$ of width $\leq 1$ and assume that $|{\Rm}| < K$ on $P'$.
Consider a subset $T \subset P'$  (e.g. a cross-sectional torus) which separates the two boundary components of $P'$ and which has distance $\geq \frac12 L + 30$ from the boundary components of $P'$ and assume that $\diam T < \td\nu$.

Then there is a torus structure $P \subset P'$ of width $\leq h$ and length $> L$ such that $T \subset P'$ and such that the pair $(P', P)$ is diffeomorphic to $(T^2 \times [-2,2], T^2 \times [-1,1])$.
\end{Lemma}
\begin{proof}
By chopping off $P'$, we first construct a torus structure $P'_1 \subset P'$ of width $\leq 1$ and length $< L + 100$ such that the boundary tori of $P'_1$ have distance at least $5$ from the boundary tori of $P'$ and such that $T$ has distance of at least $\frac12 L + 20$ from $\partial P'_1$.
Then still $T \subset P'_1$.
Choose points $z_1, z_2 \in \partial P'_1$ in each boundary component of $P'_1$ and let $\gamma \subset M$ be a minimizing geodesic from $z_1$ to $z_2$.
Then $\gamma \subset P'$ and $\gamma$ intersects $T$ in a point $z$.

By the same construction as above, we choose $P'_2 \subset P'_1$ such that the boundary tori of $P'_2$ have distance at least $5$ from the boundary tori of $P'_1$ and such that $T$ has distance of at least $\frac12 L + 10$ from $\partial P'_2$.
We still have $T \subset P'_2$.
Let now $x \in P'_2$ be an arbitrary point.
Consider minimizing geodesics $\gamma_1, \gamma_2 \subset M$ from $x$ to $z_1$, $z_2$.
Then again $\gamma_1, \gamma_2 \subset P'$ and one of these geodesics have to intersect $T$, without loss of generality assume that this geodesic is $\gamma_1$ and choose a point $x_1 \in \gamma \cap T$.
Let $y_1 \in \gamma$ be a point with $\dist(z_1, y_1) = \dist(z_1, x)$ (we can find such a point since $\dist(z_1,x) < \dist(z_1, z_2)$).
We now apply Toponogov's Theorem using the lower sectional curvature bound $- K$:
Observe that $\dist(z_1, x_1), \dist(z_1, z) < L + 100$ and $\dist(x_1, z) < \td\nu$.
So the comparison angle $\beta = \cangle x_1 z_1 z$ (in the model space of constant sectional curvature $-K$) is bounded by a quantity $\beta_0 = \beta_0 (\td\nu, L, K)$ which goes to zero in $\td\nu$ whenever $L$ and $K$ are kept fixed.
By Toponogov's Theorem, we have $\cangle x z_1 y_1 \leq \beta \leq \beta_0$ and since the comparison triangle $\td{\triangle} x z_1 y_1$ is isosceles and the lengths of the hinges are bounded by $L+100$, we conclude that $\dist(x, y_1) < \beta_1(\td\nu, L, K)$, where $\beta_1(\td\nu, L, K)$ is a quantity which goes to zero in $\td\nu$ if $L$ and $K$ are kept fixed.
This implies in particular that
\[ \dist(z_1, z_2) \leq \dist(z_1, x) + \dist( z_2, x) \leq \dist(z_1, z_2) + 2 \beta_1(\td\nu, L, K). \]
Hence, if $\td\nu$ is small enough depending on $L$ and $K$, then we have the following bound for the comparison angle at $x$:
\begin{equation} \label{eq:cangleatxtorusstruct}
 \cangle z_1 x z_2 > 0.9 \pi.
\end{equation}
For the rest of the proof, fix such a $\td\nu > 0$ for which also $\beta_1(\td\nu, L, K) < 0.1h$.

By (\ref{eq:cangleatxtorusstruct}) the function $p : \Int P'_2 \to \IR$, $p(x) = \dist(z_1, x)$ is regular in a uniform sense and hence we can find a unit vector field $\chi$ on $\Int P'_2$ such that the directional derivative of $p$ is uniformly positive everywhere, i.e.  $\chi \cdot p > c > 0$.
We can moreover choose a smoothing $p'$ of $p$ with $| p- p' | < 0.1h$ and $\chi \cdot p' > 0$ everywhere (compare with \cite{Grove-Shiohama-1977} and \cite{Meyer-1989}).
Let $P = (p')^{-1}(I)$ be the preimage of a closed subinterval $I \subset p'(P'_2)$ whose endpoints have distance $3$ from the endpoints of $p' (P'_2)$.
This implies that the preimage $P = (p')^{-1}(t)$ of every point $t \in I$ is far enough from the boundary of $P'_2$ and hence is compact.
Then in particular $T \subset P$.
So $P \approx \Sigma \times I$, for some connected, closed surface $\Sigma$ and $p'$ is the projection onto the second factor.
Since $T \subset P$, it follows that that $\pi_1(\Sigma)$ contains a subgroup isomorphic to $\IZ^2$ which implies that $\Sigma \approx T^2$.

We now estimate the diameter of $(p')^{-1}(t)$ for each $t \in I$.
Let again $x \in P$ and consider as above the geodesics $\gamma_1, \gamma_2$ as well as the point $y_1 \in \gamma$ with $\dist(z_1, y_1) = \dist (z_1, x) = p(x)$.
Additionally, we construct $y_2 \in \gamma$ with $\dist (z_2, y_2) = \dist(z_2, x)$.
Then $\dist(y_1, y_2) \leq 0.2 h$.
In the case in which $\gamma_1$ intersects $T$, we conclude as above that $\dist(x, y_1) \leq 0.1 h$.
Analogously, if $\gamma_2$ intersects $T$, we have $\dist(x, y_2) \leq 0.1 h$ and hence $\dist(x,y_1) \leq 0.3 h$.
Let $y' \in \gamma$ now be a point with $\dist(z_1, y') = p'(x)$.
Then $\dist(y', y_1) \leq | p (x) - p' (x) | < 0.1 h$ and hence $\dist(y', x) < 0.4 h$.
This implies that $\diam (p')^{-1}(t) < 0.8 h < h$ for all $t \in I$.
So $P$ has width $\leq h$.

Finally, we bound the length of $P$ from below.
Consider points $x_1, x_2 \in \partial P$ in each boundary component and let $y'_1, y'_2 \in\gamma$ be points with $\dist (z_1, y'_1) = p'(x_1)$ and $\dist (z_1, y'_2) = p'(x_2)$.
Then by the last paragraph
\begin{multline*}
 \dist(x_1, x_2) > \dist (y'_1, y'_2) - 2 \cdot 0.4 h = | p'(x_1) - p'(x_2) | - 0.8 h \\
  = \ell( p'(P'_2) ) - 2 \cdot 3 - 0.8 h > \ell( p (P'_2) ) - 6 - h.
\end{multline*}
where $\ell( p (P'_2) )$ denotes the length of the interval $p(P'_2)$.
By assumption $p(P'_2) \geq 2 (\frac12 L + 10) = L + 20$.
So $\dist(x_1, x_2) > L + 14 - h > L$ for $h < 1$.
This implies that $P$ has length $> L$.
\end{proof}

\begin{Lemma} \label{Lem:2loopstorus}
For every $K < \infty$ there is a constant $\td\varepsilon_1 = \td\varepsilon_1(K) > 0$ such that: \\
Let $(M,g)$ a complete Riemannian manifold with boundary which is diffeomorphic to $T^2 \times I$ and $p \in M$ such that $B(p, 1) \subset M \setminus \partial M$.
Assume that $| {\Rm} | < K$ and assume that there are loops $\gamma_1, \gamma_2$ based in $p$ which represent two linearly independent homotopy classes in $\pi_1(M) \cong \IZ^2$.
Then if $m = \max \{ \ell(\gamma_1), \ell(\gamma_2) \} < \td\varepsilon_1$, there is an embedded incompressible torus $T \subset M$ which separates the two ends of $M$ such that $p \in T$ and $\diam T < 10 m$.
\end{Lemma}
\begin{proof}
By a local version of the results of Cheeger, Fukaya and Gromov \cite{CFG}, there are universal constants $\rho = \rho(K) > 0$, $k < \infty$ such that we can find an open neighborhood $B(p, \rho) \subset V \subset M$ and a metric $g'$ on $V$ with $0.9 g < g' < 1.1 g$ with the following properties: 
There is a Lie group $H$ with at most $k$ connected components whose identity component $N$ is nilpotent and which acts isometrically and effectively on the universal cover $(\td{V}, \td{g}')$.
The fundamental group $\Lambda = \pi_1(V)$ can be embedded into $H$ such the action of $H$ on $(\td{V}, \td{g}')$ restricted to $\Lambda$ is the action by deck transformations.
Moreover, $H$ is generated by $\Lambda$ and $N$.
Lastly, the injectivity radius of $(\td{V}, \td{g}')$ at any lift $\td{p}$ of $p$ is larger than $\rho$.

Consider the dimension $d$ of the orbit $\td{T}$ of a lift $\td{p}$ under the action of $N$.
Since $V$ has to be non-compact, we have $d \leq 2$.
On the other hand, assuming $\td{\varepsilon}_1 < \rho$, the loops $\gamma_1, \gamma_2$ generate an infinite subgroup in $\Lambda = \pi_1(V)$ which does not have a finite index subgroup isomorphic to $\IZ$.
So $d = 2$.
Since $N \cap \Lambda$ is nilpotent and acts discontinuously on $\td{T}$, we have $N \cap \Lambda \cong \IZ^2$ and all orbits of the $N$-action are $2$ dimensional.
Consider the cover $\widehat\pi : (\widehat{V}, \widehat{g}') \to (V, g')$ corresponding to $N \cap \Lambda$.
Then $\widehat{V} \approx T^2 \times (0,1)$ and $\widehat{V} \to V$ has at most two sheets.
The action of $N$ on $(\widehat{V}, \widehat{g}')$ exhibits $(\widehat{V}, \widehat{g}')$ as a warped product of a flat torus over an interval.
We can find loops $\gamma'_1, \gamma'_2$ based at a lift $\widehat{p}$ of $p$ each of which have $\widehat{g}'$-length $< 2 (1.1)^{1/2} m$.
This implies that the $T^2$-orbit $\widehat{T}$ of $\widehat{p}$ under $N$ has $g'$-diameter $< 4 \cdot 1.1^{1/2} m$.
Let $T = \widehat\pi (\widehat{T})$ be the projection of $\widehat{T}$.
Then $\diam_{g} T < 4 \cdot 1.1^{1/2} \cdot 0.9^{-1/2} m < 10 m$ and $\widehat\pi$ restricted to $\widehat{T}$ induces a map $f : T^2 \to M$ with $f(T^2) = T$ which has at most two sheets.

We show that the intersection number of $f$ with the line $\{ \mathop{\text{pt}} \} \times I \subset M \approx T^2 \times I$ is non-zero:
Consider the composition of $f$ with the projection $M \approx T^2 \times I \to T^2$.
This is a smooth and incompressible map of the form $T^2 \to T^2$. 
Hence, its degree is non-zero which is equal to the sought intersection number.
We conclude that $T \subset M$ is a $2$-torus which separates the two boundary components of $M$.
\end{proof}

The next Lemma will be used in the proof of Lemma \ref{Lem:smallloopintorusstruc}.
\begin{Lemma} \label{Lem:1loop2d}
For every $K < \infty$, there are constants $\td\varepsilon_2 = \td\varepsilon_2 (K) > 0$ and $\Gamma' = \Gamma' (K) < \infty$ such that the following holds: \\
Let $(M,g)$ a $2$-dimensional, orientable Riemannian manifold with and $p \in M$ a point such that the $1$-ball around $x$ is relatively compact in $M$.
Assume that $| {\Rm} |  < K$ on $M$ and assume that there is a loop $\gamma : S^1 \to M$ based in $p$ which is non-contractible in $M$ and has length $\ell (\gamma) < \td\varepsilon_2$.
Then there is an embedded loop $\gamma' \subset M$ which is also based in $p$, homotopic to $\gamma$ and which satisfies the following properties:
$\ell(\gamma') < 2 \ell(\gamma)$ and the geodesic curvatures of $\gamma'$ are bounded by $\Gamma'$.
\end{Lemma}

\begin{proof}
Similarly as in the proof of Lemma \ref{Lem:2loopstorus} there is a universal constant $\rho = \rho(K) > 0$ such that we can find a neighborhood $B(p, \rho) \subset V \subset M$ and a metric $g'$ on $V$ with $0.9 g < g' < 1.1 g$ such that the same conditions as above hold.
Note that $g'$ can moreover be chosen such that $|\nabla - \nabla'| < 0.1$ and such that the curvature of $g'$ is bounded by a universal constant $K' = K' (K) < \infty$ (see \cite{CFG}).
Hence, it suffices to construct loop $\gamma'$ with $\ell_{g'} (\gamma') < 1.5 \ell_{g'} (\gamma)$ and on which the geodesic curvatures with respect to $g'$ are bounded.

As in the proof of Lemma \ref{Lem:2loopstorus} we conclude that either $(V, g')$ is a flat torus (in which case the Lemma is clear), or all orbits under the action of $N$ on the universal cover $(\td{V}, \td{g}')$ of $(V,g')$ are 1 dimensional.
In the latter case, this implies that $\Lambda \subset N$ and $(V, g')$ is a warped product of a circle over an interval $(-a,b)$ with $a, b > 0.9 \rho$ (we assume that $p$ lies in the fiber over $0 \in (-a, b)$).
Let $\varphi : (-a,b) \to (0,\infty)$ be the warping function.
By the curvature bound on $g'$ we have $| \varphi'' \varphi^{-1} | < K'$ on $(-a,b)$.
This implies a bound $| \varphi' \varphi^{-1} | < C = C (K')$ on $(- \frac12 a, \frac12 b)$.
Hence the geodesic curvature on the circle $\gamma'$ through $p$ is bounded by $C$ and for sufficiently small $\ell_{\gamma'} (\gamma)$ we have $\ell_{g'} (\gamma') < 1.5 \ell_{g'} (\gamma)$.
\end{proof}

\subsection{Existence of short loops}
In this subsection we establish the existence of short geodesic loops on surfaces of large diameter, but controlled area.
The main result of this subsection, Lemma \ref{Lem:smallloopintorusstruc}, will be used in the proof of Propositions  \ref{Prop:firstcurvboundstep3} and \ref{Prop:structontimeinterval}.
In the proof of Proposition  \ref{Prop:firstcurvboundstep3} it will enable us to apply Lemmas \ref{Lem:bettertorusstructure} and \ref{Lem:2loopstorus} and hence to find torus structures of small width.
\begin{Lemma} \label{Lem:annulus}
Let $\Sigma$ be a topological annulus and let $g$ be a symmetric non-negative $2$-form on $\Sigma$ (i.e. a possibly degenerate Riemannian metric).
Assume that with respect to $g$ any smooth curve connecting the two boundary components of $\Sigma$ has length $\geq a$ and every embedded, closed loop of non-zero winding number in $\Sigma$ has length $\geq b$.
Then $\area \Sigma \geq a b$.
\end{Lemma}
\begin{proof}
Let $g'$ be an arbitrary metric on $\Sigma$.
If the Lemma is true for $g + \varepsilon g'$ for all $\varepsilon > 0$, then we obtain the result for $g$ by letting $\varepsilon \to 0$.
So we can assume in the following that $g$ is a Riemannian metric.

We can furthermore assume that $\Sigma = A(r_1, r_2) \subset \IC$ with $0 \leq r_1 < r_2 \leq \infty$ and that $g = r^{-2} f^2(r, \theta) g_{\textnormal{eucl}}$ for polar coordinates $(r, \theta)$.
By assumption, we have
\[ \int_{r_1}^{r_2} r^{-1} f(r, \theta) dr \geq a \qquad \text{for all} \qquad \theta \in [0, 2\pi] \]
and
\[ \int_0^{2\pi} f(r, \theta) d\theta \geq b \qquad \text{for all} \qquad r \in (r_1, r_2). \]
Hence
\[  \int_0^{2\pi} \int_{r_1}^{r_2}r^{-1} f(r, \theta) dr d\theta \geq 2 \pi a \]
and
\[ \int_{r_1}^{r_2} \int_0^{2\pi} r^{-1} f(r, \theta) d\theta dr \geq b \int_{r_1}^{r_2} r^{-1} dr = b \log \Big( \frac{r_2}{r_1} \Big). \]
So
\begin{multline*}
 2 \pi a b \log \Big( \frac{r_2}{r_1} \Big) \leq \bigg( \int_0^{2\pi} \int_{r_1}^{r_2} r^{-1} f(r, \theta) dr d \theta \bigg)^2 \\ \leq \bigg( \int_0^{2\pi} \int_{r_1}^{r_2} r^{-1} f^2(r, \theta) dr d\theta \bigg) \bigg( \int_0^{2\pi} \int_{r_1}^{r_2} r^{-1} dr d \theta \bigg)
 = 2 \pi (\area \Sigma) \log \Big( \frac{r_2}{r_1} \Big) \qedhere
\end{multline*}
\end{proof}
\vspace{2mm}

\begin{Lemma} \label{Lem:multiannulus}
Let $\Sigma$ be a closed topological multi-annulus which is bounded by circles $C_0, \ldots, C_m$.
As in Lemma \ref{Lem:annulus} consider a symmetric non-negative $2$-form $g$ on $\Sigma$.
Let moreover $w_1, \ldots, w_m \in \IR$ be weights with $w_1 + \ldots + w_n \not= 0$. \\
Assume now that the distance of $C_0$ to any $C_k$ ($k=1, \ldots, m$) is $\geq a$ and every embedded closed loop $\gamma \subset \Sigma$ which encloses boundary circles $C_{k_1}, \ldots, C_{k_{m'}}$ with $w_{k_1} + \ldots + w_{k_{m'}} \not= 0$ has length $\ell(\gamma) \geq b$.
Distances and lengths are measured with respect to $g$.
Then $\area \Sigma \geq ab$.
\end{Lemma}
\begin{proof}
We will proceed by induction on $m$.
The Claim follows from Lemma \ref{Lem:annulus} for $m = 1$.
So assume $m \geq 2$.

Let $a_1$ be the supremum of all $a' > 0$ such that all boundary circles $C_1, \ldots, C_m$ are contained in the same component of the set $\{ x \in \Sigma \; : \; \dist(x, C_0) > a' \}$.
Let $\varepsilon > 0$ and consider the multi-annulus $\Sigma_1 = \{ x \in \Sigma \; : \; \dist(x, C_0) < a_1 - \varepsilon \}$.
The complement of its closure in $\Sigma$ consists of one component which contains all boundary circles $C_1, \ldots, C_m$ plus components which are topologically open disks.
Let $\Sigma_2$ be the union of the closure of $\Sigma_1$ with those open disks.
So $\Sigma_1$ is an annulus and its boundary components of $\Sigma_2$ are at least $a_1 - \varepsilon$ away from each other.
Moreover, every non-contractible circle in $\Sigma_2$ encloses all boundary components $C_1, \ldots, C_m$ and hence has length $\geq b$.
So by Lemma \ref{Lem:annulus} we have $\area \Sigma_2 \geq (a_1 - \varepsilon) b$.
In the case $a_1 \geq a$, the claim follows by letting $\varepsilon$ go to $0$.
So assume in the following that $a_1 < a$, and assume $\varepsilon$ is so small that $a_1 + 2 \varepsilon < a$.

Let $\Sigma_3$ be a connected component of $\{ x \in \Sigma \; : \; \dist(x, C_0) > a_1 + \varepsilon \}$ with the property that the weights corresponding to the boundary circles $C_{k_1}, \ldots, C_{k_{m'}}$ that it contains don't add up to zero, i.e. $w_{k_1} + \ldots + w_{k_{m'}} \neq 0$.
By assumption this is always possible and by the choice of $a_1$ we have $m' < m$.
Since $\Sigma \setminus \Sigma_3$ is connected, we conclude that $\Sigma_3$ has exactly one open end.
Choose a closed multi-annulus $\Sigma_4 \subset \Sigma_3$ which is still bounded by the circles $C_{k_1}, \ldots, C_{k_{m'}}$ and a circle $C'_0$ which approximates this open end such that $\dist(x, C_0) < a_1 + 2 \varepsilon$ for all $x \in C'_0$.
So the distance of $C'_0$ to any of the circles $C_{k_1}, \ldots, C_{k_{n'}}$ is at least $a- a_1 - 2 \varepsilon$.
By induction, this implies $\area \Sigma_4 \geq (a - a_1 - 2 \varepsilon)$.

Hence $\area \Sigma \geq \area \Sigma_2 + \area \Sigma_4 > (a_1 - \varepsilon) b + (a - a_1 - 2 \varepsilon) b = (a - 3 \varepsilon) b$.
The claim follows by letting $\varepsilon$ go to $0$.
\end{proof}

\begin{Lemma} \label{Lem:smallloopintorusstruc}
For every $\alpha > 0$ and every $A, K < \infty$ there are constants $\td{L}_0 = \td{L}_0 (\alpha, A) < \infty$ and $\td\alpha_0 = \td\alpha_0 (K) > 0$, $\td\Gamma = \td\Gamma(K) < \infty$ such that: \\
Let $(M, g)$ be a closed, irreducible Riemannian manifold which is not a spherical space form and consider a decomposition $M = M_{\textnormal{hyp}} \cup M_{\textnormal{Seif}}$ along embedded, incompressible $2$-tori such that $M_{\textnormal{Seif}}$ is graph (see Definition \ref{Def:geomdec}).

Consider a smoothly embedded solid torus $S \subset \Int M_{\textnormal{Seif}}$ ($S \approx S^1 \times D^2$) which is incompressible in $M$.
Assume that there is a Seifert decomposition of $M_{\textnormal{Seif}}$ such that one of its generic Seifert fibers is contained and non-contractible in $S$.
 
Let $P \subset S$ be a torus structure of width $\leq 1$ and length $L > \td{L}_0$ with $\partial S \subset \partial P$ (i.e. the pair $(S, \Int S \setminus \Int P)$ is diffeomorphic to $(S^1 \times D^2(1), S^1 \times D^2(\frac12))$).
Moreover assume that we have the curvature derivative bound $|{\Rm}| < K$ on $P$.

Finally, let $f : \Sigma \to M$ a filling surface for the pair $(M_{\textnormal{hyp}}, M_{\textnormal{Seif}})$ (in the sense of Definition \ref{Def:filling}) of area $\area f < A$.

Then there is a closed loop $\gamma \subset P$ which is non-contractible in $P$, but contractible in $S$.
It has length $\ell(\gamma) < \alpha$, its distance from $\partial P$ is at least $\frac13 L - 2$ and it spans a disk $D \subset M$ of area $< A$. 
Moreover if $\alpha \leq \td\alpha_0$ and if we also have the bound $|{\nabla \Rm}| < K$ on $P$, then we can choose $\gamma$ such that its geodesic curvatures are bounded by $\td\Gamma$.
\end{Lemma}

\begin{proof}
Assume that without loss of generality $\alpha < 0.1$ and set
\[ \td{L}_0 (\alpha, A) = 12 \frac{A}\alpha + 6. \]
We divide $P$ into three torus structures $P_1, P_2, P_3$ of width $\leq 1$ and length $> \frac13 L - 1$ in such a way that: $\partial S \subset \partial P_1$ and $P_k$ shares a boundary with $P_{k+1}$.
For later use, we set
\[ S_1 = S, \qquad S_2 = \Int S \setminus \Int P_1 \; \approx \; S^1 \times D^2,  \qquad S_ 3 = \Int S \setminus \Int (P_1 \cup P_2). \]
Moreover, let $\sigma \subset S \setminus P$ be an embedded loop which generates $\pi_1 (S) \cong \IZ$.
By Definition \ref{Def:filling}, there is a component $\Sigma_0 \subset \Sigma$ such that every map $f' : \Sigma_0 \to M$, which is homotopic to $f|_{\Sigma_0} : \Sigma_0 \to M$ relative boundary, intersects $\sigma$ and moreover, there is no loop $\gamma' : S^1 \to \Sigma_0$ for which $f \circ \gamma' : S^1 \to M$ is homotopic in $M$ to a non-zero multiple of $\sigma$.
Let $\varepsilon > 0$ be a small constant that we will determine later ($\varepsilon$ may depend on $M$ and $g$).
We can find a small homotopic perturbation $f_0 : \Sigma_0 \to M$ of $f : \Sigma \to M$ with $f_0|_{\partial \Sigma_0} = f|_{\partial \Sigma_0}$ which is not more than $\varepsilon$ away from $f$ such that the following holds: $f_0$ is transverse to $\partial P_2$ and its area is still bounded: $\area f_0 < A$.

We will first find a disk of bounded area whose boundary lies in $\partial S_2$ and which has non-zero intersection number with $\sigma$.
Consider for this all components $Q_1, \ldots, Q_p$ of $f_0^{-1} (S_2) \subset \Sigma_0$.
By the choice of $\Sigma_0$ we have $p > 0$.
For every $j = 1, \ldots, p$ and every topological disk $D \subset \Sigma_0 \setminus \Int Q_j$, the loop $f_0 |_{\partial D}$ is contractible in $M$ and hence also in $S_2$ (since $S_2$ is incompressible in $M$).
So we can find a map $f_{D} : D \to S_2$ whose image is contained in $S_2$ and which agrees with $f_0$ on $\partial D$.
Since $\pi_2 (M) = 0$ (see Proposition \ref{Prop:pi2irred}), the maps $f_0 |_D$ and $f_D$ are homotopic.
Let $Q'_j$ be the union of $Q_j$ with all open topological disks of $\Sigma_0 \setminus Q_j$.
Our discussion has shown that we can homotope $f_0$ to a map $f_j : \Sigma_0 \to M$ relative boundary such that $f_0 |_{Q_j} = f_j |_{Q_j}$ and $f(Q'_j) \subset S_2$.
We now show that $Q'_j$ is a disk:
Assume not.
By construction, $\Sigma_0 \setminus Q'_j$ does not contain any topological disk and hence each boundary circle of $Q'_j$ is incompressible in $\Sigma_0$.
This implies that $\pi_1(Q'_j) \to \pi_1(\Sigma_0)$ is injective and so, since $f_j$ is incompressible, also the induced map $\pi_1(Q'_j) \to \pi_1(S_2) \cong \IZ$ is injective.
So $Q'_j$ can only be a disk (in contradiction to our assumption) or an annulus and if it is an annulus, then $f_0$ restricted to either of its boundary circles needs to be non-contractible in $S_2$.
But this contradicts the choice of $\Sigma_0$.
So $Q'_j$ is in fact a disk for all $j = 1, \ldots, p$.
It remains to show that $f_0$ restricted to one of the disks $Q'_1, \ldots, Q'_p$ has non-trivial intersection number with $\sigma$.
Observe that two such disks are either disjoint or one disk is contained in the other.
So without loss of generality, we can assume that $Q_1, \ldots, Q_p$ are arranged in such a way that there is a $p' \leq p$ such that $Q'_1, \ldots, Q'_{p'}$ are pairwise disjoint and $Q'_1 \cup \ldots \cup Q'_p \subset Q'_1 \cup \ldots \cup Q'_{p'}$.
It is easy to see that if $f_0$ restricted to $Q'_j$ has zero intersection number with $\sigma$, then $f_0 |_{Q'_j}$ (being homotopic to $f_j |_{Q'_j}$) can be homotoped relative boundary into $\partial S_j$.
So if $f_0$ restricted to any of the disks $Q'_1, \ldots, Q'_{p'}$ has intersection number zero with $\sigma$, then $f_0$ can be homotoped away from $\sigma$ completely, contradicting the choice of $\Sigma_0$.

For the rest of the proof choose $j = 1, \ldots, p$ such that $f_0$ restricted to the closed disk $Q'_j$ has non-zero intersection number with $\sigma$ and such that $Q'_j$ is chosen minimal with this property, i.e. for every other $Q'_{j'} \subset Q'_j$ the intersection number of $f_0$ restricted to $Q'_{j'}$ with $\sigma$ is zero.
Let $C_0 = \partial Q'_j$ and consider the multi-annulus of $Q'' = Q'_j \setminus f_0^{-1} ( \Int S_3 )$ which is bounded by $C_0$.
Call its other boundary circles $C_1, \ldots, C_q \subset \partial Q''$.
Then $f_0(C_0)$ has distance of at least $\frac13 L - 1$ to each $f_0(C_l)$, $l = 1,\ldots, q$.
For every $l = 1, \ldots q$ define the weight $w_l$ of $C_l$ to be the intersection number of $f_0$ restricted to the disk which is bounded by $C_l$ with $\sigma$.
Observe that by the choice of $j$, the intersection number of $f_0 |_{Q'_j}$ with $\sigma$ is $w_1 + \ldots + w_q \neq 0$.
We now apply  Lemma \ref{Lem:multiannulus} to find an embedded loop $\gamma' \subset Q''$ of length 
\[ \ell (f_0 |_{\gamma'} ) < \frac{A}{\frac13 L - 1} < \frac{A}{\frac13 \td{L}_0 - 1} < \tfrac14 \alpha \]
which encloses boundary circles $C_l$ whose weights $w_l$ don't add up to zero.
Denote by $D' \subset Q''$ the disk which is bounded by $\gamma'$.
Then $f_0$ restricted to $D'$ has non-zero intersection number with $\sigma$.
We now argue that $\gamma' \cap f_0^{-1} (S_2) = \gamma' \cap (Q_1 \cup \ldots \cup Q_p) \neq \emptyset$:
If not, then $\gamma'$ is contained in an open topological disk of $Q'_j \setminus Q_j$ and $f_0^{-1} (S_2) \cap D'$ is contained in the disjoint union of some of the $Q'_1, \ldots, Q'_p$ which are contained in $Q'_j$.
However, by the choice of $j$, this implies that $f_0$ restricted to $D'$ has intersection number zero with $\sigma$ which is a contradiction.
So $\gamma'$ in fact intersects $f^{-1}_0 (S_2)$ and hence by the shortness of its length under $f_0$, we conclude that $f_0(\gamma') \subset P$ and that $f_0 (\gamma')$ has distance $> \frac13 L -1.1$ from $\partial P$.
Since $f_0 |_{\gamma'}$ is contractible in $M$ (via $f_0 |_{D'}$), it is also contractible in $S$.
Moreover, $f_0 |_{\gamma'}$ is non-contractible in $P$ for the following reason:
If not then there is a nullhomotopy $f' : D' \to P$ with $f'|_{\gamma'} = f_0|_{\gamma'}$.
Since $\pi_2(M) = 0$, the maps $f'$ and $f_0|_{D'}$ are homotopic relative boundary, which is impossible since their intersection numbers with $\sigma$ are different.
Hence, $\gamma = f_0 (\gamma')$ satisfies the first part of the claim.

We now consider the case in which $\alpha \leq \td\alpha_0$ for some $\td\alpha_0 = \td\alpha_0 ( K)$ that we will determine in the course of the proof and we will construct a curve $\gamma \subset P$ on which the geodesic curvature can be bounded by a constant depending only on $K$.
For the rest of the proof we assume without loss of generality that $f$ is area minimizing amongst all maps $f' : \Sigma_0 \to M$ which are homotopic to $f$ relative boundary.
So, $f$ parameterizes a \emph{stable} minimal surface on $\Int \Sigma_0$.
By \cite{Gul}, $f$ is an immersion on $\Int \Sigma_0$.
So we can additionally assume that the perturbation $f_0$ is a graph over $f$.

Consider the following regions:
Let $B(P_2, 1)$ and $B(P_2, 2)$ be the (open) $1$ and $2$-tubular neighborhoods of $P_2$ and let $\Sigma_1$ and $\Sigma_2$ be the components of $f_0^{-1}( B(P_2, 1 ))$ and $f_0^{-1} ( B(P_2, 2))$ which contain $\gamma'$, i.e. $\gamma' \subset \Sigma_1 \subset \Sigma_2$.
Moreover, $B(P_2, 2) \subset P$ and every point of $B(P_2, 2)$ has distance of at least $1$ from $\partial P$.
By the results of \cite{Sch}, we obtain a bound on the second fundamental form of $f(\Sigma_1)$ which only depends on $K$: $| II_{f(\Sigma_1)}| < K' (K)$.
Since $f_0$ was assumed to be a graph over $f(\Sigma)$, we conclude that
\[ \ell (f|_{\gamma'})< 2 \ell ( f_0 |_{\gamma'}) < \tfrac12 \alpha   \]
if $\varepsilon$ is small enough depending on these bounds.
Moreover, this bound and the bound on the curvature on $B(P_2, 1)$ yields a curvature bound $K'' = K'' (K) < \infty$ of the metric $f^* ( g )$ on $\Sigma_1$ which only depends on $K$.
The loop $\gamma'$ is non-contractible in $\Sigma_1$, because otherwise $f_0 |_{\gamma'}$ would be contractible in $P$.
So we can apply Lemma \ref{Lem:1loop2d} to conclude that if $\td\alpha_0 < \td\varepsilon_2 (K'')$, then there is a loop $\gamma'' \subset \Sigma_1$ which intersects $\gamma'$, is homotopic to $\gamma'$ in $\Sigma_1$ and has the following properties: $\ell(f|_{\gamma''}) < 2 \ell (f|_{\gamma'}) < \alpha$ and the geodesic curvature on $\gamma''$ in $\Sigma_1$ is bounded by $\Gamma' (K'')$.
Obviously, $\gamma''$ bounds a disk in $\Sigma_0$ whose area under $f$ is bounded by $A$.
Let now $\gamma = f (\gamma'')$.
Then the geodesic curvature on $\gamma$ in $M$ is bounded by some constant $\td\Gamma = \td\Gamma(K) <\infty$ which can be computed in terms of $\Gamma' (K'')$ and $K'(K)$.
\end{proof}

\section{The main argument} \label{sec:mainargument}
\subsection{The geometry on late, but short time-intervals} \label{subsec:firstcurvbounds}
Using the tools from section \ref{sec:maintools}, we will give a bound on the curvature and the geometry in certain areas of the manifold on a time-interval of small, but uniform size.
This description is achieved in three steps, the last step being Proposition \ref{Prop:firstcurvboundstep3}.

In the first step, we bound the curvature by a uniform constant away from some embedded solid tori.
Inside those solid tori, but in controlled distance to their boundary, we will establish a curvature bound which however deteriorates with the distance.
We will also give an approximate characterization of the geometry inside those regions of controlled distance.

\begin{Lemma}[first step] \label{Lem:firstcurvboundstep1}
There are a continuous function $\delta : [0, \infty) \to (0, \infty)$ and constants $K_1 < \infty$ and $\tau_1, \ov{w}, w^\#, \mu^\# > 0$ as well as continuous, non-decreasing functions $\Delta_1, K'_1 : (0, \infty) \to (0, \infty)$ with $\Delta_1(d) \to \infty$ as $d \to \infty$ and a non-increasing functions $\tau'_1 : (0, \infty) \to (0, \infty)$ such that the following holds: \\
For every $L < \infty$ and $\nu > 0$ there are constants $T_1 = T_1(L) < \infty$, $w_1 = w_1(L, \nu) > 0$ such that: \\
Let $\MM$ be a Ricci flow with surgery on the time-interval $[0, \infty)$ with normalized initial conditions which is performed by $\delta(t)$-precise cutoff.
Consider the constant $T_0 < \infty$ and the function $w : [T_0, \infty) \to (0, \infty)$ from Proposition \ref{Prop:thickthindec}.
Assume that $t_0 = r_0^2 \geq \max \{ T_1, 2 T_0 \}$ and that $w(t) < w_1$ for all $t \in [\frac12 t_0, t_0]$.
Assume moreover that all components of $\MM (t_0)$ are irreducible and not diffeomorphic to spherical space forms and that all surgeries on the time-interval $[\frac12 t_0, t_0]$ are trivial. \\
Let $U \subset \MM(t_0)$ be a subset with either $U = \MM(t_0)$ or
\begin{enumerate}[label=$-$]
\item $U$ is a smoothly embedded solid torus ($\approx S^1 \times D^2$).
\item There is a closed subset $U' \subset U$ which is diffeomorphic to $T^2 \times I$ with $\partial U \subset \partial U'$ whose boundary components have time-$t_0$ distance of at least $2 r_0$ and a fibration $p : U \to I$ such that the $T^2$-fiber through every $x \in U'$ has time-$t_0$ diameter $< \mu^\# \rho_{r_0} (x, t_0)$.
\item All points of $\partial U$ are $w^\#$-good at scale $r_0$ and time $t_0$.
\end{enumerate} \vspace{2mm}
Then there are sub-Ricci flows with surgery $S_1, \ldots, S_m \subset \MM$ on the time-interval $[(1-\tau_1) t_0, t_0]$ such that their final time-slices $S_1(t_0), \ldots, S_m (t_0)$ form a collection pairwise disjoint, incompressible solid tori ($\approx S^1 \times D^2$) in $\Int U$.
Moreover, there are subsubets $W_i \subset S_i(t_0)$ ($i = 1, \ldots, m$) such that for all $i = 1, \ldots, m$
\begin{enumerate}[label=(\alph*)]
\item The pair $(S_i (t_0), W_i)$ is diffeomorphic to $(S^1 \times D^2(1), S^1 \times D^2(\frac12))$.
\item For all $x \in U$ with $\dist_{t_0} (x, U \setminus (S_1(t_0) \cup \ldots \cup S_m(t_0))) \leq 100 r_0$, the point $(x, t_0)$ survives until time $(1- \tau_1) t_0$ and for all $t \in [(1- \tau_1) t_0, t_0]$ we have
\[  |{\Rm}| (x,t) < K_1 t_0^{-1} . \]
\item We have $\diam_{t_0} S_i(t_0) > 100 r_0$ and if $\diam_{t_0} S_i(t_0) \leq L r_0$, then we have the curvature bound
\begin{multline*}
 \qquad\qquad | {\Rm_t} | < K'_1 ( r_0^{-1} \diam_{t_0} S_i (t_0)) t_0^{-1} \quad \text{on} \quad S_i(t) \quad \\
  \text{for all} \quad t \in \big[ (1-\tau'_1( r_0^{-1} \diam_{t_0} S_i(t_0)) t_0, t_0 \big].
\end{multline*}
Furthermore, $S_i$ is non-singular on the time-interval $[ (1-\tau'_1( r_0^{-1} \diam_{t_0} S_i \linebreak[1] (t_0)) t_0, \linebreak[1] t_0]$.
\item For all $t \in [(1-\tau_1) t_0, t_0]$ we have
\[  \diam_t S_i(t) > \min \big\{ \Delta_1 (r_0^{-1} \diam_{t_0} S_i(t_0)), L \big\} r_0. \]
\item At time $t_0$, the closure of $S_i(t_0) \setminus W_i$ is a torus structure of width $\leq r_0$ and length 
\[ \dist_{t_0} (\partial S_i (t_0), \partial W_i) = \min \big\{ \diam_{t_0} S_i(t_0) - 2r_0 , Lr_0 \big\}. \]
\item All points of $S_i(t_0) \setminus W_i$ are locally $\ov{w}$-good at scale $r_0$ and time $t_0$.
\item For every point $x \in S_i(t_0) \setminus W_i$, there is a loop $\sigma \subset \MM(t_0)$ which is based at $x$, which is incompressible in $\MM(t_0)$ and which has length $\ell_{t_0} (\sigma) < \nu r_0$.
\item In the case in which $U = \MM(t_0)$, we have for all $t \in [(1-\tau_1) t_0, t_0]$:
\begin{enumerate}[label=$-$]
\item There is a closed subset $U'_{i,t} \subset S_i (t)$ which is diffeomorphic to $T^2 \times I$ with $\partial S_i (t) \subset \partial U'_{i,t}$ whose boundary components have time-$t$ distance of at least $2 \sqrt{t}$ and a fibration $p_{i,t}$ such that the $T^2$-fiber through every $x \in U'_{i,t}$ has time-$t$ diameter $< \mu^\# \rho_{\sqrt{t}} (x, t)$.
\item The points on $\partial S_i$ are $w^\#$-good at scale $\sqrt{t}$ and time $t$.
\end{enumerate}
\end{enumerate}
\end{Lemma}
\begin{proof}
Observe that it suffices to construct such functions $\Delta_1$ and $K'_1$ which satisfy all the claimed properties except for continuity, since all properties stay true after decreasing the values of $\Delta_1$ and increasing the values of $K'_1$.

The function $\delta(t)$ will be assumed to be bounded by the corresponding functions from Corollary \ref{Cor:Perelman68} and Propositions \ref{Prop:thickthindec}, \ref{Prop:curvcontrolgood}, \ref{Prop:curvcontrolincompressiblecollapse}, \ref{Prop:curvboundinbetween} and \ref{Prop:slowdiamgrowth}.
We also set
\[ \mu^\# = \min \{ w_0(\min \{ \mu_1, \tfrac1{10} \}, \ov{r}( \cdot, 1), K_2 (\cdot, 1)), \mu_1, \tfrac1{10} \}, \]
where $w_0$ is the constant from Proposition \ref{Prop:MorganTianMain}, $\mu_1$ is the constant from Lemma \ref{Lem:unwrapfibration} and $\ov{r}, K_2$ are the functions from Corollary \ref{Cor:Perelman68}.
If $U = \MM(t_0)$, then we set $\mu^\circ = \mu^\#$ and if $U$ is a solid torus, we set $\mu^\circ = \min \{ \mu_1, \frac1{10} \}$.

Next, we make a remark on the constant $w^\#$.
It appears in the conditions of the Lemma in the case in which $U$ is a solid torus and in assertion (h) which holds in the case $U = \MM(t_0)$.
In the following proof, both of these cases will be dealt with simultaneously.
In the case in which $U = \MM(t_0)$, the constant $w^\#$ will be determined and will never be used.
In the case in which $U$ is a solid torus, $w^\#$ will be assumed to be given and all universal constants, that are determined in this case, may depend on it.
Note that this does not create a circular argument since one could carry out the following proof first for the case $U = \MM(t_0)$, obtaining a set of constants and functions 
\begin{equation} \label{eq:constantsfromstep1}
 K_1, \tau_1, \ov{w}, \Delta_1, K'_1, \tau'_1, T_1, w_1
\end{equation}
as well as $w^\#$ and then one could carry out the proof again in the case in which $U$ is a solid torus, obtaining another set of constants and functions as listed in (\ref{eq:constantsfromstep1}).
The final set of constants and functions will then be the minima of the two values obtained for each $\tau_1, \ov{w}, \Delta_1, \tau'_1, T_1, w_1$ in each case and the maxima of the two values obtained for each $K_1, K'_1, T_1$ in each case.

We now carry out the main argument.
Apply Proposition \ref{Prop:thickthindec} to obtain a decomposition $\MM(t) = \MM_{\thick}(t) \cup \MM_{\thin}(t)$ for all $t \in [\frac12 t_0, t_0]$.
Consider for a moment the case in which $U$ is a solid torus.
Since $U$ cannot contain any incompressible torus, none of the boundary tori of $\MM_{\thick}(t_0)$ can be contained in $U$.
Let $T' \subset U'$ be a $T^2$-fiber of $p$ with $\dist_{t_0} (\partial U, T') = r_0$.
Then $\diam_{t_0} T' < \mu^\# r_0 \leq \frac1{10} r_0$.
Assuming $w(t_0) < \frac1{10}$, every component of $\partial \MM_{\thick} (t_0)$ has diameter $< \frac1{10} r_0$ (see Proposition \ref{Prop:thickthindec}(b)).
This implies that if $T' \cap \MM_{\thick} (t_0) \neq \emptyset$, then $U$ is contained in the $2r_0$-tubular neighborhood of $\MM_{\thick}(t_0)$.
In this case we have a curvature bound on $U$ on a small time-interval with final time $t_0$ (see Proposition \ref{Prop:thickthindec}(c) and (d)) and we are done by setting $m = 0$.
On the other hand, if $T' \cap \MM_{\thick} (t_0) = \emptyset$ then $U$ is contained in the $2r_0$-tubular neighborhood of $\MM_{\thin}(t_0)$.
We will assume from now on that in case in which $U$ is a solid torus, $U$ is contained in a $2r_0$-tubular neighborhood of $\MM_{\thin}(t_0)$.

We will now apply Proposition \ref{Prop:MorganTianMain} with $\mu \leftarrow \mu^\circ$.
Observe for this that next to each component of $\partial \MM_{\thick} (t_0)$ there is a torus structure of width $\leq 10 (t_0) r_0$ and length $2 r_0$ inside $\MM_{\thick} (t_0)$.
In the case in which $U = \MM (t_0)$ let $M'$ be the union of $\MM_{\thin} (t_0)$ with these torus structures.
If $U$ is a solid torus, let $M' = U$.
So either by the assumption of the Lemma or by Proposition \ref{Prop:thickthindec} for sufficiently small $w(t_0)$, condition (i) of Proposition \ref{Prop:MorganTianMain} is satisfied.
Condition (ii) follows from Proposition \ref{Prop:thickthindec}(e) assuming $w(t_0) <  \min \{ \frac1{10} w_0 (\mu^\circ, \ov{r} (\cdot, 1), \linebreak[1] K_2 (\cdot, 1)), \frac1{10} \}$.
Condition (iii) is a consequence of Corollary \ref{Cor:Perelman68} if  $t_0 > T_{\ref{Cor:Perelman68}} (w_0 (\mu^\circ, \ov{r} (\cdot, 1), K_2 (\cdot, \linebreak[1] 1)), \linebreak[1] 1, 2)$.
Note that we can assume that $T_{\ref{Cor:Perelman68}}$ is monotone in the first parameter.
Now look at the conclusions of Proposition \ref{Prop:MorganTianMain}.
We first consider the case in which a component $M''$ of $M'$ is diffeomorphic to an infra-nilmanifold or a manifold which carries a metric of non-negative sectional curvature and we have $\diam_{t_0} M'' < \mu^\circ \rho_{r_0} (x, t_0)$ for all $x \in M''$.
By the assumptions of the Lemma, $M''$ is not a spherical space form or a quotient of $S^1 \times S^2$, so it is either an infra-nilmanifold or a quotient of $T^3$.
By Lemma \ref{Lem:unwrapfibration}(v), all points in $M''$ are $w_1(\mu^\circ)$-good at scale $r_0$ and time $t_0$.
Thus for large $t_0$ we obtain a curvature bound on all of $M''$ at time $t_0$ and slightly before by Proposition \ref{Prop:curvcontrolgood}.
For the rest of the proof, we can exclude these components from $M'$ and assume that we we have a decomposition $M' = V_1 \cup V_2 \cup V'_2$ satisfying the properties (a1)--(c6) of Proposition \ref{Prop:MorganTianMain}.

Next, we apply the discussion of subsection \ref{subsec:topimplications}---in particular Proposition \ref{Prop:GGpp}---to this decomposition.
Consider the set $\mathcal{G} \subset M'$ that we obtained there in Definition \ref{Def:GG}.
\begin{Claim1}
There are universal constants $w^*_1, {w'_1}^*, \alpha^*_1 > 0$ and $K^*_1, T^*_1 < \infty$ such that if $t_0 > \max \{ T^*_1, 2 T_0 \}$ and $w(t) < w^*_1$ for all $t \in [\frac12 t_0, t_0]$, then
\[ |{\Rm}| < K^*_1 t_0^{-1} \quad \text{on} \quad P(x, t_0, 2 r_0, - (\alpha^*_1 r_0)^2) \quad  \text{for all} \quad x \in \mathcal{G} \cup \MM_{\thick}(t_0) \cup \partial U
\]
where the parabolic neighbhorhood $P(x, t_0, r_0, - (\alpha^*_1 r_0)^2)$ is always non-singular.
Moreover, all points of $\mathcal{G} \cup \partial U$ are ${w'_1}^*$-good at scale $r_0$ and time $t_0$.
\end{Claim1}
\begin{proof}
If $x \in \mathcal{G}$, then $x$ is $w_1(\mu^\circ)$-good at scale $r_0$ and time $t_0$ by Lemma \ref{Lem:unwrapfibration}.
If $x \in \partial U$, then $x$ is $w^\#$-good at scale $r_0$ and time $t_0$.
So in these cases, the curvature bound follows from Proposition \ref{Prop:curvcontrolincompressiblecollapse} for sufficiently large $t_0$ and small $w(t_0)$.
If $x \in \MM_{\thick}(t_0)$, then the curvature bound is a direct consequence of Proposition \ref{Prop:thickthindec}(d).
\end{proof}

Now consider the set $\mathcal{G}' \supset \mathcal{G}$ as in Definition \ref{Def:GGp}.
In the next claim, we extend the curvature control onto $\mathcal{G}'$.
Recall that $\partial \mathcal{G}' \subset \partial \mathcal{G} \cup \partial U$.
\begin{Claim2}
There are constants $w^*_2, \alpha^*_2 > 0$ and $K^*_2, T^*_2 < \infty$ such that:
If $t_0 > \max \{ T^*_2, 2 T_0 \}$ and $w(t) < w^*_2$ for all $t \in [\frac12 t_0, t_0]$, then
\[ |{\Rm}| < K^*_2 t_0^{-1} \quad \text{on} \quad P(x, t_0, \alpha^*_2 r_0, - (\alpha^*_2 r_0)^2) \quad \text{for all} \quad x \in \mathcal{G}' \cup \MM_{\thick}(t_0) \cup \partial U \]
where the parabolic neighborhood $P(x, t_0, \alpha^*_2 r_0, - (\alpha^*_2 r_0)^2)$ is always non-singular.
\end{Claim2}
\begin{proof}
We only need to consider the case in which $x \in \mathcal{G} \setminus \mathcal{G}'$ and $\dist_{t_0} (x, \partial \mathcal{G} \cup \partial U) > r_0$.
Let $N$ be the component of $\mathcal{G}' \setminus \mathcal{G}$ which contains $x$.
Then $\partial N \subset \partial \mathcal{G} \cup \partial U$ and $B(x, t_0, \rho_{r_0} (x, \linebreak[1] t_0)) \subset N$.
So we can apply Lemma \ref{Lem:unwrapfibration}(ii) and (iv) to conclude that for any $\td{x} \in \td{N}$ in the universal cover of $N$ we have $\vol_{t_0} B(\td{x}, t_0, \rho_{r_0} (x, t_0)) \linebreak[1] > w_1(\mu^\circ) \rho^3_{r_0} (x, t_0)$.
This implies that $x$ is $\td{c} w_1(\mu^\circ)$-good at any scale $r \leq r_0$ relatively to $N$ (here $\td{c} > 0$ is the constant from subsection \ref{subsec:goodness}).

Since all points in $\partial N$ survive until time $(1-(\alpha^*_1)^2) t_0$ and all surgeries on $[\frac12 t_0, t_0]$ are trivial, we can extend $N$ to a sub-Ricci flow with surgery $N' \subset \MM$ on the time-interval $[(1-(\alpha^*_1)^2) t_0, t_0]$.
We now apply Proposition \ref{Prop:curvboundinbetween} for $r_0 \leftarrow \min \{ \alpha_1^*, (K_1^*)^{-1/2} \} r_0$, $U \leftarrow N'$ and $w \leftarrow \td{c} w$ to obtain the desired curvature bound for sufficiently large $t_0$.
\end{proof}

Next, we find a curvature estimate in controlled distance to $\mathcal{G}'$ which however deteriorates with larger distances.
\begin{Claim3}
For every $A < \infty$ there are constants $w^*_3 = w^*_3(A), \alpha^*_3 = \alpha^*_3(A) > 0$ and $K^*_3 = K^*_3(A), T^*_3 = T^*_3(A) < \infty$ such that if $t_0 > \max \{ T^*_3, 2 T_0 \}$ and $w(t) < w^*_3$ for all $t \in [\frac12 t_0, t_0]$, then
\[ |{\Rm}| < K^*_3 t_0^{-1} \quad \text{on} \quad P(x, t_0, A r_0, - (\alpha^*_3 r_0)^2) \quad \text{for all} \quad x \in \mathcal{G}' \cup \MM_{\thick}(t_0) \cup \partial U \]
where the parabolic neighborhoods $P(x, t_0, A r_0, - (\alpha^*_3 r_0)^2)$ are non-singular.
\end{Claim3}
\begin{proof}
The case $x \in \mathcal{G}' \cup \partial U$ can be reduced to the case $x \in \partial \mathcal{G}' \subset \partial \mathcal{G} \cup \partial U$.
So in this case the claim follows immediately from Proposition \ref{Prop:curvcontrolincompressiblecollapse} for $A \leftarrow A+1$ together with distance distortion, Claim 1.
The case $x \in \MM_{\thick}(t_0)$ follows directly from Proposition \ref{Prop:thickthindec}(d).
\end{proof}

Now consider the sets $\mathcal{G}'' \subset \mathcal{G}'$ and $\mathcal{S}'' \subset M'$ as introduced in Proposition \ref{Prop:GGpp}.
Recall that $\mathcal{S}''$ is a disjoint union of smoothly embedded solid tori.
The next claim is rather geometric.
It ensures that there are no components of $V_2$ outside of $\mathcal{G}''$ in controlled distance to $\mathcal{G}''$ if $w(t)$ is assumed to be sufficiently small.
\begin{Claim4}
For every $A < \infty$ there is a $w^*_4 = w^*_4(A) > 0$ and a $T^*_4 = T^*_4(A) < \infty$ such that if $t_0 > \max \{ T^*_4, 2 T_0 \}$ and $w(t) < w^*_4(A)$ for all $t \in [\frac12 t_0, t_0]$, then the following holds: \\
For every component $\CC''$ of $\SS''$ there is a component $\CC$ of $V_1$ with $\CC \subset \CC''$ and $\partial \CC'' \subset \partial \CC$.
Moreover, one of the following cases applies:
\begin{enumerate}[label=(\alph*)]
\item $\CC \approx S^1 \times D^2$ or 
\item $\CC \approx T^2 \times I$ and $\CC$ is adjacent to a component $\CC'$ of $V'_2$ which is diffeomorphic to $S^1 \times D^2$ on the other side or
\item $\CC \approx T^2 \times I$ and the boundary components of $\CC$ have time-$t_0$ distance of at least $A r_0$.
\end{enumerate}
So in particular the components of $V_2$ which are not contained in $\mathcal{G}''$ have time-$t_0$ distance of at least $A r_0$ from $\mathcal{G}''$.
\end{Claim4}
\begin{proof}
By Proposition \ref{Prop:GGpp}(e), we only have to consider the case in which $\CC \approx T^2 \times I$ and $\CC$ is adjacent to a component $\CC'$ of $V_2$ on the other side.
Observe that then the generic Seifert fiber of $\CC'$ is contractible in $\MM(t_0)$.
Assume that $\CC'$ has time-$t_0$ distance of less than $A r_0$.
Then we can find points $x_0 \in \partial \mathcal{G} \cup \partial U$ and $x_1 \in \CC'$ with $\dist_{t_0} (x_0, x_1) < A r_0$.
Without loss of generality, we can assume that $x_1 \in \CC' \cap V_{2, \text{reg}}$ (e.g. by assuming $x_1 \in \partial \CC$).
Let $\td{x}_0, \td{x}_1 \in \td{\MM}(t_0)$ be lifts of $x_0, x_1$ in the universal cover with $\dist_{t_0}(\td{x}_0, \td{x}_1) = \dist_{t_0}(x_0, x_1)$.
Using Claim 1, we can deduce a lower bound on $\rho_{r_0} (x_0, t_0)$ and hence find a universal constant $w_1^{**} > 0$ such that $\vol_{t_0} B(\td{x}_0, t_0, r_0) >  w^{**}_1 r_0^3$.
Using Claim 3 (applied with $A \leftarrow 2A+1$), we find a curvature bound on $B(\td{x}_1, t_0, (A+1) r_0)$ for large $t_0$.
So by volume comparison we have $\vol_{t_0} B(\td{x}_1, t_0, \rho_{r_0} (x_1, t_0)) > w^{**}_2 \rho_{r_0}^3(x_1, t_0)$ for some $w_2^{**} = w_2^{**}(A) > 0$.

We will now derive a contradiction to the local collapsedness around $x_1$ for small enough $w(t_0)$.
By Proposition \ref{Prop:MorganTianMain}(c3), there is a universal constant $0 < s = s_2 (\mu^\circ, \ov{r}( \cdot, 1), K_2(\cdot, 1)) < \frac1{10}$ and a subset $U_2$ with
\[ B(x_1, t_0, \tfrac12 s \rho_{r_0} (x_1, t_0)) \subset U_2 \subset B(x_1, t_0, s \rho_{r_0}(x_1, t_0)) \]
which is diffeomorphic to $B^2 \times S^1$ such that the $S^1$ directions are isotopic to the $S^1$-fibers in $\CC' \cap V_{2, \text{reg}}$ and hence contractible in $\MM(t_0)$.
So if $\td{U}_2 \subset \td{\MM}(t_0)$ is the lift of $U_2$ which contains $\td{x}_1$, then the universal covering projection is injective on $\td{U}_2$.
Hence
\begin{multline*}
 \vol_{t_0} B(x_1, t_0, \rho_{r_0}(x_1, t_0))  \geq \vol_{t_0} U_2 = \vol_{t_0} \td{U}_2 \geq \vol_{t_0} B(\td{x}_1, t_0, \tfrac12 s \rho_{r_0}(x_1, t_0)) \\
 \geq \tfrac18 \td{c} s^3 \vol_{t_0} B(\td{x}_1, t_0, \rho_{r_0} (x_1, t_0))
 \geq \tfrac18 \td{c} w_2^{**} s^3 \rho_{r_0}^3 (x_1, t_0).
\end{multline*}
Since $\dist_{t_0} (x_1, \MM_{\thin}(t_0)) < 2 r_0$, we obtain a contradiction if we choose $w_4^*(A) < \frac18 \td{c} w_2^{**}(A) s^3$.
This finishes the proof.
\end{proof}

Next we show that the diameter of each component of $\SS''$ cannot grow too fast on a time-interval of small, but uniform size.
\begin{Claim5}
There is a constant $\alpha^*_5 > 0$ and for every $A < \infty$ there are constants $B^*_5 = B^*_5(A), T^*_5 = T^*_5(A) < \infty$ and $w^*_5(A) > 0$ such that if $t_0 > \max \{ T^*_5, 2 T_0 \}$ and $w(t) < w^*_5$ for all $t \in [\frac12 t_0, t_0]$, then we have: \\
Let $\CC$ be a component of $\SS''$.
Then there is a unique sub-Ricci flow with surgery $N \subset \MM$ on the time-interval $[t_0 - (\alpha^*_5 r_0)^2, t_0]$ with $\CC = N(t_0)$ such that the following holds:
If $\diam_{t_0} N(t_0) > B^*_5 r_0$, then $\diam_t N(t) > A r_0$ for all $t \in [t_0 - (\alpha^*_5 r_0)^2, t_0]$.
\end{Claim5}
\begin{proof}
It is clear that by Claim 1 and the fact that all surgeries on $[\frac12 t_0, t_0]$ are trivial, we can extend $\CC$ to a sub-Ricci flow with surgery $N \subset \MM$ on the time-interval $[t_0 - (\alpha_1^* r_0)^2, t_0]$.

The rest of the claim is a consequence of Proposition \ref{Prop:slowdiamgrowth}.
Choose $x_0 \in \partial \CC \subset \mathcal{G} \cup \partial U$.
So $x_0$ is ${w'_1}^*$-good at scale $r_0$.
Let $\tau^* = \min \{ (\alpha^*_1)^2, (K_1^*)^{-1}, \tau_0({w'_1}^*) \}$ where $\tau_0$ is the constant from Proposition \ref{Prop:slowdiamgrowth}.
Obviously, $\partial N (t) \subset B(x_0, t, r_0)$ for all $t \in [(1- \tau^*) t_0, t_0]$.
We can now apply Proposition \ref{Prop:slowdiamgrowth}(d) with $U \leftarrow N$, $r_0 \leftarrow r_0$, $x_0 \leftarrow x_0$, $w \leftarrow {w'_1}^*$, $A \leftarrow A$ to conclude that if for any $\tau \in (0, \tau^*]$ we have $N \subset B(x_0, t_0 - \tau r_0^2, A r_0)$, then $\CC = N(t_0) \subset B(x_0, t_0, A'({w'_1}^*, A) r_0)$.
This implies the claim for sufficiently large $t_0$ and small $w(t)$ (depending on $A$).
\end{proof}

It is clear that we can assume the functions $w_3^*(A), \alpha_3^* (A), w_4^*(A), w_5^*(A)$ to be non-increasing and the functions $K_3^*(A), T_3^*(A), T_4^*(A), B_5^*(A), T_5^*(A)$ to be non-decreasing in $A$.
In the following, we define the sub-Ricci flows with surgery $S_i$ and the sets $W_i$ and show that they satisfy the assertions (a)--(h).
In order to do this, we will denote the components of $\mathcal{S}''$ by $S''_1, \ldots, S''_{m''}$ and choose a subcollection $S^*_1, \ldots, S^*_m$ of the $S''_1, \ldots, S''_{m''}$ in the next pargraph.
The final time-slices  $S_1(t_0), \ldots, S_m(t_0)$ will arise from the sets $S^*_1, \ldots, S^*_m$ be removing a collar neighborhood of diameter $\leq 1.5 r_0$.
This is described in the following paragraph.
Fix from now on the constant $L$ and assume that $L > 102$.

Assume first that $t_0 > \max \{ T_1^*, T_3^*(L+2), 2 T_0 \}$ and $w(t) < \min \{ w_1^*, w_3^*(L+2) \}$ for all $t \in [\frac12 t_0, t_0]$.
If $d_i = r_0^{-1} \diam_{t_0} S''_i < L+2$ for some $i$, then by Claim 3, all points in $S''_i$ survive until time $t_0 - (\alpha^*_3(d_i) r_0)^2$ and we have $|{\Rm}| < K_3^*(d_i) t_0^{-1}$ on $S''_i \times [t_0 - (\alpha^*_3(d_i) r_0)^2, t_0]$.
Given the fact that the sets $S_1, \ldots, S_m$ are chosen in the way described above, this establishes the second part of assertion (c) for $K' (d) = K_3^*(d+2)$ and $\tau'_1 (d) = (\alpha^*_3 (d+2))^2$.
Moreover, assuming $\tau_1 < \tau'_1(102)$, we can remove all $S''_i$ with $\diam_{t_0} \leq 102 r_0$ and define the sets $S^*_1, \ldots, S^*_m$ to be the sets $S''_i$ with $\diam_{t_0} S''_i > 102 r_0$.
So, by a reapplication of Claim 3 assertion (b) is verified and by Claim 1 the second part of assertion (h) is true for some small but universal $\tau_1$.
Also, the first part of assertion (c) is clear.
Note that by assertion (b) and the fact that the surgeries on $[\frac12 t_0, t_0]$ are trivial, we can extend every set $S^*_i$ to a sub-Ricci flow with surgery on the time-interval $[(1-\tau_1) t_0, t_0]$ which we will in the following also denote by $S^*_i \subset \MM$.

Now assume that also $t_0 > T_4^*(L+2)$ and $w(t) < w_4^*(L+2)$ for all $t \in [\frac12 t_0, t_0]$.
For each $S^*_i$ there is a component $\CC_i$ of $V_1$, which is contained in $S^*_i (t_0)$ and shares a boundary with it.
Consider the cases (a)--(c) from Claim 4.
In cases (b), (c) we set $P_i = \CC_i$.
In case (a), we can apply Proposition \ref{Prop:MorganTianMain}(c1)($\alpha$) to find a torus structure $P_i \subset \CC_i$ such that $\partial \CC_i \subset \partial P_i$ and such that $\diam_{t_0} \CC_i \setminus P_i \leq \frac1{10} r_0$.
Observe that in all cases, the torus structure $P_i$ has width $\leq \mu^\circ r_0 \leq \frac1{10} r_0$.
In case (c) it has length $\geq (L+3) r_0$ by Claim 4 and in cases (a), (b) it has length $> \diam_{t_0} S^*_i(t_0) - \diam_{t_0} (S^*_i(t_0) \setminus P_i) - \frac1{10} r_0 > \diam_{t_0} S^*_i (t_0) - \frac2{10} r_0$ at time $t_0$.
Chop off $P_i$ on both sides such that the new boundary tori have distance of exactly $r_0$ from the corresponding boundary tori of $P_i$ and call the result $P'_i$.
Then define $S_i (t_0)$ to be the union of $P'_i$ with the component of $S^*_i(t_0) \setminus P'_i$ whose closure is diffeomorphic to a solid torus.
By assertion (b) we can extend $S_i(t_0)$ to a sub-Ricci flow with surgery $S_i \subset \MM$ on the time-interval $[(1-\tau_1) t_0, t_0]$.
Note that in all cases the torus structure $P'_i$ has length $\geq \min \{ L r_0, \diam_{t_0} S_i (t_0) - 2 r_0 \}$ at time $t_0$.
We can hence chop off $P'_i$ on the side which is contained in the interior of $S_i(t_0)$ and produce a torus structure $P''_i$ of width $\leq \frac1{10} r_0$ and length $= \min \{ L r_0, \diam_{t_0} S_i (t_0) - 2 r_0 \}$.
Let $W_i$ be the closure of $S_i (t_0) \setminus P''_i$.
Then assertion (e) holds.
Moreover, the first part of assertion (h) follows from Proposition \ref{Prop:MorganTianMain}(c1).
Assertion (a) is clear.
Observe also that $\diam_{t_0} S^*_i(t_0) - \frac{11}{10} r_0 < \diam_{t_0} S_i(t_0) \leq \diam_{t_0} S^*_i(t_0)$.

We discuss assertion (f).
Let $x \in P''_i$.
By Lemma \ref{Lem:unwrapfibration}(ii) and (iii), we conclude that $x$ is $w_1(\mu^\circ)$-good at scale $r_0$ and time $t_0$ relatively to $P_i$.
Since $B(x, t_0, \rho_{r_0} (x, t_0)) \subset P_i$ this implies that $x$ is also locally $w_1(\mu^\circ)$-good at scale $r_0$ and time $t_0$.

Next we establish assertion (d).
Observe that choosing $\tau_1$ small enough, we have at least $\diam_t S_i(t) > 50 r_0$ for all $t \in [(1-\tau_1) t_0, t_0]$ by assertion (b) and the fact that $\diam_{t_0} S_i (t_0) > 100 r_0$.
Assume now that $t_0 > T_5^*(L)$ and $w(t) < w_5^*(L)$ for all $t \in [\frac12 t_0, t_0]$.
By Claim 5 for $\CC = S_i(t_0)$ we conclude that for all $A \leq L$ we have: if $\diam_{t_0} S^*_i (t_0) > B_5^*(A) r_0$, then $\diam_t S^*_i (t) > A r_0$ for all $t \in [t_0 - (\alpha^*_5 r_0)^2, t_0]$.
So assertion (d) holds for the function
\[ \Delta_1 (d) = \sup \{ A > 0 \;\; : \;\; B_5^*(A+2) < d \} \cup \{ 50 \}. \]
It is clear that $\Delta_1$ is monotonically non-decreasing and $\lim_{d \to \infty} \Delta_1(d) = \infty$.

Finally, we prove assertion (g).
Assume that $t_0 > T_3^* (2L+10)$ and that $w(t) < w_3^*(2L + 10)$ for all $t \in [\frac12 t_0, t_0]$.
Let $x \in P''_i$ and choose an arbitrary point $x_0 \in \partial S^*_i(t_0)$.
Let $\td{x}$, $\td{x}_0$ be lifts of $x$, $x_0$ in the universal cover $\td{\MM}(t_0)$ with $\dist_{t_0} (\td{x}, \td{x}_0) = \dist_{t_0} (x, x_0)$.
As in the proof of Claim 4 we have
\[ \vol_{t_0} B(\td{x}_0, t_0, r_0) > w_1^{**} r_0^3 \]
for some universal $w_1^{**} > 0$.
By Claim 3, we have curvature control $| {\Rm_{t_0}} | < K_3^*(2 L+10) t_0^{-1}$ on $B(x, t_0, (L+5) r_0) \subset B(x_0, t_0, (2 L+10) r_0)$.
In particular, there is a $\rho^* = \rho^*(L) > 0$ such that $\rho_{r_0}(x, t_0) > \rho^* r_0$.
Without loss of generality, we can assume that $\nu < \min \{ \rho^*, 1 \}$.
Hence, by volume comparison there is some $c^* = c^*(L) > 0$ such that
\begin{multline*}
 \vol_{t_0} B(\td{x}, t_0, \nu r_0) \geq \nu^3 c^* \vol_{t_0} B(\td{x}, t_0, (L+5)r_0) \\ > \nu^3 c^* \vol_{t_0} B(\td{x}_0, t_0, r_0) > \nu^3 c^* w_1^{**} r_0^3.
\end{multline*}
On the other hand, 
\[ \vol_{t_0} B(x, t_0, \nu r_0) < \vol_{t_0} B(x, t_0, \rho_{r_0}(x,t_0)) < w(t_0) \rho_{r_0}^3(x_0,t_0) < w(t_0) r_0^3. \]
Assume first that there is no loop based at $x$ which is non-contractible in $\MM(t_0)$ and has length $< \nu r_0$.
Then
\[ w(t_0) r_0^3 > \vol_{t_0} B(x, t_0, \nu r_0) = \vol_{t_0} B(\td{x}, t_0, \nu r_0) > \nu^3 c^* w_1^{**} r_0^3. \]
So if $w(t_0) < \nu^3 c^* w_1^{**}$, we obtain a contradiction.
We conclude that if $w(t_0)$ is sufficiently small depending on $L$ and $\nu$, there is a non-contractible loop $\sigma \subset \MM(t_0)$ based at $x$ which has length $\ell_{t_0} (\sigma) < \nu r_0$.
This implies $\sigma \subset P_i \subset S^*_i (t_0)$ and hence $\sigma$ is even incompressible in $\MM(t_0)$.
\end{proof}

In the second step, we extend the uniform curvature control from Lemma \ref{Lem:firstcurvboundstep1}(b) further into the regions $S_i(t_0) \setminus W_i(t_0)$.
This will follow from Proposition \ref{Prop:curvboundnotnullinarea} and Lemma \ref{Lem:firstcurvboundstep1}(f).

\begin{Proposition}[second step] \label{Prop:firstcurvboundstep2}
There are a positive continuous function $\delta : [0, \infty) \to (0, \infty)$, constants $K_2 < \infty$, $\tau_2 > 0$ and functions $\Lambda_2, K'_2, \tau'_2 : (0, \infty) \to (0, \infty)$ with the property that $\tau'_2$ is non-increasing, $K'_2$ and $\Lambda_2$ are non-decreasing and $\Lambda_2(d) \to \infty$ as $d \to \infty$ such that: \\
For every $L < \infty$ and $\nu > 0$ there are constants $T_2 = T_2(L) < \infty$, $w_2 = w_2 (L, \nu) > 0$ such that: \\
Let $\MM$ be a Ricci flow with surgery on the time-interval $[0, \infty)$ with normalized initial conditions which is performed by $\delta(t)$-precise cutoff.
Consider the constant $T_0 < \infty$ and the function $w : [T_0, \infty) \to (0, \infty)$ obtained in Proposition \ref{Prop:thickthindec} and assume that
\begin{enumerate}[label=(\roman*)]
\item $r_0^2 = t_0 \geq \max \{ 4 T_0, T_2 \}$,
\item $w(t) < w_2$ for all $t \in [\frac14 t_0, t_0]$,
\item all components of $\MM(t_0)$ are irreducible and not diffeomorphic to spherical space forms and all surgeries on the time-interval $[\frac14 t_0, t_0]$ are trivial.
\end{enumerate}
Then there are sub-Ricci flows with surgery $S_1, \ldots, S_m \subset \MM$ on the time-interval $[(1-\tau_2) t_0, t_0]$ such that $S_1(t_0), \ldots, \linebreak[1] S_m(t_0)$ is a collection of pairwise disjoint, incompressible solid tori in $\MM(t_0)$ and there are sub-Ricci flows with surgery $W_i \subset S_i$ ($i=1,\ldots, m$) on the time-interval $[(1-\tau_2) t_0, t_0]$ such that for all $i = 1, \ldots, m$:
\begin{enumerate}[label=(\alph*)]
\item The pair $(S_i (t_0), W_i (t_0))$ is diffeomorphic to $(S^1 \times D^2(1), S^1 \times D^2(\frac12))$.
\item The set $\MM (t_0) \setminus (W_1 (t_0) \cup \ldots \cup W_m (t_0))$ is non-singular on the time-interval $[(1 - \tau_2) t_0, t_0]$ and
\[ \qquad \qquad | {\Rm} | < K_2 t_0^{-1} \qquad \text{on} \qquad \big(\MM (t_0) \setminus (W_1(t_0) \cup \ldots \cup W_m (t_0)) \big) \times [(1-\tau_2) t_0, t_0]. \]
\item If $\diam_{t_0} S_i (t_0) \leq L r_0$, then $S_i$ is non-singular on the time-interval $[(1- \tau'_2( r_0^{-1} \diam_{t_0} S_i ) ) t_0, t_0 ]$ and we have the curvature bound
\[  \qquad\quad |{\Rm}| < K'_2 ( r_0^{-1} \diam_{t_0} S_i ) t_0^{-1}  \quad \text{on} \quad S_i (t_0) \times  [(1- \tau'_2( r_0^{-1} \diam_{t_0} S_i (t_0) ) ) t_0, t_0 ].
\]
\item The set $S_i(t_0) \setminus \Int W_i(t_0)$ is a torus structure of width $\leq r_0$ and length
\[ \dist_{t_0} (\partial S_i (t_0), \partial W_i (t_0)) = \min \big\{ \Lambda_2 ( r_0^{-1} \diam_{t_0} S_i (t_0)), L \big\} r_0. \]
\item For every point $x \in S_i (t_0) \setminus W_i (t_0)$, there is a loop $\sigma \subset \MM(t_0)$ based at $x$ which is incompressible in $\MM(t_0)$ and has length $\ell_{t_0}(\sigma) < \nu r_0$.
\end{enumerate}
\end{Proposition}
The most important statement of this Lemma is the fact that the uniform curvature bound in (b) also holds on $S_i (t_0) \setminus W_i (t_0)$ and on a time-interval whose size does not depend on $r_0^{-1} \diam_{t_0} S_i(t_0)$.
Since this enables us to estimate the metric distortion of the regions $S_i(t_0) \setminus W_i(t_0)$ on this time-interval, we don't need to list the lower diameter estimate from Lemma \ref{Lem:firstcurvboundstep1}(d).
We have also omitted the statement from Lemma \ref{Lem:firstcurvboundstep1}(f) since we won't make use of it anymore.

Observe that we can only establish the curvature bound in assertion (c) at times \emph{close to time $t_0$} and that the length of the torus structure in assertion (d) cannot be bounded from below by a constant depending on the diameter of $S_i$ at time $(1-\tau_2) t_0$.
The reason for this comes from the fact that the geometry on $S_i$ could be close to that of a cigar soliton times $S^1$.
In fact after rescaling by a proper constant, regions of the cigar soliton of large diameter can shrink very rapidly under the Ricci flow.
It is the content of Proposition \ref{Prop:slowdiamgrowth}, which we have applied in the proof of Lemma \ref{Lem:firstcurvboundstep1}, that however the opposite behaviour cannot occur, i.e. regions of bounded diameter can never grow too fast in a short time.

\begin{proof}
We will fix several constants and functions which we will use in the course of the proof.
First we assume that $\delta(t)$ is bounded by the corresponding functions from Lemma \ref{Lem:firstcurvboundstep1} and Proposition \ref{Prop:curvboundnotnullinarea}.
Consider moreover the functions $\Delta_1$, $K'_1$, $\tau'_1$ from Lemma \ref{Lem:firstcurvboundstep1} and choose $D^* < \infty$ such that
\[ D^* > 100 \qquad \text{and} \qquad \Delta_1(D^*) > 100. \]
Define the functions $L_1^*, \ldots, L_5^*(d) : [D^*, \infty) \to (1, \infty)$ by
\begin{alignat*}{1}
 L_1^*(d) &= \tfrac14 \Delta_1(d) - 10  \\
 L_2^*(d) &= \tfrac12 \min \{ d-3, L_1^*(d) \} \\
 L_3^* (d) &= \min \{ L_1^*(d) - 1, L_2^* (d) \} \\
 L_4^* (d) &= L_3^*(d) - 1 \\
 L_5^* (d) &= \tfrac12 L_4^*(d)
\end{alignat*}
Observe, that $L_1^*, \ldots , L_5^*$ are continuous, monotonically non-decreasing and $L^*_i (d) \to \infty$ as $d \to \infty$.
Using these functions we define
\[ \Lambda_2 (d) = \begin{cases} \min \{ L^*_5(d) - 1, d-2 \} & \text{if $d \geq D^*$} \\ 1 & \text{if $d < D^*$} \end{cases} \]
Then $\Lambda_2(d)$ is also non-increasing and $\Lambda_2(d) \to \infty$ as $d \to \infty$.

Given the constant $L$, we pick
\[ L^\circ = L^\circ (L) > \max \{ D^*, 10L + 100 \}. \]
Finally, using the constants $w_1$ and $T_1$, from Lemma \ref{Lem:firstcurvboundstep1} we assume
\[ w_2 (L, \nu) < w_1 (L^\circ, \nu) \qquad \text{and} \qquad T_2 (L, \nu) > 2 T_1 (L^\circ, \nu) . \]

We first apply Lemma \ref{Lem:firstcurvboundstep1} at time $t_0 \leftarrow t_0$ with $U \leftarrow \MM(t_0)$, $L \leftarrow L^\circ$ and $\nu \leftarrow \nu$.
We obtain sub-Ricci flows with surgery on the time-interval $[(1-\tau^*_0) t_0, t_0]$ which we will denote by $S'_1, \ldots, S'_{m'} \subset \MM$ and subsets which we will denote by $W''_i \subset S'_i (t_0)$ for $i = 1, \ldots, m'$.
By Lemma \ref{Lem:firstcurvboundstep1}(c), if $\diam_{t_0} S'_i (t_0) \leq D^* r_0$, then $S'_i$ is non-singular on the time-interval $[(1-\tau'_1(D^*)) t_0, t_0]$ and we have a curvature bound there.
Let now $S_1, \ldots, S_m$ be the subcollection of the $S'_1, \ldots, S'_{m'}$ for which $d_i = r_0^{-1}\diam_{t_0} S'_i (t_0) > D^*$ and pick the sets $W'_i \subset S_i(t_0)$ accordingly.
Consider the torus structures $P'_i = \Int S_i(t_0) \setminus W'_i$ of width $\leq r_0$ and length 
\[ \min \{ d_i - 2, L^\circ \} r_0 \geq \min \{ \Lambda_2(d_i), L + 1\} r_0. \]
Chop off each $P'_i$ on the side which is not adjacent to $\partial S_i (t_0)$ and produce torus structures $P_i$ of width $\leq r_0$ and length exactly $\min \{ \Lambda_2(d_i), L \} r_0$.
Then we set $W_i = \Int S_i(t_0) \setminus \Int P_i$.
We will later be able to extend $W_i$ to a sub-Ricci flow with surgery on a small, but uniform time-interval. 

\begin{Claim0}
There are universal constants $\tau^*_0, w^*_0 > 0$ and $K^*_0 < \infty$ such that: \\
For all $x \in \MM(t_0)$ with $\dist_{t_0} (x, \MM(t_0) \setminus (S_1(t_0) \cup \ldots \cup S_m(t_0))) \leq 100 r_0$ the point $(x, t_0)$ survives until time $(1-\tau^*_0) t_0$ and
\[
|{\Rm}|(x,t) < K^*_0 t_0^{-1} \qquad \text{for all} \qquad t \in [(1-\tau^*_0) t_0, t_0].
\]
Moreover, assertions (a), (c), (d) and (e) of this Proposition hold.
\end{Claim0}

\begin{proof}
The first statement is a direct consequence of Lemma \ref{Lem:firstcurvboundstep1}(b) and (c).
Here we assume that $\tau^*_0 < \min \{ \tau_1, \tau'_1(D^*) \}$ and $K^*_0 > \max \{ K_1, K'_1(D^*) \}$.

Assertion (a) is clear and assertion (c) is a consequence of Lemma \ref{Lem:firstcurvboundstep1}(c).
Assertion (d) follows by the choice of $\Lambda_2$ and assertion (e) by Lemma \ref{Lem:firstcurvboundstep1}(g) and the fact that $W'_i(t_0) \subset W_i$.
\end{proof}

So it remains to extend the curvature bound from Claim 0 to the subsets $S_i (t_0) \setminus W_i$ on a uniform time-interval.
The proof of this fact will involve the application of Lemma \ref{Lem:firstcurvboundstep1} at times $t \in [(1-\tau_0^*) t_0, t_0]$ for $U \leftarrow S_i(t)$, $L \leftarrow L^\circ$ and $\nu \leftarrow \nu$.
By assertion (h) from the previous application of Lemma \ref{Lem:firstcurvboundstep1}, the extra conditions of Lemma \ref{Lem:firstcurvboundstep1} in the solid torus case are satisfied.
The remaining conditions hold by the choice of $w_2$ and $T_2$.

The desired curvature bound is established in the following Claims 1--5.
In Claims 1--3 we will first derive a local goodness bound for points which are in controlled distance from $\partial S_i(t)$ for any $t$ of a uniform time-interval.
An important tool will hereby be the notion of ``torus collars of length up to'' a certain constant as introduced in Definition \ref{Def:toruscollars}.
In Claim 4 we will derive a curvature bound using this local goodness bound together with Proposition \ref{Prop:curvboundnotnullinarea}.
Claim 5 will translate this result into the final form.

In the following, fix some $i = 1, \ldots, m$ and recall that $d_i = r_0^{-1} \diam_{t_0} S_i(t_0) > D^*$.

\begin{Claim1}
There are universal constants $K_1^* < \infty$ and $0 < \tau_1^* < \tau^*_0$ such that for all $t \in [(1-\tau_1^*) t_0, t_0]$ the following holds:
Consider numbers
\[ 0 < \td{L} \leq \min \{ L_1^*(d_i), 4 (L+2) \}, \qquad 1 \leq a \leq 2 \]
and assume that $S_i(t)$ does not have torus collars of width $\leq ar_0$ and length up to $\td{L} r_0$, but it has torus collars of width $\leq a r_0$ and length up to $(\td{L}-1) r_0$ if $\td{L} > 1$. \\
Then $|{\Rm}| (x, t) < K_1^* t_0^{-1}$ for all $x \in S_i(t)$ with $\dist_t (x, \partial S_i(t)) < (\td{L} + 10) r_0$.
\end{Claim1}
\begin{proof}
Observe first that in the case $\td{L} \leq 1$ we are done by Claim 0 and a sufficiently small choice of $\tau^*_1$.
So assume in the following that $\td{L} > 1$.

Assume that $\tau_1^* < \tau_0^*$ and fix some $t \in [(1-\tau_1^*) t_0, t_0]$.
So we can apply Lemma \ref{Lem:firstcurvboundstep1} at time $t$ with $U \leftarrow S'_i$ and $L \leftarrow L^{\circ}$ and obtain the sub-Ricci flows with surgery $S_1^*, \ldots, S^*_{m^*} \subset \MM$ and the subsets $W^*_i \subset S^*_i( t)$.
Observe that the parameter $r = \sqrt{t}$ changes only slightly, i.e. we may assume that for the right choice of $\tau_1^*$ we have $0.9 r_0 < r \leq r_0$.
Moreover, we assume that $\tau^*_1$ is chosen small enough that $\diam_t \partial S_i(t) < 2 \diam_{t_0} \partial S_i(t) \leq2 r_0$.

If $\dist_t (S^*_{i^*}(t), \partial S_i (t)) \geq (\td{L} - 30) r_0$ for all $i^* = 1, \ldots, m^*$, then we are done by Lemma \ref{Lem:firstcurvboundstep1}(b) applied at time $t$.
So all we need to do is to assume that there is an $i^* = 1, \ldots, m^*$ with 
\begin{equation} \label{eq:torcollarcontradiction}
 \dist_t (S_{i^*}^*(t), \partial S_i(t)) < (\td{L} - 30) r_0
\end{equation}
and derive a contradiction.

Observe that by Lemma \ref{Lem:firstcurvboundstep1}(c) and (e) applied at time $t$ we have $\dist_t (\partial S^*_{i^*} (t), \linebreak[1] \partial W^*_{i^*}) \geq \min \{ 97 r, L^\circ r \} \geq 50 r_0$.
So we can choose a point $y \in S_{i^*}^*(t) \setminus W^*_{i^*}$ such that $\dist_t( y, \partial S^*_{i^*}(t)) = 3 r_0$.
Then we have at least $\dist_t (y, \partial S_i(t)) < (\td{L} - 20)r_0$ and hence by our assumption, there is a set $P \subset S_i(t)$ which is bounded by $\partial S_i(t)$ and a torus $T \subset S_i(t)$ with $y \in T$ and $\diam_t T \leq a r_0 \leq 2 r_0$.
By the choice of $y$ we have $T \subset S_{i^*}^*(t)$.
This implies
\begin{equation} \label{eq:SpPSs}
 S_i(t) = P \cup S_{i^*}^*(t)
\end{equation}
and we conclude using assertion (d) of Lemma \ref{Lem:firstcurvboundstep1} applied at time $t_0$ that
\[ \diam_t S^*_{i^*} (t) + \diam_t P \geq \diam_t S_i(t) > \min \{ \Delta_1(d_i), L^{\circ} \} r_0. \]
Observe now that by Lemma \ref{Lem:collardiameter}, we know that
\begin{equation} \label{eq:diamboundonP}
 \diam_t P < (\td{L} - 20 + 4a) r_0 < (\td{L} - 10) < \min \{ L_1^*(d_i), \linebreak[1] 4 (L+2) \} r_0.
\end{equation}
So using the fact that $r \leq r_0$, we obtain
\begin{multline*}
 \diam_t S^*_{i^*}(t) > \big( \min \{ \Delta_1(d_i), L^{\circ} \} - \min \{ L_1^* (d_i), 4(L+2) \}  \big) r_0 \\
 \geq \min \big\{ \Delta_1(d_i) - L_1^*(d_i) , L^\circ - 4 (L + 2) \big\} r .
\end{multline*}
Observe that the right hand side is larger than $10 r$.
We conclude further using Lemma \ref{Lem:firstcurvboundstep1}(e) applied at time $t$ that
\begin{multline}
 \dist_t (W^*_{i^*}, \partial S_i(t)) \geq \dist_t (\partial S^*_{i^*}(t), \partial W^*_{i^*}) = \min \{ \diam_t S^*_{i^*}(t) - 2 r, L^\circ r \} \\
\geq 0.9 \min \big\{ \Delta_1(d_i) - L_1^*(d_i) - 2, L^\circ - 4 (L + 3), L^\circ \big\} r_0 \\
> \min \big\{ L_1^* (d_i) + 1, 4(L+2) \big\} r_0 \geq \td{L} r_0. \label{eq:distSpWs}
\end{multline}

So far we have only used the first assumption that $S_i (t)$ \emph{does} have torus collars of width $\leq a r_0$ and length up to $(\td{L} - 1)r_0$.
We will now show that in contradiction to the second assumption of the claim, $S_i(t)$ also has torus collars of width $\leq a r_0$ an length up to $\td{L} r_0$.
So assume that $x \in S_i(t)$ with $(\td{L}-1) r_0 < \dist_t(x, \partial S_i(t)) \leq \td{L} r_0$.
By (\ref{eq:SpPSs}), the diameter bound (\ref{eq:diamboundonP}) on $P$ and (\ref{eq:distSpWs}), we conclude $x \in S_{i^*}^*(t) \setminus W_{i^*}^*$.
So, we can find a set $P^* \subset S_{i^*}^*(t) \setminus W_{i^*}^*$ which is diffeomorphic to $T^2 \times I$ and bounded by $\partial S_{i^*}^*(t)$ and a $2$-torus $T^* \subset S^*_{i^*} (t) \setminus W^*_{i^*}$ with $x \in T^*$ and $\diam_t T^* \leq r \leq r_0 \leq a r_0$.
Again by (\ref{eq:diamboundonP}) we find $T^* \cap P = \emptyset$.
It follows from Lemma \ref{Lem:2T2timesI}, that $P \cup P^*$ is diffeomorphic to $T^2 \times I$.
This finishes the contradiction argument and shows that (\ref{eq:torcollarcontradiction}) does not hold for any $i^* = 1, \ldots, m^*$.
\end{proof}

\begin{Claim2}
There are universal constants $0 < \tau_2^* < \tau^*_0$ and $T_2^* < \infty$ such that if $t_0 > T_2^*$ then at all times $t \in [(1-\tau_2^*) t_0, t_0]$ the set $S_i(t)$ has torus collars of width $\leq 2r_0$ and length up to $\min \{ L_2^*(d_i), 2(L+2) \} r_0$.
\end{Claim2}
\begin{proof}
Choose $\tau_2^* < \tau_1^*$ such that $\exp(2K_1^* \tau_2^*) < 2$.
By Lemma \ref{Lem:firstcurvboundstep1}(e) we already know that \emph{at time $t_0$}, the set $S_i(t_0)$ has torus collars of width $\leq r_0$ and length up to $L_i r_0$ where
\[ L_i = 2\min \{ L_2^* (d_i) , 2(L+2) \} \leq \min \{ d_i-2, L^\circ \}. \]
Let $t^* \in [(1-\tau_2^*) t_0, t_0]$ be minimal with the property that for all $t \in (t^*, t_0]$ the set $S_i(t)$ has torus collars of width $\leq \exp (2K_1^* t_0^{-1} (t_0 - t)) r_0$ and length up to $\exp (-2K_1^* t_0^{-1} (t_0 - t)) L_i r_0$ at time $t$.
We are done if $t^* = (1-\tau_2^*) t_0$.
So consider the case $t^* > (1-\tau_2^*) t_0$.

Let $\varepsilon > 0$ be a small constant which we will determine later.
It will not be a universal constant.
By the choice of $t^*$, we find times $t_1 \leq t^* \leq t_2$ with $t_2 - t_1 < \varepsilon$ such that at time $t_2$ the set $S_i(t_2)$ has torus collars of width $\leq \exp (2K_1^* t_0^{-1} (t_0 - t_2))r_0$ and length up to $\exp ( -2 K_1^* t_0^{-1} (t_0 - t_2) ) L_i r_0$, but at time time $t_1$ it does not have torus collars of width $\leq \exp (2 K_1^* t_0^{-1} (t_0 - t_1))r_0$ and length up to $\exp ( - 2 K_1^* t_0^{-1} (t_0 - t_1) ) L_i r_0$.

Choose $\td{L} \leq \exp ( - 2 K_1^* t_0^{-1} (t_0 - t_1) ) L_i$ such that at time $t_1$ the set $S_i(t_1)$ does not have torus collars of width $\leq \exp (2 K_1^* t_0^{-1} (t_0 - t_1))r_0$ and length up to $\td{L}r_0$, but it does have torus collars of width $\leq \exp (2 K_1^* t_0^{-1} (t_0 - t_1))r_0$ and length only up to $(\td{L} - 1) r_0$ if $\td{L} > 1$.
Observe that $\td{L} \leq \min \{ L^*_1 (d_i), 4(L+2) \}$ and $\exp (2 K_1^* t_0^{-1} (t_0 - t_1))r_0 < 2r_0$.
So we can apply Claim 1 to conclude that
\[ |{\Rm}| (x,t_1)  < K_1^* t_0^{-1} \qquad \text{if} \qquad x \in S_i(t_1) \quad \text{and} \quad  \dist_{t_1} (x, \partial S_i(t_1)) < (\td{L} + 10) r_0. \]
Let $Q < \infty$ be a bound on the curvature around all surgery points in $\MM$ on the time-interval $[t_1, t_2]$.
Then all strong $\delta(t)$-necks around surgery points as described in Definition \ref{Def:precisecutoff}(4) are defined on a time-interval of length $> \frac1{100} Q^{-1/2}$ and the curvature $|{\Rm}|$ there is bounded from below by $> c' \delta^{-2}(t)$ for some $t \in [t_1, t_2]$ and a universal $c' > 0$.
So if we choose $\varepsilon < \frac1{100} Q^{-1/2}$ and assume $t_0$ to be large enough, then we can exclude surgery points of the form $(x,t)$ with $\dist_{t_1} (x, \partial S_i(t_1)) < (\td{L} + 10) r_0$ and $t \in [t_1, t_2]$.
Moreover, again by choosing $\varepsilon$ sufficiently small, we can assume that curvatures at points which survive until time $t_2$ cannot grow by more than a factor of $2$ such that we have
\[ | {\Rm} | (x,t) < 2 K_1^* t_0^{-1} \quad \text{if} \quad (x,t) \in S_i(t) \times [t_1, t_2] \; \text{and} \; \dist_{t_1} (x, \partial S_i(t)) < (\td{L} + 10) r_0. \]
(We remark, that we could have also excluded surgery points using property (2) of the canonical neighborhood assumptions in Definition \ref{Def:CNA}.)

Now let $x \in S_i(t_1)$ be a point with $\dist_{t_1} (x, \partial S_i(t_1)) \leq \td{L} r_0$.
Then by the curvature bound we conclude
\[ \dist_{t_2} (x, \partial S_i(t_2)) \leq \exp(2 K_1^* t_0^{-1} (t_2 - t_1)) \td{L} r_0 \leq \exp(-2 K_1^* t_0^{-1} (t_0 - t_2)) L_i r_0 . \]
So there is a set $P \subset S_i(t_2)$ which is bounded by $\partial S_i(t_2)$ and an embedded $2$-torus $T \subset S_i(t_2)$ with $x \in T$ and $\diam_{t_2} T \leq \exp ( 2K_1^* t_0^{-1} (t_0 - t_2)) r_0$.
By Lemma \ref{Lem:collardiameter} we have under the assumption that $\varepsilon$ so small that $\exp (2 K^*_1 \varepsilon) L_i < L_i + 1$
\[ \diam_{t_2} P \leq \exp (2 K^*_1 t_0^{-1} (t_2 - t_1)) \td{L} r_0 + 8 r_0 < (\td{L} + 9) r_0. \]
Again by assuming $\varepsilon$ small, we conclude that the distance distortion on $P$ for times $[t_1, t_2]$ is bounded by $r_0$ and hence $\MM$ is non-singular on $P \times [t_1, t_2]$.
So at time $t_1$ the set $P$ is bounded by $\partial S_i(t_1)$ and a torus of diameter $\leq \exp ( 2 K_1^* t_0^{-1} (t_0 - t_1)) r_0$.

We have just showed that $S_i(t_1)$ does indeed have torus collars of width $\leq \exp ( 2 K_1^* t_0^{-1} (t_0 - t_1)) r_0$ and length up to $\td{L} r_0$ contradicting our assumption.
\end{proof}

\begin{Claim3}
There are constants $0 < \tau_3^* < \tau_0^*$, $w_3^* > 0$ and $K_3^*, T_3^* < \infty$ such that:
Assume that $t_0 > T_3^*$.
Then for every $t \in [(1-\tau_3^*)t_0, t_0]$ and every point $x \in S_i(t)$ with $\dist_t (\partial S_i(t), x) < \min \{ L_3^*(d_i), 2(L+2) \} r_0$ either $|{\Rm}|(x,t) < K^*_3 t_0^{-1}$ or $x$ is locally $w_3^*$-good at scale $r_0$ and time $t$.
\end{Claim3}
\begin{proof}
Assume that $\tau_3^* < \min \{ \tau^*_0, \tau^*_1, \tau_2^* \}$ and fix some time $t \in [(1-\tau_3^*) t_0, t_0]$.
We will argue as in the first part of the proof of Claim 1 with $\td{L} = \min \{ L^*_3 (d_i), 2 (L+2) \} + 1 \leq \min \{ L^*_1(d_i), 4 (L+2) \}$ and $a = 2$.
Observe hereby that by Claim 2 the set $S_i(t)$ has torus collars of width $\leq 2 r_0$ and length up to $\min \{ L^*_3 (d_i), 2 (L+2) \} \leq \min \{ L^*_2(d_i), 2(L+2) \}$.

So we apply again Lemma \ref{Lem:firstcurvboundstep1} at time $t$ with $U \leftarrow S_i(t)$ and $L \leftarrow L^\circ$ and obtain pairs of subsets $(S^*_1(t), W^*_1), \ldots, (S^*_{m^*}(t), W^*_{m^*})$ of $S_i(t)$.
If $\dist_t ( S^*_{i^*} (t), \partial S_i (t) ) \geq (\td{L} - 30) r_0$ for all $i^* = 1, \ldots, m^*$, then we obtain a curvature bound as before using Lemma \ref{Lem:firstcurvboundstep1}(b).
If not, i.e. if (\ref{eq:torcollarcontradiction}) is satisfied for some $i^* = 1, \ldots, m^*$, then we obtain from (\ref{eq:distSpWs}) that
\[ \dist_t ( W^*_{i^*}, \partial S_i(t) ) > \min \{ L^*_2 (d_i), 2 (L+2) \} r_0. \]
This implies that every $x \in S_i(t)$ with $\dist_t (\partial S_i(t), x) < \min \{ L_2^*(d_i), 2(L+2) \} r_0$ is either contained in $S_i(t) \setminus (S^*_1(t) \cup \ldots \cup S^*_{m^*} (t))$ in which case we obtain a curvature bound from Lemma \ref{Lem:firstcurvboundstep1}(b) or contained in $S^*_{i^*} (t) \setminus W^*_{i^*}$ in which case $x$ is $\ov{w}$-good at scale $r$ and time $t$ by Lemma \ref{Lem:firstcurvboundstep1}(f).
\end{proof}

\begin{Claim4}
There are universal constants $K_4^*, T_4^* < \infty$ and $0 < \tau_4^* < \tau_0^*$ such that if $t_0 > T_4^*$, then $|{\Rm}| (x,t) \linebreak[1] < K^*_4 t_0^{-1}$ for all $t \in [(1-\tau_4^*) t_0, t_0]$ and $x \in S_i(t)$ with 
\[ \dist_t(\partial S_i(t), x) \leq \min \{ L^*_4( d_i), 2 (L + 1 ) \} r_0 \]
and none of these points are surgery points.
\end{Claim4}
\begin{proof}
We use Claim 3 and Proposition \ref{Prop:curvboundnotnullinarea} with $U \leftarrow S_i$, $r_0 \leftarrow r_0$, $r_1 \leftarrow (\tau^*_3)^{1/2} r_0$, $A \leftarrow K_0^* \tau^*_3$, $w \leftarrow w^*_3$, $b \leftarrow \min \{ L^*_3( d_i), 2 (L+2) \} r_0$ to conclude that for all $t \in [(1-\frac12 \tau^*_3) t_0, t_0]$ and $x \in S_i(t)$ with $\dist_t(\partial S_i(t), x) \leq \min \{ L^*_2 (d_i), 2(L+1) \} r_0 \leq (b -1)r_0$ we have
\[ |{\Rm}| (x,t) < K_{\ref{Prop:curvboundnotnullinarea}} (w^*_3, K_0^* \tau^*_3) \big( r_0^{-2} + (\tfrac12 \tau^*_3 t_0)^{-1} \big) .  \]
This implies the claim.
\end{proof}

\begin{Claim5}
There are constants $K_5^*, T_5^* < \infty$ and $0 < \tau_5^* < \tau^*_0$ such that if $t_0 > T^*_5$, then for all $x \in S_i(t_0)$ for which 
\[ \dist_{t_0}(\partial S_i(t), x) \leq \min \{ L^*_5 ( d_i) ,  L + 1 \} r_0 \]
the point $(x,t_0)$ survives until time $(1-\tau^*_5) t_0$ and for all $t \in [(1-\tau^*_5) t_0, t_0]$ we have $|{\Rm}|(x,t) < K^*_5 t_0^{-1}$.
\end{Claim5}
\begin{proof}
This follows by a distance distortion estimate and Claim 4.
We just need to choose $\tau_5^* < \tau_4^*$ so small that distances don't shrink by more than a factor of $2$ on a time-interval of size $\leq \tau_5^* t_0$ in a region in which the curvature is bounded by $K_4^* t_0^{-1}$.
\end{proof}

To conclude the proof of the Proposition, we just need to use Claim 5 and observe that the torus structure $\Int S_i (t_0) \setminus \Int W_i$ has width $\leq r_0$ and length $\leq \min \{ L_5^* (d_i) - 1, L \} r_0$.
Hence, all its points satisfy the distance bound from Claim 5.
So assuming $\tau_2 < \tau^*_5$, we obtain a curvature bound on the non-singular neighborhood $(\Int S_i (t_0) \setminus \Int W_i) \times [(1-\tau_2) t_0, t_0]$ and we can extend $W_i$ to a sub-Ricci flow with surgery $W_i \subset \MM$ on the time-interval $[(1-\tau_2) t_0, t_0]$.
\end{proof}

In a third step we make the additional fact that we can find filling surfaces of controlled area inside a time-slice.
This will enable us to improve the bound on the width of the torus structures.
We will also present the dependences of the involved parameters in a way which will be more suitable for the following subsection.
\begin{Proposition}[third step] \label{Prop:firstcurvboundstep3}
There are a positive continuous function $\delta : [0, \infty) \linebreak[1] \to (0, \infty)$ and constants $K < \infty$, $\tau > 0$ and for every $A < \infty$ there are non-increasing functions $D_A, K'_A : (0,1] \to (0, \infty)$ such that for every $\eta > 0$ there are $w_3 = w_3(\eta, A) > 0$, $T_3 = T_3(\eta, A) < \infty$ such that: \\
Let $\MM$ be a Ricci flow with surgery on the time-interval $[0, \infty)$ with normalized initial conditions which is performed by $\delta(t)$-precise cutoff.
Consider the constant $T_0 < \infty$, the function $w : [T_0, \infty) \to (0, \infty)$ as well as the decomposition $\MM(t) = \MM_{\thick} (t) \cup \MM_{\thin} (t)$ for all $t \in [T_0, \infty)$ obtained in Proposition \ref{Prop:thickthindec}.
Assume that
\begin{enumerate}[label=(\roman*)]
\item $r_0^2 = t_0 \geq \max \{ 4T_0, T_3 \}$,
\item $w(t) < w_3$ for all $t \in [\frac14 t_0, t_0]$,
\item all components of $\MM(t_0)$ are irreducible and not diffeomorphic to spherical space forms and all surgeries on the time-interval $[\frac12 t_0, t_0]$ are trivial,
\item there is a filling map $f : \Sigma \to \MM(t_0)$ (in the sense of Definitinon \ref{Def:filling}) for the pair $(\MM_{\thick}(t_0), \MM_{\thin}(t_0))$ of area $\area_{t_0} f < A t_0$.
\end{enumerate}
Then there are closed subsets $P_1, \ldots, P_m \subset \MM(t_0)$ and sub-Ricci flows with surgery $U_1, \ldots, U_m \subset \MM$ on the time-interval $[(1-\tau) t_0, t_0]$ as well as numbers $h_1, \linebreak[1] \ldots, \linebreak[1] h_m \in [\eta, 1]$ such that the sets $P_1 \cup U_1(t_0), \ldots, P_m \cup U_m(t_0)$ are pairwise disjoint and such that for all $i = 1, \ldots, m$
\begin{enumerate}[label=(\alph*)]
\item The set $U_i(t_0)$ is a smoothly embedded, incompressible solid torus ($\approx S^1 \times D^2$), $P_i \approx T^2 \times I$ and $P_i, U_i(t_0)$ share a torus boundary, i.e. $P_i \cap U_i(t_0) = \partial U_i(t_0)$.
So $(P_i \cup U_i(t_0), U_i(t_0)) \approx (S^1 \times D^2(1), S^1 \times D^2(\frac12))$.
\item The set $P_i$ is an $h_i$-precise torus structure at scale $r_0$ and time $t_0$.
\item If $h_i > \eta$, then 
\begin{multline*}
 \qquad\qquad \diam_{t_0} (P_i \cup U_i (t_0)) < D(h_i) r_0 \qquad \text{and} \\ \qquad  | {\Rm_{t_0}} | < K'(h_i) t_0^{-1} \quad \text{on} \quad P_i \cup U_i (t_0).
\end{multline*}
\item The set $\MM(t_0) \setminus (U_1(t_0) \cup \ldots \cup U_m(t_0))$ is non-singular on the time-interval $[(1-\tau) t_0, t_0]$ and
\[ \qquad\qquad | {\Rm} | < K t_0^{-1} \qquad \text{on} \qquad \big( \MM(t_0) \setminus (U_1(t_0) \cup \ldots \cup U_m(t_0)) \big) \times [(1-\tau) t_0, t_0]. \]
\end{enumerate}
\end{Proposition}

\begin{proof}
We first define the constants $K, \tau$, the functions $D_A, K'_A$ and the quantities $w_3$, $T_3$.
Consider the functions $\Lambda_2, K'_2, \tau'_2$ and the constants $K_2, \tau_2$ from Proposition \ref{Prop:firstcurvboundstep2}.
Choose $D^* < \infty $ such that $\Lambda_2 (D^*) > 1000$ and set
\[ K = \max \{ K_2, K'_2 (D^*) \} \qquad \text{and} \qquad \tau = \min \{ \tau_2, \tau'_2 (D^*) \}. \]

Now fix the constant $A < \infty$.
Before defining $D_A$ and $K'_A$, we need to fix a few other quantities and functions that we will be important in the course of the proof.
Using the constants $\td{L}_0$ from Lemma \ref{Lem:smallloopintorusstruc}, $\td\nu$ from Lemma \ref{Lem:bettertorusstructure} and $\td\varepsilon_1$ from Lemma \ref{Lem:2loopstorus} we set
\[ \ov{L}_{A} = \max \Big\{ \td{L}_0 \big( \min \big\{ \tfrac1{10} \td\nu ( K_2, 1, 1 ), \td\varepsilon_1 (K_2) \big\}, \; A \big), \; 10 \big(\td\nu ( K_2, 1, 1 ) \big)^{-1} \Big\}.  \]
Then we define the functions $L^{**}_A, L^*_A : (0,1] \to (0, \infty)$ by
\begin{alignat*}{1}
 L^{**}_A(h) &= \max \Big\{ \ov{L}_A, \; \td{L}_0 \big( \min \big\{ \tfrac1{10} \td\nu ( K_2, h^{-1}, h ), \td\varepsilon_1(K_2) \big\}, \; A \big) , \\
 & \qquad\qquad \qquad\qquad\qquad\qquad 2 h^{-1} + 100, \; 10 \big(\td\nu ( K_2, h^{-1}, h ) \big)^{-1} \Big\}. \\
 L^*_A(h) &= \inf \big\{ L^{**}_A (h'') \;\; : \;\; 0 < h'' \leq \tfrac12 h \big\}.
\end{alignat*}
Then $L^*_A$ is non-increasing.
Using this function, we define the functions $D_A, K'_A$ by
\begin{alignat*}{1}
 D_A(h) &= \sup \{ d > 0 \;\; : \;\; \Lambda_2(d) \leq L^*_A (h) \}. \\
 K'_A(h) &= K'_2(D_A(h))
\end{alignat*}
Observe that also $D_A$ and $K'_A$ are non-decreasing.

Now also fix the constant $\eta > 0$ and set
\[
 L^\circ = \max \{ L^*_A(\eta) + 1, D_A(\eta), 1000 \} \qquad \text{and} \qquad \nu^\circ = \min \Big\{ \frac{1}{L^\circ}, \td\varepsilon_1 (K_2) \Big\} .
\]
Then we define
\[ w_3(\eta, A) = w_2(L^\circ, \nu^\circ) \qquad \text{and} \qquad T_3(\eta, A) = T_2 (L^\circ). \]

By this choice of $w_3$ and $T_3$, we can apply Proposition \ref{Prop:firstcurvboundstep2} with $L \leftarrow L^\circ$, $\nu \leftarrow \nu^\circ$ and obtain sub-Ricci flows with surgery $S_1, \ldots, S_m \subset \MM$ on the time-interval $[(1-\tau) t_0, t_0]$ and subsets $W_i \subset S_i(t_0)$.
For each $i = 1, \ldots, m$ set $d_i = r_0^{-1} \diam_{t_0} S_i$.
Then $P'_i = S_i \setminus \Int W_i$ are torus structures of width $\leq r_0$ and length $L_i r_0$ for
\[ L_i = \min \{ \Lambda_2 ( d_i), L^\circ \}. \]
We can assume without loss of generality that $L_i \geq 1000$ for all $i=1, \ldots, m$: If $L_i < 1000$ for some $i = 1, \ldots, m$, then by Proposition \ref{Prop:firstcurvboundstep2}(d) $\Lambda_2(d_i) < 1000 \leq L^\circ$.
So by monotonicity of $\Lambda_2$ we have $d_i \leq D^*$ and by Proposition \ref{Prop:firstcurvboundstep2}(c) and the choice of $\tau$, $K$ the flow $S_i$ is non-singular on the time-interval $[(1-\tau) t_0, t_0]$ and we have $| {\Rm} | < K t_0^{-1}$ on $S_i(t_0) \times [(1-\tau)t_0, t_0]$.
Hence we can remove the pair $S_i$ and $W_i$ from the list.
We summarize
\[
1000 \leq L_i \leq L^\circ \qquad \text{for all} \qquad i = 1, \ldots, m.
\]

Next, we define numbers $h'_i > 0$ which will give rise to the $h_i$.
If $L_i \leq \ov{L}_A$, then we just set $h'_i = 1$.
Observe that by definition of $L^*_A$, we then immediately get $L_i \leq L^*_A(h_i)$.
In the case in which $L_i > \ov{L}_A$, we choose an $h'_i \in (0,1]$ such that the following three equations are satisfied:
\begin{alignat}{1}
L_i &> \td{L}_0 \big( \min \big\{ \tfrac1{10} \td\nu ( K_2, (h'_i)^{-1}, h'_i ), \td\varepsilon_1(K_2) \big\}, \; A \big) \notag  \\
L_i &> 2 (h'_i)^{-1} + 100 \label{eq:Listwicehinverse} \\
L_i &>  10 (\td\nu ( K_2, (h'_i)^{-1}, h'_i ))^{-1} \label{eq:Li10nuinverse}
\end{alignat}
The conditions $L_i > \ov{L}_A$ and $L_i \geq 1000$ ensure that we can find such an $h'_i$.
We can furthermore assume that $h'_i$ is so small such that for any $h'' \leq \frac12 h'_i$, at least one of these conditions is not fulfilled.
So for any $h'' \leq \frac12 h'_i$ we have $L_i \leq L^{**}_A (h'')$ and thus $L_i \leq L^*_A (h'_i)$.
Furthermore, observe that by (\ref{eq:Li10nuinverse}) we have
\begin{equation} \label{eq:lowerboundon110nu}
 \tfrac1{10} \td\nu (K_2, (h'_i)^{-1}, h'_i) > \frac{1}{L_i} \geq \frac1{L^\circ} = \nu^\circ.
\end{equation}
Lastly, we define
\[ h_i = \max \{ h'_i, \eta \}. \]
So for all $i = 1, \ldots, m$ for which $h_i > \eta$ we have if $L_i \leq L^*_A(h'_i) = L^*_A (h_i)$.
Proposition \ref{Prop:firstcurvboundstep2}(d) together with the fact that $L^\circ \leq L^*_A (h_i)$ implies that $\Lambda_2(d_i) \leq L^*_A (h_i)$.
Hence $d_i \leq D_A(h_i)$ and we have the curvature bound $|{\Rm_{t_0}}| < K'(h_i) t_0^{-1}$ on $S_i(t_0)$.
Since $P_i \cup U(t_0)$ will be strictly contained in $S_i(t_0)$, this establishes assertion (c).

For the next paragraphs fix $i = 1, \ldots, m$.
We will now construct the sets $P_i$ and the sub-Ricci flows with surgery $U_i \subset \MM$.
$U_i$ and $P_i$ will be a modification of the sets $W_i$ and $P'_i$.
The torus structure $P_i$ will be a subset of $P'_i$.

First consider the case $L_i \leq \ov{L}_A$.
Then $h_i = 1$.
In this case, we simply set $P_i = S_i \setminus \Int W_i$ and $U_i = W_i$.
Obviously, $P_i$ is $1$-precise.

Assume in the following that $L_i > \ov{L}_A$.
From Lemma \ref{Lem:smallloopintorusstruc} applied to $M \leftarrow \MM(t_0)$, $M_{\textnormal{hyp}} \leftarrow \MM_{\thick} (t_0)$, $M_{\textnormal{Seif}} \leftarrow \MM_{\thin} (t_0)$, $P \leftarrow P'_i$, $S \leftarrow S_i$ and $f \leftarrow f$, we obtain a closed loop $\gamma_i \subset P'_i$ which is non-contractible in $P'_i$, but contractible in $S_i(t_0)$, has length
\[ \ell_{t_0} (\gamma_i) < \min \big\{ \tfrac1{10} \td\nu ( K_2, (h'_i)^{-1}, h'_i ), \td\varepsilon_1(K_2) \big\} r_0 \]
and which has time-$t_0$ distance of at least $\frac13 L_i - 2$ from $P'_i$.
Let $p_i \in \gamma_i$ be an arbitrary base point.
By Proposition \ref{Prop:firstcurvboundstep2}(e), there is a closed loop $\sigma_i \subset P'_i$ based at $p_i$ which is non-contractible in $S_i$ and has length (see (\ref{eq:lowerboundon110nu})) 
\[ \ell_{t_0} (\sigma_i) < \nu^\circ r_0 \leq \min \big\{ \tfrac1{10} \td\nu ( K_2, (h'_i)^{-1}, h'_i ), \td\varepsilon_1 (K_2) \big\} r_0. \]
In particular, $\gamma_i$ and $\sigma_i$ represent two linearly independent homotopy classes in $\pi_1(P'_i) \cong \IZ^2$.
By Lemma \ref{Lem:2loopstorus}, there is an embedded torus $T_i \subset P'_i$ with $p_i \in T_i$ which separates the two ends of $P'_i$ and which has diameter
\[ \diam_{t_0} T_i < \td\nu ( K_2, (h'_i)^{-1}, h'_i ) r_0. \]
Observe that $T_i$ has distance of at least $\frac13 L_i - 3 \geq \frac12 (h'_i)^{-1} + 30$ from $\partial P'_i$ (see (\ref{eq:Listwicehinverse})).
We can hence apply Lemma \ref{Lem:bettertorusstructure} and obtain a torus structure $P_i \subset P'_i$ of width $\leq h'_i r_0$ and length $> (h'_i)^{-1} r_0$ such that the pair $(P'_i, P_i)$ is diffeomorphic to $(T^2 \times [-2, 2], T^2 \times [-1,1])$.
Finally, we let $U_i (t_0)$ be the closure of the component of $S_i(t_0) \setminus P_i$ which is diffeomorphic to a solid torus.
It is clear that we can extend it to a sub-Ricci flow with surgery $U_i \subset \MM$ on the time-interval $[(1-\tau) t_0, t_0]$.
Then assertions (a), (b), (d) are clear and (c) was established before.
\end{proof}

\subsection{The geometry on late and long time-intervals}
In this subsection, we relate the conclusions from Proposition \ref{Prop:firstcurvboundstep3} applied at each time of a larger time-interval towards one another and obtain a geometric description of the flow on the given time-interval.
More precisely, we will be able to conclude that, if the size of the time-interval is controlled and its initial time is large enough, then we either have a curvature bound on the final time-slice or we can find a curve whose length at all times of the time-interval is very small and which bounds a disk of controlled area.
\begin{Proposition} \label{Prop:structontimeinterval}
There is a positive continuous function $\delta : [0, \infty) \to (0, \infty)$ and for every $L, A < \infty$, $\alpha > 0$ there are constants $K_4 = K_4 (L, \linebreak[1] A, \linebreak[1] \alpha), \linebreak[1] \Gamma_4 = \Gamma_4 (L, A), T_4 = T_4 (L, A, \alpha) < \infty$ and $w_4 = w_4 (L, A, \alpha) > 0$ (observe that $\Gamma_4$ does not depend on $\alpha$) such that: \\
Let $\MM$ be a Ricci flow with surgery on the time-interval $[0, \infty)$ with normalized initial conditions which is performed by $\delta(t)$-precise cutoff.
Consider the constant $T_0 < \infty$, the function $w : [T_0, \infty) \to (0, \infty)$ as well as the decomposition $\MM (t) = \MM_{\thick} (t) \cup \MM_{\thin} (t)$ for all $t \in [T_0, \infty)$ obtained in Proposition \ref{Prop:thickthindec} and assume that
\begin{enumerate}[label=(\roman*)]
\item $t_\omega > r_0^2 = t_0 \geq \max \{ 4T_0, T_4 \}$ and $t_\omega \leq L t_0$,
\item $w(t) < w_4$ for all $t \in [\frac14 t_0, t_\omega]$,
\item for every $t \in [\frac14 t_0, t_\omega]$ all components of $\MM(t)$ are irreducible and not diffeomorphic to spherical space forms and all surgeries on the time-interval $[\frac14 t_0, t_\omega]$ are trivial,
\item for every $t \in [t_0, t_\omega]$ there is a filling map $f : \Sigma \to \MM(t)$ (in the sense of Definitinon \ref{Def:filling}) for the pair $(\MM_{\thick}(t), \MM_{\thin}(t))$ of area $\area_t f < A t$.
\end{enumerate}
Then there is a collection of sub-Ricci flows with surgery $U_1, \ldots, U_m \subset \MM$ on the time-interval $[t_0, t_\omega]$ such that for all $t \in [t_0, t_\omega]$, the sets $U_1 (t), \ldots, U_m(t) \subset \MM (t)$ are pairwise disjoint, incompressible, solid tori.
Moreover, there are collars $P_i \subset \MM (t_\omega) \setminus \Int (U_1(t_\omega) \cup \ldots \cup U_m (t_\omega))$ of each $U_i (t_\omega)$ which are diffeomorphic to $T^2 \times I$ and non-singular on the time-interval $[t_0, t_\omega]$ and there are closed loops $\gamma_i \subset P_i$ which at time $t_0$ bound disks $D_i \subset \MM(t_0)$ such that
\begin{enumerate}[label=(\alph*)]
\item $| {\Rm_{t_\omega}} | < K_4 t_\omega^{-1}$ on $\MM(t_\omega) \setminus (U_1 (t_\omega) \cup \ldots \cup U_{m} (t_\omega))$,
\item $\ell_t (\gamma_i) < \alpha \sqrt{t}$ and $\diam_t \partial U_i(t) < \alpha \sqrt{t}$ for all $t \in [t_0, t_\omega]$ and $i = 1, \ldots, m$,
\item $\max \curv_t \gamma_i < \Gamma_4 t_0^{-1}$ for all $t \in [t_0, t_\omega]$ and all $i = 1, \ldots, m$,
\item $\area_{t_0} D_i < A t_0$ for all $i = 1, \ldots, m$,
\item $\gamma_i$ is non-contractible in $P_i$.
\end{enumerate}
\end{Proposition}

\begin{proof}
Let $\tau$ be the constant from Proposition \ref{Prop:firstcurvboundstep3} and pick $N \in \IN$ minimal with the property that $(1+\tau)^N t_0 \geq t_\omega$.
Subdivide the time-interval $[t_0, t_\omega]$ by times $t_k = (1+\tau)^k t_0$ for $k = 0, 1, \ldots, N$.
For simplicity, we assume that $t_{N} = t_{\omega}$ (if not, we can decrease $\tau$ slightly or extend the flow past time $t_\omega$) and that $N \geq 3$.
Observe that $N$ depends on $L$.

In the following proof, we will apply Proposition \ref{Prop:firstcurvboundstep3} at the times $t_0 \leftarrow t_k$ with $A \leftarrow A$ and $\eta \leftarrow \eta^\circ$.
Here $\eta^\circ = \eta^\circ(L, A, \alpha) > 0$ is a constant which we are going to determine in the course of the proof.
It will be clear that $\eta^\circ$ can be chosen such that it only depends on $L$, $A$ and $\alpha$.
In order to be able to apply Proposition \ref{Prop:firstcurvboundstep3}, we assume $t_0 > T_3(\eta^\circ, A)$ and that $w(t) < w_3(\eta^\circ, A)$ for all $t \in [\frac14 t_0, t_\omega]$.
Then at each time $t_k$, Proposition \ref{Prop:firstcurvboundstep3} provides subsets $P_i^{(k)} \subset \MM(t_k)$, sub-Ricci flows with surgery $U_i^{(k)} \subset \MM$ on the time-interval $[t_{k-1}, t_k]$ as well as numbers $h_i^{(k)} > 0$ for $i = 1, \ldots, m^{(k)}$.
The $P_i^{(k)}$ are $h_i^{(k)}$-precise torus structures at scale $\sqrt{t_k}$ and time $t_k$ and using the universal constant $K$ from Proposition \ref{Prop:firstcurvboundstep3} we have for each $k \in \{ 0, \ldots, N \}$
\begin{equation} \label{eq:summarycurvboundgoodpart} 
| {\Rm} | < K t^{-1} \quad \text{on} \quad \big( \MM(t_k) \setminus \big( U_1^{(k)} (t_k) \cup \ldots \cup U_{m^{(k)}}^{(k)} (t_k) \big)\big) \times [t_{k-1}, t_k].
\end{equation}

It is not difficult to see that for all $k = 0, \ldots, N$ and $i = 1, \ldots, m^{(k)}$ we can divide $P_i^{(k)}$ into $N + 2$ approximately equally long torus structures as
\[ P_i^{(k)} = P_{i,1}^{(k)} \cup \ldots \cup P_{i, N+2}^{(k)} \]
such that if $h^{(k)}_i < \frac1{2(N+1)}$, then the $P_{i,j}^{(k)}$ are $2N h_i^{(k)}$-precise at scale $\sqrt{t_k}$ and time $t_k$ and in such a way that $P_{i,j}^{(k)}$ and $P_{i,j+1}^{(k)}$ are adjacent for $j < N + 2$ and $P_{i, N+2}^{(k)}$ is adjacent to $U_i^{(k)}(t_k)$.
Using this subdivision, we define the sub-Ricci flows with surgery $V^{(k)}_{i,1}, \ldots, V^{(k)}_{i, N+3} \subset \MM$ on the time-interval $[t_{k-1}, t_k]$ as follows:
$V^{(k)}_{i,j}$ is the extension of the subset
\[ P^{(k)}_{i, j} \cup \ldots \cup P^{(k)}_{i, N+2} \cup U^{(k)}_i (t_k) \subset \MM(t_k) \]
to the time-interval $[t_{k-1}, t_k]$.

\begin{Claim1}
There are a constant $\eta^*_1 = \eta^*_1 (L, A) > 0$ and a non-decreasing function $\varphi^*_1 = \varphi^*_{1, L, A} : (0, \infty) \to (0, \infty)$ which both depend on $L$ and $A$ such that $\varphi^*_1(h) < h$ for all $h > 0$ and such that the following holds: \\
Let $k_1, k_2 \in \{0, \ldots, N \}$, $k_1 \neq k_2$, $i_1 \in \{ 1, \ldots, m^{(k_1)} \}$, $i_2 \in \{ 1, \ldots, m^{(k_2)} \}$.
Assume that there is some $j \leq N+1$ such that
\begin{enumerate}[label=(\roman*)]
\item if $k_1 < k_2$: all points in $P_{i_1, j}^{(k_1)} \cup P_{i_1, j+1}^{(k_1)}$ survive until time $t_{k_2-1}$ and $P_{i_1, j}^{(k_1)} \cap  U_{i_2}^{(k_2)} (t_{k_2 - 1}) \neq \emptyset$.
Moreover, we have $(P_{i_1, j}^{(k_1)} \cup P_{i_1, j+1}^{(k_1)}) \cap U_{i'}^{(k')}(t_{k'}) = \emptyset$ for all $k'$ strictly between $k_1$ and $k_2$ and all  $i' \in \{ 1, \ldots, m^{(k')} \}$.
\item if $k_2 < k_1$: all points in $P_{i_1, j}^{(k_1)} \cup P_{i_1, j+1}^{(k_1)}$ survive until time $t_{k_2}$ and $P_{i_1, j}^{(k_1)} \cap  U_{i_2}^{(k_2)} (t_{k_2}) \neq \emptyset$.
Moreover, we have $(P_{i_1, j}^{(k_1)} \cup P_{i_1, j+1}^{(k_1)}) \cap U_{i'}^{(k')}(t_{k'}) = \emptyset$ for all $k'$ strictly between $k_1$ and $k_2$ and all  $i' \in \{ 1, \ldots, m^{(k')} \}$.
\end{enumerate}
Then if $h_{i_1}^{(k_1)} < \eta^*_1$, we can uniquely extend the sub-Ricci flow with surgery $V_{i_1, j+2}^{(k_1)}$ to the time-interval $[t_{k_1-1}, t_{k_2 - 1}]$ in case (i) and $[t_{k_2}, t_{k_1}]$ in case (ii).
These extensions satisfy $V_{i_1, j+2}^{(k_1)} (t_{k_2 - 1}) \subsetneq U_{i_2}^{(k_2)} (t_{k_2 - 1} )$ in case (i) and $V_{i_1, j+2}^{(k_1)} (t_{k_2}) \subsetneq U_{i_2}^{(k_2)} (t_{k_2} )$ in case (ii).

Lastly in case (ii), without imposing any restriction on $h^{(k_1)}_{i_1}$, we have $\varphi_1^* (h_{i_2}^{(k_2)}) \leq \max \{ h_{i_1}^{(k_1)}, \varphi_1^*(\eta^\circ) \}$.
\end{Claim1}

\begin{proof}
By (\ref{eq:summarycurvboundgoodpart}), we have $|{\Rm_t}| < K t^{-1}$ on $P_{i_1, j}^{(k_1)} \cup P_{i_1, j-1}^{(k_1)}$ for all $t \in [t_{k_1-1}, t_{k_2-1}]$ (in case (i)) or $t \in [t_{k_2}, t_{k_1}]$ (in case (ii)).
So $P_{i_1, j}^{(k_1)}$ and $P_{i_1, j+1}^{(k_1)}$ are still $\varphi'(h_{i_1}^{(k_1)})$-precise torus structures at scale $\sqrt{t_{k_2-1}}$ or $\sqrt{t_{k_2}}$ and time $t_{k_2-1}$ or $t_{k_2}$ (depending on whether we are in case (i) or (ii)) for some $\varphi' : (0, \infty) \to (0, \infty)$ with $\varphi'(h) \to 0$ as $h \to 0$ (observe that $\varphi'$ can be chosen depending only on $L$).
It is clear that we can extend the sub-Ricci flows with surgery $V_{i_1, j+2}^{(k_1)}$ to the desired time-interval

We now show that if $h_{i_1}^{(k_1)}$ is small enough, then we must have $V_{i_1, j+2}^{(k_1)} (t_{k_2 /- 1}) \subsetneq U_{i_2}^{(k_2)} (t_{k_2 /- 1} )$ (by ``$k_2/-1$'' we mean $k_2$ in case (i) and $k_2 -1$ in case (ii)):
By the previous conclusion and by assuming $h_{i_1}^{(k_1)}$ to be small enough, we can assume that at time $t_{k_2 /- 1}$, the set $P^{(k_1)}_{i_1, j}$ is far enough away from $V^{(k_1)}_{i_1, j+2} (t_{k_2 /- 1})$ to ensure that $\partial U^{(k_2)}_{i_2}(t_{k_2 /- 1})$ cannot intersect both $P^{(k_1)}_{i_1, j}$ and $V^{(k_1)}_{i_1, j+2}(t_{k_2 /- 1})$.
Consider first the case in which $V^{(k_1)}_{i_1, j+2}(t_{k_2 /- 1}) \cap \partial U^{(k_2)}_{i_2} (t_{k_2 /- 1}) = \emptyset$.
Then either $V_{i_1, j+2}^{(k_1)} (t_{k_2 /- 1}) \subsetneq U_{i_2}^{(k_2)} (t_{k_2 /- 1} )$ and we are done or $V_{i_1, j+2}^{(k_1)} (t_{k_2 /- 1}) \cap U_{i_2}^{(k_2)} (t_{k_2 /- 1} )= \emptyset$.
However, in the latter case we conclude---assuming $h_{i_1}^{(k_1)}$ to be small---that a cross-sectional torus of $P^{(k_1)}_{i_1, j}$ is contained in $P^{(k_2)}_{i_2} \cup U^{(k_2)}_{i_2} (t_{k_2 /-1})$ and hence $P^{(k_1)}_{i_1, j} \cup P^{(k_1)}_{i_1, j+1} \cup V^{(k_1)}_{i_1}(t_{k_2/-1})$ and $P^{(k_2)}_{i_2} \cup U^{(k_2)}_{i_2}(t_{k_2 /-1 })$ cover a component of $\MM(t_{k_2 /- 1})$.
This however contradicts Lemma \ref{Lem:coverMbysth} and assumption (iii).
It remains to consider the case in which $P^{(k_1)}_{i_1, j} \cap \partial U^{(k_2)}_{i_2} (t_{k_2 /- 1}) = \emptyset$, but $V^{(k_1)}_{i_1, j+2} (t_{k_2/-1})\cap \partial U^{(k_2)}_{i_2} (t_{k_2/-1}) \neq \emptyset$.
This implies together with the assumption in the claim that $P^{(k_1)}_{i_1, j} \subset U^{(k_2)}_{i_2} (t_{k_2 /-1})$ and so a component of $\MM (t_{k_2/-1})$ is covered by $P^{(k_1)}_{i_1, j} \cup P^{(k_1)}_{i_1, j+1} \cup V^{(k_1)}_{i_1}(t_{k_2/-1})$ and $U^{(k_2)}_{i_2}(t_{k_2 /-1})$, contradicting again Lemma \ref{Lem:coverMbysth}.
We have hence established the first part of the claim.

For the second part of the claim we can choose $\varphi^*_1$ such that $\varphi_1^*(\eta) < \eta_1^*$ for all $\eta > 0$, so we only need to consider the case $h^{(k_1)}_{i_1} < \eta_1^*$.
Moreover, we can assume that $h^{(k_2)}_{i_2} > \eta^\circ$, because otherwise the statement is trivial.
Observe now that since $U_{i_2}^{(k_2)} (t_{k_2})$ intersects both $P_{i_1, j}^{(k_1)}$ and $V_{i_1, j+2}^{(k_1)} (t_{k_2})$ and the boundary components of $P_{i_1, j+1}^{(k_1)}$, two sets which have time-$t_{k_2}$ distance of at least $(\varphi'(h_{i_1}^{(k_1)}))^{-1} \sqrt{t_{k_2}}$, we find using Proposition \ref{Prop:firstcurvboundstep3}(c)
\[ \big( \varphi'(h_{i_1}^{(k_1)}) \big)^{-1} \sqrt{ t_{k_2}} \leq \diam_{t_{k_2}} U_{i_2}^{(k_2)} (t_{k_2})
 < D(h_{i_2}^{(k_2)}) \sqrt{t_{k_2}}. \]
This establishes the second part.
\end{proof}

Consider now the index set $I = \{ (k, i) \; : \; 0 \leq k \leq N, 1 \leq i \leq m^{(k)} \}$.
We will write $(k_1, i_1) \prec (k_2, i_2)$ whenever we are in the situation of Claim 1, i.e. if there is a $j  \leq N + 1$ such case (i) or (ii) of this claim holds.

\begin{Claim2}
There are a constant $\eta_2^* = \eta_2^*(L,A)  > 0$ and a monotonically non-decreasing function $\varphi_2^* = \varphi_{2, L, A}^* : (0, \infty) \to (0, \infty)$ which both depend on $L$ and $A$ such that if $\eta^\circ < \eta_2^*$, then the following holds: \\ Whenever we have a chain
\[ (k_1, i_1) \prec (k_2, i_2) \prec \ldots \prec (k_m, i_m) \]
such that $k_2, \ldots, k_m \leq k_1$ and $h_{i_1}^{(k_1)} < \eta_3^*$, then $m \leq N + 1$ and $k_1 > k_2 > \ldots > k_m$.
Moreover, there are indices $j_1, \ldots, j_{m-1} \in \{ 1, \ldots, N+1 \}$ such that the sub-Ricci flow with surgery $V_{i_1, j_1+2}^{(k_1)}$ can be extended to the time-interval $[t_{k_2}, t_{k_1}]$, $V_{i_2, j_2+2}^{(k_2)}$ can be extended to the time-interval $[t_{k_3}, t_{k_2}]$, \ldots and such that
\begin{multline*}
 V_{i_1, j_1+2}^{(k_1)} (t_2) \subsetneq U_{i_2}^{(k_2)} (t_2), \qquad V_{i_2, j_2+2}^{(k_2)} (t_3) \subsetneq U_{i_3}^{(k_3)} (t_3),  \qquad \ldots, \\
V_{i_{m-1}, j_{m-1}+2}^{(k_{m-1})} (t_m) \subsetneq U_{i_m}^{(k_m)} (t_m)
\end{multline*}
Lastly, $\varphi_2^* (h_{i_m}^{(k_m)}) \leq \max \{ h_{i_1}^{(k_1)}, \eta^\circ \}$.
\end{Claim2}
\begin{proof}
Set $\eta_2^* = {\varphi_1^*}^{(N)}( \eta_1^* )$ and $\varphi_2^* (h) = {\varphi_1^*}^{(N)} (h)$ where $\eta_1^*, \varphi_1^*$ are taken from Claim 1 and the upper index in parentheses indicates multiple application.

Without loss of generality, we can assume that $m \leq N + 2$, because otherwise we can shorten the chain to size $N + 2$.
We will first show the claim without the last line by induction on $m$.
Additionally, we show that
\begin{equation} \label{eq:extraindasspt}
 h_{i_m}^{(k_m)} \leq  {\varphi_1^*}^{(N - m +1 )} (\eta_1^*),
\end{equation}
if $m \leq N + 1$.
For $m=1$ there is nothing to show.
Assume that the induction hypothesis holds for $m-1$, i.e. that we can extend the sub-Ricci flows with surgery $V_{i_1, j_1+2}^{(k_1)}, \ldots, V_{i_{m-2}, j_{m-2}+2}^{(k_1)}$ to the appropriate time-intervals so that they satisfy the inclusion property above, that $k_1 > k_2 > \ldots > k_{m-1}$ and that $h_{i_m}^{(k_{m-1})}$ satisfies inequality (\ref{eq:extraindasspt}) above with $m$ replaced by $m-1$.

So $h_{i_{m-1}}^{(k_{m-1})} \leq {\varphi^*_1}^{(N - m +2 )} (\eta_1^*) < \eta_1^*$ and we conclude by Claim 1 that there is a $j_{m-1}$ such that we can extend the sub-Ricci flow with surgery $V^{(k_{m-1})}_{i_{m-1}, j_{m-1} +2}$ to the time-interval $[t_{k_m}, t_{k_{m-1}}]$ or $[t_{k_{m-1} - 1}, t_{k_m - 1}]$ depending on whether $k_m < k_{m-1}$ or $k_m > k_{m-1}$ and we have
\begin{equation} \label{eq:VkminUkm}
V^{(k_{m-1})}_{i_{m-1}, j_{m-1} + 2} (t_{k_m /- 1}) \subsetneq U^{(k_m)}_{i_m}( t_{k_m /-1}).
\end{equation}

We now show in the next two paragraphs that we must have $k_m < k_{m-1}$:
Assume that $k_m > k_{m-1}$.
Then there is an $l = 1, \ldots, m-2$ such that $k_{l+1} < k_m \leq k_l$.
We first show that
\begin{equation} \label{eq:Ulp1inVkm}
U^{(k_{l+1})}_{i_{l+1}} (t_{k_{l+1}}) \subset V^{(k_{m-1})}_{i_{m-1}, j_{m-1} + 2} (t_{k_{l+1}}).
\end{equation}
Choose $l+1 \leq l^* \leq m-1$ minimal with the property that $U^{(k_{l^*})}_{i_{l^*}} (t_{k_{l^*}})\subset V^{(k_{m-1})}_{i_{m-1}, j_{m-1} + 2} (t_{k_{l^*}})$.
This is possible, since the inclusion is true for $l^* = m-1$.
Then by the induction assumption
\[ V^{(k_{l^*-1})}_{i_{l^*-1}, j_{l^*-1} + 2} (t_{k_{l^*}}) \subset U^{(k_{l^*})}_{i_{l^*}} (t_{k_{l^*}})\subset V^{(k_{m-1})}_{i_{m-1}, j_{m-1} + 2} (t_{k_{l^*}}) \]
which implies that if $l^* > l+2$, then
\[ U^{(k_{l^*-1})}_{i_{l^*-1}} (t_{k_{l^*-1}}) \subset V^{(k_{l^*-1})}_{i_{l^*-1}, j_{l^*-1} + 2} (t_{k_{l^*-1}}) \subset V^{(k_{m-1})}_{i_{m-1}, j_{m-1} + 2} (t_{k_{l^*-1}}) \]
in contradiction to the choice of $l^*$.
So $l^* = l+2$ and (\ref{eq:Ulp1inVkm}) holds.
Furthermore, we conclude that $k_m < k_l$, because otherwise by the same argument as before, but with $l+1$ replaced by $l$, we had
\[ U^{(k_{l})}_{i_{l}} (t_{k_{l}-1}) \subset V^{(k_{m-1})}_{i_{m-1}, j_{m-1} + 2} (t_{k_{l}-1}) \subsetneq U^{(k_m)}_{i_m}( t_{k_l - 1}). \]

Since $(k_l, i_l) \prec (k_{l+1}, i_{l+1})$, we know that that the flow is non-singular on $( P^{(k_l)}_{i_l, j_l} \cup P^{(k_l)}_{i_l, j_l+1} ) \times [t_{k_{l+1}}, t_{k_l}]$ and that $P^{(k_l)}_{i_l, j_l} \cap U^{(k_{l+1})}_{i_{l+1}} (t_{k_{l+1}}) \neq \emptyset$.
So by (\ref{eq:Ulp1inVkm}), we have $P^{(k_l)}_{i_l, j_l} \cap V^{(k_{m-1})}_{i_{m-1}, j_{m-1} + 2} (t_{k_{l+1}}) \neq \emptyset$.
But this implies using (\ref{eq:VkminUkm}) that $P^{(k_l)}_{i_l, j_l} \cap U^{(k_m)}_{i_m} (t_{k_m }) \neq \emptyset$ in contradiction to the definition of the relation $(k_l, i_l) \prec (k_{l+1}, i_{l+1})$.
So, indeed we have $k_m < k_{m-1}$.
We can thus apply Claim 1 to conclude from (\ref{eq:extraindasspt}) for $m-1$ that
\begin{multline*}
 \varphi_1^* ( h_{i_m}^{(k_m)} ) \leq \max \big\{ {\varphi_1^*}^{(N - m + 2 )} (\eta_1^*), \varphi_1^*(\eta^\circ) \big\} \\ \leq \max \big\{ {\varphi_1^*}^{(N - m + 2 )} (\eta_1^*), \varphi_1^*(\eta_2^*) \big\} 
 \leq  {\varphi_1^*}^{(N - m + 2 )} (\eta_1^*).
\end{multline*}
This implies (\ref{eq:extraindasspt}) for $m$, by the monotonicity of $\varphi_1^*$ and finishes the induction.

For the last line in the claim, we can use Claim 1 to conclude
\begin{alignat*}{1}
\varphi_2^*( h^{(k_m)}_{i_m}) &= {\varphi_1^*}^{(N)}( h^{(k_m)}_{i_m} ) \leq {\varphi_1^*}^{(N-1)} \big( \max \{ h^{(k_{m-1})}_{i_{m-1}} , \varphi_1^*(\eta^\circ) \} \big) \\
&\leq \max \big\{ {\eta_2^*}^{(N-1)} (  h^{(k_{m-1})}_{i_{m-1}} ), \eta^\circ \big\}
\\
& \qquad \leq \max \big\{ {\varphi_1^*}^{(N-2)} \big(  \max \{ h^{(k_{m-2})}_{i_{m-2}}, \varphi_1^*(\eta^\circ) \} \big), \eta^\circ \big\} \\
 &\leq \max \big\{ {\varphi_1^*}^{(N-2)} ( h^{(k_{m-2})}_{i_{m-2}} ), \eta^\circ \big\} 
 \\ &\leq \ldots \\ 
 &\leq \max \big\{ {\varphi_1^*}^{(N-m+1)} ( h^{(k_1)}_{i_1} ), \eta^\circ \big\}
 \leq  \max \big\{ h^{(k_1)}_{i_1} , \eta^\circ \big\} \\[-2.1\baselineskip]
\end{alignat*}
\end{proof}

\begin{Claim3}
Assume that $\eta^\circ < \eta_2^*$.
If there are indices $(k, i) \in I$ such that $h_i^{(k)} < \eta_2^*$, then there are indices $(k^*, i^*) \in I$ with $k^* \leq k$ such that the following holds:
The set $P^{(k^*)}_{i^*, 1}$ is non-singular on $[t_{-1}, t_k]$ and we can extend $V^{(k^*)}_{i^*, 2}$ to a sub-Ricci flow with surgery on the time-interval $[t_{-1}, t_k]$.
In particular, $P^{(k^*)}_{i^*, 1} \cap U^{(k')}_{i'} (t_{k'}) = \emptyset$ for all $(k', i') \in I$ with $k' \leq k$ and we have $U^{(k)}_i (t_k) \subsetneq V^{(k^*)}_{i^*, 2} (t_k)$.
Lastly, $\varphi_2^* ( h_{i^*}^{(k^*)} ) \leq \max \{ h_i^{(k)}, \eta^\circ \}$.
\end{Claim3}
\begin{proof}
Consider a maximal chain as in Claim 2 with $(k_1, i_1) = (k, i)$ and $k_2, \ldots, k_m \linebreak[2] \leq k_1$ and set $(k^*, i^*) = (k_m, i_m)$.
By Claim 2, we have $k_1 > k_2 > \ldots > k_m$ and we obtain indices $j_1, \ldots, j_{m-1} \in \{ 1, \ldots, N+1 \}$ together with extensions of the flows $V^{(k_1)}_{i_1, j_1 +2}, \ldots, V^{(k_{m-1})}_{i_{m-1}, j_{m-1} + 2}$ which satisfy the inclusion property mentioned above.
Moreover, $\varphi^*_2( h^{(k^*)}_{i^*} ) \leq \max \{ h_i^{(k)}, \eta^\circ \}$ which establishes the last part of the claim.

Assume now that $P^{(k^*)}_{i^*, 1}$ is singular on the time-interval $[t_{-1}, t_k]$.
Then we can find some $k' \leq k$ such that $P^{(k^*)}_{i^*, 1}$ is non-singular on $[t_{k^*}, t_{k'-1}]$ (if $k' > k^*$) or on $[t_{k'}, t_{k^*}]$ (if $k' < k^*$) and there is an index $i' \in \{ 1, \ldots, m^{(k')} \}$ such that $P^{(k^*)}_{i^*, 1} \cap U^{(k')}_{i'} (t_{k' - 1}) \neq \emptyset$ (if $k' > k^*$) or $P^{(k^*)}_{i^*, 1} \cap U^{(k')}_{i'} (t_{k'}) \neq \emptyset$ (if $k' < k^*$).
Then consider all triples $(k'', i'', j^*)$ of indices with $(k'', i'') \in I$, $k'' \leq k$ and $j^* \in \{ 1, \ldots, N + 2 \}$ for which $P^{(k^*)}_{i^*, j^*}$ is non-singular on $[t_{k^*}, t_{k''-1}]$ (if $k'' > k^*$) or on $[t_{k''}, t_{k^*}]$ (if $k'' < k^*$) and
\[ P_{i^*, j^*}^{(k^*)} \cap U_{i''}^{(k'')} (t_{k''-1})  \not= \emptyset \quad \text{if $k'' > k^*$} \quad \text{or} \quad P_{i^*, j^*}^{(k^*)} \cap U_{i''}^{(k'')} (t_{k''})  \not= \emptyset \quad \text{if $k'' < k^*$}. \]
(Note that we have exchanged the $1$ for $j^*$. So $(k', i', 1)$ is one of these triples.)
We can assume that we have picked $(k'', i'', j^*)$ amongst all these triples of indices, such that $j^* + |k^* - k''|$ is minimal and amongst such triples of indices for which this number is the same, we can assume that $| k^* - k'' |$ is minimal.

Now observe that by maximality of $(k_m, i_m)$ with respect to $\prec$ we must have $(k^*, i^*) \not\prec (k'', i'')$.
So either $j^* = N + 2$ or $j^* \leq N+1$ and there are indices $(k''', i''') \in I$ with $k'''$ strictly between $k^*$ and $k''$ such that the following holds:
The set $P_{i^*, j^* + 1}^{(k^*)}$ is non-singular on the time-interval $[t_{k^*}, t_{k'''-1}]$ (if $k'' > k^*$) or on $[t_{k'''}, t_{k^*}]$ (if $k'' < k^*$) and we have $(P_{i^*, j^*}^{(k^*)} \cup P_{i^*, j^* + 1}^{(k^*)} ) \cap U_{i'''}^{(k''')} (t_{k''' - 1}) \neq \emptyset$ (if $k'' > k^*$) or $(P_{i^*, j^*}^{(k^*)} \cup P_{i^*, j^* + 1}^{(k^*)} ) \cap U_{i'''}^{(k''')} (t_{k'''}) \neq \emptyset$ (if $k'' < k^*$).
But the latter possibility implies that we could replace the triple $(k'', i'', j^*)$ by either $(k''', i''', j^*)$ or $(k''', i''', j^* + 1)$, contradicting its minimal choice.
So $j^* = N + 2$.
However, the triple $(k', i', 1)$ would make $j^* + |k^* - k'|$ smaller than the triple $(k'', i'', N+2)$.
This yields the desired contradiction and shows that $P^{(k^*)}_{i^*, 1}$ is non-singular on the time-interval $[t_{-1}, t_k]$ as well as the fact that $P^{(k^*)}_{i^*, 1} \cap U^{(k')}_{i'} (t_{k'}) = \emptyset$ for all $(k', i') \in I$ with $k' \leq k$.
Moreover, it is clear that the sub-Ricci flow with surgery $V^{(k^*)}_{i^*, 2}$ can be extended to the time-interval $[t_{-1}, t_k]$.

We finally show by induction that $U^{(k_l)}_{i_l} (t_{k_l}) \subset V^{(k^*)}_{i^*, 2} (t_l)$ for all $l = m, \ldots, 1$.
This implies the claim for $l = 1$.
The statement is clear for $l = m$, so assume that $l < m$ and that it holds for $l + 1$.
By Claim 2 we have $V^{(k_l)}_{i_l, j_l + 2} (t_{l+1}) \subsetneq U^{(k_{l+1})}_{i_{l+1}} (t_{k_{l+1}}) \subset V^{(k^*)}_{i^*, 2} (t_{l+1})$.
So $U^{(k_l)}_{i_l} (t_{k_l}) \subset V^{(k_l)}_{i_l, j_l + 2} (t_{l}) \subsetneq V^{(k^*)}_{i^*, 2} (t_{l})$, finishing the induction.
\end{proof}

Assume in the following that $\eta^\circ < \eta_2^*$.
We now apply Claim 3 for $k = N$.
So for every $i \in \{ 1, \ldots, m^{(N)} \}$ for which $h_i^{(N)} < \eta_2^*$, we can find indices $(k^*_i, i^*_i) \in I$ such that the following holds: $P^{(k^*_i)}_{i^*_i, 1}$ is non-singular on the time-interval $[t_{-1}, t_\omega]$ and $P^{(k^*_i)}_{i^*_i, 1} \cap U^{(k')}_{i'} (t_{k'}) = \emptyset$ for all indices $(k', i') \in I$.
Moreover, we can extend $V^{(k^*_i)}_{i^*_i, 2}$ to the time-interval $[t_0, t_\omega]$ and we have $U^{(k)}_i (t_\omega) \subsetneq V^{(k^*_i)}_{i^*_i, 2} (t_\omega)$.
Observe that this implies that $| {\Rm} | < K t^{-1}$ on $P^{(k^*_i)}_{i^*_i, 1}$ for all $t \in [t_{-1}, t_\omega]$ and by Shi's estimates there is a universal $K_1 > K$ such that $ |{\nabla \Rm}| < K_1 t^{-3/2}$ on $P^{(k^*_i)}_{i^*_i, 1}$, but at time-$t$ distance of at least $\sqrt{t}$ away from its boundary for all $t \in [t_0, t_\omega]$.
So as in the proof of Claim 1, there is a function $\varphi'' : (0, \infty) \to (0, \infty)$ with $\varphi''(h) \to 0$ as $h \to 0$ such that $P^{(k^*_i)}_{i^*_i, 1}$ is a $\varphi''( h^{(k^*_i)}_{i^*_i} )$-precise torus structure at scale $\sqrt{t}$ at every time $t \in [t_{-1}, t_\omega]$.
We remark, that $\varphi''$ can be chosen depending only on $L$ and $A$.
Using Claim 3, it is then easy to see that there is a constant $\eta_4^* = \eta_4^*(L,A, \alpha) > 0$ with $\eta_4^* < \eta_2^*$ such that assuming $\eta^\circ < \eta_4^*$, the following holds: 
For all $i \in \{ 1, \ldots, m^{(N)} \}$ with $h_i^{(N)} < \eta_4^*$, the set $P^{(k^*_i)}_{i^*_i, 1}$ is a $\min \{ \alpha, \frac1{10} \}$-precise torus structure at scale $\sqrt{t}$ and at every time $t \in [t_0, t_\omega]$ and at time $t_0$, the torus structure $P^{(k^*_i)}_{i^*_i, 1}$ is even $(\td{L}_0 (\min \{ e^{-LK} \alpha, \td{\alpha}_0 (K_1) \}, A) + 10)^{-1}$-precise.
Here $\td{L}_0$ and $\td{\alpha}_0$ are the constants from Lemma \ref{Lem:smallloopintorusstruc}.

We finally choose $\eta^\circ = \eta^\circ (L, A, \alpha) > 0$ such that $\eta^\circ < \eta_4^*$.
For each $i \in \{ 1, \ldots, m^{(N)} \}$ with $h_i^{(N)} < \eta_4^*$ we set $U'_i = V^{(k^*_i)}_{i^*_i, 2}$ and $P'_i = P^{(k^*_i)}_{i^*_i, 1}$.
We now pick a subcollection of the $U'_1, \ldots, U'_{m^{(N)}}$ which are pairwise disjoint at time $t_\omega$.
If $U'_{i_1} (t_\omega) \cap U'_{i_2} (t_\omega) \neq \emptyset$, then by the fact that $P'_{i_1}$ and $P'_{i_2}$ are $\frac1{10}$-precise and Lemma \ref{Lem:coverMbysth} we have $U'_{i_1} (t_\omega) \subset P'_{i_2} \cup U'_{i_2} (t_\omega)$ or $U'_{i_2} (t_\omega) \subset P'_{i_1} \cup U'_{i_1} (t_\omega)$.
In the first case we remove the index $i_1$ from the list and in the second case, we remove $i_2$ (if both cases hold we remove either $i_1$ or $i_2$).
We can repeat this process until we arrive at a collection $U_1, \ldots, U_m$ whose time-$t_\omega$ slices are pairwise disjoint.
This implies that the time-$t$ slices are pairwise disjoint as well for any $t \in [t_0, t_\omega]$.
Let $P_1, \ldots, P_m$ be the corresponding collection of torus structures.
Observe that at each step of this process, the set $\bigcup_i U'_i (t_\omega) \setminus \bigcup_i P'_i$ does not decrease.
Thus $U_1(t_\omega) \cup \ldots \cup U_m (t_\omega) \supset \bigcup_i U'_i (t_\omega) \setminus \bigcup_i  P'_i$.
We conclude that every point of $\MM(t_\omega) \setminus ( U_1(t_\omega) \cup \ldots \cup U_m (t_\omega))$, which does not belong to $\MM(t_\omega) \setminus ( U^{(N)}_1 (t_\omega) \cup \ldots \cup U^{(m)}_{m^{(N)}} (t_\omega))$, is either contained in some $U^{(N)}_j(t_\omega)$ for which $h_j^{(N)} \geq \eta_4^*$ or it is contained in some $U'_j (t_\omega)$ in which case it must belong to $\bigcup_i P'_i$.
So assertion (a) holds for $K_4 = \max \{ K, K'(\eta^*_4) \}$ (see Proposition \ref{Prop:firstcurvboundstep3}(c)).
The second part of assertion (b) holds by the choice of $\eta^*_4$.

It remains to construct the loops $\gamma_i \subset P_i$ and the disks $D_i$ such that assertions (b)--(e) hold.
Fix $i = 1, \ldots, m$.
By the choice of $\eta_4^*$ we find a torus structure $P^*_i \subset P_i$ which is $(\td{L}_0 ( \min \{ e^{-LK} \alpha, \td{\alpha}_0 (K_1) \}, A ) )^{-1}$-precise at scale $\sqrt{t_0}$ and time $t_0$ and which has time-$t_0$ distance of at least $\sqrt{t_0}$ from the boundary of $P_i$ in such a way that the pair $(P_i, P^*_i)$ is diffeomorphic to $(T^2 \times [-2,2], T^2 \times [-1,1])$.
So $|{\Rm}| < K_1 t_0^{-1}$ and $|{\nabla \Rm}| < K_1 t_0^{-3/2}$ on $P^*_i$.
We can hence apply Lemma \ref{Lem:smallloopintorusstruc} with $\alpha \leftarrow \min \{ e^{-LK} \alpha, \td{\alpha}_0 (K_1) \}$, $A \leftarrow A$, $K \leftarrow K_1$, $M = \MM(t_0)$, $M_{\textnormal{hyp}} \leftarrow \MM_{\thick} (t_0)$, $M_{\textnormal{Seif}} \leftarrow \MM_{\thin} (t_0)$.
This gives us a loop $\gamma_i \subset P^*_i \subset P_i$ of length $\ell_{t_0} (\gamma_i) < e^{-LK} \alpha$ which is non-contractible in $P^*_i$, which spans a disk $D_i \subset \MM(t_0)$ of time-$t_0$ area $\area_{t_0} D_i < A t_0$ and whose geodesic curvatures at time $t_0$ are bounded by $\td\Gamma(K_1)$.
So assertions (d) and (e) hold and the first part of assertion (b) follows by a distance distortion estimate.
For assertion (d) observe that $|{\Rm}| < K t^{-1}$ and $|{\nabla \Rm}| < K_1 t^{-3/2}$ on $\gamma_i$ for all $t \in [t_0, t_\omega]$.
\end{proof}

\subsection{The final argument}
We will finally show that the general picture presented in Proposition \ref{Prop:structontimeinterval} cannot persist for a long time, i.e. that for an appropriate choice of the parameters $L$, $A$ and $\alpha$ we must have $m = 0$.
This will imply a curvature bound on the final time-slice $\MM (t_\omega)$ and hence establish Theorem \ref{Thm:MainTheorem}.

As a preparation, we first prove that after some large time, all time-slices are irreducible and all surgeries are trivial (see also \cite[Proposition 18.9]{MTRicciflow}).
\begin{Proposition} \label{Prop:irreducibleafterfinitetime}
Let $\MM$ be a Ricci flow with surgery and precise cutoff whose time-slices are closed manifolds, defined on the time-interval $[T, \infty)$ $(T\geq0)$.
Then there is some $T_1 \in [T, \infty)$ such that all surgeries on $[T_1, \infty)$ are trivial and we can find a sub-Ricci flow with surgery $\MM' \subset \MM$ on the time-interval $[T, \infty)$ (whose time-slices have no boundary) such that:
For all $t \in [T, \infty)$ the complement $\MM(t) \setminus \MM'(t)$ consists of a disjoint union of spheres and for all $t \in [T_1, \infty)$ all components of $\MM'(t)$ are irreducible and not diffeomorphic to spherical space forms.

Moreover, if there is a time $T^* \geq T_1$ such that if $\MM'$ is non-singular on $[T^*, \infty)$, then there is also a time $T^{**} \geq T^*$ such that $\MM$ is non-singular on $[T^{**}, \infty)$ and $\MM'(t) = \MM(t)$ for all $t \in [T^{**}, \infty)$.
\end{Proposition}
\begin{proof}
Let $\MM = (( T^i )_i, (M^i \times I^i), (g^i_t)_i, (\Omega_i)_i, (U^i_\pm)_i)$ (see Definition \ref{Def:RFsurg}).
By Definition \ref{Def:precisecutoff}, for any surgery time $T^i$, the topological manifold $M^i$ can be obtained from $M^{i+1}$ by possibly adding spherical space forms or copies of $S^1 \times S^2$ to the components of $\MM(t_2)$ and then performing connected sums between some of those components.
So every component of $M^{i+1}$, which is not diffeomorphic to a sphere, forms the building block of a component of $M^i$ which is also not diffeomorphic to a sphere.
It is then clear that we can choose $\MM' \subset \MM$ such that for every $t \in [T, \infty)$ the set $\MM'(t)$ is the union of all components of $\MM(t)$ which are not diffeomorphic to spheres.

By the existence and uniqueness of the prime decomposition (see e.g. \cite[Theorem 1.5]{Hat}) and the conclusion above, there are only finitely many times when the topology of $\MM'(t)$ can change.
By finite-time extinction of spherical components (see \cite{PerelmanIII}, \cite{ColdingMinicozziextinction}), we conclude that $\MM'(t)$ cannot have components which are diffeomorphic to spherical space forms for $t \in [T_1, \infty)$.
This implies that there is some $T_1 \in [T, \infty)$ such that the time-slices $\MM'(t)$ are all diffeomorphic to each other for all $t \in [T_1, \infty)$, i.e. that all surgeries on the time-interval $[T_1, \infty)$ are trivial on $\MM'$ and hence also on $\MM$.

Assome that $\MM' (T_1)$ was not irreducible.
Then by Proposition \ref{Prop:pi2irred}, we have $\pi_2(N) \not= 0$ for some component $N$ of $\MM'(T_1)$.
We can thuse use Lemma \ref{Lem:evolsphere} to obtain a contradiction.

The last part of the proposition follows again from finite-time extinction.
\end{proof}

\begin{proof}[Proof of Theorem \ref{Thm:MainTheorem}]
Let the function $\delta(t)$ be the one given in Proposition \ref{Prop:structontimeinterval}.

Consider the constant $T_0 < \infty$ and the function $w : [T_0, \infty) \to (0, \infty)$ from Proposition \ref{Prop:thickthindec} as well as the constant $T_1 < \infty$ and the sub-Ricci flow with surgery from Proposition \ref{Prop:irreducibleafterfinitetime} and set $T_2 = \max \{ T_0, T_1 \}$.
By the last statement of Proposition \ref{Prop:irreducibleafterfinitetime}, we can assume without loss of generality that $\MM' = \MM$.
So all surgeries on the time-interval $[T_2, \infty)$ are trivial and for all times $t \in [T_2, \infty)$ the components of $\MM(t)$ are irreducible and not diffeomorphic to spherical space forms and there is a thick-thin decomposition $\MM (t) = \MM_{\thick}(t) \cup \MM_{\thin}(t)$ as described in Proposition \ref{Prop:thickthindec}.

Denote by $\MM_{\textnormal{torus}} (T_2)$ the union of all components of $\MM_{\textnormal{thin}} (T_2)$ which are diffeomorphic to $T^2 \times I$ and set $\MM_{\textnormal{hyp}} (T_2) = \MM_{\thick}(T_2) \cup \MM_{\textnormal{torus}} (T_2)$ and $\MM_{\textnormal{Seif}} (T_2) = \MM_{\thin}(T_2) \setminus \MM_{\textnormal{torus}} (T_2)$.
We will now show that the decomposition $\MM (T_2) = \MM_{\textnormal{hyp}} (T_2) \cup \MM_{\textnormal{Seif}} (T_2)$ can be refined to a minimal geometric decomposition.
By the results of \cite{MorganTian} or \cite{KLcollapse}, which led to the resolution of the Geometrization Conjecture (essentially their statement is Proposition \ref{Prop:MorganTianMain} plus a topological discussion), we know that $\MM_{\textnormal{Seif}} (T_2)$ is a graph manifold (see Definition \ref{Def:geomdec} and the subsequent discussion).
So there are pairwise disjoint, embedded, incompressible $2$-tori $T^*_1, \ldots, T^*_k \subset \MM_{\textnormal{Seif}} (T_2)$ which induce a minimal Seifert decomposition of $\MM_{\textnormal{Seif}} (T_2)$ (observe that these tori are even incompressible in $\MM (T_2)$).
Let moreover $T^*_{k+1}, \ldots, T^*_{k'} \subset \MM_{\textnormal{hyp}} (T_2)$ be pairwise disjoint, embedded, incompressible $2$-tori, which decompose $\MM_{\textnormal{hyp}} (T_2)$ into its hyperbolic components.
Then $T^*_1, \ldots, T^*_{k'}$ together with the boundary tori in $\partial \MM_{\textnormal{hyp}} (T_2)$ provide a geometric decomposition of $\MM (T_2)$.
After removing some of these tori, we obtain a \emph{minimal} geometric decomposition.
If we had removed a torus of $\partial \MM_{\textnormal{hyp}} (T_2)$ in this process, then we could find a component $\CC$ of the new and minimal decomposition which contains a Seifert component $\CC_1$ of $\MM_{\textnormal{Seif}} (T_2) \setminus (T^*_1 \cup \ldots \cup T^*_k)$ and a hyperbolic component $\CC_2$ of $\MM_{\textnormal{hyp}} (T_2) \setminus (T^*_{k+1} \cup \ldots \cup T^*_{k'})$.
Then $\CC$ cannot be Seifert.
So it is hyperbolic and hence $\CC_1 \approx T^2 \times I$ which contradicts the minimal choice of $T^*_1, \ldots, T^*_k$.

Now observe that $\MM(T_2)$ satisfies property $\TT'_2$ (see Definition \ref{Def:TT2}).
So there is a filling surface $f : \Sigma \to \MM (T_2)$ for the pair $(\MM_{\textnormal{hyp}} (T_2), \MM_{\textnormal{Seif}} (T_2))$.
For every component $\CC$ of $\MM_{\textnormal{torus}} (T_2)$ we consider a pair of immersed annuli which connect the boundary components of $\CC$ and whose central loops generate the fundamental group of $\CC$.
Observe that by an intersection number argument, every non-contractible loop in $\CC$ intersects every homotopic deformation (relative boundary and inside $\MM(T_2)$) of at least one of these annuli.
Hence, the union $f' : \Sigma' \to \MM(T_2)$ of $f$ with all these pairs of immersed annuli is filling for the pair $(\MM_{\thick} (T_2), \MM_{\thin} (T_2))$.

We now apply Lemma \ref{Lem:evolminsurfgeneral} to construct a homotopic representative of $f'$ of bounded area.
To do this, we choose almost geodesic representatives $\gamma_{i, T_2}$ inside $\partial \MM_{\thick} \MM (T_2)$ of the boundary circles of $f' |_\gamma$, $\gamma \subset \partial \Sigma'$ which are homotopic to $f' |_\gamma$ inside $\partial\MM_{\thick}(T_2)$.
By Proposition \ref{Prop:thickthindec}, it is possible to move these circles by isotopies and obtain families of curves $\gamma_{i, t} \subset \partial\MM_{\thick}(t)$, parameterized by $t \in [T_2, \infty)$, which satisfy the properties of Lemma \ref{Lem:evolminsurfgeneral} for certain constants $\Gamma, a, b$.
So by Lemma \ref{Lem:evolminsurfgeneral}(b), we obtain filling surfaces $f'_t : \Sigma' \to \MM (t)$ for the pair $(\MM_{\thick} (t), \MM_{\thin} (t))$ and all $t \in [T_2, \infty)$ such that $\limsup_{t \to \infty} t^{-1} \area_t f'_t \leq 4 ( - 2 \pi \chi(\Sigma') + m a (\Gamma + b))$.
Hence, there is a constant $A < \infty$ such that $t^{-1} \area_t f'_t < A$ for all $t \in [T_2, \infty)$.

Next, we set
\[ L = \Big( 1 + \frac{A}{4 \pi} \Big)^4 \]
and consider the constant $\Gamma_4 = \Gamma_4 (L, A)$ from Proposition \ref{Prop:structontimeinterval}.
Set
\[ \alpha = \frac{\pi}{\Gamma_4} \]
and choose $T_4 = T_4 (L, A, \alpha)$ and $w_4 = w_4(L, A, \alpha)$ according to this Proposition.
Choose now $T^* > \max \{ 4 T_2, T_4 \}$ such that $w(t) < w_4$ for all $t \in [\frac14 T^*, \infty)$.
Hence, we can apply Proposition \ref{Prop:structontimeinterval} with the parameters $L, A, \alpha$ for any $t_\omega > L T^*$ and $t_0 = L^{-1} t_\omega$.
We obtain sub-Ricci flows with surgery $U_1, \ldots U_m \subset \MM$ and disks $D_1, \ldots, D_m \subset \MM(t_0)$ with $\area_{t_0} D_i < A t_0$ such that the geodesic curvature of their boundary loops are bounded by $\Gamma_4 t^{-1}$ and their lengths are bounded by $\alpha \sqrt{t}$ for all $t \in [t_0, t_\omega]$.
Assume first that $m \geq 1$.
Since $\alpha \Gamma_4 = \pi < 2 \pi$, we obtain a contradiction by Lemma \ref{Lem:evolminsurfgeneral}(a):
\[ t_\omega < \Big( 1 + \frac{A}{4 (2 \pi - \alpha \Gamma_4 )} \Big)^4 t_0 = L t_0 = t_\omega. \]
So $m = 0$, and thus we have $|{\Rm_{t_\omega}}| < K_4 t_\omega^{-1}$ on $\MM(t_\omega)$.

This shows that $|{\Rm_t}| < K_4 t^{-1}$ for all $t \geq L T^*$.
So in particular, the Ricci flow with surgery $\MM$ does not develop any singularities past time $L T^*$ and hence there are no singularities on $[L T^*, \infty)$.
This finishes the proof of the Theorem.
\end{proof}


\begin{thebibliography}{xxxxx}
\bibitem[Asa]{Asa}{K. Asano, ``Homeomorphisms of Prism Manifolds'', Yokohama Mathematical Journal 26, no. 1 (1978): 19-25.}
\bibitem[Bam1]{Bamler-diploma}{R. Bamler, ``Ricci flow with surgery'', diploma thesis, Ludwig-Maximilians-Universit\"at Munich (2007)}
\bibitem[Bam2]{Bamler-longtime-I}{R. Bamler, ``Long-time analysis of 3 dimensional Ricci flow I'', arXiv:{\linebreak[1]}math/{\linebreak[1]}1112.5125v1, (December 21, 2011), http://arxiv.org/abs/1112.5125v1}
\bibitem[BBBMP1]{BBBMP}{L. Bessi\`eres, G. Besson, M. Boileau, S. Maillot, J. Porti, ``Geometrisation of 3-Manifolds'', Zuerich, Switzerland: European Mathematical Society Publishing House, (2010), http://www.ems-ph.org/books/book.php?proj\_nr=122}
\bibitem[BBBMP2]{BBMP2}{L. Bessi\`eres, G. Besson, M. Boileau, S. Maillot, J. Porti, ``Collapsing irreducible 3-manifolds with nontrivial fundamental group'', Invent. Math., 179(2):435-460, (2010)}
\bibitem[BBI]{BBI}{D. Burago, Y. Burago, S. Ivanov, ``A course in metric geometry'', Vol. 33. Graduate Studies in Mathematics. Providence, RI: American Mathematical Society, (2001)}
\bibitem[BGP]{BGP}{Y. Burago, M. Gromov, G. Perelman. ``AD Alexandrov spaces with curvature bounded below'', Russian Mathematical Surveys 47 (1992): 1-58}
\bibitem[CFG]{CFG}{J. Cheeger, K. Fukaya, M. Gromov, ``Nilpotent Structures and Invariant Metrics on Collapsed Manifolds'', Journal of the American Mathematical Society 5, no. 2 (1992): 327-372.}
\bibitem[CG]{CG}{J. Cao, J. Ge, ``A simple proof of Perelman's collapsing theorem for 3-manifolds'', arXiv:math/1003.2215v3  (2010, March 10), http://arxiv.org/abs/1003.2215v3}
\bibitem[CM]{ColdingMinicozziextinction}{T. H Colding and William P Minicozzi II, ``Width and finite extinction time of Ricci flow'', arXiv:0707.0108 (July 1, 2007), http://arxiv.org/abs/0707.0108}
\bibitem[Fae]{Faessler}{D. Faessler, ``On the topology of locally volume collapsed Riemannian 3-orbifolds'', arXiv:1101.3644 (January 19, 2011), http://arxiv.org/abs/1101.3644}
\bibitem[FY]{FY}{K. Fukaya, T. Yamaguchi, ``The fundamental groups of almost non-negatively curved manifolds'', Annals of Mathematics. Second Series 136, no. 2 (1992): 253-333}
\bibitem[GS]{Grove-Shiohama-1977}{K. Grove, K. Shiohama, ``A Generalized Sphere Theorem'', Ann. of Math. (1977) 106, no. 2: 201-211}
\bibitem[Gul]{Gul}{R. D. Gulliver, ``Regularity of minimizing surfaces of prescribed mean curvature'', Annals of Mathematics. Second Series 97 (1973): 275-305}
\bibitem[Ham]{Ham}{R. Hamilton, ``Non-singular solutions of the Ricci flow on three-manifolds'', Communications in Analysis and Geometry 7, no. 4 (1999): 695-729}
\bibitem[Has]{Hass}{J. Hass, ``Minimal Surfaces in Manifolds with $S^1$ Actions and the Simple Loop Conjecture for Seifert Fibered Spaces'', Proceedings of the American Mathematical Society 99, no. 2: 383-388.}
\bibitem[Hat]{Hat}{A. Hatcher, ``Notes on basic 3-manifold topology'', in preparation, available on the web at http://www.math.cornell.edu/$\sim$hatcher}
\bibitem[HH]{HH}{E. Heinz, S. Hildebrandt, ``Some remarks on minimal surfaces in Riemannian manifolds'', Communications on Pure and Applied Mathematics 23 (1970): 371-377}
\bibitem[KL1]{KLnotes}{B. Kleiner, J. Lott, ``Notes on Perelman's papers'', arXiv:math/0605667 (May 25, 2006), http://arxiv.org/abs/math/0605667}
\bibitem[KL2]{KLcollapse}{B. Kleiner, J. Lott, ``Locally Collapsed 3-Manifolds'', arXiv:math/1005.5106v2  (May 27, 2010), http://arxiv.org/abs/1005.5106v2}
\bibitem[Lot1]{LottTypeIII}{J. Lott, ``On the long-time behavior of type-III Ricci flow solutions'', Mathematische Annalen 339, no. 3 (2007): 627-666}
\bibitem[Lot2]{LottDimRed}{J. Lott, ``Dimensional reduction and the long-time behavior of Ricci flow'', Commentarii Mathematici Helvetici. A Journal of the Swiss Mathematical Society 85, no. 3 (2010): 485-534}
\bibitem[LS]{LottSesum}{J. Lott, N. Sesum, ``Ricci flow on three-dimensional manifolds with symmetry'', arXiv.org, February 21, 2011}
\bibitem[Mey]{Meyer-1989}{W. Meyer, ``Toponogov's Theorem and Applications'', Lecture Notes, Trieste (1989)}
\bibitem[Mor]{Mor}{C. B. Morrey, ``The problem of Plateau on a Riemannian manifold'', Annals of Mathematics. Second Series 49 (1948): 807Ð851}
\bibitem[MT1]{MTRicciflow}{J. W. Morgan, G. Tian, ``Ricci Flow and the Poincare Conjecture'', American Mathematical Society, (2007)}
\bibitem[MT2]{MorganTian}{J. Morgan, G. Tian, ``Completion of the Proof of the Geometrization Conjecture'', arXiv:0809.4040 (September 23, 2008), http://arxiv.org/abs/0809.4040}
\bibitem[MY]{MY}{W. Meeks, S.-T. Yau. ``Topology of three-dimensional manifolds and the embedding problems in minimal surface theory'', Annals of Mathematics. Second Series 112, no. 3 (1980): 441-484}
\bibitem[Neu]{Neumann}{W. D. Neumann, ``Immersed and Virtually Embedded $\pi_1$-Injective Surfaces in Graph Manifolds'', Algebraic \& Geometric Topology 1: 411-426 (electronic).}
\bibitem[Per1]{PerelmanI}{G. Perelman, ``The entropy formula for the Ricci flow and its geometric applications'', math/0211159 (November 2002): 39}
\bibitem[Per2]{PerelmanII}{G. Perelman, ``Ricci flow with surgery on three-manifolds'', arXiv:math/0303109 (March 10, 2003), http://arxiv.org/abs/math/0303109}
\bibitem[Per3]{PerelmanIII}{G. Perelman, ``Finite extinction time for the solutions to the Ricci flow on certain three-manifolds'', arXiv:math/0307245 (July 17, 2003), http://arxiv.org/abs/math/0307245}
\bibitem[Sch]{Sch}{R. Schoen, ``Estimates for Stable Minimal Surfaces in Three-Dimensional Manifolds'', in Seminar on Minimal Submanifolds, 103:111-126. Princeton, NJ: Princeton Univ. Press, 1983.}
\bibitem[SU1]{SU81}{J. Sacks, K. Uhlenbeck, ``The Existence of Minimal Immersions of 2-Spheres'', Annals of Mathematics, vol. 113, no. 1, pp. 1-24 (1981)}
\bibitem[SU2]{Sacks-Uhlenbeck-1982}{J. Sacks, K. Uhlenbeck, ``Minimal Immersions of Closed Riemann Surfaces'', Transactions of the American Mathematical Society 271, no. 2 (1982): 639-652.}
\bibitem[ScY]{Schoen-Yau-1979}{R. Schoen, S.-T. Yau, ``Existence of Incompressible Minimal Surfaces and the Topology of Three Dimensional Manifolds with Non-Negative Scalar Curvature'', The Annals of Mathematics 110, no. 1 (July 1979): 127.}
\bibitem[ShY]{ShioyaYamaguchi}{T. Shioya, T. Yamaguchi, ``Volume collapsed three-manifolds with a lower curvature bound'', Mathematische Annalen 333, no. 1 (2005): 131-155}
\bibitem[WY]{WangYu}{S. Wang, F. Yu, ``Graph Manifolds with Non-Empty Boundary Are Covered by Surface Bundles'', Mathematical Proceedings of the Cambridge Philosophical Society 122, no. 3: 447-455.}
\end{thebibliography}
\end{document}